\documentclass[11pt, reqno]{amsart}
\usepackage{amsfonts}
\usepackage{bbm}
\usepackage{amscd,amsfonts}
\usepackage{amssymb, eucal, amsfonts, amsmath, xypic, latexsym}
\usepackage{pifont}
\usepackage{mathrsfs,color}
\usepackage{amsthm,indentfirst,bm,fancyhdr,dsfont}
\usepackage{graphicx}
\usepackage[all]{xy}
\usepackage[CJKbookmarks=true]{hyperref}

\newtheorem{theorem}{Theorem}[section]
\newtheorem{lemma}[theorem]{Lemma}
\newtheorem{prop}[theorem]{Proposition}
\newtheorem{corollary}[theorem]{Corollary}
\theoremstyle{definition}

\newtheorem{defn}[theorem]{Definition}

\newtheorem{rem}[theorem]{Remark}
\newtheorem{conjecture}[theorem]{Conjecture}

\newtheorem{conventions}[theorem]{Conventions}

\numberwithin{equation}{section}

\def\ggg{\mathfrak{g}}

\def\gl{\mathfrak{gl}}

\def\ggg{\mathfrak{g}}
\def\mmm{\mathfrak{m}}
\def\ppp{\mathfrak{p}}

\def\hhh{\mathfrak{h}}
\def\aaa{\mathfrak{a}}

\def\bbc{\mathbb{C}}

\def\bbz{\mathbb{Z}}
\def\bbq{\mathbb{Q}}

\def\bo{{\bar 1}}
\def\bz{{\bar 0}}

\def\sfr{\textsf{r}}

\def\ad{\text{ad}}

\def\communication#1{\def\@comm{#1}}

{\vskip-\lastskip\medskip
  \noindent
  {\em #1.}\enspace
  }%
{\qed\par\medskip
  }

\makeatletter
\@namedef{subjclassname@2020}{\textup{2020} Mathematics Subject Classification}
\makeatother

\begin{document}
\title[Highest weight theory for finite minimal $W$-superalgebras]{Highest weight theory for minimal finite $W$-superalgebras and related Whittaker categories}
\author{Yang Zeng and Bin Shu}
\thanks{Communicated by T. Arakawa. Received September 2, 2022. Revised February 14, 2023.}
\thanks{This work is  supported partially by NSFC (12071136, 12271345, 12461005), Science and Technology Commission of Shanghai Municipality (No. 22DZ2229014).}
%\thanks{\nonumber{{\it{Mathematics Subject Classification}} (2020):
%Primary 17B10, Secondary 17B35.
% {\it{The Key words}}: basic classical Lie superalgebras, minimal nilpotent elements, finite $W$-(super)algebras; highest weight modules; Whittaker categories; category $\mathcal{O}$ for minimal finite $W$-superalgebras.
% This work is  supported partially by NSFC (Nos. 12071136; 11701284), Shanghai Key Laboratory of PMMP (No. 13dz2260400). }}
\address{School of Mathematics, Nanjing Audit University, Nanjing, Jiangsu Province 211815, China}
\email{zengyang@nau.edu.cn}
\address{School of Mathematical Sciences, East China Normal University, Shanghai 200241, China}
\email{bshu@math.ecnu.edu.cn}
\begin{abstract}
Let $\ggg=\ggg_{\bar0}+\ggg_{\bar1}$ be a basic classical Lie superalgebra over $\mathbb{C}$, and $e=e_{\theta}\in\ggg_{\bar0}$ with $-\theta$ being a minimal root of $\ggg$. Set $U(\ggg,e)$ to be the minimal finite $W$-superalgebras associated with the pair $(\ggg,e)$. In this paper we study the highest weight theory for $U(\ggg,e)$, introduce the Verma modules and give a complete isomorphism classification of finite-dimensional  irreducible modules, via the parameter set consisting of pairs of weights and levels. Those Verma modules can be further described via parabolic induction from Whittaker modules for $\mathfrak{osp}(1|2)$ or $\mathfrak{sl}(2)$ respectively, depending on the detecting parity of $\sfr:=\dim{\ggg}(-1)_{\bar1}$. %and the composition multiplicities of the latter modules can be calculated by using the Kazhdan-Lusztig algorithm.
We then introduce and investigate the BGG category $\mathcal{O}$ for $U(\ggg,e)$, establishing highest weight theory, as a counterpart of the works for finite $W$-algebras by Brundan-Goodwin-Kleshchev \cite{BGK} and Losev \cite{Los}, respectively.

In comparison with the non-super case, the significant difference here lies in the situation when  $\sfr$
is odd, which is a completely new phenomenon. The difficulty and complicated computation arise from there.
\end{abstract}
\maketitle
\setcounter{tocdepth}{1}
\tableofcontents
\setcounter{section}{-1}
\section{Introduction}\label{Introduction}
This work is a sequel to \cite{ZS5}. In \cite{ZS5} we obtained the PBW theorem of minimal refined $W$-superalgebras over $\mathbb{C}$, along with their generators and the commutator formulas. %In virtue of these, we further showed that minimal refined $W$-superalgebras afford one-dimensional representations over $\bbc$. So the lower-bounds of dimensions for modular representations of basic Lie superalgebras (i.e. the super Kac-Weisfeiler property in \cite{WZ}) are attainable.
In this paper, we continue to study the representations of minimal finite and refined $W$-superalgebras.  Our purpose is to develop the highest weight  theory for minimal finite and refined $W$-superalgebras, and determine all simple objects in the corresponding BGG category $\mathcal{O}$.

\subsection{}
A finite $W$-algebra (resp.\ superalgebra) $U(\ggg,e)$ is a certain associative algebra (resp.\ superalgebra) associated with a complex semi-simple Lie algebra (resp.\ basic classical Lie superalgebra) $\ggg$ and a nilpotent element $e\in \ggg$ (resp.\ $e\in\ggg_\bz$ for $\ggg=\ggg_\bz\oplus\ggg_\bo$).
In the last two decades,  finite $W$-algebras  and finite $W$-superalgebras have been developed rapidly since Premet studied finite $W$-algebras in full generality in \cite{P2} (see \cite{P3,P4,P7}, etc.).

In particular, Brundan-Goodwin-Kleshchev in \cite{BGK} considered the highest weight theory for finite $W$-algebras and showed that the irreducible highest weight modules can be parameterized by some unknown set $\mathcal{L}$. In virtue of \cite{BGK}, Losev in \cite{Los} further studied the so-called BGG category $\mathcal{O}$ for finite $W$-algebras first introduced by Brundan-Goodwin-Kleshchev, and established an equivalence of these categories and  the generalized Whittaker categories studied by Mili\v{c}i\'{c}-Soergel \cite{Mi} when the nilpotent element $e$ is principal.
\subsection{}\label{1.3}
Let $G$ be a connected semi-simple algebraic group over $\bbc$ with $\ggg=\text{Lie}(G)$. When the element $e=e_{\theta}$ with $-\theta$ being a minimal root of $\ggg$, the corresponding finite $W$-algebra $U(\ggg,e)$ is called minimal. In the case of basic classical Lie superalgebras, one can also take (even) nilpotent elements  associated with  minimal roots,    and then obtain  minimal finite $W$-superalgebras. The study of minimal (affine) $W$-superalgebras can be traced back from \cite{KRW} (see also \cite[\S5]{KW} for more details).

In the present paper, we will study the highest weight theory for minimal finite $W$-superalgebras and for refined minimal $W$-superalgebras. %in the same spirit of \cite{P3} and \cite{BGK}.
This work can be regarded as a counterpart of Premet's work on  the minimal finite $W$-algebras in \cite{P3}, and also a counterpart of Brundan-Goodwin-Kleshchev's work \cite{BGK} on highest weight theory for finite $W$-algebras. We will follow Premet's strategy in \cite[\S7]{P3}, and also the methods applied by Brundan-Goodwin-Kleshchev in \cite[\S4]{BGK},   with a lot of modifications.\footnote{Depending on the parity of the dimension for a particular subspace of $\ggg$,
one can observe critical distinctions for the construction of Verma modules for minimal finite and refined $W$-superalgebras, and also for the formulation of the corresponding BGG category $\mathcal{O}$,  which never appear in the non-super case.} Let us briefly introduce what we will do.
 %Thanks to the different structure of minimal finite and refined $W$-superalgebras,
%For the finite-dimensional irreducible modules of finite $W$-superalgebras associated with arbitrary even nilpotent elements, Shu-Xiao initiated to study them in

\subsection{} For a given complex basic classical Lie superalgebra  ${\ggg}={\ggg}_{\bar0}+{\ggg}_{\bar1}$ and a nilpotent element $e\in\ggg_\bz$, let $(\cdot,\cdot)$ be a non-degenerate even supersymmetric invariant bilinear form on ${\ggg}$. Define $\chi\in{\ggg}^{*}$ by letting $\chi(x)=(e,x)$ for all $x\in{\ggg}$.
Fix an $\mathfrak{sl}(2)$-triple $(e,h,f)$ in ${\ggg}_{\bar0}$, and denote by ${\ggg}^e:=\text{Ker}(\ad\,e)$ in ${\ggg}$. The linear
operator ad\,$h$ defines a ${\bbz}$-grading ${\ggg}=\bigoplus_{i\in{\bbz}}{\ggg}(i)$ with $e\in{\ggg}(2)_{\bar{0}}$ and $f\in{\ggg}(-2)_{\bar{0}}$, and we have ${\ggg}^e=\bigoplus_{i\geqslant0}{\ggg}^e(i)$ by the $\mathfrak{sl}(2)$-representation theory. Up to a scalar, we further assume that $(e,f)=1$.

Choose $\mathbb{Z}_2$-homogenenous bases $\{u_1,\cdots,u_{s}\}$ of ${\ggg}(-1)_{\bar0}$ and $\{v_1,\cdots,v_\sfr\}$ of ${\ggg}(-1)_{\bar1}$ such that $\chi([u_i, u_j])=i^*\delta_{i+j,s+1}$ for $1\leqslant i,j\leqslant s$, where $i^*:=\left\{\begin{array}{ll}-1&\text{if}~1\leqslant i\leqslant \frac{s}{2};\\ 1&\text{if}~\frac{s}{2}+1\leqslant i\leqslant s\end{array}\right.$, and $\chi([v_i,v_j])=\delta_{i+j,\sfr+1}$ for $1\leqslant i,j\leqslant \sfr$.
Write ${\ggg}(-1)^{\prime}_{\bar0}$ for the ${\bbc}$-span of $u_{{s\over 2}+1},\cdots,u_{s}$ and
${\ggg}(-1)^{\prime}_{\bar1}$ the ${\bbc}$-span of $v_{\frac{\sfr}{2}+1},\cdots,v_\sfr$ (resp.\ $v_{\frac{\sfr+3}{2}},\cdots,v_\sfr$) when $\sfr$ is even (resp.\ odd).
Denote by ${\ggg}(-1)^{\prime}:={\ggg}(-1)^{\prime}_{\bar0}\oplus{\ggg}(-1)^{\prime}_{\bar1}$, and set the ``$\chi$-admissible subalgebra" and the ``extended $\chi$-admissible subalgebra" of $\ggg$ to be
\begin{align*}
\mathfrak{m}:=\bigoplus_{i\leqslant -2}{\ggg}(i)\oplus{\ggg}(-1)^{\prime},\qquad
\mathfrak{m}^{\prime}:=\left\{\begin{array}{ll}\mathfrak{m}&\text{if}~\sfr~\text{is even;}\\
\mathfrak{m}\oplus {\bbc}v_{\frac{\sfr+1}{2}}&\text{if}~\sfr~\text{is odd.}\end{array}\right.
\end{align*} respectively.
Then the generalized Gelfand-Graev ${\ggg}$-module associated with $\chi$ is defined by
 $Q_\chi:=U({\ggg})\otimes_{U(\mathfrak{m})}{\bbc}_\chi,$
where ${\bbc}_\chi={\bbc}1_\chi$ is a one-dimensional $\mathfrak{m}$-module such that $x.1_\chi=\chi(x)1_\chi$ for all $x\in\mathfrak{m}$.

A finite $W$-superalgebra $U(\ggg,e):=(\text{End}_{\ggg}Q_{\chi})^{\text{op}}$ is by definition isomorphic to $Q_{\chi}^{\text{ad}\,{\mmm}}$,
the invariants of $Q_{\chi}$ under the adjoint action of ${\mmm}$.
\iffalse
A refined $W$-superalgebra $W_\chi'$ is defined by $Q_{\chi}^{\text{ad}\,\mmm'}$, the invariants of $Q_{\chi}$ under the adjoint action of $\mmm'$.
%A PBW theorem of $U(\ggg,e)$ (see \cite[Theorem 0.1]{ZS2}) shows that the structure of $U(\ggg,e)$ is crucially dependent on the parity of a discriminant number $\mathsf{r}$ (the meaning of this notation can be seen in the above of (\ref{extend admiss alg})). The parity of $\mathsf{r}$ is therefore called the judging parity.
It is obvious that $W_\chi'=U(\ggg,e)$ when $\mathsf{r}$ is even, and $W_\chi'$ is a proper subalgebra of $U(\ggg,e)$ when $\mathsf{r}$ is odd.
\fi
Recall in \cite[Theorem 4.5]{ZS2} we introduced the PBW theorem for finite $W$-superalgebra $U(\ggg,e)$. %To study related topics on its representation theory,
As an important ingredient in our arguments,  we will first introduce a refined version of the PBW Theorem of $U(\ggg,e)$,
 %with aid of the structure theorem for refined $W$-superalgebra %$W_\chi'$ in \cite[Theorem 3.7]{ZS5},
 which will be present in Theorem \ref{PBWQC} and Proposition \ref{fiiso}.

\subsection{}From now on, we will focus on the minimal case.
\subsubsection{}
Recall that  Cartan subalgebras of $\ggg$ are conjugate, which are by definition just the ones of $\ggg_{\bar0}$.
Let $\mathfrak{h}$ be a standard Cartan subalgebra of ${\ggg}$, and let $\Phi$ be the root system of ${\ggg}$ relative to $\mathfrak{h}$. Denote by $\Phi_{\bar0}$ and $\Phi_{\bar1}$ the set of all even roots and odd roots, respectively. %whose simple root system $\Delta=\{\alpha_1,\cdots,\alpha_k\}$.
From the detailed description of $\Phi$ in \cite[\S2.5.4]{K} and \cite[Theorem 3.10]{PS}, we can choose a minimal root $-\theta$ of $\ggg$, and then a positive root system $\Phi^+$ of  $\Phi$ with simple root system $\Delta=\{\alpha_1,\cdots,\alpha_k\}$ satisfies that $\alpha_k=\theta$ when $\sfr$ is even,
and $\alpha_k=\frac{\theta}{2}$ when $\sfr$ is odd (see the arguments for Convention \ref{conventions}).
%Take a standard Borel subalgebra $\mathfrak{b}$.
%Correspondingly, we have a, and the associated simple root system $\Delta=\{\alpha_1,\ldots,\alpha_k\}$.
%is distinguished, for which the corresponding Dynkin diagram has the form represented in \cite[Table 1]{K}.
%Let $\Phi^+$ be the positive system of $\Phi$ relative to $\Delta$.
%, and set $\Phi^-:=-\Phi^+$. Denote by $\Phi_{\bar0}$ and $\Phi_{\bar1}$ the set of all even roots and odd roots, respectively. Set $\Phi^\pm_{\bar0}:=\Phi^\pm\cap\Phi_{\bar0}$ and $\Phi^\pm_{\bar1}:=\Phi^\pm\cap\Phi_{\bar1}$, respectively.
\iffalse
Set $\Delta_{\bar0}'=\{\alpha_1',\cdots,\alpha_n'\}$ (numbered as in \cite{B}) to be the simple root system of a simple component of $\ggg_{\bar0}$.
If the simple component of $\ggg_{\bar0}$ is not of type $A(n)$ or $C(n)$, there is unique long root in $\Delta_{\bar0}'$ linked with the lowest root $-\beta$ on the extended Dynkin diagram of the corresponding simple component of $\ggg_{\bar0}$; we call it $\theta$. If the simple component of $\ggg_{\bar0}$ is of type $A(n)$ and $C(n)$ we set $\theta=\alpha'_n$.
\fi
Furthermore, we may choose root vectors $e=e_\theta,f=e_{-\theta}$ corresponding to roots $\theta$ and $-\theta$ such that $(e_\theta,[e_\theta,e_{-\theta}],e_{-\theta})$ is an $\mathfrak{sl}(2)$-triple and put $h=h_\theta=[e_\theta,e_{-\theta}]$.

%Due to the minimality of $-\theta$, the eigenspace decomposition of $\text{ad}\,h$ gives rise to a short $\mathbb{Z}$-grading
%$\ggg=\ggg(-2)\oplus\ggg(-1)\oplus\ggg(0)\oplus\ggg(1)\oplus\ggg(2),$
%where $\ggg(2)=\mathbb{C}e$ and $\ggg(-2)=\mathbb{C}f$, with $\ggg(1)\oplus\ggg(2)$ and $\ggg(-1)\oplus\ggg(-2)$ being Heisenberg Lie superalgebras. %One knows that $\ggg^e(i)=\ggg(i)$ for $i=1,2$, and that $\ggg^e(0)$ coincides with the image of the Lie superalgebra endomorphism
%\begin{equation*}\label{xh}
%\sharp: \ggg(0)\rightarrow\ggg(0), x\mapsto x-\frac{1}{2}(h,x)h,
%\end{equation*}
%whose kernel of $\bbc h$ is a central ideal of $\ggg(0)$.
\iffalse
Under the above settings, the PBW structural Theorem of minimal refined $W$-superalgebra $W_\chi'$
was formed in our previous paper; see Theorems \ref{ge} and \ref{maiin1} for more detail.
\fi
\subsubsection{}\label{1.5}

As explained in \cite{ZS2}, the structure of finite $W$-superalgebras critically depends on the detecting parity of $\sfr:=\dim{\ggg}(-1)_{\bar1}$. %From the discussion in \S\ref{1.3}, we have %already known that $U(\ggg,e)$ when $\sfr$ is even.
In the  situation when $\sfr$ is even, the concerned results are very similar to that of the non-super case, despite the discussion here is more difficult.
However, when $\sfr$ is odd,
finite $W$-superalgebra is significantly different from the finite $W$-algebra case. In this situation, the emergency of  odd root $\frac{\theta}{2}$  makes the situation much more complicated, and an extra restriction \eqref{c0clambda} must be imposed to make the procedure go smoothly.
So in the following we will mainly consider the case when $\sfr$ is odd.

Let $\Phi_e=\{\alpha\in\Phi\mid\alpha(h)=0~\text{or}~1\}$, and write $\Phi_e^\pm:=\Phi_e\cap\Phi^\pm$ where $\Phi^-=-\Phi^+$. For $i=0,1$ set $\Phi^\pm_{e,i}:=\{\alpha\in\Phi_e^\pm\mid\alpha(h)=i\}$, $(\Phi^+_{e,0})_{\bar0}:=\Phi^+_{e,0}\cap\Phi_{\bar0}$. Set $\mathfrak{h}^e:=\mathfrak{h}\cap\ggg^e$ to be a Cartan subalgebra in $\ggg^e(0)$ with $\{h_1,\cdots,h_{k-1}\}$ being a basis such that $(h_i,h_j)=\delta_{i,j}$ for $1\leqslant i,j\leqslant k-1$. Let $\delta$, $\rho$ and $\rho_{e,0}$ be defined as in \eqref{deltarho}.
Given a linear function $\lambda$ on $\mathfrak{h}^e$ and $c\in\mathbb{C}$ satisfying the equation \eqref{c0clambda}, we call $(\lambda,c)$ a matchable pair. Associated with such a pair, let $I_{\lambda,c}$ be a left ideal of $U(\ggg,e)$ defined as in  \S\ref{3.2.1}.
We call the $U(\ggg,e)$-module $Z_{U(\ggg,e)}(\lambda,c):=U(\ggg,e)/I_{\lambda,c}$
the {\it Verma module of level $c$ corresponding to $\lambda$.} %, which has a nice PBW basis. %In Proposition \ref{verma} we show %that
Moreover, $Z_{U(\ggg,e)}(\lambda,c)$ is proved to contain a unique maximal submodule which we
denote $Z_{U(\ggg,e)}^{\rm max}(\lambda,c)$ (see \S\ref{3.2.2}). Thus to every matchable pair
$(\lambda,c)\in(\mathfrak{h}^e)^*\times\bbc$, there corresponds an irreducible
highest weight $U(\ggg,e)$-module
$L_{U(\ggg,e)}(\lambda,c):=\,Z_{U(\ggg,e)}(\lambda,c)/Z_{U(\ggg,e)}^{\rm max}(\lambda,c)$.

Recall that an irreducible module is of type $M$ if its endomorphism ring is one-dimensional and it is of type $Q$ if its endomorphism ring is two-dimensional. As the first main result of the paper, we have
\begin{theorem}\label{verma2}Assume that $\sfr$ is odd.
Given a matchable pair $(\lambda,c)\in(\mathfrak{h}^e)^*\times\bbc$,
the following statements hold:
\begin{itemize}
\item[(1)] $Z^{\rm max}_{U(\ggg,e)}(\lambda,c)$ is the unique maximal submodule of the Verma module $Z_{U(\ggg,e)}(\lambda,c)$, and $L_{U(\ggg,e)}(\lambda,c)$ is a simple $U(\ggg,e)$-module of type $Q$.
\item[(2)] The simple $U(\ggg,e)$-modules $L_{U(\ggg,e)}(\lambda,c)$ and $L_{U(\ggg,e)}(\lambda',c')$ are isomorphic if and only if $(\lambda,c)=(\lambda',c')$.
\item[(3)] Any finite-dimensional simple $U(\ggg,e)$-module  (up to parity switch) is isomorphic to one of the modules $L_{U(\ggg,e)}(\lambda,c)$ for some $\lambda\in(\mathfrak{h}^e)^*$ satisfying $\lambda(h_\alpha)\in\mathbb{Z}_+$ for all $\alpha\in(\Phi^+_{e,0})_{\bar0}$. We further have that $c$ is a rational number in the case when $\ggg=\mathfrak{spo}(2|m)$ with $m$ being odd such that $\ggg^e(0)=\mathfrak{so}(m)$, or when $\ggg=\mathfrak{spo}(2m|1)$ with $m\geqslant2$ such that $\ggg^e(0)=\mathfrak{spo}(2m-2|1)$, or when $\ggg=G(3)$ with $\ggg^e(0)=G(2)$.
\end{itemize}
\end{theorem}
%We further have that $c$ is a rational number if $\ggg$ is one of the basic Lie superalgebras of type  $\mathfrak{spo}(m|1)$ with $m\geqslant4$ such that $\ggg^e(0)=\mathfrak{spo}(m-2|1)$, $\mathfrak{spo}(2|m)$ with $m$ being odd  such that $\ggg^e(0)=\mathfrak{so}(m)$,  or $G(3)$ with $\ggg^e(0)=G(2)$.
The proof of Theorem \ref{verma2} will be given in \S\ref{3.2.4}.
\subsubsection{}\label{1.6}
For $\chi=(e,\cdot)$ we let $\mathcal{C}_\chi$ denote the category of Whittaker modules, i.e., all ${\ggg}$-modules on which $x-\chi(x)$ acts locally nilpotently for all $x\in\mmm$. Given a ${\ggg}$-module $M$ we set
$$\text{Wh}(M)=\{m\in M\mid x.m=\chi(x)m, \forall x\in{\mmm}\}.$$
%Recall that a ${\ggg}$-module $M$ is called a Whittaker module if $a-\chi(a)$ acts on $M$ locally nilpotently for each $a\in{\mmm}$. A Whittaker vector in a Whittaker is a vector $v\in M$ which satisfies $(a-\chi(a))v=0,~\for$.
%Let $\mathcal{C}_\chi$ finitely generated Whittaker, and assume all the morphisms are even. Write the subspace of all Whittaker vectors in $M$.

We will link the Verma modules $Z_{U(\ggg,e)}(\lambda,c)$ with $\ggg$-modules obtained by
parabolic induction from Whittaker modules for $\mathfrak{osp}(1|2)$.
%To determine the composition multiplicities of we them
Set $\mathfrak{s}_\theta= \bbc e\oplus\bbc h\oplus\bbc f\oplus\bbc[v_{\frac{\sfr+1}{2}},e]\oplus\bbc v_{\frac{\sfr+1}{2}}$, and put
\begin{equation*}
\mathfrak{p}_\theta:=\mathfrak{s}_\theta+\mathfrak{h}+\sum_{\alpha\in\Phi^+}\bbc e_\alpha,\qquad
\mathfrak{n}_\theta:=\sum_{\alpha\in\Phi^+\backslash\{\frac{\theta}{2},\theta\}}\bbc e_\alpha,\qquad \widetilde{\mathfrak{s}}_\theta:=\mathfrak{h}^e\oplus\mathfrak{s}_\theta,
\end{equation*}where $e_\alpha$ denotes a root vector with $\alpha\in\Phi$.
Let $C_\theta:=2ef+\frac{1}{2}h^2-\frac{3}{2}h-2v_{\frac{\sfr+1}{2}}[v_{\frac{\sfr+1}{2}},e]$ be a central element of $U(\widetilde{\mathfrak{s}}_\theta)$. Given $\lambda\in(\mathfrak{h}^e)^*$, write $I_\theta(\lambda)$ for the left ideal of $U(\mathfrak{p}_\theta)$ generated by $f-1$, $[v_{\frac{\sfr+1}{2}},e]-\frac{3}{4}v_{\frac{\sfr+1}{2}}+\frac{1}{2}v_{\frac{\sfr+1}{2}}h$, $C_\theta+\frac{1}{8}$, all $t-\lambda(t)$ with $t\in\mathfrak{h}^e$, and all $e_\gamma$ with $\gamma\in\Phi^+\backslash\{\frac{\theta}{2},\theta\}$.
Set $Y(\lambda):=U(\mathfrak{p}_\theta)/I_\theta(\lambda)$ to be a $\mathfrak{p}_\theta$-module with the trivial action of $\mathfrak{n}_\theta$, which
is isomorphic to a Whittaker module for the Levi subalgebra $\widetilde{\mathfrak{s}}_\theta$. Now define \begin{equation*}
M(\lambda):=U(\ggg)\otimes_{U(\mathfrak{p}_\theta)}Y(\lambda).
\end{equation*}
%Recall that each $u_i^*$ (resp. $v_i^*$) with $1\leqslant i\leqslant \frac{s}{2}$ (resp. $1\leqslant i\leqslant\frac{\sfr-1}{2}$) is a root vector in $\ggg(-1)_{\bar0}$ (resp. $\ggg(-1)_{\bar1}$) corresponding to $\gamma_{\bar0i}^*=-\theta-\gamma_{\bar0i}\in\Phi^+_{\bar0}$ (resp. $\gamma_{\bar1i}^*=-\theta-\gamma_{\bar1i}\in\Phi_{\bar1}^+$).
%Let $\delta=\frac{1}{2}(\gamma^*_{\bar01}+\cdots+\gamma^*_{\bar0\frac{s}{2}}
%-\gamma^*_{\bar11}-\cdots-\gamma^*_{\bar1\frac{\sfr-1}{2}})$ and $\rho=\frac{1}{2}\sum_{\alpha\in\Phi^{+}}(-1)^{|\alpha|}\alpha$, where $|\alpha|$ denotes the parity of $\alpha$.

Since the restriction of $(\cdot,\cdot)$ to $\mathfrak{h}^e$ is non-degenerate, for any $\eta\in(\mathfrak{h}^e)^*$ there is a unique $t_\eta$ in $\mathfrak{h}^e$ with $\eta=(t_\eta,\cdot)$. Hence $(\cdot,\cdot)$ induces a non-degenerate bilinear form on $(\mathfrak{h}^e)^*$ via $(\mu,\nu):=(t_\mu,t_\nu)$ for all $\mu,\nu\in(\mathfrak{h}^e)^*$.
For a linear function $\varphi\in\mathfrak{h}^*$ we denote by $\bar{\varphi}$ the restriction of $\varphi$ to $\mathfrak{h}^e$. Keep the notation $\epsilon$ as in \eqref{defofepsilon}. Under the twisted action of $U(\ggg,e)$ on the Verma modules as defined
in \S\ref{5.3}, we have the following second main result of the paper.
\begin{theorem}\label{Whcchi}Assume that $\sfr$ is odd.
Every $\ggg$-module $M(\lambda)$ is an object of the category $\mathcal{C}_\chi$. Furthermore, %if $\lambda\in(\mathfrak{h}^e)^*$ satisfies the equation \eqref{c0c'lambda'}, then
$\text{Wh}(M(\lambda))\cong Z_{U(\ggg,e)}(\lambda+\bar{\delta},-\frac{1}{8}+(\lambda+2\bar\rho,\lambda)+\epsilon)$ as $U(\ggg,e)$-modules.
\end{theorem}
The proof of Theorem \ref{Whcchi} will be given in \S\ref{5.3}.
In virtue of Theorem \ref{Whcchi} and Skryabin's equivalence in \cite[Theorem 2.17]{ZS2} (see also \S\ref{skryabin} for more details), we can translate the problem of computing of the composition multiplicities of the Verma modules $Z_{U(\ggg,e)}(\lambda,c)$ to the one of the parabolic induced modules from Whittaker modules (i.e., the standard Whittaker modules) for $\mathfrak{osp}(1|2)$. These standard Whittaker modules have been studied in much detail in
\cite{C,C3}, and it is known that their
composition multiplicities can be determined by the composition factors of Verma modules for $U(\ggg)$ in the ordinary BGG category $\mathcal{O}$.%, which makes the problem much easier.
\subsubsection{}
For $\alpha\in(\mathfrak{h}^e)^*$, let $\ggg_\alpha=\bigoplus_{i\in\mathbb{Z}}\ggg_\alpha(i)$
denote the $\alpha$-weight space of $\ggg$ with respect to $\mathfrak{h}^e$, then we have
$\ggg=\ggg_0 \oplus \bigoplus_{\alpha\in \Phi'_e} \ggg_\alpha$,
where $\Phi'_e \subset (\mathfrak{h}^e)^*$ is the set of nonzero weights
of $\mathfrak{h}^e$ on $\ggg$.
Let $(\Phi'_e)^+:=\Phi^+\backslash\{\frac{\theta}{2},\theta\}$ be a system of positive roots in the restricted root system $\Phi'_e$. Setting $(\Phi'_e)^-:=-(\Phi'_e)^+$, we define
$\ggg_{\pm} := \bigoplus_{\alpha \in (\Phi'_e)^{\pm}} \ggg_\alpha$, so that
$\ggg = \ggg_- \oplus \ggg_0 \oplus \ggg_+,
\mathfrak{q} = \ggg_0 \oplus \ggg_+$.
The choice $(\Phi'_e)^+$ of positive roots
induces a dominance ordering $\leqslant$ on $(\mathfrak{h}^e)^*$:
$\mu\leqslant\lambda$ if
$\lambda-\mu\in \mathbb{Z}_{\geqslant 0}(\Phi'_e)^+$.

Under the above settings, we can define the highest weight $U(\ggg,e)$-module $M_e(\lambda)$ with highest weight $\lambda$ as in \eqref{highestweigmodu} and its irreducible quotient $L_e(\lambda)$ as in \eqref{lelambda}. Comparing Theorem \ref{verma2} with Theorem \ref{Tverma}, we can find that $M_e(\lambda)$ and $L_e(\lambda)$ share the same meaning as the Verma module $Z_{U(\ggg,e)}(\lambda,c)$ and its irreducible quotient $L_{U(\ggg,e)}(\lambda,c)$; see Remark \ref{compare} for more details.

Now we introduce an analogue of the BGG  category $\mathcal{O}$. Let $\mathcal{O}(e)=\mathcal{O}(e;\mathfrak{h},\mathfrak{q})$ denote the category of all finitely generated $U(\ggg,e)$-modules $V$, that are semi-simple over $\mathfrak{h}^e$
with finite-dimensional $\mathfrak{h}^e$-weight spaces, such that the set
$\{\lambda \in (\mathfrak{h}^e)^*\:|\:V_\lambda \neq \{0\}\}$ is contained in a finite union of sets of the form $\{\nu\in(\mathfrak{h}^e)^*\:|\: \nu\leqslant\mu\}$ for $\mu \in (\mathfrak{h}^e)^*$.
Then we obtain
\begin{theorem}\label{proofcateO} Assume that $\sfr$ is odd.
For the category $\mathcal O(e)$, the following statements hold:
\begin{itemize}
\item[(1)] There is a complete set of isomorphism classes of simple objects which is $\{L_e(\lambda)\:|\: \lambda\in (\hhh^e)^*\}$ as in \eqref{lelambda}.
\item[(2)] The category $\mathcal O(e)$ is Artinian. In particular, every object has finite length of composition series.
\item[(3)] The category $\mathcal O(e)$ has a block decomposition as $\mathcal O(e) = \bigoplus_{\psi^\lambda}  \mathcal O_{\psi^\lambda}(e)$, where the direct sum is over all central characters $\psi^\lambda:Z(U(\ggg,e)) \rightarrow \bbc$,
and $\mathcal O_{\psi^\lambda}(e)$ denotes the Serre subcategory of $\mathcal O(e)$
generated by the irreducible modules
$\{L_e(\mu)\:|\:\mu \in (\hhh^e)^*~\text{such that}~\psi^\mu=\psi^\lambda\}$.
\end{itemize}
\end{theorem}
The proof of Theorem \ref{proofcateO} will be given in \S\ref{5.1.6}, for which we roughly follow the strategy in \cite{BGK}, but the situation is quite different from the case of finite $W$-algebras. The role of ``Cartan subalgebra" is taken over by a finite $W$-superalgebra arising from the sum of a Lie subsuperalgebra which is isomorphic to $\mathfrak{osp}(1|2)$ and an abelian subalgebra which commutes with this Lie subsuperalgebra. The precise structural information of $U(\ggg,e)$ previously presented %in Theorem \ref{main3}
enables us to successfully establish such a desired highest weight theory.

Similar theory can also be established for minimal refined $W$-superalgebra $W'_{\chi}$ (see Appendix \ref{5.2}).
\subsection{}
The paper is organized as follows. In \S\ref{Backgrounds} some basics on finite and refined $W$-superalgebras are recalled, and the PBW theorem of finite $W$-superalgebra $U(\ggg,e)$ associated with arbitrary nilpotent element is refined. In \S\ref{3}, we first
study the topics of the Verma module $Z_{U(\ggg,e)}(\lambda,c)$ and its simple quotient $L_{U(\ggg,e)}(\lambda,c)$ for minimal  finite $W$-superalgebra $U(\ggg,e)$, modulo Lemma \ref{left ideal2} whose lengthy proof is postponed until Appendix \ref{lengthy proof}. And then the Verma module $Z_{W_\chi'}(\lambda,c)$ and its simple quotient $L_{W_\chi'}(\lambda,c)$ for minimal refined $W$-superalgebras $W_\chi'$ are introduced.
We finally demonstrate a complete set of isomorphism classes of irreducible ``highest weight" modules.
\S\ref{4} is devoted to the correspondence between the Verma modules for finite $W$-superalgebras and their associated Whittaker categories, where the most important tool we used there is Skryabin's equivalence in \cite[Theorem 2.17]{ZS2}.
In \S\ref{mathcalO}, we introduce the abstract universal highest weight modules for minimal finite $W$-superalgebras of type odd, consider their corresponding BGG category $\mathcal{O}$, and finally give a proof of Theorem \ref{proofcateO}.
Appendix \ref{5.2} is a counterpart of \S\ref{mathcalO} for minimal refined $W$-superalgebras of both types.
Appendix \ref{lengthy proof} is dedicated to the proof of Lemma \ref{left ideal2}.
%recall the PBW Theorem of minimal refined $W$-superalgebras $W_\chi'$, and then. Parallel to the discussions in previous section, in \S\ref{4} we will first refine the PBW theorem of finite $W$-superalgebra $U(\ggg,e)$ associated with arbitrary nilpotent element, and introduce the corresponding Verma module $Z_{U(\ggg,e)}(\lambda,c)$ and its simple quotient $L_{U(\ggg,e)}(\lambda,c)$ for the minimal case. The concluding section \S\ref{5} will be devoted to the computing of composition multiplicities of Verma modules for minimal finite and refined $W$-superalgebras
\subsection{}\label{innao} Throughout the paper we work with complex field $\mathbb{C}$ as the ground field.
Let $\mathbb{Z}_+$ and $\bbq_+$ be the sets of all the non-negative integers in $\mathbb{Z}$ and all the non-negative rational numbers in $\mathbb{Q}$ respectively, and denote by
\begin{equation*}
\begin{split}
\mathbb{Z}_+^k:=\{(i_1,\cdots,i_k)\mid i_j\in\mathbb{Z}_+\},&\qquad
\Lambda_k:=\{(i_1,\cdots,i_k)\mid i_j\in\{0,1\}\},\\
\mathbf{a}:=(a_1,\cdots,a_k),&\qquad
|\mathbf{a}|:=\sum_{i=1}^ka_i.
\end{split}
\end{equation*}
For any real number $a\in\mathbb{R}$, let $\lceil a\rceil$ denote the largest integer lower bound of $a$, and $\lfloor a\rfloor$ the least integer upper bound of $a$.

A superspace is a $\mathbb{Z}_2$-graded vector space $V=V_{\bar0}\oplus V_{\bar1}$, in which we call elements in $V_{\bar0}$ and $V_{\bar1}$ even and odd, respectively. Write $|v|\in\mathbb{Z}_2$ for the parity (or degree) of $v\in V$, which is implicitly assumed to be $\mathbb{Z}_2$-homogeneous.

All Lie superalgebras ${\ggg}$ will be assumed to be finite-dimensional.
We consider vector spaces, subalgebras, ideals, modules, and submodules, {\textit{etc}}. in the super sense unless otherwise specified, throughout the paper. A supermodule homomorphism is assumed to be a $\mathbb{Z}_2$-graded parity-preserving linear map that is a homomorphism in the usual sense.

\section{A refined  PBW theorem for finite $W$-superalgebras}\label{Backgrounds}

In this section, we give a refined version of the PBW theorem for finite $W$-superalgebras which will be very important to the subsequent arguments. For this, we will have a glance at basic classical Lie superalgebras, and recall some basic structure of finite $W$-superalgebras.

\subsection{Basic classical Lie superalgebras}\label{1.1}
We refer the readers to  \cite{CW,K,K2,M1} for basic classical Lie superalgebras,  and \cite{P2,W,WZ,ZS2} for finite $W$-(super)algebras.

%Let ${\ggg}$ be a basic Lie superalgebra over ${\bbc}$.
%Recall that among all simple finite-dimensional Lie superalgebras it %is characterized by the properties that its even part $\ggg_{\bar0}$ %is reductive and that it admits a non-degenerate invariant %supersymmetric bilinear form $(\cdot,\cdot)$.
A complete list of basic classical simple Lie superalgebras consists of simple finite-dimensional Lie algebras and the Lie superalgebras $\mathfrak{sl}(m|n)$$(=A(m-1|n-1))$  with $m,n\geqslant1, m\neq n$; $\mathfrak{psl}(m|m)$$(=A(m-1|m-1))$  with $m\geqslant2$; $\mathfrak{osp}(m|2n)=\mathfrak{spo}(2n|m)$ (type $B,C,D$); $D(2,1;\alpha)$ with $\alpha\in\mathbb{C}, a\neq0,-1$; $F(4)$; $G(3)$. Those Lie superalgebras are divided into two types as below.
\begin{center}\label{Table 1}
({\sl{Table 1}}): The classification of basic classical Lie superalgebras
\vskip0.3cm
{\begin{tabular}{|c|ccccc|}
\hline Type I & $A(m|n)~(m\neq n)$ &$A(n|n)$&  $C(n)$& &\\
\hline Type II & $B(m|n)$ & $D(m|n)$ & $D(2,1;\alpha)$ & $F(4)$&$G(3)$\\\hline
\end{tabular}}
\end{center}
\vskip0.3cm

Let ${\ggg}=\ggg_\bz\oplus \ggg_\bo$ be a basic classical Lie superalgebra over ${\bbc}$ (We will simply call it a basic Lie superalgebra for short).
Let $\mathfrak{h}$ be a standard Cartan subalgebra of $\ggg$ and %let $\mathfrak{b}$ be a distinguished Borel subalgebra in ${\ggg}$ (i.e., for which the corresponding Dynkin diagram has the form represented in \cite[Table 1]{K2}), containing $\mathfrak{h}$.
$\Phi$ be a root system of ${\ggg}$ relative to $\mathfrak{h}$. %, whose Dynkin diagram is as in \cite[Table 1]{K2}. %with Dynkin diagram being as in \cite[Table 1]{K2}.
We take a simple root system $\Delta=\{\alpha_1,\cdots,\alpha_k\}$. % , which will be also called distinguished.
By \cite[\S3.3]{FG}, we can choose a Chevalley
basis $B=\{e_\gamma\mid\gamma\in\Phi\}\cup\{h_\alpha\mid\alpha\in\Delta\}$ of ${\ggg}$ (In the case of $\ggg=D(2,1;\alpha)$ with $\alpha\notin\bbz$ such that $\alpha\in\overline\bbq$, one needs to adjust the definition of Chevalley basis by changing $\bbz$ to $\bbz[\alpha]$ (where $\bbz[\alpha]$ denotes the $\bbz$-algebra generated by $\alpha$) in the range of construction constants; see \cite[\S3.1]{Gav}. If $\alpha\notin\overline\bbq$, we just assume that $B$ is a basis of ${\ggg}$).
Denote by $\Phi^+$ the positive system of $\Phi$ relative to $\Delta$, and set $\Phi^-:=-\Phi^+$. Denote by $\Phi_{\bar0}$ and $\Phi_{\bar1}$ the set of all even roots and odd roots, respectively.
We always write $|\alpha|=\bar0$ for any $\alpha\in\Phi_{\bar0}$ and $|\alpha|=\bar1$ for any $\alpha\in\Phi_{\bar1}$.
Set $\Phi^\pm_{\bar0}:=\Phi^\pm\cap\Phi_{\bar0}$ and $\Phi^\pm_{\bar1}:=\Phi^\pm\cap\Phi_{\bar1}$, respectively.

\subsection{Finite $W$-superalgebras}\label{1.2}
Given  a nonzero nilpotent element $e\in \ggg_\bz$, by the Jacobson-Morozov theorem there is an $\frak{sl}_2$-triple $(e,f,h)$ with $f,h\in\ggg_\bz$. Let $(\cdot,\cdot)$ be a non-degenerate even supersymmetric invariant bilinear form on ${\ggg}$. Define $\chi\in{\ggg}^{*}$ by letting $\chi(x)=(e,x)$ for all $x\in{\ggg}$. Up to a scalar, we can further assume that $(e,f)=1$.

The linear operator ad\,$h$ on $\ggg$ defines a ${\bbz}$-grading ${\ggg}=\bigoplus_{i\in{\bbz}}{\ggg}(i)$ with $e\in{\ggg}(2)_{\bar{0}}$ and $f\in{\ggg}(-2)_{\bar{0}}$. Set ${\ppp}:=\bigoplus_{i\geqslant 0}{\ggg}(i)$ to be a parabolic subalgebra of $\ggg$. Denote by $\ggg^e$ (resp.\ $\ggg^f$) the centralizer of $e$ (resp.\ $f$) in $\ggg$, then we have ${\ggg}^e=\bigoplus_{i\geqslant0}{\ggg}^e(i)$ (resp.\ ${\ggg}^f=\bigoplus_{i\leqslant0}{\ggg}^f(i)$) by the $\mathfrak{sl}(2)$-representation theory.
Define a symplectic (resp.\ symmetric) bilinear form $\langle\cdot,\cdot\rangle$ on the ${\bbz}_2$-graded subspace ${\ggg}(-1)_{\bar{0}}$ (resp.\ ${\ggg}(-1)_{\bar{1}}$) given by $\langle x,y\rangle=(e,[x,y])=\chi([x,y])$ for all $x,y\in{\ggg}(-1)_{\bar0}~(\text{resp.}\,x,y\in{\ggg}(-1)_{\bar1})$. Set $s:=\text{dim}\,\ggg(-1)_{\bar0}$ (note that $s$ is an even number), and $\sfr:=\text{dim}\,\ggg(-1)_{\bar1}$. Choose $\mathbb{Z}_2$-homogenenous bases $\{u_1,\cdots,u_{s}\}$ of ${\ggg}(-1)_{\bar0}$ and $\{v_1,\cdots,v_\sfr\}$ of ${\ggg}(-1)_{\bar1}$  contained in ${\ggg}$ such that $\langle u_i, u_j\rangle =i^*\delta_{i+j,s+1}$ for $1\leqslant i,j\leqslant s$, where $i^*:=\left\{\begin{array}{ll}-1&\text{if}~1\leqslant i\leqslant \frac{s}{2};\\ 1&\text{if}~\frac{s}{2}+1\leqslant i\leqslant s\end{array}\right.$, and $\langle v_i,v_j\rangle=\delta_{i+j,\sfr+1}$ for $1\leqslant i,j\leqslant \sfr$.
We further assume that the $u_i$'s with $1\leqslant i\leqslant \frac{s}{2}$ and $v_i$'s with $1\leqslant i\leqslant \lceil\frac{\sfr}{2}\rceil$ are root vectors corresponding to negative roots $\gamma_{\bar 0i}\in\Phi_{\bar0}^-$ and $\gamma_{\bar 1i}\in\Phi_{\bar1}^-$, respectively. When $\sfr$ is odd, we assume that the element $v_{\frac{\sfr+1}{2}}$ is also a negative root vector in $\Phi_{\bar1}^-$.

Let $z_\alpha:=u_\alpha$ for $1\leqslant\alpha\leqslant s$, and $z_{\alpha+s}:=v_\alpha$ for $1\leqslant\alpha\leqslant\sfr$. Set $S(-1):=\{1,2,\cdots,s+\sfr\}$, and then $\{z_\alpha\mid\alpha\in S(-1)\}$ %denote the union of $\{z_\alpha\mid1\leqslant\alpha\leqslant s\}$ and $\{z_{\alpha+s}\mid1\leqslant\alpha\leqslant\sfr\}$, which
is a basis of ${\ggg}(-1)$. Set $z_\alpha^*:=\alpha^\natural z_{s+1-\alpha}$ for $1\leqslant\alpha\leqslant s$, where $\alpha^\natural:=\left\{\begin{array}{ll}1&\text{if}~1\leqslant \alpha\leqslant \frac{s}{2};\\ -1&\text{if}~\frac{s}{2}+1\leqslant \alpha\leqslant s\end{array}\right.$, and $z_{\alpha+s}^*:=z_{\sfr+1-\alpha+s}$ for $1\leqslant\alpha\leqslant \sfr$; i.e., $u_i^*=\left\{\begin{array}{ll}u_{s+1-i}&\text{if}~1\leqslant i\leqslant \frac{s}{2};\\ -u_{s+1-i}&\text{if}~\frac{s}{2}+1\leqslant i\leqslant s\end{array}\right.$, and $v_{i}^*=v_{\sfr+1-i}$ for $1\leqslant i\leqslant \sfr$. Then $\{z_\alpha^*\mid\alpha\in S(-1)\}$ is a dual basis of $\{z_\alpha\mid\alpha\in S(-1)\}$ such that $\langle z^*_\alpha,z_\beta\rangle =\delta_{\alpha,\beta}$ for $\alpha,\beta\in S(-1)$.

{\it From now on, for any $\alpha\in S(-1)$ we will denote the parity of $z_\alpha$ by $|\alpha|$ for simplicity.} It is straightforward that $z_\alpha$ and $z_\alpha^*$ have the same parity, and %each $u_i^*$ (resp. $v_i^*$) with $1\leqslant i\leqslant \frac{s}{2}$ (resp. $1\leqslant i\leqslant \lceil\frac{\sfr}{2}\rceil$) is also a root vector in $\ggg(-1)_{\bar0}$ (resp. $\ggg(-1)_{\bar1}$).
%Therefore, we have that each $u_i^*$ (resp. $v_i^*$) with $\frac{s}{2}+1\leqslant i\leqslant s$ (resp. $\lceil\frac{\sfr}{2}\rceil+1\leqslant i\leqslant\sfr$) is a root vector in $\ggg(-1)_{\bar0}$ (resp. $\ggg(-1)_{\bar1}$) corresponding to $-\gamma_{\bar0,s-i}\in\Phi^+_{\bar0}$ (resp. $\gamma_{\bar1,\sfr-i}^*\in\Phi_{\bar1}^+$) by definition.
each $u_i^*$ (resp.\ $v_i^*$) with $1\leqslant i\leqslant \frac{s}{2}$ (resp.\ $1\leqslant i\leqslant \lceil\frac{\sfr}{2}\rceil$) is a root vector in $\ggg(-1)_{\bar0}$ (resp.\ $\ggg(-1)_{\bar1}$) corresponding to $\gamma_{\bar0i}^*:=-\theta-\gamma_{\bar0i}\in\Phi^+_{\bar0}$ (resp.\ $\gamma_{\bar1i}^*:=-\theta-\gamma_{\bar1i}\in\Phi_{\bar1}^+$).
Moreover, $\{u_1,\cdots,u_{\frac{s}{2}},u^*_1,\cdots,u^*_{\frac{s}{2}}\}$ constitutes a $\bbc$-basis of $\ggg(-1)_{\bar0}$. On the other hand, $\{v_1,\cdots,v_{\frac{\sfr}{2}},v^*_1,\cdots,v^*_{\frac{\sfr}{2}}\}$ (resp.\ $\{v_1,\cdots,v_{\frac{\sfr-1}{2}},v_{\frac{\sfr+1}{2}},v^*_1,\cdots,v^*_{\frac{\sfr-1}{2}}\}$) is a $\bbc$-basis of $\ggg(-1)_{\bar1}$ for $\sfr$ being even (resp.\ odd).
Write ${\ggg}(-1)^{\prime}_{\bar0}$ for the ${\bbc}$-span of $u_{{s\over 2}+1},\cdots,u_{s}$ and
${\ggg}(-1)^{\prime}_{\bar1}$ the ${\bbc}$-span of $v_{\frac{\sfr}{2}+1},\cdots,v_\sfr$ (resp.\ $v_{\frac{\sfr+3}{2}},\cdots,v_\sfr$) when $\sfr$ is even (resp.\ odd). Denote by ${\ggg}(-1)^{\prime}:={\ggg}(-1)^{\prime}_{\bar0}\oplus{\ggg}(-1)^{\prime}_{\bar1}$.

Now we can introduce the so-called ``$\chi$-admissible subalgebra" of $\ggg$ by
\begin{align}\label{admissible alg}
\mathfrak{m}:=\bigoplus_{i\leqslant -2}{\ggg}(i)\oplus{\ggg}(-1)^{\prime}.
 \end{align} Then $\chi$ vanishes on the derived subalgebra of $\mathfrak{m}$.  We also have an ``extended $\chi$-admissible subalgebra" of $\ggg$ as below
\begin{align}\label{extend admiss alg}
\mathfrak{m}^{\prime}:=\left\{\begin{array}{ll}\mathfrak{m}&\text{if}~\sfr~\text{is even;}\\
\mathfrak{m}\oplus {\bbc}v_{\frac{\sfr+1}{2}}&\text{if}~\sfr~\text{is odd.}\end{array}\right.
\end{align}
Define the generalized Gelfand-Graev ${\ggg}$-module associated with $\chi$ by \begin{equation}\label{generalized Gelfand-Graev}
Q_\chi:=U({\ggg})\otimes_{U(\mathfrak{m})}{\bbc}_\chi,
\end{equation}
where ${\bbc}_\chi={\bbc}1_\chi$ is a one-dimensional  $\mathfrak{m}$-module such that $x.1_\chi=\chi(x)1_\chi$ for all $x\in\mathfrak{m}$. The super structure of $Q_\chi$ is dependent on the parity of $\bbc_\chi$, which is indicated to be even hereafter.
%\begin{defn}\label{finitewsuper}
The finite $W$-superalgebra associated with the pair ($\ggg,e)$ is defined as
 $$U({\ggg},e)=(\text{End}_{\ggg}Q_{\chi})^{\text{op}},$$
where $(\text{End}_{\ggg}Q_{\chi})^{\text{op}}$ denotes the opposite algebra of the endomorphism algebra of ${\ggg}$-module $Q_{\chi}$.
%\end{defn}

Let $I_\chi$ denote the ${\bbz}_2$-graded left ideal in $U({\ggg})$ generated by all $x-\chi(x)$ with $x\in\mathfrak{m}$. The fixed point space $(U({\ggg})/I_\chi)^{\ad\,\mmm}$ carries a natural algebra structure given by $(x+I_\chi)\cdot(y+I_\chi):=(xy+I_\chi)$ for all $x,y\in U({\ggg})$. Then $U({\ggg})/I_\chi\cong Q_\chi$ as ${\ggg}$-modules via the ${\ggg}$-module map sending $1+I_\chi$ to $1_\chi$, and $U({\ggg},e)\cong Q_{\chi}^{\ad\,\mmm}$ as $\bbc$-algebras. Explicitly speaking, any element of $U({\ggg},e)$ is uniquely determined by its effect on the generator $1_\chi\in Q_\chi$, and the canonical isomorphism between $U({\ggg},e)$ and $Q_{\chi}^{\ad\,\mmm}$ is given by $u\mapsto u(1_\chi)$ for any $u\in U({\ggg},e)$. In what follows we will often identify $Q_\chi$ with $U({\ggg})/I_\chi$ and $U({\ggg},e)$ with $Q_{\chi}^{\ad\,\mmm}$.

Let $w_1,\cdots, w_c$ be a basis of $\ggg$ over $\bbc$. For any given  $w_{i_1}\in\mathfrak{g}(j_1),\cdots,w_{i_k}\in\mathfrak{g}(j_k)$, set $\text{wt}(w_{i_1}\cdots w_{i_k}):=j_1+\cdots+j_k$ to be the weight of $w_{i_1}\cdots w_{i_k}$, and let $U({\ggg})=\bigcup_{i\in{\bbz}}\text{F}_iU({\ggg})$ be a filtration of $U({\ggg})$, where $\text{F}_iU({\ggg})$ is the ${\bbc}$-span of all $w_{i_1}\cdots w_{i_k}$ with %$w_{i_1}\in\mathfrak{g}(j_1),\cdots,w_{i_k}\in\mathfrak{g}(j_k)$ and
$(j_1+2)+\cdots+(j_k+2)\leqslant  i$. This filtration is called {\sl Kazhdan filtration}.  The Kazhdan filtration on $Q_{\chi}$ is defined by $\text{F}_iQ_{\chi}:=\text{Pr}(\text{F}_iU({\ggg}))$ with $\text{Pr}:U({\ggg})\twoheadrightarrow U({\ggg})/I_\chi$ being the canonical homomorphism, which makes $Q_{\chi}$ into a filtered $U({\ggg})$-module. Then there is an induced Kazhdan filtration
$\text{F}_i U({\ggg},e)$ on the subspace $U({\ggg},e)=Q_{\chi}^{\ad\,\mmm}$ of $Q_{\chi}$ such that $\text{F}_j U({\ggg},e)=0$
unless $j\geqslant0$. %, and also a Kazhdan filtration $\text{F}_i W_\chi'$ on its subalgebra $W_\chi'$.
%For any $\Theta\in U({\ggg},e)$, we will denote by $\text{deg}_e(\Theta)$ the degree of $\Theta$ under the Kazhdan filtration.
For any element $\Theta\in U({\ggg},e)$, we will denote by $\text{deg}_e(\Theta)$ the degree of $\Theta$ under the Kazhdan grading.

Denote by $\text{gr}$ the corresponding graded algebras under the Kazhdan filtration as above. As $U({\ggg},e)\subset U({\ggg})/I_\chi$ by definition, it is not hard to see that $\text{gr}(U({\ggg}))$ is supercommutative, and then $\text{gr}(U({\ggg},e))$ is also supercommutative. %For any $X\in U({\ggg},e)$ we denote by $\text{gr}(X)$ the corresponding element in $\text{gr}(U({\ggg},e))$.

\subsection{Refined $W$-superalgebras}\label{2.3}
Recall in \cite[Definition 4.8]{ZS2} we introduced the so-called refined  $W$-superalgebras $W'_\chi$ via $$W'_\chi:=(U({\ggg})/I_\chi)^{\text{ad}\,{\mmm}'}\cong Q_\chi^{\text{ad}\,{\mmm}'}
\equiv\{\text{Pr}(y)\in U({\ggg})/I_\chi \mid [a,y]\in I_\chi, \forall a\in{\mmm}'\},$$
and $\text{Pr}(y_1)\cdot\text{Pr}(y_2):=\text{Pr}(y_1y_2)$ for all $\text{Pr}(y_1),\text{Pr}(y_2)\in W'_\chi$.
 By definition,  $W'_\chi$ coincides with $U(\ggg,e)$ if $\sfr$ is even, while $W'_\chi$ is a proper subalgebra of $U(\ggg,e)$ if $\sfr$ is odd. Moreover, the PBW theorem of $W'_\chi$ was introduced in \cite[Theorem 3.7]{ZS5}, and the Kazhdan filtration on $Q_{\chi}$ induces a Kazhdan filtration $\text{F}_i W_\chi'$ on the subalgebra $W_\chi'$ of $U({\ggg},e)$.
 The adoption of this notion enables us conveniently  to prove  the existence of  Kac-Weisfeiler modules when $e=e_{\theta}$ with $-\theta$ being a minimal root in \cite{ZS5}. Apart from this achievement, the consideration of  refined $W$-superalgebras can give rise to some other advantage in the arguments. In the sequel, we can see more about this point, owing to the isomorphism from  $W'_\chi$ onto  a quantum finite $W$-superalgebra introduced in \cite{suh}, the latter of which will be used in our arguments immediately.

%Keep the notations in the second paragraph of \S\ref{3.1}.
Set $\mathfrak{n}:=\bigoplus_{i\leqslant-2}{\ggg}(i)$ and $\mathfrak{n}':=\bigoplus_{i\leqslant-1}{\ggg}(i)$ to be nilpotent subalgebras of $\ggg$.
 Let $I^{\text{fin}}$ be the left ideal of $U(\ggg)$ generated by the elements $\{x-\chi(x) \mid x\in\mathfrak{n}\}$, and let $Q^{\text{fin}}_\chi:=U(\ggg)/I^{\text{fin}}$ be an induced $\ggg$-module. In what follows we denote by $\text{Pr}':U(\ggg)\rightarrow U(\ggg)/I^{\text{fin}}$ the canonical projection.
Suh \cite{suh} introduced quantum finite $W$-superalgebra associated with the pair $(\ggg,e)$ as
$$W^{\text{fin}}(\ggg,e):=(Q^{\text{fin}}_\chi)^{\text{ad}\,\mathfrak{n}'},$$
where $(Q^{\text{fin}}_\chi)^{\text{ad}\,\mathfrak{n}'}$ denotes the invariant subspace of $Q^{\text{fin}}_\chi$ under the adjoint action of $\mathfrak{n}'$, and the associated product of $(Q^{\text{fin}}_\chi)^{\text{ad}\,\mathfrak{n}'}$ is defined by $$(x+I^{\text{fin}})\cdot(y+I^{\text{fin}}):=xy+I^{\text{fin}}$$for $x+I^{\text{fin}},y+I^{\text{fin}}\in (Q^{\text{fin}}_\chi)^{\text{ad}\,\mathfrak{n}'}$.
The Kazhdan grading on $U(\ggg)$-module $Q^{\text{fin}}_\chi$ and the algebra $W^{\text{fin}}(\ggg,e)$ can also be defined similarly as before.

Recall in \cite[Theorem 4.11]{ZS5} we showed that $W_\chi'\cong W^{\text{fin}}(\ggg,e)$ as Kazhdan filtered algebras excluding the case of $\ggg=D(2,1;\alpha)$ with $\alpha\not\in\overline\bbq$. In fact, the above isomorphism is also valid for this special case. When the detecting parity of $\sfr$ is odd, by definition there must exist a  root $\alpha\in\Phi_{\bar1}^+$ of $\ggg$ such that $2\alpha\in\Phi_{\bar0}^+$. From the detailed description of the system of root $\Phi$ of $\ggg$ in \cite[\S2.5.4]{K}, this happens only when $\ggg$ is of type $B(m|n)$ or $G(3)$. Therefore, this special case corresponds to the situation when $\sfr$ is even. Let $\text{gr}(W'_\chi)$ and $\text{gr}(W^{\text{fin}}(\ggg,e))$ denote the graded algebra of $W'_\chi$ and $W^{\text{fin}}(\ggg,e)$ under the Kazhdan grading, respectively.
On one hand, it follows from \cite[Corollary 3.9(1)]{SX} that $\text{gr}(W'_\chi)\cong S(\ggg^e)$ as $\mathbb{C}$-algebras. On the other hand, we also have $\text{gr}(W^{\text{fin}}(\ggg,e))\cong S(\ggg^e)$ by \cite[Proposition 4.9]{ZS5}. Therefore, we obtain $\text{gr}(W'_\chi)\cong\text{gr}(W^{\text{fin}}(\ggg,e))$, and then $W'_\chi\cong W^{\text{fin}}(\mathfrak{g},e)$ as $\mathbb{C}$-algebras.

From now on, we will consider quantum finite $W$-superalgebra $W^{\text{fin}}(\ggg,e)$ as refined $W$-superalgebra $W_\chi'$, which will cause no confusion.

\subsection{A variation of the definition of $U(\ggg,e)$}\label{2.4}
In the sequel arguments, we need to vary the definition of finite $W$-superalgebras $U(\ggg,e)$ for the convenience of arguments.
\subsubsection{}\label{2.4.1}
As we have addressed in the previous subsection, the refined $W$-superalgebra $W'_\chi$ coincides with the finite $W$-superalgebra $U(\ggg,e)$ if $\sfr$ is even. Therefore, by the discussion in \S\ref{2.3} we can take quantum finite $W$-superalgebra $W^{\text{fin}}(\ggg,e)$ as $U(\ggg,e)$ when $\sfr$ is even. However, the situation for $\sfr$ being odd is much more difficult.
From now till the end of this section, we will always assume that $\sfr$ is odd, and the final statement will be given in Proposition \ref{fiiso}.
\subsubsection{}
Let $\mathfrak{l}$ be the ${\bbc}$-span of $u_{1},\cdots,u_{s}$ and $v_{1},\cdots,v_{\frac{\sfr-1}{2}},v_{\frac{\sfr+3}{2}},\cdots,v_\sfr$. It is immediate that $\ggg(-1)=\mathfrak{l}\oplus\mathbb{C}v_{\frac{\sfr+1}{2}}$ as vector space. Set \begin{equation}\label{n0}
\mathfrak{n}^0:=\bigoplus_{i\leqslant-2}{\ggg}(i)\oplus\mathfrak{l},
\end{equation} which is also a nilpotent subalgebra of $\ggg$.

\begin{defn}\label{U'}
Define the algebra $U^{\text{fin}}(\ggg,e)$ associated with the pair $(\ggg,e)$ by
\begin{equation}\label{a finw}
U^{\text{fin}}(\ggg,e):=(Q^{\text{fin}}_\chi)^{\text{ad}\,\mathfrak{n}^0}
\equiv\{\text{Pr}'(y)\in U({\ggg})/I^{\text{fin}} \mid [a,y]\in I^{\text{fin}}, \forall a\in\mathfrak{n}^0\},
\end{equation}
where $\text{Pr}':U(\ggg)\rightarrow U(\ggg)/I^{\text{fin}}$ is the canonical projection as defined in \S\ref{2.3}, and $(Q^{\text{fin}}_\chi)^{\text{ad}\,\mathfrak{n}^0}$ denotes the invariant subspace of $Q^{\text{fin}}_\chi$ under the adjoint action of $\mathfrak{n}^0$. The associated product of $(Q^{\text{fin}}_\chi)^{\text{ad}\,\mathfrak{n}^0}$ is defined by $$(x+I^{\text{fin}})\cdot(y+I^{\text{fin}}):=xy+I^{\text{fin}}$$for $x+I^{\text{fin}},y+I^{\text{fin}}\in (Q^{\text{fin}}_\chi)^{\text{ad}\,\mathfrak{n}^0}$.
\end{defn}

Obviously, $W^{\text{fin}}(\ggg,e)$ is a subalgebra of $U^{\text{fin}}(\ggg,e)$. There is also an induced Kazhdan filtration $F_{i}U^{\text{fin}}(\ggg,e)$ on the subspace $(Q^{\text{fin}}_\chi)^{\text{ad}\,\mathfrak{n}^0}$ of $Q^{\text{fin}}_\chi$ (see \S\ref{2.3}). To obtain the PBW theorem of $U^{\text{fin}}(\ggg,e)$, we need some preparation. First note that
\begin{lemma}\label{chisir}
Let $x\in\bigoplus_{i\geqslant-1}\ggg(i)$. Then $\chi([\mathfrak{n}^0,x])=0$ if and only if $x\in\ggg^e\oplus\mathbb{C}v_{\frac{\sfr+1}{2}}$.
\end{lemma}
\begin{proof}
Note that if $x\in\ggg(i)$ and $Y\in\ggg(j)$, then $\chi([Y,x])\neq0$ implies $i+j=-2$. Therefore if $x\in\mathfrak{p}$, the condition $\chi([\mathfrak{n}^0,x])=0$ implies $\chi([\ggg,x])=0$, and thus $x\in\ggg^e$. If $x\in\ggg(-1)$, it follows from $\chi([\mathfrak{n}^0,x])=0$ that $x\in\mathbb{C}v_{\frac{\sfr+1}{2}}$.
\end{proof}

%Retain the notations as in \S\ref{2.3}.
By the PBW theorem, the graded algebra $\text{gr}(U(\ggg)/I^{\text{fin}})$ under the Kazhdan grading is isomorphic to $S(\mathfrak{\mathfrak{p}}\oplus\ggg(-1))$ as vector space. The $\mathbb{Z}$-grading of $\ggg$ as defined at the beginning of \S\ref{1.2} induces a grading on $S(\mathfrak{\mathfrak{p}}\oplus\ggg(-1))$.
For any $X\in S(\mathfrak{\mathfrak{p}}\oplus\ggg(-1))$ we denote by
$\bar X$  the element of highest degree under the $\mathbb{Z}$-grading. Following Poletaeva-Serganova's arguments in \cite[\S2.2]{PS2}, we can prove the following statement.
\begin{prop}\label{contained in}
If $X\in\text{gr}(U^{\text{fin}}(\ggg,e))$, then $\bar X\in S(\ggg^e\oplus\mathbb{C}v_{\frac{\sfr+1}{2}})$.
\end{prop}
\begin{proof}
Let $X\in\text{gr}(U^{\text{fin}}(\ggg,e))$. Passing to the graded version of \eqref{a finw}, for any $Y\in\mathfrak{n}^0$ we have
\begin{equation}\label{pryx}
\text{Pr}'([Y,X])=0.
\end{equation}
Define $\gamma:\mathfrak{n}^0\otimes S(\mathfrak{\mathfrak{p}}\oplus\ggg(-1))\rightarrow S(\mathfrak{\mathfrak{p}}\oplus\ggg(-1))$ by putting
$\gamma(Y,Z)=\text{Pr}'([Y,Z])$
for all $Y\in\mathfrak{n}^0, Z\in S(\mathfrak{\mathfrak{p}}\oplus\ggg(-1))$. It is easy to see that if $Y\in\ggg(-i)$ with $i>0$, and $Z\in S(\mathfrak{\mathfrak{p}}\oplus\ggg(-1))(j)$, then $\gamma(Y,Z)\in S(\mathfrak{\mathfrak{p}}\oplus\ggg(-1))(j-i)\oplus S(\mathfrak{\mathfrak{p}}\oplus\ggg(-1))(j-i+2)$. Hence we can write $\gamma=\gamma_0+\gamma_2$ where $\gamma_0(Y,Z)$ is the projection on $S(\mathfrak{\mathfrak{p}}\oplus\ggg(-1))(j-i)$ and $\gamma_2(Y,Z)$ is the projection on $S(\mathfrak{\mathfrak{p}}\oplus\ggg(-1))(j-i+2)$. The condition \eqref{pryx} implies that for any $X\in\text{gr}(U^{\text{fin}}(\ggg,e))$,
\begin{equation}\label{gamma2}
\gamma_2(\mathfrak{n}^0,\bar X)=0.
\end{equation}
On the other hand, $\mathfrak{n}^0\times S(\mathfrak{\mathfrak{p}}\oplus\ggg(-1))\rightarrow S(\mathfrak{\mathfrak{p}}\oplus\ggg(-1))$ is a derivation with respect to the second argument defined by the condition $\gamma_2(Y,Z)=\chi([Y,Z])$
for any $Y\in\mathfrak{n}^0, Z\in\mathfrak{\mathfrak{p}}\oplus\ggg(-1)$. Now by induction on the polynomial degree of $\bar X$
in $S(\mathfrak{\mathfrak{p}}\oplus\ggg(-1))$, using Lemma \ref{chisir}, one can show that \eqref{gamma2} implies $\bar X\in\ggg^e\oplus\mathbb{C}v_{\frac{\sfr+1}{2}}$.
\end{proof}

%A refined PBW theorem
\subsubsection{}
%We are going to introduce the PBW theorem for $U^{\text{fin}}(\ggg,e)$.
Choose homogeneous elements $x_1,\cdots,x_l,x_{l+1},\cdots,x_m\in{\ppp}_{\bar{0}}, y_1,\cdots, y_q, y_{q+1},\cdots$,\\$y_n\in{\ppp}_{\bar{1}}$ as a basis of $\ppp$ such that
\begin{itemize}
\item[(1)] $x_i\in{\ggg}(k_i)_{\bar{0}}, y_j\in{\ggg}(k'_j)_{\bar{1}}$, where $k_i,k'_j\in{\bbz}_+$ with $1\leqslant i\leqslant m$ and $1\leqslant j\leqslant n$;
\item[(2)] $x_1,\cdots,x_l$ is a basis of ${\ggg}^e_{\bar{0}}$ and $y_1,\cdots,y_q$ is a basis of ${\ggg}^e_{\bar{1}}$;
\item[(3)] $x_{l+1},\cdots,x_m\in[f,{\ggg}_{\bar{0}}]$ and
$y_{q+1},\cdots,y_n\in[f,{\ggg}_{\bar{1}}]$.
\end{itemize}
%$x_1,\cdots,x_l$ is a basis of ${\ggg}^e_{\bar{0}}$ and $y_1,\cdots,y_q$ is a basis of ${\ggg}^e_{\bar{1}}$, respectively.
Also recall the bases $\{u_1,\cdots,u_{s}\}$ of ${\ggg}(-1)_{\bar0}$ and $\{v_1,\cdots,v_\sfr\}$ of ${\ggg}(-1)_{\bar1}$ as defined in \S\ref{1.2}.

Keep the notations as in \S\ref{innao}. Given $(\mathbf{a},\mathbf{b},\mathbf{c},\mathbf{d})\in{\bbz}^m_+\times\Lambda_n\times{\bbz}^{s}_+\times\Lambda_{\sfr}$, let $x^\mathbf{a}y^\mathbf{b}u^\mathbf{c}v^\mathbf{d}$ denote the monomial $x_1^{a_1}\cdots x_m^{a_m}y_1^{b_1}\cdots y_n^{b_n}u_1^{c_1}\cdots u_{s}^{c_{s}}v_1^{d_1}\cdots v_{\sfr}^{d_{\sfr}}$ in $U({\ggg})$.  It is obvious that the ${\ggg}$-module $Q^{\text{fin}}_\chi$  has a free basis $\{x^\mathbf{a}y^\mathbf{b}u^\mathbf{c}v^\mathbf{d}\otimes1_\chi\mid(\mathbf{a},\mathbf{b},\mathbf{c},\mathbf{d})\in{\bbz}^m_+\times\Lambda_n\times{\bbz}^{s}_+\times\Lambda_{\sfr}\}$. %Given $(\mathbf{a},\mathbf{b},\mathbf{c},\mathbf{d})\in{\bbz}^m_+\times\Lambda_n\times
%{\bbz}^{s}_+\times\Lambda_{\sfr}$,
Write
\begin{equation*}
|(\mathbf{a},\mathbf{b},\mathbf{c},\mathbf{d})|_e:=\sum_{i=1}^ma_i(k_i+2)+\sum_{i=1}^nb_i(k'_i+2)+\sum_{i=1}^{s}c_i+\sum_{i=1}^{\sfr}d_i,
\end{equation*} which is exactly the Kazhdan degree of $x^{\mathbf{a}}y^\mathbf{b}u^\mathbf{c}v^\mathbf{d}$. %Write $\text{deg}_e(x^{\mathbf{a}}y^\mathbf{b}u^\mathbf{c}v^\mathbf{d}):=|(\mathbf{a},\mathbf{b},\mathbf{c},\mathbf{d})|_e$.

Set
\begin{equation}\label{Y_i}
Y_i:=\left\{
\begin{array}{ll}
x_i&\text{if}~1\leqslant  i\leqslant  l;\\
y_{i-l}&\text{if}~l+1\leqslant  i\leqslant  l+q;\\
v_{\frac{\sfr+1}{2}}&\text{if}~i=l+q+1.
\end{array}
\right.
\end{equation}
To simplify  notations, we always assume that $Y_i$ belongs to ${\ggg}(m_i)$ for $1\leqslant  i\leqslant  l+q$, and $Y_{l+q+1}\in\ggg(-1)_{\bar1}$ by our earlier settings.

%\subsubsection{Quantum finite $W$-superalgebras}

\subsubsection{} We are in a position to introduce the PBW theorem of $U^{\text{fin}}(\ggg,e)$.
\begin{theorem}\label{PBWQC} The following statements  concerning the PBW structure of $U^{\text{fin}}(\ggg,e)$ hold.
\begin{itemize}
\item[(1)] There exist homogeneous elements $\Theta_1,\cdots,\Theta_{l}\in U^{\text{fin}}(\ggg,e)_{\bar0}$ and $\Theta_{l+1},\cdots,\Theta_{l+q+1}\in U^{\text{fin}}(\ggg,e)_{\bar1}$ such that
\begin{equation}\label{m'}
\begin{split}
\Theta_k=&\bigg(Y_k+\sum\limits_{\mbox{\tiny $\begin{array}{c}|\mathbf{a},\mathbf{b},\mathbf{c},\mathbf{d}|_e=m_k+2,\\|\mathbf{a}|
+|\mathbf{b}|+|\mathbf{c}|+|\mathbf{d}|\geqslant 2\end{array}$}}\lambda^k_{\mathbf{a},\mathbf{b},\mathbf{c},\mathbf{d}}x^{\mathbf{a}}
y^{\mathbf{b}}u^{\mathbf{c}}v^{\mathbf{d}}\\&+\sum\limits_{|\mathbf{a},\mathbf{b},\mathbf{c},\mathbf{d}|_e<m_k+2}\lambda^k_{\mathbf{a},\mathbf{b},\mathbf{c},\mathbf{d}}x^{\mathbf{a}}
y^{\mathbf{b}}u^{\mathbf{c}}v^{\mathbf{d}}\bigg)\otimes1_\chi
\end{split}
\end{equation}
for $1\leqslant  k\leqslant  l+q$ with $\lambda^k_{\mathbf{a},\mathbf{b},\mathbf{c},\mathbf{d}}\in\bbq$, where $\lambda^k_{\mathbf{a},\mathbf{b},\mathbf{c},\mathbf{d}}=0$ if ${\bf c=d=0}$ and $a_{l+1}=\cdots=a_m=b_{q+1}=\cdots=b_n=0$,
and $\Theta_{l+q+1}=v_{\frac{\sfr+1}{2}}\otimes1_\chi$.
\item[(2)] The monomials $\Theta_1^{a_1}\cdots\Theta_l^{a_l}\Theta_{l+1}^{b_1}\cdots\Theta_{l+q+1}^{b_{q+1}}$ with $a_i\in\mathbb{Z}_+,~b_j\in\Lambda_1$ for $1\leqslant i\leqslant l$ and $1\leqslant j\leqslant q+1$ form a basis of $U^{\text{fin}}(\ggg,e)$ over $\mathbb{C}$.
\item[(3)] For $i,j$ satisfying $1\leqslant  i<j\leqslant  l+q+1$ and $l+1\leqslant  i=j\leqslant  l+q+1$, there exist
polynomial superalgebras
$F_{i,j}\in\bbq[X_1,\cdots,X_l;X_{l+1},\cdots,X_{l+q+1}]$ with $X_1,\cdots,X_l$ being even and $X_{l+1},\cdots,X_{l+q+1}$ being odd,
such that \begin{equation}\label{fijrelations}
[\Theta_i,\Theta_{j}]=F_{i,j}(\Theta_1,\cdots,\Theta_{l+q+1}),
\end{equation}
where
\begin{equation}\label{Theta1}
F_{i,l+q+1}(\Theta_1,\cdots,\Theta_{l+q+1})=0,\qquad F_{l+q+1,l+q+1}(\Theta_1,\cdots,\Theta_{l+q+1})=1
\end{equation}for $1\leqslant  i\leqslant  l+q$.
Moreover, if the elements $Y_i,~Y_j\in {\ggg}^e$ with $1\leqslant i,j\leqslant l+q$ satisfy $[Y_i,Y_j]=\sum\limits_{k=1}^{l+q}\alpha_{ij}^kY_k$ in ${\ggg}^e$, then
\begin{equation}\label{Theta2}
\begin{array}{llllll}
&F_{i,j}(\Theta_1,\cdots,\Theta_{l+q+1})&\equiv&\sum\limits_{k=1}^{l+q}\alpha_{ij}^k\Theta_k+q_{ij}(\Theta_1,\cdots,\Theta_{l+q+1})&(\mbox{mod }\text{F}_{m_i+m_j+1}U^{\text{fin}}(\ggg,e)),
\end{array}
\end{equation}where $q_{ij}$ is a polynomial superalgebra in $l+q+1$ variables in $\bbq$ whose constant term and linear part are zero. %, and $\text{F}_{m_i+m_j+1}U^{\text{fin}}(\ggg,e)$ denotes the component of Kazhdan filtration of $U(\ggg,e)$ with degree $m_i+m_j+1$.
\item[(4)] The algebra $U^{\text{fin}}(\ggg,e)$ is generated by the $\mathbb{Z}_2$-homogeneous elements $\Theta_1,\cdots,\Theta_{l}\in U^{\text{fin}}(\ggg,e)_{\bar0}$ and $\Theta_{l+1},\cdots,\Theta_{l+q+1}\in U^{\text{fin}}(\ggg,e)_{\bar1}$ subjected to the relations in \eqref{fijrelations}
with $1\leqslant  i<j\leqslant  l+q+1$ and $l+1\leqslant  i=j\leqslant  l+q+1$.
\end{itemize}
\end{theorem}
\begin{proof}
Since $\mathfrak{n}^0\subset \mathfrak{n}'$ by definition, then $W^{\text{fin}}(\ggg,e)=(Q_\chi^{\text{fin}})^{\text{ad}\,\mathfrak{n}'}$ is a subalgebra of $U^{\text{fin}}(\ggg,e)=(Q_\chi^{\text{fin}})^{\text{ad}\,{\mathfrak{n}}^0}$. In virtue of \cite[Theorems 3.7, 4.11]{ZS5}, we can choose the elements $\Theta_k$ for $1\leqslant  k\leqslant  l+q$ as in \eqref{m'}, to be the generators of $W^{\text{fin}}(\ggg,e)$ (all the coefficients $\lambda^k_{\mathbf{a},\mathbf{b},\mathbf{c},\mathbf{d}}\in\bbq$ can be assured by the knowledge of field theory as in the proof of \cite[Theorem 4.6]{P2}). Moreover, it follows from the definition of $\mathfrak{n}^0$ in \eqref{n0} that
\begin{equation*}
[\mathfrak{n}^0,\Theta_{l+q+1}]=[\mathfrak{n}^0,v_{\frac{\sfr+1}{2}}\otimes1_\chi]=0.
\end{equation*}
 So all the elements in \eqref{m'}, together with $\Theta_{l+q+1}=v_{\frac{\sfr+1}{2}}\otimes1_\chi$, belong to $U^{\text{fin}}(\ggg,e)$. Thanks to Proposition \ref{contained in}, we see that all these elements constitute a set of generators of $U^{\text{fin}}(\ggg,e)$.

On one hand, since $\Theta_k\in(Q_\chi^{\text{fin}})^{\text{ad}\,\mathfrak{n}'}$ for $1\leqslant  k\leqslant  l+q$, and $v_{\frac{\sfr+1}{2}}\in\mathfrak{n}'$ by definition, we have
\begin{equation}\label{kl+q+1}
[\Theta_k,\Theta_{l+q+1}]=[\Theta_k,v_{\frac{\sfr+1}{2}}\otimes1_\chi]=0,
\end{equation}
which implies $F_{i,l+q+1}(\Theta_1,\cdots,\Theta_{l+q+1})=0$ in \eqref{Theta1}.
On the other hand, by assumption we have
\begin{equation*}
[\Theta_{l+q+1},\Theta_{l+q+1}]=[v_{\frac{\sfr+1}{2}}\otimes1_\chi,v_{\frac{\sfr+1}{2}}\otimes1_\chi]
=[v_{\frac{\sfr+1}{2}},v_{\frac{\sfr+1}{2}}]\otimes1_\chi=1\otimes1_\chi.
\end{equation*}
Thus $F_{l+q+1,l+q+1}(\Theta_1,\cdots,\Theta_{l+q+1})=1$, which is just the second equation in \eqref{Theta1}. The remainder of the theorem can be obtained by the same discussion as in \cite[Theorems 4.5, 4.7]{ZS2}. %from \cite[Theorems 3.7, 4.11]{ZS5}.
\end{proof}
%By the analogous arguments of \cite[Lemma 4.1]{P4}, one can obtain the following result.
%\begin{theorem}\label{relationc}
%The algebra $U^{\text{fin}}(\mathfrak{g},e)$ is generated by the $\mathbb{Z}_2$-homogeneous elements $\Theta_1,\cdots,\Theta_{l}\in U^{\text{fin}}(\mathfrak{g},e)_{\bar0}$ and $\Theta_{l+1},\cdots,\Theta_{l+q+1}\in U^{\text{fin}}(\mathfrak{g},e)_{\bar1}$ subject to the relations
%$$[\Theta_i,\Theta_j]=F_{i,j}(\Theta_1,\cdots,\Theta_{l+q+1})$$
%with $1\leqslant  i<j\leqslant  l+q+1$ and $l+1\leqslant  i=j\leqslant  l+q+1$, where the $F_{i,j}$'s are defined as in \eqref{Theta1} and \eqref{Theta2}.
%\end{theorem}

As an immediate consequence of Theorem \ref{PBWQC}, we have
\begin{corollary}\label{ufin}Under the Kazhdan grading,
$\text{gr}(U^{\text{fin}}(\mathfrak{g},e))\cong S(\ggg^e)\otimes\mathbb{C}[\Lambda]$ as $\mathbb{C}$-algebras, where $\Lambda$ is the exterior algebra generated by one element $\Lambda$.
\end{corollary}
\subsubsection{}
Now we turn to finite $W$-superalgebra $U({\ggg},e)\cong Q_{\chi}^{\ad\,\mmm}$. In \cite[Theorem 0.1]{ZS2}, we showed that
\begin{prop}\label{graded W}(\cite{ZS2})
When $\sfr$ is odd,
$\text{gr}(U(\ggg,e))\cong S(\ggg^e)\otimes\mathbb{C}[\Lambda]$ as vector space under the Kazhdan grading, where $\Lambda$ is the exterior algebra generated by $\Lambda$.
\end{prop}

We will improve Proposition \ref{graded W}, by adopting a new approach which is different from the one used by Shu-Xiao under the settings of Poisson geometric realization of finite $W$-superalgebras in \cite[Corollary 3.9(2)]{SX}. We also refer to \cite[\S2.2]{PS2} for more details.

Recall that
we introduced the PBW theorem of refined $W$-superalgebra $W'_\chi$ in \cite[Theorem 3.7]{ZS5}.
Since $W'_\chi$ is a subalgebra of $U(\mathfrak{g},e)$, by \cite[Lemma 4.3]{ZS2} %and Proposition \ref{graded W}
we can choose the elements given in \cite[Theorem 3.7(1)]{ZS5} and also $\Theta_{l+q+1}=v_{\frac{\sfr+1}{2}}\otimes1_\chi$ as the generators of $U(\mathfrak{g},e)$. In particular, such a choice ensures the equation \eqref{Theta1} still valid in this case.
Now the same discussion as  the proofs of Theorem \ref{PBWQC} and Corollary \ref{ufin},
%thanks to \cite[Theorems 4.5, 4.7]{ZS2},
one can conclude that %the vector spaces isomorphism in Proposition \ref{graded W} is in fact an algebras isomorphism (see also \cite[\S2.2]{PS2} for more details), i.e.,
\begin{theorem}\label{algiso4.6}(\cite{SX})
The isomorphism of vector spaces in Proposition \ref{graded W} is in fact an isomorphism of $\bbc$-algebras.
\end{theorem}
\begin{rem}\label{alphanotinbbq}
%For the original version of Proposition \ref{graded W}, the case for $\ggg=D(2,1;\alpha)$ with $\alpha\not\in\overline\bbq$ is excluded since we employed
%the ``admissible" procedure from the modular finite $W$-superalgebras there. Therefore, we can not retrieve Theorem \ref{algiso4.6} as in \cite[Corollary 3.9(2)]{SX} for this case.
In the proof of Theorem \ref{algiso4.6} as above, we can observe that the construction of the generators of $W'_\chi$ introduced in \cite[Theorem 3.7]{ZS5} plays a key role. In fact, in the procedure of formulating the PBW theorem for  $W'_\chi$ there, since
the ``admissible" procedure from the modular finite $W$-superalgebras is employed,  we have always assumed that the associated $\ggg$ is a basic Lie superalgebra excluding the case of $D(2,1;\alpha)$ with $\alpha\not\in\overline\bbq$. However, this does not affect the proof here.  As  mentioned at the end of \S\ref{2.3}, the case of $\sfr$ being odd appears only when $\ggg$ is of type $B(m|n)$ or $G(3)$. Then  $\ggg$ can not be of type $D(2,1;\alpha)$.
%It is remarkable that all the discussion in this subsection are concentrated on the case,
%for which $U(\ggg,e)$ contains $W'_\chi$ as a proper subalgebra. Moreover, the formulating of the PBW theorem for $U(\ggg,e)$ relies heavily on the structure of $W'_\chi$. On one hand, since we employed
%the ``admissible" procedure from the modular finite $W$-superalgebras in the procedure of formulating the PBW theorem for  $W'_\chi$ in \cite{ZS5}, we always assume that $\ggg$ is a basic Lie superalgebra excluding the case of $D(2,1;\alpha)$ with $\alpha\not\in\overline\bbq$.
%On the other hand,  if  $\sfr$ is odd, by definition there must exist a non-isotropic $\alpha\in\Phi_{\bar1}^+$ such that $2\alpha\in\Phi_{\bar0}^+$.
%From the detailed description of the system of root $\Phi$ of $\ggg$ in \cite[\S2.5.4]{K}, this happens only when $\ggg$ is of type $B(m|n)$ or $G(3)$. Therefore, we need not to consider the case with $\ggg=D(2,1;\alpha)$ with $\alpha\not\in\overline\bbq$ in this subsection.
\end{rem}

Combining Theorem \ref{algiso4.6} with Corollary \ref{ufin}, we now obtain
\begin{prop}\label{griso}
When $\sfr$ is odd, $\text{gr}(U^{\text{fin}}(\mathfrak{g},e))\cong\text{gr}(U(\ggg,e))$ as $\mathbb{C}$-algebras under the Kazhdan grading.
\end{prop}

Translating Proposition \ref{griso} into the corresponding Kazhdan-filtrated algebras, we have

\begin{prop}\label{fiiso}
When $\sfr$ is odd, there is an isomorphism  $U^{\text{fin}}(\mathfrak{g},e)\cong U(\ggg,e)$ as $\mathbb{C}$-algebras.
\end{prop}

Owing to Proposition \ref{fiiso}, {\it from now on we will regard $U^{\text{fin}}(\mathfrak{g},e)$ in Definition \ref{U'} as finite $W$-superalgebra $U(\ggg,e)$, i.e., $U(\ggg,e)\triangleq U^{\text{fin}}(\mathfrak{g},e)$.}
\begin{rem}
%take $Q^{\text{fin}}_\chi$ as the generalized Gelfand-Graev module $Q_\chi$ in \eqref{generalized Gelfand-Graev}, and consider the
%the quantum finite $W$-superalgebra $W^{\text{fin}}(\ggg,e)$ in Definition \ref{quafi} as refined $W$-superalgebra $W_\chi'$ introduced in Definition \ref{rewcc}, which will cause no confusion.
%
%
By all the discussion above, Theorem \ref{PBWQC} can be considered as the PBW theorem of finite $W$-superalgebra $U(\ggg,e)$. Compared with \cite[Theorem 4.5]{ZS2}, %besides the generators $\Theta_{k}$ for $1\leqslant k\leqslant l+q+1$ are chosen differently,
the most important difference lies in \eqref{Theta1}, where $F_{i,l+q+1}(\Theta_1,\cdots,\Theta_{l+q+1})$ for $1\leqslant  i\leqslant  l+q$ can not be easily determined in \cite[(4.2)]{ZS2}, while $F_{i,l+q+1}(\Theta_1,\cdots,\Theta_{l+q+1})=0$ in our case, and such a choice makes the construction of $U(\ggg,e)$ much easier to determine.
\end{rem}

For further discussion, we need the following ring-theoretic property of finite and refined $W$-superalgebras, which is parallel to the non-super case in \cite[Lemma 1.1(2)]{SZc}.
\begin{prop}\label{zero-divisor}
Both $U(\ggg,e)$ and $W_\chi'$ are Noetherian rings.
\end{prop}
\begin{proof}
%By definition both $U(\ggg,e)$ and $W_\chi'$ are sub-quotients of the Noetherian ring $U(\ggg)$, then (1) follows.
Under the Kazhdan grading, we have shown in Theorem \ref{algiso4.6} that $\text{gr}(U(\ggg,e))\cong S(\ggg^e)\otimes\mathbb{C}[\Lambda]$ as $\bbc$-algebra, and $\text{gr}(W'_\chi)\cong S(\ggg^e)$ by \cite[Corollary 3.8]{ZS5}. So the gradation of $U(\ggg,e)$ and $W_\chi'$ are isomorphic to polynomial superalgebras. Then it follows from \cite[Theorem 1.6.9]{MR} that both the filtration algebras $U(\ggg,e)$ and $W_\chi'$ are Noetherian, as the Noetherian property hold for their associated graded algebras.
%The detailed proof will be omitted here.
%{\color{blue}Thanks to Theorem \ref{PBWQC}, the finite $W$-superalgebra $U(\ggg,e)$ is isomorphic to a factor algebra of the free associative superalgebra $\bbc\langle T_1,\cdots,T_l;T_{l+1},\cdots,T_{l+q+1}\rangle$ subjected to the two-sided ideal $I$ generated by all $[T_i,T_j]-F_{i,j}(T_1,\cdots,T_{l+q+1})$ with $1\leqslant  i<j\leqslant  l+q+1$ and $l+1\leqslant  i=j\leqslant  l+q+1$. By the knowledge of ring theory (e.g., \cite[Corollary 7.6 and Proposition 7.1]{AM}) we see that $U(\ggg,e)$ is a Noetherian ring. By the same discussion as above, the Noetherian property of $W_\chi'$ follows from \cite[Theorems 3.7]{ZS5}.}
\end{proof}
\section{Verma modules and isomorphism classes of their irreducibles quotients for minimal finite and refined $W$-superalgebras}\label{3}
In this section we will study Verma modules for minimal finite $W$-superalgebra $U(\ggg,e)$.
\subsection{}\label{2.1}
We first recall some basics on minimal finite and refined $W$-(super)algebras in this part. We refer the readers to \cite{AKFPP,AKFPP2,ACKL,ArM,KFP,P3,suh,ZS5}.

\subsubsection{}\label{3.1.1}
A root $-\theta$ is called minimal if it is even and there exists an additive function $\varphi:\Phi\rightarrow\mathbb{R}$ such that $\varphi_{\mid{\Phi}}\neq0$ and $\varphi(\theta)>\varphi(\eta)$ for all $\eta\in\Phi\backslash\{\theta\}$. In the ordering defined by $\varphi$, a minimal root $-\theta$ is the lowest root %of one of the simple components
of $\ggg_{\bar0}$. %, and therefore the adjoint orbit is the unique nonzero nilpotent orbit of minimal dimension in
%a simple component of $\ggg_{\bar0}$. %In particular, $\theta$ is an even simple root in $\Delta$.
Conversely, it is easy to see, using the description of $\Phi$ given in \cite{K}, that a long root %of any simple component
of $\ggg_{\bar0}$ (with respect to the bilinear form $(\cdot,\cdot)$) is minimal except when $\ggg=\mathfrak{osp}(3|n)$ and the simple component of $\ggg_{\bar0}$ is $\mathfrak{so}(3)$.

For a fixed minimal root $-\theta$, we can choose a simple root system $\Delta$ of $\Phi$ such that  it  contains
$\frac{\theta}{2}$ whenever  this guy is a root. Otherwise, we can choose a simple root system such that it contains $\theta$ (see, e.g., \cite[Theorem 3.10]{PS}).
Thus we can make the following convention for our later arguments.
\begin{conventions}\label{conventions}
For the minimal root $-\theta$ of $\ggg$, fix a simple root system $\Delta=\{\alpha_1,\cdots,\alpha_k\}$ satisfying that $\alpha_k=\theta$ when $\sfr$ is even,
and $\alpha_k=\frac{\theta}{2}$ when $\sfr$ is odd. %satisfy that $\theta=\alpha_k$ is an even simple root in the associated simple root system $\Delta=\{\alpha_1,\cdots,\alpha_k\}$ when $\sfr$ is even,
%and $\frac{\theta}{2}=\alpha_k$ is an odd simple root in $\Delta$ when $\sfr$ is odd.
%the associated simple root system $\Delta=\{\alpha_1,\cdots,\alpha_k\}$ satisfies that $\alpha_k=\theta$ when $\sfr$ is even, and $\alpha_k=\frac{\theta}{2}$ when $\sfr$ is odd.
\end{conventions}

Fix a minimal root $-\theta$ of $\ggg$, we may choose root vectors $e:=e_\theta$ and $f:=e_{-\theta}$ such
that $$[e,f]=h:=h_\theta\in\mathfrak{h},\quad[h,e]=2e,\quad[h,f]=-2f.$$ As $(e,f)=1$ by our earlier assumption, we %can normalize the invariant bilinear form $(\cdot,\cdot)$ on $\ggg$ by the condition
have $(\theta,\theta)=2$.
It is well-known that the eigenspace decomposition of $\text{ad}\,h$ gives rise to a short $\mathbb{Z}$-grading
\begin{equation}\label{short}
\ggg=\ggg(-2)\oplus\ggg(-1)\oplus\ggg(0)\oplus\ggg(1)\oplus\ggg(2).
\end{equation}
Moreover, $\ggg(2)=\mathbb{C}e$ and $\ggg(-2)=\mathbb{C}f$, with $\ggg(1)\oplus\ggg(2)$ and $\ggg(-1)\oplus\ggg(-2)$ being Heisenberg Lie superalgebras.
We thus have a bijective correspondence between short $\mathbb{Z}$-gradings (up to an automorphism of $\ggg$) and minimal roots (up to the
action of the Weyl group). Furthermore, one has
$$\ggg^e=\ggg(0)^\sharp\oplus\ggg(1)\oplus\ggg(2),$$where  $\ggg(0)^\sharp=\ggg^e(0)=\{x\in\ggg(0)\mid[x,e]=0\}$.
Note that $\ggg(0)^\sharp$ is the centralizer of the triple $(e,f,h)$ by $\mathfrak{sl}(2)$-theory.
Moreover, $\ggg(0)^\sharp$ is the orthogonal complement to $\mathbb{C}h$ in $\ggg(0)$, and coincides with the image of the Lie superalgebra endomorphism
\begin{equation}\label{xh}
\sharp: \ggg(0)\rightarrow\ggg(0), x\mapsto x-\frac{1}{2}(h,x)h.
\end{equation} Obviously $\ggg(0)^{\sharp}$ is an ideal of codimensional $1$ in the Levi subalgebra $\ggg(0)$.

Denote by $\Phi_e$ the set of all $\alpha\in\Phi$ with $\alpha(h)\in\{0,1\}$, and write $\Phi_e^\pm:=\Phi_e\cap\Phi^\pm, \Phi^\pm_{e,i}:=\{\alpha\in\Phi_e^\pm\mid\alpha(h)=i\}$ for $i=0,1$, $(\Phi^+_{e,0})_{\bar0}:=\Phi^+_{e,0}\cap\Phi_{\bar0}$.
Write $\mathfrak{h}^e:=\mathfrak{h}\cap\ggg^e$, a Cartan subalgebra in $\ggg(0)^{\sharp}$.
Then $\ggg^e$ is spanned by $\mathfrak{h}^e\bigcup\{e_\alpha\:|\:\alpha\in\Phi_e\}\bigcup\{e\}$. Note that the restrictions of  $(\cdot,\cdot)$ to $\ggg(0)^{\sharp}$ and $\mathfrak{h}^e$ are both non-degenerate. Take a basis $\{h_1,\cdots,h_{k-1}\}$ of $\mathfrak{h}^e$ such that $(h_i,h_j)=\delta_{i,j}$ for $1\leqslant i,j\leqslant k-1$, and denote by $\mathfrak{n}^\pm(i)$ the span of all $e_\alpha$ with $\alpha\in\Phi_{e,i}^\pm$. Then $\mathfrak{n}^+(0)$ and $\mathfrak{n}^-(0)$ are maximal subalgebras of $\ggg(0)^{\sharp}$. Take bases $\{x_1,\cdots,x_w\}$ and
$\{x^*_1,\cdots,x^*_w\}$ of $\mathfrak{n}^-_{\bar0}(0)$ and $\mathfrak{n}^+_{\bar0}(0)$ respectively, %consisting of root vectors $e_\alpha$ with $\alpha\in\Phi_{\bar0}$
such that $(x_i,x_j)=(x_i^*,x_j^*)=0$ and $(x^*_i,x_j)=\delta_{i,j}$ for $1\leqslant i,j\leqslant w$.
Furthermore, we can assume that each $x_i$ (resp. $x_i^*$) with $1\leqslant i\leqslant w$ is a root vector for $\mathfrak{h}$ corresponding to $-\beta_{\bar0i}\in\Phi_{\bar0}^-$ (resp. $\beta_{\bar0i}\in\Phi_{\bar0}^+$). Set
$\{y_1,\cdots,y_\ell\}$ and $\{y^*_1,\cdots,y^*_\ell\}$ to be bases of $\mathfrak{n}^-_{\bar1}(0)$ and $\mathfrak{n}^+_{\bar1}(0)$ %consisting of root vectors $e_\alpha$ with $\alpha\in\Phi_{\bar1}$
such that $(y_i,y_j)=(y_i^*,y_j^*)=0$ and $(y^*_i,y_j)=\delta_{i,j}$ for $1\leqslant i,j\leqslant\ell$.
We also assume that each $y_i$ (resp. $y_i^*$) with $1\leqslant i\leqslant \ell$ is a root vector for $\mathfrak{h}$ corresponding to $-\beta_{\bar1i}\in\Phi_{\bar1}^-$ (resp. $\beta_{\bar1i}\in\Phi_{\bar1}^+$).
Recall in \S\ref{1.2} we have assumed that $u_i, u_j^*$ (resp. $v_i, v_j^*$) with $1\leqslant i,j\leqslant \frac{s}{2}$ (resp. $1\leqslant i,j\leqslant \lceil\frac{\sfr}{2}\rceil$) are elements in $\ggg(-1)_{\bar0}$ (resp. $\ggg(-1)_{\bar1}$).
For $1\leqslant i\leqslant\frac{s}{2}$, set $f_i=[e,u_i]$ and $f^*_i=[e,u_i^*]$, then
$f_i$ (resp. $f^*_i$) is a root vector for $\mathfrak{h}$ corresponding to the root $\theta+\gamma_{\bar0i}\in\Phi_{e,1}^-$ (resp. $\theta+\gamma^*_{\bar0i}\in\Phi_{e,1}^+$).
For $1\leqslant i\leqslant \lceil\frac{\sfr}{2}\rceil$, write $g_i=[e,v_i]$ and $g^*_i=[e,v_i^*]$, then $g_i$ (resp. $g^*_i$) is a root vector for $\mathfrak{h}$ corresponding to the root $\theta+\gamma_{\bar1i}\in\Phi_{e,1}^-$ (resp. $\theta+\gamma^*_{\bar1i}\in\Phi_{e,1}^+$). When $\sfr=\dim{\ggg}(-1)_{\bar1}$ is odd,
by our discussion in \S\ref{1.2} the elements $v_{\frac{\sfr+1}{2}}$ and $[v_{\frac{\sfr+1}{2}},e]$ are  root vectors corresponding to negative root $-\frac{\theta}{2}\in\Phi_{\bar1}^-$ and positive root $\frac{\theta}{2}\in\Phi_{e,1}^+$, respectively.

%For further discussion, we will always assume that $\theta=\alpha_k$ is an even simple root in $\Delta=\{\alpha_1,\cdots,\alpha_k\}$ when $\sfr$ is even,
%and $\frac{\theta}{2}=\alpha_k$ is an odd simple root in $\Delta$ when $\sfr$ is odd in the paper.
\subsubsection{}\label{3.2}
Since $\ggg^e(0)=\mathfrak{h}^e\oplus\mathfrak{n}^-_{\bar0}(0)\oplus\mathfrak{n}^+_{\bar0}(0)\oplus\mathfrak{n}^-_{\bar1}(0)\oplus\mathfrak{n}^+_{\bar1}(0)$ as vector space, then \begin{equation}\label{basisofge0}
\{h_1,\cdots,h_{k-1},x_1,\cdots,x_w,x^*_1,\cdots,x^*_w,y_1,\cdots,y_\ell,y^*_1,\cdots,y^*_\ell\}
\end{equation} is a base of $\ggg^e(0)$ by our earlier assumption.
%As the bilinear form $(\cdot,\cdot)$ is non-degenerate on $\ggg^e(0)$, $\{a_i\mid i\in I\}$ and $\{b_i\mid i\in I\}$  of $\ggg^e(0)$ such that $(a_i,b_j)=\delta_{i,j}$. For any given $i\in I$, since $a_i$ and $b_i$ have the same parity, we will denote this parity by $|i|$.
Now let\begin{equation}\label{C0def} C_0:=\sum_{i=1}^{k-1}h_i^2+\sum_{i=1}^{w}x_ix^*_i+\sum_{i=1}^{w}x^*_ix_i+\sum_{i=1}^{\ell}y_iy^*_i-
\sum_{i=1}^{\ell}y^*_iy_i\end{equation} be the corresponding Casimir element of $U(\ggg^e(0))$.

Recall in \cite[Proposition 1.2]{ZS5}, we introduced a set of generators for minimal refined $W$-superalgebras as below.
\begin{theorem}\label{ge}(\cite{ZS5}) Let $-\theta$ be a minimal root, and
$e=e_\theta\in\ggg$ a root vector for $\theta$. Suppose $v\in\ggg^e(0)$, $w\in\ggg^e(1)$, $C$ is a central element of $W_\chi'$, and set $s=\text{dim}\,\ggg(-1)_{\bar0}$, $\sfr=\text{dim}\,\ggg(-1)_{\bar1}$. Then the following are $\mathbb{Z}_2$-homogeneous generators of minimal refined $W$-superalgebra $W_\chi'$:
\begin{equation*}
\begin{split}
\Theta_v=&\bigg(v-\frac{1}{2}\sum\limits_{\alpha\in S(-1)}z_\alpha[z_\alpha^*,v]\bigg)\otimes1_\chi;\\
\Theta_w=&\bigg(w-\sum\limits_{\alpha\in S(-1)}z_\alpha[z_\alpha^*,w]+\frac{1}{3}\Big(\sum\limits_{\alpha,\beta\in S(-1)}z_\alpha z_\beta[z_\beta^*,[z_\alpha^*,w]]-2[w,f]\Big)\bigg)\otimes1_\chi;\\
C=&\bigg(2e+\frac{h^2}{2}-\Big(1+\frac{s-\sfr}{2}\Big)h+C_0
+2\sum\limits_{\alpha\in S(-1)}(-1)^{|\alpha|}[e,z_\alpha^*]z_\alpha\bigg)\otimes1_\chi.
\end{split}
\end{equation*}
\end{theorem}

For any $\mathbb{Z}_2$-homogeneous element $w\in\ggg$, denote by $|w|$ the parity of $w$. Set \begin{equation}\label{castheta} \Theta_{\text{Cas}}:=\sum_{i=1}^{k-1}\Theta_{h_i}^2+\sum_{i=1}^{w}\Theta_{x_i}\Theta_{x^*_i}+
\sum_{i=1}^{w}\Theta_{x^*_i}\Theta_{x_i}+\sum_{i=1}^{\ell}\Theta_{y_i}\Theta_{y^*_i}-
\sum_{i=1}^{\ell}\Theta_{y^*_i}\Theta_{y_i},\end{equation}
an element %of the associative subalgebra of $W_\chi'$ generated by the Lie superalgebra $\Theta_{\ggg^e(0)}$
commutes with all operators $\Theta_v$ for $v\in\ggg^e(0)$ (see \cite[Proposition 5.7]{ZS5} for more details).
Then the commutators between the generators in Theorem \ref{ge} are presented in \cite[Theorem 1.3]{ZS5}, i.e.,
\begin{theorem}\label{maiin1}(\cite{ZS5})
The minimal refined $W$-superalgebra $W_\chi'$ is generated by the Casimir element $C$ and the subspaces $\Theta_{\ggg^e(i)}$ for $i=0,1$, as described in Theorem \ref{ge}, subjected to the following relations:
\begin{itemize}
\item[(1)] $[\Theta_{v_1},\Theta_{v_2}]=\Theta_{[v_1,v_2]}$ for all $v_1, v_2\in\ggg^e(0)$;
\item[(2)] $[\Theta_{v},\Theta_{w}]=\Theta_{[v,w]}$ for all $v\in\ggg^e(0)$ and $w\in\ggg^e(1)$;
\item[(3)] $[\Theta_{w_1},\Theta_{w_2}]=\frac{1}{2}([w_1,w_2],f)(C-\Theta_{\text{Cas}}-c_0)-\frac{1}{2}\sum\limits_{\alpha\in S(-1)}\bigg(\Theta_{[w_1,z_\alpha]^{\sharp}}\Theta_{[z_\alpha^*,w_2]^{\sharp}}\\ \hspace*{2.2cm} -(-1)^{|w_1||w_2|}\Theta_{[w_2,z_\alpha]^{\sharp}}\Theta_{[z_\alpha^*, w_1]^{\sharp}}\bigg)$
    for all $w_1, w_2\in\ggg^e(1)$;
\item[(4)]  $[C,W_\chi']=0$.
\end{itemize}
In (3), the notation $\sharp$ is defined as in \eqref{xh}, and the constant $c_0$ is decided by the following equation:
\begin{equation}\label{decidec0}
c_0([w_1,w_2],f)=\frac{1}{12}\sum\limits_{\alpha,\beta\in S(-1)}[[[w_1, z_\alpha],z_\beta],[z_\beta^*,[z_\alpha^*,w_2]]]\otimes 1_\chi
-\frac{3(s-\sfr)+4}{12}([w_1,w_2],f),
\end{equation}
where $s=\text{dim}\,\ggg(-1)_{\bar0}$ and $\sfr=\text{dim}\,\ggg(-1)_{\bar1}$.
\end{theorem}

%In this section we will first refine and introduce a variation on the PBW theorem of finite $W$-superalgebra $U(\ggg,e)$ in \cite[Theorem 4.2]{ZS2} for the general case, and then give its explicit realization associated with minimal nilpotents. In virtue of these results, we will study highest weight modules for minimal finite $W$-superalgebra $U(\ggg,e)$.
%
%In this section we will always assume that $\sfr$ is odd. In particular, when we refer to minimal finite $W$-superalgebras, we continue to keep the conventions as in Remark \ref{types}.
%In this section,
%we will always assume that $\ggg$ is a Lie superalgebra of type II.

\subsubsection{}
\label{2.4.1}
When $\sfr$ is even, the refined $W$-superalgebra $W'_\chi$ coincides with the finite $W$-superalgebra $U(\ggg,e)$ by definition. Therefore, Theorems \ref{ge} and \ref{maiin1} can also be considered as the PBW theorem of minimal finite $W$-superalgebra $U(\ggg,e)$ in this case.

Now we assume that $\sfr$ is odd.
%In virtue of the results we obtained in \S\ref{3.1}, now we will write out the generators of minimal finite $W$-superalgebras and their relations explicitly, which will be used for the discussion on the representation of minimal finite $W$-superalgebras. %In this part we will always assume that $\text{dim}\,\ggg(-1)_{\bar1}$ is an odd number.
In fact, taking Theorems \ref{ge}, \ref{maiin1}, \ref{PBWQC} and Proposition \ref{fiiso} into consideration, we have
\begin{theorem}\label{main3}
Let $-\theta$ be a minimal root, and
$e=e_\theta\in\ggg$ a root vector for $\theta$. Suppose $v\in\ggg^e(0)$, $w\in\ggg^e(1)$. Let $\Theta_v, \Theta_w, C$ be as defined in Theorem \ref{ge}, and $\Theta_{l+q+1}$ as in Theorem \ref{PBWQC}(1). When $\sfr$ is odd, the minimal finite $W$-superalgebra $U(\ggg,e)$ is generated by $\Theta_v, \Theta_w, \Theta_{l+q+1}$, and the Casimir element $C$, subjected to the relations as in Theorem \ref{maiin1}(1)---(4), $[\Theta_{l+q+1},\Theta_v]=[\Theta_{l+q+1},\Theta_w]=[\Theta_{l+q+1},C]=0$, and also $[\Theta_{l+q+1},\Theta_{l+q+1}]=1\otimes1_\chi$.
\end{theorem}
In order to highlight the key role in which the odd element $v_{\frac{\sfr+1}{2}}$ plays, {\it from now on we will write $\Theta_{v_{\frac{\sfr+1}{2}}}$ instead of $\Theta_{l+q+1}$ in Theorems \ref{PBWQC} and \ref{main3}.}

%It should be mentioned that although all the results above are under the assumption that the nilpotent element $e_{\beta}$ corresponds to a minimal root $\beta$, one can easily prove that refined $W$-superalgebra $W_\chi'$ depends only on the adjoint orbit $G_\ev.e_{\beta}$ of $e_{\beta}$. By our earlier discussion we see that the roots $\beta$ and $\theta$ are in the same orbit under the adjoint action of $G_\ev$, so Proposition \ref{ge} and Theorem \ref{maiin1} still go through in our case with $e=e_{\theta}$.
\begin{rem}\label{types}
For the convenience of following discussion, %on relevant topics of minimal $W$-superalgebras,
we will make some conventions.
%Although we mainly discuss the related topics of minimal finite $W$-superalgebras $U(\ggg,e)$ in this section, some conclusions also apply for minimal refined $W$-superalgebra $W_\chi'$ in certain special cases.
To be explicit, from the explicit description of the system of roots $\Phi$ and of simple roots $\Delta$ relative to $\mathfrak{h}$ of $\ggg$ in \cite[\S2.5.4]{K}, and also the description of the corresponding $\ggg^e(0)$-module $\ggg(1)$ ($\cong\ggg^*(-1)$) given in
\cite[Tables 1-3]{KW} (note that $\ggg^f(0)=\ggg^e(0)$ is denoted by $\ggg^\natural$, and $\ggg(1)$ is written as $\ggg_{\frac{1}{2}}$ in their settings), taking  Convention \ref{conventions} into consideration, one can observe that
\begin{itemize}
\item[(1)] Let $\ggg$ be a Lie algebra, or a Lie superalgebra of type $A(m|n) (m\neq n)$, $A(n|n)$, $C(n)$, $D(m|n)$, $D(2,1;\alpha)$, $F(4)$,  then $\theta$ is a simple root of $\ggg$ relative to $\Phi$. Or let $\ggg$ be a certain subclass of type $B(m|n)$ or $G(3)$ such that $\theta$ is a simple root of $\ggg$ relative to $\Phi$ (see Table 2 for more details).
    Since $-\frac{\theta}{2}$ is not a root in $\Phi$, we have $\mathfrak{m}'=\mathfrak{m}$ with $\sfr$ being an even number. Then the corresponding minimal refined $W$-superalgebra $W_\chi'$ coincides with minimal finite $W$-superalgebra $U(\ggg,e)$, which will be called minimal finite (refined) $W$-superalgebra of type even.
\item[(2)] Let $\ggg$ be a certain subclass of type $B(m|n)$ or $G(3)$ such that $\frac{\theta}{2}$ is a simple root in $\Phi_{\bar1}$ (see Table 3 for more details), and then $\mathfrak{m}$ is a proper subalgebra of $\mathfrak{m}'$ with $\sfr$ being odd. So $W_\chi'$ is also a proper subalgebra of  $U(\ggg,e)$. In this case we will refer them to minimal refined $W$-superalgebra of type odd and minimal finite $W$-superalgebra of type odd, respectively.
\end{itemize}

For the convenience of readers, we will list the classification of minimal $W$-superalgebras discussed above in Tables 2 and 3.
\end{rem}
%
%
%
%
%In fact, taking the discussions in \S\ref{2.1} into consideration, we have that
%\begin{itemize}
%\item[(1)]for $\ggg$ being a Lie algebras, or a Lie superalgebra of type $A(m|n), C(n), D(m|n),$ $D(2,1;\alpha)$ or $F(4)$, as $-\frac{\theta}{2}$ is not a root in $\Phi$ we have $\mathfrak{m}'=\mathfrak{m}$ with $\sfr(=\text{dim}\,\ggg(-1)_{\bar1})$ being an even number. Then the corresponding minimal finite $W$-superalgebra $U(\ggg,e)$ coincides with minimal refined $W$-superalgebra $W_\chi'$;
%\item[(2)]for $\ggg$ being of type $B(m|n)$ or $G(3)$, we have $e_{-\frac{\theta}{2}}\neq0$, and $\mathfrak{m}$ is a proper subalgebra of $\mathfrak{m}'$ with $\sfr$ being an odd number. So minimal refined $W$-superalgebra $W_\chi'$ is also a proper subalgebra of the  corresponding minimal finite $W$-superalgebra $U(\ggg,e)$.
%\end{itemize}
%
%{\it In this paper, when we consider minimal refined $W$-superalgebra of type even, we refer to %minimal finite $W$-algebras, or
%%minimal refined $W$-superalgebra associated with a Lie superalgebra of type I, and in this case
%minimal refined $W$-superalgebra which coincides with the corresponding finite $W$-superalgebra. %or minimal finite $W$-superalgebra of types $A(m|n), C(n), D(m|n), D(2,1;\alpha)$, $F(4)$.
%For minimal refined $W$-superalgebra of type odd, we refer to minimal refined $W$-superalgebra which is a proper subalgebra of the corresponding finite $W$-superalgebra.
%On the other hand, when we consider minimal finite $W$-superalgebra, we refer to minimal finite $W$-superalgebra of type $B(m|n)$ or $G(3)$.}
%
%
%\setcounter{equation}{4}

\begin{center}\label{Table 2}
({\sl{Table 2}}): The classification of $\ggg$ involving minimal finite (refined) $W$-superalgebras of type even
\vskip0.3cm
{\begin{tabular}{|ccc|}
\hline $\ggg$ & $\ggg^e(0)$ & $\ggg(1)$\\
\hline Simple Lie algebras & see \cite[Table 1]{KW} & see \cite[Table 1]{KW}\\
\hline $\mathfrak{sl}(2|m)$ $(m\neq2)$ & $\mathfrak{gl}(m)$ & $\bbc^m\oplus\bbc^{m*}$\\
\hline $\mathfrak{sl}(m|n)$ $(m\neq n,m>2)$ & $\mathfrak{gl}(m-2|n)$ & $\bbc^{m-2|n}\oplus\bbc^{m-2|n*}$\\
\hline $\mathfrak{psl}(2|2)$ & $\mathfrak{sl}(2)$ & $\bbc^2\oplus\bbc^2$\\
\hline $\mathfrak{psl}(m|m)$ $(m>2)$ & $\mathfrak{sl}(m-2|m)$ & $\bbc^{m-2|m}\oplus\bbc^{m-2|m*}$\\
\hline $\mathfrak{spo}(2|m)$ ($m$ even) & $\mathfrak{so}(m)$ & $\bbc^m$\\
\hline $\mathfrak{osp}(4|2m) $ & $\mathfrak{sl}(2)\oplus\mathfrak{sp}(2m)$ & $\bbc^2\otimes\bbc^{2m}$\\
\hline $\mathfrak{spo}(2n|m)$ $(n\geqslant2$, $m$ even)  & $\mathfrak{spo}(2n-2|m)$ & $\bbc^{2n-2|m}$\\
\hline $\mathfrak{osp}(m|2n)$ $(m\geqslant5)$ & $\mathfrak{osp}(m-4|2n)\oplus\mathfrak{sl}(2)$ & $\bbc^{m-4}\otimes\bbc^2$\\
\hline $D(2,1;\alpha)$ & $\mathfrak{sl}(2)\oplus\mathfrak{sl}(2)$ & $\bbc^2\otimes\bbc^2$\\
\hline $F(4)$ & $\mathfrak{so}(7)$ & $\text{spin}(7)$\\
\hline $F(4)$ & $D(2,1;2)$ & $\stackrel{1}{\circ} \leftarrow \!\!\otimes\!\!\rightarrow\!\!\circ~~(6|4)\text{-dim}$\\
\hline $G(3)$ & $\mathfrak{osp}(3|2)$ & $\stackrel{-3}{\otimes} \Longrightarrow
\stackrel{1}{\circ}~~(4|4)\text{-dim}$\\
\hline
\end{tabular}}
\end{center}
\vskip0.3cm
\begin{center}\label{Table 3}
({\sl{Table 3}}): The classification of $\ggg$ involving minimal refined $W$-superalgebras of type odd \& minimal finite $W$-superalgebras of type odd
\vskip0.3cm
{\begin{tabular}{|ccc|}
\hline $\ggg$ & $\ggg^e(0)$ & $\ggg(1)$\\
\hline $\mathfrak{spo}(2|m)$ ($m$ odd) & $\mathfrak{so}(m)$ & $\bbc^m$\\
\hline $\mathfrak{spo}(2n|m)$ $(n\geqslant2$, $m$ odd)  & $\mathfrak{spo}(2n-2|m)$ & $\bbc^{2n-2|m}$\\
\hline $G(3)$ & $G(2)$ & $7$-dim\\
\hline
\end{tabular}}
\end{center}
\vskip0.3cm

%\begin{center}\label{Table 4}
%({\sl{Table 4}}): Minimal finite $W$-superalgebras
%\vskip0.3cm
%{\begin{tabular}{|cc|}
%\hline ${\mathrm B}(m|n)$ & ${\mathrm G}(3)$\\
%\hline
%\end{tabular}}
%\end{center}
%\vskip0.3cm
{\it In the paper, when we refer to minimal finite (refined) $W$-superalgebras, we will always keep the conventions as in Remark \ref{types}.}

\subsection{}\label{proof of thverma}
By the classification of minimal $W$-superalgebras in Remark \ref{types}, the most complicated and also being of significantly different from the non-super situation is the case  in Table 3. So
in this part we are dedicated to introduce the Verma modules for minimal finite $W$-superalgebra $U(\ggg,e)$ of type odd, and other cases will be considered in \S\ref{3.3}.

\subsubsection{}\label{Verma}
We begin with the following observation.
\begin{lemma}\label{semi} The following  statements hold:
\begin{itemize}
\item[(1)] When $\ggg$ is a simple Lie algebra not of type $A$, $\ggg^e(0)$ is a semi-simple Lie algebra;
\item[(2)] When $\ggg$ is not a Lie algebra, but $\ggg^e(0)$ is and $\ggg(1)$ is purely odd, then $\ggg^e(0)$ is a semi-simple Lie algebra if
\begin{itemize}
\item[(i)] $\ggg$ is of type $\mathfrak{psl}(2|2)$, $\mathfrak{spo}(2|m)$ with $m$ being even such that $\ggg^e(0)=\mathfrak{so}(m)$, $\mathfrak{osp}(4|2m)$ with $\ggg^e(0)=\mathfrak{sl}(2)\oplus\mathfrak{sp}(2m)$, $D(2,1;\alpha)$, or $F(4)$ with $\ggg^e(0)=\mathfrak{so}(7)$; in these cases $\sfr$ are always even;
\item[(ii)] $\ggg$ is of type $\mathfrak{spo}(2|m)$ with $m$ being odd such that $\ggg^e(0)=\mathfrak{so}(m)$, or $G(3)$ with $\ggg^e(0)=G(2)$; in these cases $\sfr$ are always odd;
\end{itemize}
\item[(3)] When $\ggg$ and $\ggg^e(0)$ are not Lie algebras, all $\ggg^e(0)$-modules are completely reducible if and only if

\begin{itemize}
\item[(i)] $\ggg$ is of type $\mathfrak{osp}(5|2m)$ with $\ggg^e(0)=\mathfrak{osp}(1|2m)\oplus\mathfrak{sl}(2)$; in this case $\sfr$ is even;
\item[(ii)] $\ggg$ is of type $\mathfrak{spo}(2m|1)$ with $m\geqslant2$ such that $\ggg^e(0)=\mathfrak{spo}(2m-2|1)$; in this case $\sfr$ is odd;
\end{itemize}
\item[(4)] In other cases, not all finite-dimensional representations of $\ggg^e(0)$ are completely reducible.
\end{itemize}
\end{lemma}
\begin{proof}
Recall that $\sfr$ is the dimension of $\ggg(-1)_{\bar1}$, and by the non-degeneracy of the bilinear form $(\cdot,\cdot)$ we know that $\sfr$ has the same parity as that of $\ggg(1)_{\bar1}$.
Then Statements (1) and (2) are immediate consequences of \cite[Tables 1-2]{KW}.

Recall that all finite-dimensional representations of a Lie (super)algebra $\mathfrak{L}$ are completely reducible if and only if $\mathfrak{L}$ is isomorphic to the direct product of a semi-simple Lie algebra with finitely many Lie superalgebra of the type $\mathfrak{osp}(1|2m)$ with  $m\geqslant1$ (see, e.g., \cite[Chapter III, \S3.1, Theorem 1]{MS}). Applying \cite[Table 3]{KW} again, we see that when $\ggg=\mathfrak{spo}(2m|1)$ with $m\geqslant2$ such that $\ggg^e(0)=\mathfrak{spo}(2m-2|1)$, the $\ggg^e(0)$-module $\ggg(1)$ is isomorphic to $\mathbb{C}^{2m-2|1}$, thus $\text{dim}\,\ggg(1)_{\bar1}=1$. We also observe that when $\ggg=\mathfrak{osp}(5|2m)$ with $\ggg^e(0)=\mathfrak{osp}(1|2m)\oplus\mathfrak{sl}(2)$, the $\ggg^e(0)$-module $\ggg(1)$ is isomorphic to $\mathbb{C}^{1|2m}\otimes\mathbb{C}^2$, thus $\text{dim}\,\ggg(1)_{\bar1}$ is even. For other cases in \cite[Table 3]{KW}, $\ggg^e$ is not isomorphic to the direct product of a semi-simple Lie algebra with finitely many $\mathfrak{osp}(1|2m)$. %, thus not all of whose finite-dimensional representations are completely reducible.
Then Statement (3) is proved.
Statement (4) is just an immediate consequence of Statements (1)---(3) and \cite[Chapter III, \S3.1, Theorem 1]{MS}.
%To complete the proof of the lemma, it remains to note that finite-dimensional representations of a semi-simple Lie algebra are completely reducible by Weyl's theorem.
\end{proof}

\subsubsection{}\label{3.2.1}
Keep the notations as in \S\ref{innao}. For any $\alpha,\beta\in\ggg^*$ and $x\in\ggg$, we will write $(\alpha+\beta)(x):=\alpha(x)+\beta(x)$ and $(\alpha\cdot\beta)(x):=\alpha(x)\cdot\beta(x)$ for simplicity. Put \begin{equation}\label{deltarho}
\begin{split}
\delta=&\frac{1}{2}\bigg(\sum\limits_{i=1}^{\frac{s}{2}}\gamma^*_{\bar0i}-\sum\limits_{i=1}^{\frac{\sfr-1}{2}}\gamma^*_{\bar1i}\bigg)
=\frac{1}{2}\bigg(\sum\limits_{i=1}^{\frac{s}{2}}(-\theta-\gamma_{\bar0i})+\sum\limits_{i=1}^{\frac{\sfr-1}{2}}(\theta+\gamma_{\bar1i})\bigg)\\
=&\frac{1}{2}\bigg(-\sum\limits_{i=1}^{\frac{s}{2}}\gamma_{\bar0i}+\sum\limits_{i=1}^{\frac{\sfr-1}{2}}\gamma_{\bar1i}\bigg)-\frac{s-\sfr+1}{4}\theta,\\ \rho=&\frac{1}{2}\sum_{\alpha\in\Phi^{+}}(-1)^{|\alpha|}\alpha,\\ \rho_{e,0}=&\rho-2\delta-\bigg(\frac{s-\sfr}{4}+\frac{1}{2}\bigg)\theta=\frac{1}{2}\sum_{\alpha\in\Phi^{+}_{e,0}}(-1)^{|\alpha|}\alpha
=\frac{1}{2}\Big(\sum_{j=1}^{w}\beta_{\bar0j}-\sum_{j=1}^{\ell}\beta_{\bar1j}\Big),
\end{split}
\end{equation}where $\gamma^*_{\bar0i}\in\Phi^+_{\bar0},\gamma^*_{\bar1j}\in\Phi_{\bar1}^+$, $\gamma_{\bar0i}\in\Phi^-_{\bar0},\gamma_{\bar1j}\in\Phi_{\bar1}^-$ for $1\leqslant i \leqslant \frac{s}{2}$ and $1\leqslant j \leqslant \frac{\sfr-1}{2}$ are defined in \S\ref{1.2},  $\beta_{\bar0i}\in\Phi_{\bar0}^+,\beta_{\bar1j}\in\Phi_{\bar1}^+$ for $1\leqslant i \leqslant w$ and $1\leqslant j \leqslant \ell$ are defined in \S\ref{3.1.1}, and $|\alpha|$ denotes the parity of $\alpha$.
For a linear function $\lambda$ on $\mathfrak{h}^e$ and $c\in\bbc$, we will call $(\lambda,c)$ a matchable pair if they satisfy the following equation:
\begin{equation}\label{c0clambda}
\begin{split}
c=&c_0+\sum_{i=1}^{k-1}\lambda(h_i)^2+\sum_{i=1}^{k-1}\bigg(\lambda\cdot
\Big(\sum_{j=1}^{w}\beta_{\bar0j}
-\sum_{j=1}^{\ell}\beta_{\bar1j}-\sum_{j=1}^{\frac{s}{2}}\gamma_{\bar0j}+
\sum_{j=1}^{\frac{\sfr-1}{2}}\gamma_{\bar1j}\Big)(h_i)\bigg)\\
=&c_0+\sum_{i=1}^{k-1}\lambda(h_i)^2+2\sum_{i=1}^{k-1}(\lambda\cdot(\rho_{e,0}+\delta))(h_i),
\end{split}
\end{equation}
where %$\beta_{\bar0j}\in\Phi_{\bar0}^+,\beta_{\bar1j}\in\Phi_{\bar1}^+$ and $\gamma_{\bar0j}\in\Phi_{\bar0}^-,\gamma_{\bar1j}\in\Phi_{\bar1}^-$ are defined in \S\ref{2.1} and \S\ref{1.2} respectively, and
$c_0$ has the same meaning as in \eqref{decidec0}.
Given a matchable pair $(\lambda,c)\in(\mathfrak{h}^e)^*\times\bbc$, denote by $I_{\lambda,c}$ the linear span in $U(\ggg,e)$ of all PBW monomials of the form
\begin{equation*}
\begin{array}{ll}
&\prod_{i=1}^w\Theta_{x_i}^{a_i}\cdot\prod_{i=1}^\ell\Theta_{y_i}^{c_i}\cdot
\prod_{i=1}^{\frac{s}{2}}\Theta_{f_i}^{m_i}\cdot\prod_{i=1}^{\frac{\sfr-1}{2}}
\Theta_{g_i}^{p_i}\cdot\Theta_{v_{\frac{\sfr+1}{2}}}^\iota\cdot\prod_{i=1}^{k-1}
(\Theta_{h_i}-\lambda(h_i))^{t_i}\cdot(C-c)^{t_k}\\
&\cdot\Theta_{[v_{\frac{\sfr+1}{2}},e]}^\varepsilon\cdot\prod_{i=1}^{\frac{s}{2}}\Theta_{f^*_i}^{n_i}\cdot
\prod_{i=1}^{\frac{\sfr-1}{2}}\Theta_{g^*_i}^{q_i}\cdot
\prod_{i=1}^w\Theta_{x^*_i}^{b_i}\cdot
\prod_{i=1}^\ell\Theta_{y^*_i}^{d_i},
\end{array}
\end{equation*}
where ${\bf a},{\bf b}\in\mathbb{Z}_+^w$, ${\bf c},{\bf d}\in\Lambda_\ell$, ${\bf m},{\bf n}\in\mathbb{Z}_+^{\frac{s}{2}}$, ${\bf p},{\bf q}\in\Lambda_{\frac{\sfr-1}{2}}, \iota,\varepsilon\in\Lambda_1, {\bf t}\in\mathbb{Z}_+^k$ with
$\sum_{i=1}^{k}t_i+\varepsilon+\sum_{i=1}^{\frac{s}{2}}n_i+\sum_{i=1}^{\frac{\sfr-1}{2}}q_i+
\sum_{i=1}^{w}b_i+\sum_{i=1}^{\ell}d_i>0$. {\it In what follows, when refer to the notation $I_{\lambda,c}$, we will always assume that $(\lambda,c)$ is a matchable pair.}% for ${\bf a},{\bf b}\in\mathbb{Z}_+^w$, ${\bf c},{\bf d}\in\Lambda_\ell$, ${\bf
%m},{\bf n}\in\mathbb{Z}_+^{\frac{s}{2}}$, ${\bf p},{\bf q}\in\Lambda_{\lceil\frac{\sfr}{2}\rceil},
%{\bf t}\in\mathbb{Z}_+^k$ with .

By the same strategy as in \cite[Lemma 7.1]{P3}, we have
\begin{lemma}\label{left ideal2}
The subspace $I_{\lambda,c}$ is a left ideal of minimal finite $W$-superalgebra $U(\ggg,e)$ of type odd.
\end{lemma}
Since the proof of Lemma \ref{left ideal2} is rather lengthy, we will postpone it till Appendix \ref{lengthy proof}.

\subsubsection{}\label{3.2.2}
Put $Z_{U(\ggg,e)}(\lambda,c):=U(\ggg,e)/I_{\lambda,c}$, and let $v_0$ denote the image of $1$ in $Z_{U(\ggg,e)}(\lambda,c)$. Clearly, $Z_{U(\ggg,e)}(\lambda,c)$ is a cycle $U(\ggg,e)$-module generated by $v_0$. {\it We will call $Z_{U(\ggg,e)}(\lambda,c)$ the Verma module of level $c$ corresponding to $\lambda$.} By Lemma \ref{left ideal2}, the vectors
\begin{equation*}
\bigg\{\prod_{i=1}^w\Theta_{x_i}^{a_i}\cdot\prod_{i=1}^\ell\Theta_{y_i}^{c_i}\cdot
\prod_{i=1}^{\frac{s}{2}}\Theta_{f_i}^{m_i}\cdot\prod_{i=1}^{\frac{\sfr-1}{2}}
\Theta_{g_i}^{p_i}\cdot\Theta_{v_{\frac{\sfr+1}{2}}}^\iota(v_0)\mid ({\bf a}, {\bf c}, {\bf m}, {\bf p}, \iota)\in\mathbb{Z}_+^w\times\Lambda_\ell
\times\mathbb{Z}_+^{\frac{s}{2}}\times\Lambda_{\frac{\sfr-1}{2}}\times\Lambda_1\bigg\}
\end{equation*}form a $\bbc$-basis of the Verma module $Z_{U(\ggg,e)}(\lambda,c)$ over $\bbc$. Denote by $Z^+_{U(\ggg,e)}(\lambda,c)$ the $\mathbb{C}$-span of all $\prod_{i=1}^w\Theta_{x_i}^{a_i}\cdot\prod_{i=1}^\ell\Theta_{y_i}^{c_i}\cdot
\prod_{i=1}^{\frac{s}{2}}\Theta_{f_i}^{m_i}\cdot\prod_{i=1}^{\frac{\sfr-1}{2}}
\Theta_{g_i}^{p_i}\cdot\Theta_{v_{\frac{\sfr+1}{2}}}^\iota(v_0)$ with $\sum_{i=1}^{w}a_i+\sum_{i=1}^{\ell}c_i+\sum_{i=1}^{\frac{s}{2}}m_i+
\sum_{i=1}^{\frac{\sfr-1}{2}}p_i>0$. Set $Z^{\rm max}_{U(\ggg,e)}(\lambda,c)$ to be the sum of all $U(\ggg,e)$-submodules of $Z_{U(\ggg,e)}(\lambda,c)$ contained in $Z^+_{U(\ggg,e)}(\lambda,c)$, and let
\begin{equation*}
L_{U(\ggg,e)}(\lambda,c):=Z_{U(\ggg,e)}(\lambda,c)/Z^{\rm max}_{U(\ggg,e)}(\lambda,c).
\end{equation*}
\subsubsection{}\label{3.2.4}
\noindent\textbf{The proof of Theorem \ref{verma2}}.
 Now we are in a position to prove  Theorem \ref{verma2}.
Generally speaking, we can repeat the proof of \cite[Proposition 7.1]{P3}, with a lot of modifications.
We will complete the proof by steps.

(1) Given a root $\alpha=\sum_{i=1}^{k}n_i\alpha_i\in\Phi$, set
\begin{equation*}
\text{ht}_\theta(\alpha):= \sum_{\alpha_i\neq\frac{\theta}{2}}n_i.
\end{equation*}
Since $\frac{\theta}{2}$ is an odd simple root by  Convention \ref{conventions}, then $\text{ht}_\theta(\alpha)=0$ if and only if $\alpha=\pm\frac{\theta}{2},\pm\theta$. %Thanks to \cite[Proposition 2.1]{ZS2}, there exists an even non-degenerate super-symmetric invariant bilinear form on $\ggg$.
By  \cite[Proposition 5.1.2]{K} we know that all derivations of $\ggg$ are inner. Therefore, we can find a unique $h_0\in\mathfrak{h}$ such that $[h_0,e_\alpha]=\text{ht}_\theta(\alpha)e_\alpha$ for any $\alpha\in\Phi$. By definition we have $[h_0,e_{\pm\theta}]=0$ (recall that $e_\theta=e$ and $e_{-\theta}=f$), thus $h_0\in\mathfrak{h}^e$. It is obvious that $\Theta_{h_0}(v_0)=\lambda(h_0)v_0$ by definition, and we have the decomposition $Z_{U(\ggg,e)}(\lambda,c)=\mathbb{C}v_0\oplus\mathbb{C}\Theta_{v_{\frac{\sfr+1}{2}}}(v_0)\oplus Z^+_{U(\ggg,e)}(\lambda,c)$ as $\mathbb{C}$-vector space.
As all $x_i$, $y_i$, $f_i$ and $g_i$ are root vectors for $\mathfrak{h}$, corresponding to negative roots different from $-\frac{\theta}{2}$ and $-\theta$, it follows from Theorem \ref{main3} that the subspace $Z^+_{U(\ggg,e)}(\lambda,c)$ decomposes into eigenspaces for $\Theta_{h_0}$, and the eigenvalues of $\Theta_{h_0}$ on $Z^+_{U(\ggg,e)}(\lambda,c)$ are of the form $\lambda(h_0)-k$
with $k$ being a positive integer.

Let $V$ be a nonzero $\mathbb{Z}_2$-graded submodule of $U(\ggg,e)$-module $Z_{U(\ggg,e)}(\lambda,c)$. If $V\nsubseteq Z^+_{U(\ggg,e)}(\lambda,c)$, it follows from the discussion above %and the knowledge of linear algebra
that $v_0,\Theta_{v_{\frac{\sfr+1}{2}}}(v_0)\in V$, which entails that $V=Z_{U(\ggg,e)}(\lambda,c)$. Therefore, any proper submodule of $Z_{U(\ggg,e)}(\lambda,c)$ is contained in $Z^+_{U(\ggg,e)}(\lambda,c)$, and $Z_{U(\ggg,e)}^{\rm max}(\lambda,c)$ is a unique maximal submodule of $Z_{U(\ggg,e)}(\lambda,c)$. Obviously $L_{U(\ggg,e)}(\lambda,c)$ is a simple module of type $Q$, for which the odd endomorphism is induced by the element $\Theta_{v_{\frac{\sfr+1}{2}}}\in U(\ggg,e)$. Now we complete the proof of Statement (1).

(2) From the discussion in (1) we know that each $U(\ggg,e)$-module $L_{U(\ggg,e)}(\lambda,c)$ decomposed into eigenspaces for $\Theta_{h_0}$, and the eigenvalues of $\Theta_{h_0}$ are in the set $\lambda(h_0)-\mathbb{Z}_+$. Moreover, the eigenspace $L_{U(\ggg,e)}(\lambda,c)_{\lambda(h_0)}$ of the $U(\ggg,e)$-module $L_{U(\ggg,e)}(\lambda,c)$ is spanned by the elements $v_0$ and $\Theta_{v_{\frac{\sfr+1}{2}}}(v_0)$. If $L_{U(\ggg,e)}(\lambda,c)\cong L_{U(\ggg,e)}(\lambda',c')$ as $U(\ggg,e)$-modules, it follows from the discussion above that $\lambda(h_0)\in\lambda'(h_0)-\mathbb{Z}_+$ and $\lambda'(h_0)\in\lambda(h_0)-\mathbb{Z}_+$.
This implies that $\lambda(h_0)=\lambda'(h_0)$ and $L_{U(\ggg,e)}(\lambda,c)_{\lambda(h_0)}\cong L_{U(\ggg,e)}(\lambda',c')_{\lambda'(h_0)}$ as modules over the commutative subalgebra $\Theta_{\mathfrak{h}^e}\oplus\mathbb{C}C$ of $U(\ggg,e)$.
So we have $\lambda=\lambda'$ and $c=c'$, and the proof of Statement (2) is completed.

(3) Let $M$ be a finite-dimensional simple $U(\ggg,e)$-module. As the even element $C$ is in the center of $U(\ggg,e)$, Schur's lemma entails that $C$ acts on $M$ as $c~\text{id}$ for some $c\in\mathbb{C}$.
As $\Theta_{\mathfrak{h}^e}$ is an abelian by Theorem \ref{main3} (more precisely, Theorem \ref{maiin1}(1)), by the knowledge of linear algebra we know that $M$ contains at least one weight subspace for $\Theta_{\mathfrak{h}^e}$. Applying Theorem \ref{main3} again we know that the vector space $\bigoplus_{\mu\in(\mathfrak{h}^e)^*}M_{\mu}$ of all weight subspaces of $M$ is a $U(\ggg,e)$-submodule of $M$. From the irreducibility of $U(\ggg,e)$-module $M$ we know that $M$ decomposes into weight spaces relative to $\Theta_{\mathfrak{h}^e}$.

Since  $\mathfrak{h}=\mathbb{C}h\oplus\mathfrak{h}^e$ as vector space, and $[e_\theta,e_{-\theta}]=[e,f]=h$ by definition, one can easily conclude that any linear function vanishing on $\mathfrak{h}^e$ is a scalar multiple of $\theta$. As $\sfr$ is odd, $\frac{\theta}{2}$ is a simple root by  Convention \ref{conventions}, then any sum of roots from $\Phi^+_e\backslash\{\frac{\theta}{2},\theta\}$ restricts to a nonzero function on $\mathfrak{h}^e$. Therefore, we can define a partial ordering on $(\mathfrak{h}^e)^*$ by
\begin{equation}\label{parord}
\psi\leqslant\phi\Leftrightarrow\phi=\psi+\bigg(\sum_{\gamma\in\Phi^+_e\backslash\{\frac{\theta}{2},\theta\}}r_\gamma\gamma\bigg)_{|\mathfrak{h}^e},
\quad r_\gamma\in\mathbb{Z}_+\qquad(\forall\phi,\psi\in(\mathfrak{h}^e)^*).
\end{equation}%where the summands of $\gamma$ in \eqref{parord} $\Phi^+_e$ (resp. $\Phi^+_e\backslash\{\frac{\theta}{2}\}$)   when $U(\ggg,e)$ is of type even (resp. odd).
Recall that the set of $\Theta_{\mathfrak{h}^e}$-weights of $M$ is finite, then it contains at least one maximal element with the ordering we just defined above, and we put it as $\lambda$. For a nonzero vector $m$ in $M_\lambda$, we have $\Theta_{x_i^*}.m=\Theta_{y_i^*}.m=\Theta_{f_i^*}.m=\Theta_{g_i^*}.m=0$ for all admissible $i$  (Since $M$ is finite-dimensional, we can further assume that
$\Theta_{[v_{\frac{\sfr+1}{2}},e]}.m=0$). So there must exist a $U(\ggg,e)$-module homomorphism $\xi$ from either $Z_ {U(\ggg,e)}(\lambda,c)$ or $\prod Z_ {U(\ggg,e)}(\lambda,c)$ (Here $\prod$ denotes the parity switching functor) to $M$ such that $\xi(v_0)=m$. Moreover, the simplicity of $M$ entails that $\xi$ is surjective, and Statement (1) yields $\text{Ker}\,\xi=Z_{U(\ggg,e)}^{\rm max}(\lambda,c)$. Taking Theorem \ref{main3} (more precisely, Theorem \ref{maiin1}(1)) into consideration, when we restrict $M$ to the $\mathfrak{sl}(2)$-triple $(\Theta_{e_\alpha},\Theta_{h_\alpha},\Theta_{e_{-\alpha}})\subset U(\ggg,e)$ with $\alpha\in(\Phi^+_{e,0})_{\bar0}$, one can easily conclude that $\lambda(h_\alpha)\in\mathbb{Z}_+$ for any $\alpha\in(\Phi^+_{e,0})_{\bar0}$.

Finally, let $\ggg=\mathfrak{spo}(2|m)$ with $m$ being odd such that $\ggg^e(0)=\mathfrak{so}(m)$, or $\ggg=\mathfrak{spo}(2m|1)$ with $m\geqslant2$ such that $\ggg^e(0)=\mathfrak{spo}(2m-2|1)$, or  $\ggg=G(3)$ with $\ggg^e(0)=G(2)$ in Table 3, then all finite-dimensional representations of $\ggg^e(0)$ are completely reducible by Lemma \ref{semi}. It follows from Theorem \ref{main3} (more precisely, Theorem \ref{maiin1}(1)) that the linear map $\Theta:\ggg^e(0)\rightarrow \Theta_{\ggg^e(0)}, x\mapsto\Theta_x$ is a Lie superalgebra isomorphism. Let $M$ be a finite-dimensional simple $U(\ggg,e)$-module, then $M$ is completely reducible as a $\Theta_{\ggg^e(0)}$-module. %Without loss of generality,
%we can further assume that $M$ is irreducible of highest weight.
Let $\ggg_\mathbb{Q}$ be the $\mathbb{Q}$-form in $\ggg$ spanned by the Chevalley basis from \S\ref{1.1}, and write $\ggg_\mathbb{Q}^e(i):=\ggg_\mathbb{Q}\cap\ggg^e(i)$ with $i=0,1$. Choose $u,v\in\ggg_\mathbb{Q}^e(1)$ with $([u,v],f)=2$, and also assume that $z_\alpha, z^*_\alpha\in\ggg_\mathbb{Q}(-1)$ for all $\alpha\in S(-1)$, then $[u,z_\alpha]^{\sharp}, [z_\alpha^*,u]^{\sharp}, [v,z_\alpha]^{\sharp}, [z_\alpha^*,v]^{\sharp}\in\ggg^e_\mathbb{Q}$. The highest weight theory implies that there is a $\mathbb{Q}$-form in $M$ stable under the action of $\Theta_{\ggg_\mathbb{Q}^e(0)}$. So we have $\text{tr}_M\big(\Theta_{[u,z_\alpha]^{\sharp}}\Theta_{[z_\alpha^*,v]^{\sharp}}\big), \text{tr}_M\big(\Theta_{[v,z_\alpha]^{\sharp}}\Theta_{[z_\alpha^*,u]^{\sharp}}\big)\in\mathbb{Q}$ for all $\alpha\in S(-1)$. As $\text{tr}_M[\Theta_u,\Theta_v]=0$, it follows from Theorem \ref{main3} (more precisely, Theorem \ref{maiin1}(3)) that $(c-c_0) \text{dim}\,M\in\mathbb{Q}$. Since $c_0\in\mathbb{Q}$ by \eqref{decidec0}, we have $c\in\mathbb{Q}$.

\iffalse
Let $\ggg_\mathbb{Q}$ be the $\mathbb{Q}$-form in $\ggg$ spanned by the Chevalley basis from \S\ref{1.1}, and write $\ggg_\mathbb{Q}^e(i):=\ggg_\mathbb{Q}\cap\ggg^e(i)$ with $i=0,1$. Choose $u,v\in\ggg_\mathbb{Q}^e(1)$ such that $([u,v],f)\in\mathbb{Q}\backslash\{0\}$. As $z_\alpha,z_\beta, z_\alpha^*,z_\beta^*\in\ggg_\mathbb{Q}(-1)$ for all $\alpha,\beta\in S(-1)$, we have $[[[u, z_\alpha],z_\beta],[z_\beta^*,[z_\alpha^*,v]]]\in\ggg^e_\mathbb{Q}$. Then it follows from \eqref{decidec0} that $c_0\in\mathbb{Q}$.

Recall in \S\ref{3.1.1} we introduced
$\mathfrak{n}^\pm(0)$ as the span of all $e_\alpha$ with $\alpha\in\Phi_{e,0}^\pm$. Moreover, $\{x_1,\cdots,x_w\}$ and
$\{x^*_1,\cdots,x^*_w\}$ are dual bases of $\mathfrak{n}^-_{\bar0}(0)$ and $\mathfrak{n}^+_{\bar0}(0)$, with each $x_i$ (resp. $x_i^*$) being a root vector for $\mathfrak{h}$ corresponding to $-\beta_{\bar0i}\in\Phi_{\bar0}^-$ (resp. $\beta_{\bar0i}\in\Phi_{\bar0}^+$). In the meanwhile,
$\{y_1,\cdots,y_\ell\}$ and $\{y^*_1,\cdots,y^*_\ell\}$ are dual bases of $\mathfrak{n}^-_{\bar1}(0)$ and $\mathfrak{n}^+_{\bar1}(0)$, with each $y_i$ (resp. $y_i^*$) being a root vector for $\mathfrak{h}$ corresponding to $-\beta_{\bar1i}\in\Phi_{\bar1}^-$ (resp. $\beta_{\bar1i}\in\Phi_{\bar1}^+$).

On the other hand, since all the root vectors we have chosen are linear combinations of the Chevalley basis of $\ggg$, and $\lambda(h_i)\in\mathbb{Z}_+$ for any $1\leqslant i\leqslant k-1$ by our earlier remark, then $c\in\mathbb{Q}$ comes as an immediate consequence of \eqref{c0clambda}.
\fi
\begin{rem}
By the same discussion as in Lemma \ref{left ideal2}, one can conclude that the linear span in $U(\ggg,e)$ of all PBW monomials as in \eqref{jlamdbdaccc} with
$\sum_{i=1}^{k}t_i+\sum_{i=1}^{\frac{s}{2}}n_i+\sum_{i=1}^{\frac{\sfr-1}{2}}q_i+
\sum_{i=1}^{w}b_i+\sum_{i=1}^{\ell}d_i>0$ is also a left ideal of minimal finite $W$-superalgebra $U(\ggg,e)$ of type odd (note that there is no restriction on the pair $(\lambda,c)$ as in \eqref{c0clambda}), and write it as $I'_{\lambda,c}$. Then we can introduce the $U(\ggg,e)$-module $Z'_{U(\ggg,e)}(\lambda,c):=U(\ggg,e)/I'_{\lambda,c}$ as in \S\ref{3.2.2} correspondingly. The reason why we did not consider $Z'_{U(\ggg,e)}(\lambda,c)$ lies in the fact that for the vector $[v_{\frac{\sfr+1}{2}},e]$ associated with the simple odd root $\frac{\theta}{2}$, we may have $\Theta_{[v_{\frac{\sfr+1}{2}},e]}.m\neq0$ for every $m\in Z'_{U(\ggg,e)}(\lambda,c)$, and then $Z'_{U(\ggg,e)}(\lambda,c)$ is not necessarily a highest weight module for $U(\ggg,e)$.
\end{rem}

\subsection{}\label{3.3}
In this part we will consider Verma modules for %minimal finite (refined) $W$-superalgebra of type even and also minimal refined $W$-superalgebra $W'_\chi$ of type odd
other cases as in Remark \ref{types}, which is parallel to those in \S\ref{proof of thverma}. Recall that for the type even case, $U(\ggg,e)$  coincides with  $W'_\chi$. So we just need to consider  $W'_\chi$ for both types.  Since the proofs are similar, we will just sketch them.

\subsubsection{}
%Denote by
%\begin{equation*}
%\mathbb{Z}_+^k:=\{(i_1,\cdots,i_k)\mid i_j\in\mathbb{Z}_+\},~~
%\Lambda_k:=\{(i_1,\cdots,i_k)\mid i_j\in\{0,1\}\},~~
%\mathbf{a}:=(a_1,\cdots,a_k).
%\end{equation*}
Keep the notations as in \S\ref{innao}. For a linear function $\lambda$ on $\mathfrak{h}^e$ and  $c\in\bbc$, we denote by $J_{\lambda,c}$ the linear span in $W_\chi'$ of all PBW monomials of the form
\begin{equation}\label{jlamdbdaccc}
\begin{array}{ll}
&\prod_{i=1}^w\Theta_{x_i}^{a_i}\cdot\prod_{i=1}^\ell\Theta_{y_i}^{c_i}\cdot
\prod_{i=1}^{\frac{s}{2}}\Theta_{f_i}^{m_i}\cdot\prod_{i=1}^{\lceil\frac{\sfr}{2}\rceil}
\Theta_{g_i}^{p_i}\cdot\prod_{i=1}^{k-1}
\big(\Theta_{h_i}-\lambda(h_i)\big)^{t_i}\cdot(C-c)^{t_k}\\
&\cdot\Theta_{[v_{\frac{\sfr+1}{2}},e]}^\varepsilon\cdot\prod_{i=1}^{\frac{s}{2}}\Theta_{f^*_i}^{n_i}\cdot
\prod_{i=1}^{\lceil\frac{\sfr}{2}\rceil}\Theta_{g^*_i}^{q_i}\cdot
\prod_{i=1}^w\Theta_{x^*_i}^{b_i}\cdot
\prod_{i=1}^\ell\Theta_{y^*_i}^{d_i},
\end{array}
\end{equation}
where ${\bf a},{\bf b}\in\mathbb{Z}_+^w$, ${\bf c},{\bf d}\in\Lambda_\ell$, ${\bf m},{\bf n}\in\mathbb{Z}_+^{\frac{s}{2}}$, ${\bf p},{\bf q}\in\Lambda_{\lceil\frac{\sfr}{2}\rceil}, {\bf t}\in\mathbb{Z}_+^k$, $\varepsilon\in\Lambda_1$ with
$\sum_{i=1}^{k}t_i+\varepsilon+\sum_{i=1}^{\frac{s}{2}}n_i+\sum_{i=1}^{\lceil\frac{\sfr}{2}\rceil}q_i+
\sum_{i=1}^{w}b_i+\sum_{i=1}^{\ell}d_i>0$
{\it(The term $\Theta_{[v_{\frac{\sfr+1}{2}},e]}$ occurs if and only if $\sfr$ is odd. In this subsection, when we consider  the type even case, we always assume that $\varepsilon=0$)}. %for ${\bf a},{\bf b}\in\mathbb{Z}_+^w$, ${\bf c},{\bf d}\in\Lambda_\ell$, ${\bf
%m},{\bf n}\in\mathbb{Z}_+^{\frac{s}{2}}$, ${\bf p},{\bf q}\in\Lambda_{\lceil\frac{\sfr}{2}\rceil},
%{\bf t}\in\mathbb{Z}_+^k$.

%By the similar discussion as in the proof of Lemma \ref{left ideal2}, we can obtain

\begin{lemma}\label{left ideal}Under the above settings, we have the following results:
\begin{itemize}
\item[(1)]The subspace $J_{\lambda, c}$ is a left ideal of minimal refined $W$-superalgebra $W_\chi'$ of type even;
\item[(2)]The subspace $J_{\lambda,c}$ with the pair $(\lambda,c)$ satisfying \eqref{c0clambda} is a left ideal of minimal refined $W$-superalgebra $W_\chi'$ of type odd.
\end{itemize}
\end{lemma}
\begin{proof}Since the proof is much the same as the one for Lemma \ref{left ideal2}, we will just sketch the differences.
In fact, one can observe that all the considerations in Appendix \ref{lengthy proof} are still valid for Statement (1) except Appendix \ref{Step d}, where the emergence of $\Theta_{[v_{\frac{\sfr+1}{2}},e]}^2$ in \eqref{1vr+12fir}  makes it necessary to impose the restriction \eqref{c0clambda}. Since the element $\Theta_{[v_{\frac{\sfr+1}{2}},e]}$ will never appear when $W_\chi'$ is of type even, then $\lambda\in(\mathfrak{h}^e)^*$ and  $c\in\bbc$ in (1) can be chosen arbitrarily, which is much the same as the minimal finite $W$-algebra case. On the other hand, repeat verbatim as in Appendix \ref{lengthy proof} we can obtain Statement (2).
\end{proof}

\begin{rem}\label{keyremark}
Lemma \ref{left ideal} will play a key role for the exposition on $W_\chi'$ below. {\it To ease notation, from now on when we consider the pair $(\lambda,c)$ for $J_{\lambda, c}$, we always assume that $\lambda\in(\mathfrak{h}^e)^*$ and  $c\in\bbc$ are arbitrarily chosen for the minimal refined $W$-superalgebras of type even, and $\lambda,c$ should satisfy \eqref{c0clambda} (i.e., $(\lambda,c)$ is a matchable pair as defined in \S\ref{3.2.1}) for the minimal refined $W$-superalgebras of type odd, unless otherwise specified.} \end{rem}

\subsubsection{}
Retain the conventions as in Remark \ref{keyremark}.
Write $Z_{W_\chi'}(\lambda,c):=W_\chi'/J_{\lambda,c}$, and denote by $v_0$ the image of $1$ in $Z_{W_\chi'}(\lambda,c)$. By definition we know that $Z_{W_\chi'}(\lambda,c)$ is a cycle $W_\chi'$-module generated by $v_0$. {\it We will call $Z_{W_\chi'}(\lambda,c)$ the Verma module of level $c$ corresponding to $\lambda$.} Moreover, Lemma \ref{left ideal} entails that the vectors
\begin{equation*}
\bigg\{\prod_{i=1}^w\Theta_{x_i}^{a_i}\cdot\prod_{i=1}^\ell\Theta_{y_i}^{c_i}\cdot
\prod_{i=1}^{\frac{s}{2}}\Theta_{f_i}^{m_i}\cdot\prod_{i=1}^{\lceil\frac{\sfr}{2}\rceil}
\Theta_{g_i}^{p_i}(v_0)\mid {\bf (a,c,m,p)}\in\mathbb{Z}_+^w\times\Lambda_\ell
\times\mathbb{Z}_+^{\frac{s}{2}}\times\Lambda_{\lceil\frac{\sfr}{2}\rceil}\bigg\}
\end{equation*}form a basis of the Verma module $Z_{W_\chi'}(\lambda,c)$ over $\bbc$. Denote by $Z^+_{W_\chi'}(\lambda,c)$ the $\mathbb{C}$-span of all $\prod_{i=1}^w\Theta_{x_i}^{a_i}\cdot\prod_{i=1}^\ell\Theta_{y_i}^{c_i}\cdot
\prod_{i=1}^{\frac{s}{2}}\Theta_{f_i}^{m_i}\cdot\prod_{i=1}^{\lceil\frac{\sfr}{2}\rceil}
\Theta_{g_i}^{p_i}(v_0)$ with $\sum_{i=1}^{w}a_i+\sum_{i=1}^{\ell}c_i+\sum_{i=1}^{\frac{s}{2}}m_i+
\sum_{i=1}^{\lceil\frac{\sfr}{2}\rceil}p_i>0$. Set $Z^{\rm max}_{W_\chi'}(\lambda,c)$ to be the sum of all $W_\chi'$-submodules of $Z_{W_\chi'}(\lambda,c)$ contained in $Z^+_{W_\chi'}(\lambda,c)$, and let
\begin{equation*}
L_{W_\chi'}(\lambda,c):=Z_{W_\chi'}(\lambda,c)/Z^{\rm max}_{W_\chi'}(\lambda,c).
\end{equation*}%then we have
%\begin{prop}\label{verma}
%The following hold:
%\begin{itemize}
%\item[(1)] $Z^{\rm max}_{W_\chi'}(\lambda,c)$ is the unique maximal submodule of the Verma module $Z_{W_\chi'}(\lambda,c)$, and $L_{W_\chi'}(\lambda,c)$ is a simple $W_\chi'$-module.
%\item[(2)] The simple $W_\chi'$-modules $L_{W_\chi'}(\lambda,c)$ and $L_{W_\chi'}(\lambda',c')$ are isomorphic if and only if $\lambda=\lambda'$ and $c=c'$.
%\item[(3)] Any finite-dimensional simple $W_\chi'$-module is isomorphic to one of the modules $L_{W_\chi'}(\lambda,c)$ for some $\lambda\in(\mathfrak{h}^e)^*$ satisfying $\lambda(h_\alpha)\in\mathbb{Z}_+$ for all $\alpha\in\Phi^+_{e,0}$. We further have that $c$ is a rational number if $\ggg$ is a simple Lie algebra except type $A(n)$, or $\ggg$ is one of the basic Lie superalgebras of types $A(1|1)$, $\mathfrak{spo}(2|m)$, $\mathfrak{osp}(4|m)$ and $D(2,1;\alpha)$.
%\end{itemize}
%\end{prop}

Under the settings above, we can introduce the main result of this subsection.
\begin{theorem}\label{verma}Keep the conventions as above. The following statements hold:
\begin{itemize}
\item[(1)] $Z^{\rm max}_{W_\chi'}(\lambda,c)$ is the unique maximal submodule of the Verma module $Z_{W_\chi'}(\lambda,c)$, and $L_{W_\chi'}(\lambda,c)$ is a simple $W_\chi'$-module of type $M$.
\item[(2)] The simple $W_\chi'$-modules $L_{W_\chi'}(\lambda,c)$ and $L_{W_\chi'}(\lambda',c')$ are isomorphic if and only if $(\lambda,c)=(\lambda',c')$.
\item[(3)] Any finite-dimensional simple $W_\chi'$-module is isomorphic to one of the modules $L_{W_\chi'}(\lambda,c)$ for some $\lambda\in(\mathfrak{h}^e)^*$ satisfying $\lambda(h_\alpha)\in\mathbb{Z}_+$ for all $\alpha\in(\Phi^+_{e,0})_{\bar0}$. We further have that $c$ is a rational number in the case when $\ggg$ is a simple Lie algebra except type $A(m)$, or when $\ggg=\mathfrak{psl}(2|2)$, $\ggg=\mathfrak{spo}(2m|1)$ with $m\geqslant2$  such that $\ggg^e(0)=\mathfrak{spo}(2m-2|1)$, or when $\ggg=\mathfrak{spo}(2|m)$ with $\ggg^e(0)=\mathfrak{so}(m)$, or when $\mathfrak{osp}(4|2m)$ with $\ggg^e(0)=\mathfrak{sl}(2)\oplus\mathfrak{sp}(2m)$, or when $\ggg=\mathfrak{osp}(5|2m)$ with $\ggg^e(0)=\mathfrak{osp}(1|2m)\oplus\mathfrak{sl}(2)$, or when $\ggg=D(2,1;\alpha)$ with $\alpha\in\overline\bbq$, or when $\ggg=F(4)$ with $\ggg^e(0)=\mathfrak{so}(7)$, or when $\ggg=G(3)$ with $\ggg^e(0)=G(2)$.% with $\alpha$ being an algebraic number
\end{itemize}
\end{theorem}
\begin{proof}
Take Lemmas \ref{semi} and \ref{left ideal} into consideration, repeat verbatim the proof of Theorem \ref{verma2}. Then the theorem can be proved.
\end{proof}

\section{The associated Whittaker categories of $\ggg$}\label{4}
In this section, we will relate Verma modules $Z_{U(\ggg,e)}(\lambda,c)$ for minimal finite $W$-superalgebra $U(\ggg,e)$ of both types to $\ggg$-modules obtained by parabolic induction from Whittaker modules for $\mathfrak{osp}(1|2)$ or $\mathfrak{sl}(2)$, respectively. %In which Skryabin's equivalence \cite[Theorem 2.17]{ZS2} will plays the key role.
Combining this with the related results on Whittaker categories in \cite{C,C3}, we obtain a complete solution to the problem
of determining the composition multiplicities of Verma modules $Z_{U(\ggg,e)}(\lambda,c)$ in terms
of composition factors of Verma modules for $U(\ggg)$ in the ordinary BGG category $\mathcal{O}$.

%consider  the composition multiplicities of Verma modules $Z_{U(\ggg,e)}(\lambda,c)$ for minimal finite $W$-superalgebra $U(\ggg,e)$ of both types. In virtue of Skryabin's equivalence in \cite[Theorem 2.17]{ZS2}, these can be translated into the one with $\ggg$-modules obtained by parabolic induction from Whittaker modules for $\mathfrak{osp}(1|2)$ or $\mathfrak{sl}(2)$ respectively, depending on the type of $U(\ggg,e)$.
%For the situation of type even, some certain cases can be computed via the Kazhdan-Lusztig combinatorics.

Although the tools we applied for both types are similar,  the discussion for minimal finite $W$-superalgebras of type odd is much more difficult. Therefore, we will give a detailed exposition for the case of type odd, and then sketch the case of type even.

\subsection{}\label{skryabin}
Recall that a ${\ggg}$-module $M$ is called a Whittaker module if $a-\chi(a)$ acts on $M$ locally nilpotently for each $a\in{\mmm}$. A Whittaker vector in a Whittaker ${\ggg}$-module $M$ is a vector $v\in M$ which satisfies $(a-\chi(a))v=0,~\forall a\in{\mmm}$.
Let $\mathcal{C}_\chi$ denote the category of finitely generated Whittaker ${\ggg}$-modules. %, and assume all the morphisms are even.
Write $$\text{Wh}(M)=\{v\in M\mid (a-\chi(a))v=0,\forall a\in{\mmm}\}$$ the subspace of all Whittaker vectors in $M$. For $M\in\mathcal{C}_\chi$, it is obvious that $\text{Wh}(M)=0$ if and only of $M=0$.

For any $y\in U({\ggg})$, denote by $\text{Pr}(y)\in U({\ggg})/I_\chi$ the coset associated to $y$. Given a Whittaker ${\ggg}$-module $M$ with an action map $\rho$, since $U({\ggg},e)\cong Q_{\chi}^{\ad\,\mmm}$ as $\bbc$-algebras, $\text{Wh}(M)$ is naturally a $U({\ggg},e)$-module by letting $\text{Pr}(y).v=\rho(y)v$ for $v\in\text{Wh}(M)$ and $\text{Pr}(y)\in U({\ggg})/I_\chi$.
For a $U({\ggg},e)$-module $M$, $Q_\chi\otimes_{U({\ggg},e)}M$ is a Whittaker ${\ggg}$-module by letting $y.(q\otimes v)=(y.q)\otimes v$ for $y\in U({\ggg})$ and $q\in Q_\chi, v\in M$.

Let $U({\ggg},e)$-$\mathbf{mod}$ be the category of finitely generated $U({\ggg},e)$-modules (here $U({\ggg},e)$ denotes a finite $W$-superalgebra in the general case, not just for minimal ones).
In \cite[Theorem 2.17]{ZS2}, we introduced Skryabin's equivalence between the finitely generated Whittaker ${\ggg}$-modules and finitely generated $U({\ggg},e)$-modules, i.e.,
\begin{theorem}\label{skry}The functor $Q_\chi\otimes_{U({\ggg},e)}-:U({\ggg},e)$-$\mathbf{mod}\longrightarrow
\mathcal{C}_\chi$ is an equivalence of categories, with $\text{Wh}:\mathcal{C}_\chi\longrightarrow U({\ggg},e)$-$\mathbf{mod}$ as its quasi-inverse.
\end{theorem}

\subsection{}\label{4.1}
In this part we consider minimal finite $W$-superalgebras $U(\ggg,e)$ of type odd, with $\ggg$ given in Table 3.%, as we assumed in Remark \ref{types}.
\subsubsection{}\label{4.1.1}
To describe the composition factors of the Verma modules  $Z_{U(\ggg,e)}(\lambda,c)$ with their multiplicities, we are going to establish a link between these $U(\ggg,e)$-modules and the $\ggg$-modules obtained by parabolic induction from Whittaker modules for $\mathfrak{osp}(1|2)$. The Skryabin's equivalence in Theorem \ref{skry} will be relied on; the Kazhdan filtration of $U(\ggg,e)$ will play an important role too.

For the topic of  Whittaker modules for Lie superalgebras, there have been a lot of results on it. %On one hand, the classification of simple $\mathfrak{osp}(1|2)$-modules was given by Bavula-Oystaeyen in \cite{Bav}, in which Whittaker modules play a very important role. On the other hand,
Whittaker categories for Lie superalgebras were defined and a category decomposition was presented by Bagci-Christodoulopoulou-Wiesner in \cite{Bag}. As a further work, Chen \cite{C} classified simple Whittaker modules for classical Lie superalgebras in terms of their parabolic decompositions, and established a type of Mili\v{c}i\'{c}-Soergel equivalence of a category of Whittaker modules and a category of Harish-Chandra bimodules. Furthermore, for classical Lie superalgebras of type I, the problem of composition factors of standard Whittaker modules  (i.e., the parabolic induced modules from Whittaker modules) was reduced to that of Verma modules in their BGG category $\mathcal{O}$ there.
%However, $\mathfrak{osp}(1|2)$ is a basic classical Lie superalgebras of type II.
For any quasi-reductive Lie superalgebra (including all the basic classical ones), Chen-Cheng \cite[Theorem 1]{C3} recently gave a complete solution to the problem of determining the composition factors of the standard Whittaker modules in terms of composition factors of Verma modules for $U(\ggg)$ in the ordinary BGG category $\mathcal{O}$. In most cases (including all basic Lie superalgebras of type $A,B,C,D$), the latter can be computed by related works (e.g., \cite{Bao,BW,Br,Br2,CCM,CLW2,CLW,CMW}). %on the irreducible characters of the BGG %category $\mathcal{O}$.
%As far as we know, the knowledge of parabolic induced modules from Whittaker modules %for basic classical Lie superalgebras of type II is very limited, even for the %$\mathfrak{osp}(1|2)$ case.
%As the related theory for both Lie algebras and basic classical Lie superalgebras of %type I have already been fruitful (see, e.g., \cite{AB,Ba,Bag,C,C2,CCM,Mc,Mi,R}), its %generalization to the basic classical Lie superalgebras of type II will be an %interesting topic, especially for the translation of composition factors of standard %Whittaker modules.%, which is needed to be developed in the future.
\subsubsection{}\label{4.2.2}
Denote by $\mathfrak{s}_\theta$ the subalgebra of $\ggg$ spanned by
\begin{equation}\label{(e,h,f,E,F)}
(e,h,f,E,F):=\bigg(e_\theta,h_\theta,e_{-\theta},[\sqrt{-2} v_{\frac{\sfr+1}{2}},e_\theta],\sqrt{-2}v_{\frac{\sfr+1}{2}}\bigg).
\end{equation}
Taking Theorem \ref{main3}, \eqref{veve} and \eqref{vev} into account, it is readily to check that
\begin{lemma}\label{osp12genecom}
The subalgebra $\mathfrak{s}_\theta$ of $\ggg$ is isomorphic to Lie superalgebra $\mathfrak{osp}(1|2)$ with even subalgebra generated by $\{e,h,f\}$ and odd subalgebra generated by $\{E,F\}$. The commutators of these basis elements are given by
\[
\begin{array}{llll}
[h,e]=2e,&[h,f]=-2f,&[e,f]=h,&[h,E]=E,\\
\,[h,F]=-F,&[e,E]=0,&[e,F]=-E,&[f,E]=-F,\\
\,[f,F]=0,&[E,E]=2e,&[E,F]=h,&[F,F]=-2f.
\end{array}\]
\end{lemma}

Put
\begin{equation}\label{pnsodd}
\mathfrak{p}_\theta:=\mathfrak{s}_\theta+\mathfrak{h}+\sum_{\alpha\in\Phi^+}\bbc e_\alpha,\qquad
\mathfrak{n}_\theta:=\sum_{\alpha\in\Phi^+\backslash\{\frac{\theta}{2},\theta\}}\bbc e_\alpha,\qquad \widetilde{\mathfrak{s}}_\theta:=\mathfrak{h}^e\oplus\mathfrak{s}_\theta.
\end{equation}%where $e_\alpha$ denotes a root vector with $\alpha\in\Phi$.
It is obvious that $\mathfrak{p}_\theta=\widetilde{\mathfrak{s}}_\theta\oplus\mathfrak{n}_\theta$ is a parabolic subalgebra of $\ggg$ with nilradical $\mathfrak{n}_\theta$ and $\widetilde{\mathfrak{s}}_\theta$ is a Levi subalgebra of $\mathfrak{p}_\theta$. Set \begin{equation}\label{Ctheta}
C_\theta:=ef+fe+\frac{1}{2}h^2-\frac{1}{2}EF+\frac{1}{2}FE=2ef+\frac{1}{2}h^2-\frac{3}{2}h+FE
\end{equation} to be a Casimir element of $U(\mathfrak{s}_\theta)$. Given $\lambda\in(\mathfrak{h}^e)^*$, write $I_\theta(\lambda)$ for the left ideal of $U(\mathfrak{p}_\theta)$ generated by $f-1$, $E-\frac{3}{4}F+\frac{1}{2}Fh$ (this requirement will be explained in \eqref{EFh}), $C_\theta+\frac{1}{8}$ (this requirement will be explained in \eqref{ctheta1lambda}), all $t-\lambda(t)$ with $t\in\mathfrak{h}^e$, and all $e_\gamma$ with $\gamma\in\Phi^+\backslash\{\frac{\theta}{2},\theta\}$.

Set $Y(\lambda):=U(\mathfrak{p}_\theta)/I_\theta(\lambda)$ to be a $\mathfrak{p}_\theta$-module with the trivial action of $\mathfrak{n}_\theta$, and let $1_{\lambda}$ denote the image of $1$ in $Y(\lambda)$. Since $f.1_{\lambda}=1_{\lambda}$ by definition, then
\begin{equation*}
\begin{split}
e.1_{\lambda}=&\frac{1}{2}\bigg(C_\theta-\frac{1}{2}h^2+\frac{3}{2}h-FE\bigg).1_{\lambda}\\
=&\bigg(-\frac{1}{4}h^2+\frac{3}{4}h
+\frac{1}{2}\Big(-\frac{3}{4}F^2+\frac{1}{2}F^2h\Big)-\frac{1}{16}\bigg).1_{\lambda}\\
=&\bigg(-\frac{1}{4}h^2+\frac{1}{4}h-\frac{5}{16}\bigg).1_{\lambda}.
\end{split}
\end{equation*}
Combining this with the PBW theorem we see that the vectors $\{F^kh^l\cdot1_{\lambda}\mid k\in\Lambda_1, l\in\mathbb{Z}_+\}$ form a $\bbc$-basis of $Y(\lambda)$ (the independence of these vectors follows from the fact that $Y(\lambda)$ is infinite-dimensional). Moreover, one can easily conclude that $Y(\lambda)$
is isomorphic to a Whittaker module for $\mathfrak{s}_\theta\cong\mathfrak{osp}(1|2)$.

It follows from the discussion above that the vectors
\begin{equation*}
\begin{split}
&m({\bf i,j,\iota,k,l,m,n},t):=\\
&u_1^{i_1}\cdots u_{\frac{s}{2}}^{i_{\frac{s}{2}}}\cdot v_1^{j_1}\cdots v_{\frac{\sfr-1}{2}}^{j_{\frac{\sfr-1}{2}}}\cdot v_{\frac{\sfr+1}{2}}^{\iota}\cdot x_1^{k_1}\cdots x_w^{k_w}\cdot y_1^{l_1}\cdots y_\ell^{l_\ell}\cdot f_1^{m_1}\cdots f_{\frac{s}{2}}^{m_{\frac{s}{2}}}\cdot
g_1^{n_1}\cdots g_{\frac{\sfr-1}{2}}^{n_{\frac{\sfr-1}{2}}}\cdot h^t(1_{\lambda})
\end{split}
\end{equation*}
with ${\bf i,m}\in\mathbb{Z}_+^{\frac{s}{2}}$, ${\bf j,n}\in\Lambda_{\frac{\sfr-1}{2}}$, $\iota\in\Lambda_1$, ${\bf k}\in\mathbb{Z}_+^w$, ${\bf l}\in\Lambda_\ell$, and $t\in\mathbb{Z}_+$ form a $\bbc$-basis of the induced $\ggg$-module
\begin{equation*}
M(\lambda):=U(\ggg)\otimes_{U(\mathfrak{p}_\theta)}Y(\lambda).
\end{equation*}
\subsubsection{}\label{5.3}
Keep the notations as in \eqref{deltarho}. Denote by
\begin{equation}\label{defofepsilon}
\epsilon:=c_0+\frac{1}{8}+2\sum_{i=1}^{k-1}\rho_{e,0}(h_i)\delta(h_i)+3\sum_{i=1}^{k-1}\delta(h_i)^2.
\end{equation}
Since the element $C$  lies in the center of $U(\ggg,e)$, and the Verma module $Z_{U(\ggg,e)}(\lambda,c)$ is a cycle $U(\ggg,e)$-module, then $C$ acts on $Z_{U(\ggg,e)}(\lambda,c)$ as the scalar $c$.
We introduce the twisted action of $U(\ggg,e)$ on $Z_{U(\ggg,e)}(\lambda,c)$ as follows: for any $v\in Z_{U(\ggg,e)}(\lambda,c)$, set
\begin{equation}\label{twisted}
C.v=\text{tw}\,(C)(v):=(C-\epsilon)(v)=(c-\epsilon)(v),
\end{equation}
while keep the action of all the other generators of $U(\ggg,e)$ (as defined in Theorem \ref{main3}) on $v$ as usual.

Since the restriction of $(\cdot,\cdot)$ to $\mathfrak{h}^e$ is non-degenerate, for any $\eta\in(\mathfrak{h}^e)^*$ there exists a unique $t_\eta$ in $\mathfrak{h}^e$ with $\eta=(t_\eta,\cdot)$. Hence $(\cdot,\cdot)$ induces a non-degenerate bilinear form on $(\mathfrak{h}^e)^*$ via $(\mu,\nu):=(t_\mu,t_\nu)$ for all $\mu,\nu\in(\mathfrak{h}^e)^*$. For a linear function $\varphi$ on $\mathfrak{h}$ we denote by $\bar{\varphi}$ the restriction of $\varphi$ to $\mathfrak{h}^e$.

Under the above settings, we can introduce the proof of Theorem \ref{Whcchi}.
It is remarkable that there exists great distinction between the structure theory of finite $W$-algebra in \cite{P2} and its super case in \cite{ZS2}.
Therefore, as a super version of \cite[Theorem 7.1]{P3}, one can observe significant differences not only in the exposition, but also for the proofs.
  \vskip0.2cm
  \noindent\textsc{\bf The proof of Theorem \ref{Whcchi}}.
(1) Set $M:=M(\lambda)$, and let $M_0$ and $M_1$ denote the $\bbc$-span of all $m({\bf i,j,\iota,k,l,m,n},t)$ in $M$ with $|{\bf i}|+|{\bf j}|+t=0$ and $|{\bf i}|+|{\bf j}|+t>0$, respectively. Then $M=M_0\oplus M_1$ as vector space. Denote by $\text{pr}:M=M_0\oplus M_1\twoheadrightarrow M_0$ the first projection.

Let $M^k$ denote the $\bbc$-span in $M$ of all $m({\bf i,j,\iota,k,l,m,n},t)$ of Kazhdan degree $\leqslant k$. Then $\{M^k\mid k\in\bbz_+\}$ is an increasing filtration in $M$ and $M^0=\bbc1_{\lambda}$. Taking $U(\ggg)$ with its Kazhdan filtration we thus regard $M$ as a filtrated $U(\ggg)$-module.

Set $z=\lambda f+\sum_{i=1}^{\frac{s}{2}}\mu_iu_i^*+\sum_{i=1}^{\frac{\sfr-1}{2}}\nu_iv_i^*\in
\mathfrak{m}$, where $\lambda, \mu_i, \nu_i\in\bbc$. Since $u_i^*, v_j^*\in\mathfrak{n}_\theta$
for $1\leqslant i\leqslant \frac{s}{2}$ and $1\leqslant j\leqslant\frac{\sfr-1}{2}$, and $f.1_{\lambda}=1_{\lambda}$ by definition, we have $z.1_{\lambda}=\lambda\cdot1_{\lambda}=\chi(z)\cdot1_{\lambda}$. As $z$ acts locally nilpotently on $U(\ggg)$, we can deduce that $z-\chi(z)$ acts locally nilpotently on $M$ for all $z\in\mathfrak{m}$. Therefore, $M$ is an object of $\mathcal{C}_\chi$. It follows from the discussion in \S\ref{skryabin} that $\text{Wh}(M)\neq0$, the algebra $U(\ggg,e)$ acts on $\text{Wh}(M)$, and $M\cong Q_\chi\otimes_{U(\ggg,e)}\text{Wh}(M)$ as $\ggg$-modules.

(2) For $1\leqslant l\leqslant\frac{\sfr-1}{2}$, since
$(v_l^2).1_{\lambda}=\frac{1}{2}[v_l,v_l].1_{\lambda}=0$,
%and
%\begin{equation*}
%\Big(v_{\frac{\sfr+1}{2}}\Big)^2.1_{\lambda}=\frac{1}{2}
%[v_{\frac{\sfr+1}{2}},v_{\frac{\sfr+1}{2}}].1_{\lambda}=\frac{1}{2}\cdot1_{\lambda},
%\end{equation*}
now observe that
\begin{equation*}
\begin{split}
u^*_k.m({\bf i,j,\iota,k,l,m,n},t)
\in %i_k\cdot m({\bf i-e}_k,{\bf j,\iota,k,l,m,n},t)\\
&\sum_{j=0}^ti_k\cdot C_t^j2^j\cdot m({\bf i-e}_k,{\bf j,\iota,k,l,m,n},t-j)\\
&+\text{span}\{m({\bf i',j',\iota',k',l',m',n'},t')\mid|{\bf i'}|+|{\bf j'}|\geqslant|{\bf i}|+|{\bf j}|\},\\
v^*_l.m({\bf i,j,\iota,k,l,m,n},t)
\in&%(-1)^{\sum_{r=1}^{l-1}j_r}j_l\cdot m({\bf i},{\bf j-e}_l,\iota,{\bf k,l,m,n},t)\\&\quad+
\sum_{i=0}^t(-1)^{\sum_{r=1}^{l-1}j_r}j_l\cdot C_t^i2^i\cdot m({\bf i},{\bf j-e}_l,\iota,{\bf k,l,m,n},t-i)
\\&+\text{span}\{m({\bf i',j',\iota',k',l',m',n'},t')\mid|{\bf i'}|+|{\bf j'}|\geqslant|{\bf i}|+|{\bf j}|\}\\
\end{split}
\end{equation*}
for all $1\leqslant k\leqslant \frac{s}{2}, 1\leqslant l\leqslant\frac{\sfr-1}{2}$, and for $t>0$ we have
\begin{equation*}
\begin{split}
(f-1).m({\bf i,j,\iota,k,l,m,n},t)
\in%2t\cdot m({\bf i},{\bf j,\iota,k,l,m,n},t-1)\\&+
&2^t\cdot m({\bf i},{\bf j,\iota,k,l,m,n},0)
\\&+\text{span}\{m({\bf i',j',\iota',k',l',m',n'},l')\mid l'>0\}.
\end{split}
\end{equation*}
%\begin{equation*}
%\begin{split}
%(f-1).m({\bf i,j,\iota,k,l,m,n},t)
%\in%2t\cdot m({\bf i},{\bf j,\iota,k,l,m,n},t-1)\\&+
%&\sum_{i=1}^tC_t^i2^i\cdot m({\bf i},{\bf j,\iota,k,l,m,n},t-i)
%\\&+\text{span}\{m({\bf i',j',\iota',k',l',m',n'},t)\mid |{\bf i'}|+|{\bf j'}|\geqslant|{\bf i}|+|{\bf j}|+1\}.
%\end{split}
%\end{equation*}
From all these it is immediate that the map $\text{pr}:\text{Wh}(M)\rightarrow M_0$ is injective.

(3) (i) Note that $1_{\lambda}\in\text{Wh}(M)$, and for all $t\in\mathfrak{h}^e$ we have
\begin{equation}\label{calcuoft}
\begin{split}
\Theta_{t}(1_{\lambda})&=\bigg(t-\frac{1}{2}\sum\limits_{\alpha\in S(-1)}z_\alpha[z_\alpha^*,t]\bigg)(1_{\lambda})\\
&=\bigg(t-\frac{1}{2}\sum\limits_{i=\frac{s}{2}+1}^{s}u_i[u_i^*,t]
-\frac{1}{2}\sum\limits_{i=\frac{\sfr+1}{2}}^{\sfr}v_i[v_i^*,t]\bigg)(1_{\lambda})\\
&=\bigg(t+\frac{1}{2}\sum\limits_{i=1}^{\frac{s}{2}}[u_i^*,[u_i,t]]
-\frac{1}{2}\sum\limits_{i=1}^{\frac{\sfr-1}{2}}[v_i^*,[v_i,t]]-\frac{1}{2}
v_{\frac{\sfr+1}{2}}[v_{\frac{\sfr+1}{2}},t]\bigg)(1_{\lambda})\\
&=\bigg(\lambda(t)-\frac{1}{2}\sum\limits_{i=1}^{\frac{s}{2}}\gamma_{\bar0i}(t)f
+\frac{1}{2}\sum\limits_{i=1}^{\frac{\sfr-1}{2}}\gamma_{\bar1i}(t)f-\frac{1}{8}\theta(t)f\bigg)
(1_{\lambda})\\
&=\bigg(\lambda+\delta+\frac{2s-2\sfr+1}{8}\theta\bigg)(t)\cdot(1_{\lambda})
\\
&=(\lambda+\delta)(t)\cdot1_{\lambda}.
\end{split}
\end{equation}

Suppose $v\in\ggg^e(0)$ is a root vector for $\mathfrak{h}$ corresponding to root $\gamma\in\Phi^+_{e,0}$. As $\frac{\theta}{2}$ is a simple root  by  Convention \ref{conventions}, and $0=[h_\theta,v]=\gamma(h_\theta)v$, then $v,[u_i^*,[u_i,v]]\in\mathfrak{n}_\theta$ for all $1\leqslant i\leqslant \frac{s}{2}$, and
$[v_i^*,[v_i,v]]\in\mathfrak{n}_\theta$ for all $1\leqslant i\leqslant\frac{\sfr-1}{2}$. As we also have $[v_{\frac{\sfr+1}{2}},v]\in\mathfrak{n}_\theta$, then
\begin{equation*}
\begin{split}
\Theta_v(1_{\lambda})&=\bigg(v-\frac{1}{2}\sum\limits_{\alpha\in S(-1)}z_\alpha[z_\alpha^*,v]\bigg)(1_{\lambda})\\
&=\bigg(v-\frac{1}{2}\sum\limits_{i=\frac{s}{2}+1}^{s}u_i[u_i^*,v]
-\frac{1}{2}\sum\limits_{i=\frac{\sfr+1}{2}}^{\sfr}v_i[v_i^*,v]\bigg)(1_{\lambda})\\
&=\bigg(v+\frac{1}{2}\sum\limits_{i=1}^{\frac{s}{2}}[u_i^*,[u_i,v]]
-\frac{1}{2}\sum\limits_{i=1}^{\frac{\sfr-1}{2}}[v_i^*,[v_i,v]]-\frac{1}{2}
v_{\frac{\sfr+1}{2}}[v_{\frac{\sfr+1}{2}},v]\bigg)(1_{\lambda})\\
&=0.\\
\end{split}
\end{equation*}

(ii) Now let $w\in\ggg^e(1)$ be a root vector for $\mathfrak{h}$ corresponding to root $\gamma'\in\Phi^+_{e,1}$. Then
\begin{equation*}
\begin{split}
\Theta_w(1_{\lambda})=&\bigg(w-\sum\limits_{\alpha\in S(-1)}z_\alpha[z_\alpha^*,w]+\frac{1}{3}\Big(\sum\limits_{\alpha,\beta\in S(-1)}z_\alpha z_\beta[z_\beta^*,[z_\alpha^*,w]]-2[w,f]\Big)\bigg)(1_{\lambda})\\
=&\bigg(w-\sum\limits_{i=\frac{s}{2}+1}^{s}u_i[u_i^*,w]-\sum\limits_{i=\frac{\sfr+1}{2}}^{\sfr}v_i[v_i^*,w]
+\frac{1}{3}\Big(\sum\limits_{i,j=\frac{s}{2}+1}^{s}u_i u_j[u_j^*,[u_i^*,w]]\\
&+\sum\limits_{i=\frac{s}{2}+1}^{s}\sum\limits_{j=1}^{\frac{s}{2}}u_i u_j[u_j^*,[u_i^*,w]]
+\sum\limits_{i=1}^{\frac{s}{2}}\sum\limits_{j=\frac{s}{2}+1}^{s}u_i u_j[u_j^*,[u_i^*,w]]\\
&+2\sum\limits_{i=1}^{\frac{s}{2}}\sum\limits_{j=\frac{\sfr+1}{2}}^{\sfr}u_i v_j[v_j^*,[u_i^*,w]]
+2\sum\limits_{i=\frac{s}{2}+1}^{s}\sum\limits_{j=1}^{\frac{\sfr-1}{2}}u_i v_j[v_j^*,[u_i^*,w]]\\
&+2\sum\limits_{i=\frac{s}{2}+1}^{s}\sum\limits_{j=\frac{\sfr+1}{2}}^{\sfr}u_i v_j[v_j^*,[u_i^*,w]]
+\sum\limits_{i=1}^{\frac{\sfr-1}{2}}\sum\limits_{j=\frac{\sfr+1}{2}}^{\sfr}v_i v_j[v_j^*,[v_i^*,w]]\\
&+\sum\limits_{i=\frac{\sfr+1}{2}}^{\sfr}\sum\limits_{j=1}^{\frac{\sfr-1}{2}}v_i v_j[v_j^*,[v_i^*,w]]+\sum\limits_{i=\frac{\sfr+1}{2}}^{\sfr}\sum\limits_{j=\frac{\sfr+1}{2}}^{\sfr}v_i v_j[v_j^*,[v_i^*,w]]-2[w,f]\Big)\bigg)(1_{\lambda})\\
%\end{split}
%\end{equation*}
%\begin{equation*}
%\begin{split}
=&\bigg(w+\sum\limits_{i=1}^{\frac{s}{2}}[u_i^*,[u_i,w]]-\Big(\sum\limits_{i=1}^{\frac{\sfr-1}{2}}
[v_i^*,[v_i,w]]+v_{\frac{\sfr+1}{2}}[v_{\frac{\sfr+1}{2}},w]\Big)\\
&+\frac{1}{3}\Big(\sum\limits_{i,j=1}^{\frac{s}{2}}u_i^* u_j^*[u_j,[u_i,w]]-\sum\limits_{i,j=1}^{\frac{s}{2}}u_i^* u_j[u_j^*,[u_i,w]]
-\sum\limits_{i,j=1}^{\frac{s}{2}}u_i u_j^*[u_j,[u_i^*,w]]\\
&+2\sum\limits_{i=1}^{\frac{s}{2}}\sum\limits_{j=1}^{\frac{\sfr-1}{2}}u_i v_j^*[v_j,[u_i^*,w]]+2\sum\limits_{i=1}^{\frac{s}{2}}u_i v_{\frac{\sfr+1}{2}}[v_{\frac{\sfr+1}{2}},[u_i^*,w]]
-2\sum\limits_{i=1}^{\frac{s}{2}}\sum\limits_{j=1}^{\frac{\sfr-1}{2}}u_i^* v_j[v_j^*,[u_i,w]]\\
&-2\sum\limits_{i=1}^{\frac{s}{2}}\sum\limits_{j=1}^{\frac{\sfr-1}{2}}u_i^* v_j^*[v_j,[u_i,w]]-2\sum\limits_{i=1}^{\frac{s}{2}}u_i^* v_{\frac{\sfr+1}{2}}[v_{\frac{\sfr+1}{2}},[u_i,w]]
+\sum\limits_{i,j=1}^{\frac{\sfr-1}{2}}v_i v_j^*[v_j,[v_i^*,w]]\\
&+\sum\limits_{i=1}^{\frac{\sfr-1}{2}}v_i v_{\frac{\sfr+1}{2}}[v_{\frac{\sfr+1}{2}},[v_i^*,w]]+\sum\limits_{i,j=1}^{\frac{\sfr-1}{2}}v_i^* v_j[v_j^*,[v_i,w]]+\sum\limits_{i=1}^{\frac{\sfr-1}{2}}v_{\frac{\sfr+1}{2}}v_i [v_i^*,[v_{\frac{\sfr+1}{2}},w]]\\
&+\sum\limits_{i,j=1}^{\frac{\sfr-1}{2}}v_i^*v_j^*[v_j,[v_i,w]]+\sum\limits_{i=1}^{\frac{\sfr-1}{2}}v_{\frac{\sfr+1}{2}} v_i^*[v_i,[v_{\frac{\sfr+1}{2}},w]]+\sum\limits_{i=1}^{\frac{\sfr-1}{2}} v_i^*v_{\frac{\sfr+1}{2}}[v_{\frac{\sfr+1}{2}},[v_i,w]]\\
&+v_{\frac{\sfr+1}{2}}^2 [v_{\frac{\sfr+1}{2}},[v_{\frac{\sfr+1}{2}},w]]-2[w,f]\Big)\bigg)(1_{\lambda})
\end{split}
\end{equation*}
\begin{equation}\label{discuss}
\begin{split}
=&\bigg(w+\sum\limits_{i=1}^{\frac{s}{2}}[u_i^*,[u_i,w]]-\Big(\sum\limits_{i=1}^{\frac{\sfr-1}{2}}
[v_i^*,[v_i,w]]+v_{\frac{\sfr+1}{2}}[v_{\frac{\sfr+1}{2}},w]\Big)\\
&+\frac{1}{3}\Big(\sum\limits_{i,j=1}^{\frac{s}{2}}[u_i^*,[u_j^*,[u_j,[u_i,w]]]]
-\sum\limits_{i=1}^{\frac{s}{2}}[u_i^*,[u_i,w]]f-\sum\limits_{i,j=1}^{\frac{s}{2}}u_j[u_i^*,[u_j^*,[u_i,w]]]\\
&-\sum\limits_{i,j=1}^{\frac{s}{2}}u_i [u_j^*,[u_j,[u_i^*,w]]]+2\sum\limits_{i=1}^{\frac{s}{2}}\sum\limits_{j=1}^{\frac{\sfr-1}{2}}u_i [v_j^*,[v_j,[u_i^*,w]]]+2\sum\limits_{i=1}^{\frac{s}{2}}u_i v_{\frac{\sfr+1}{2}}[v_{\frac{\sfr+1}{2}},[u_i^*,w]]\\
&-2\sum\limits_{i=1}^{\frac{s}{2}}\sum\limits_{j=1}^{\frac{\sfr-1}{2}} v_j [u_i^*,[v_j^*,[u_i,w]]]-2\sum\limits_{i=1}^{\frac{s}{2}}\sum\limits_{j=1}^{\frac{\sfr-1}{2}}[u_i^*, [v_j^*,[v_j,[u_i,w]]]]-2\sum\limits_{i=1}^{\frac{s}{2}} v_{\frac{\sfr+1}{2}}[u_i^*,[u_i,[v_{\frac{\sfr+1}{2}},w]]]\\
&+\sum\limits_{i,j=1}^{\frac{\sfr-1}{2}}v_i [v_j^*,[v_j,[v_i^*,w]]]+2\sum\limits_{i=1}^{\frac{\sfr-1}{2}}v_i v_{\frac{\sfr+1}{2}}[v_{\frac{\sfr+1}{2}},[v_i^*,w]]+\sum\limits_{i=1}^{\frac{\sfr-1}{2}}[v_i^*,[v_i,w]]f\\
&
-\sum\limits_{i,j=1}^{\frac{\sfr-1}{2}}v_j [v_i^*,[v_j^*,[v_i,w]]]%+\sum\limits_{i=1}^{\frac{\sfr-1}{2}}v_{\frac{\sfr+1}{2}}v_i [v_i^*,[v_{\frac{\sfr+1}{2}},w]]
+\sum\limits_{i,j=1}^{\frac{\sfr-1}{2}}[v_i^*,[v_j^*,[v_j,[v_i,w]]]]+2\sum\limits_{i=1}^{\frac{\sfr-1}{2}}v_{\frac{\sfr+1}{2}} [v_i^*,[v_i,[v_{\frac{\sfr+1}{2}},w]]]\\
&%-\sum\limits_{i=1}^{\frac{\sfr-1}{2}}v_{\frac{\sfr+1}{2}}
%[v_i^*,[v_{\frac{\sfr+1}{2}},[v_i,w]]]
+\frac{1}{2}[v_{\frac{\sfr+1}{2}},[v_{\frac{\sfr+1}{2}},w]]f-2[w,f]\Big)\bigg)(1_{\lambda})
\end{split}
\end{equation}

Now we will discuss the terms in \eqref{discuss}. To ease notation, we will call $i$ admissible for $u_i$ if $1\leqslant i\leqslant\frac{s}{2}$, and also for $v_i$ if $1\leqslant i\leqslant\frac{\sfr-1}{2}$. %$w\in \mathfrak{n}_\theta$, and
For any admissible $i,j$, by degree consideration we see that $[u_i^*,[u_j^*,[u_j,[u_i,w]]]]\in\ggg(-3)=\{0\}$. The same discussion entails that $[u_i^*, [v_j^*,[v_j,[u_i,w]]]]=[v_i^*,[v_j^*,[v_j,[v_i,w]]]]=0$ for all admissible $i,j$.

On the other hand, one can easily conclude that $[u_i^*,[u_j^*,[u_i,w]]]\in\ggg(-2)$ for all admissible $i,j$, thus $[u_i^*,[u_j^*,[u_i,w]]]=kf$ for some $k\in\mathbb{C}$. As $[u_i^*,[u_j^*,[u_i,w]]]$ is a root vector for $\mathfrak{h}$ corresponding to root $\gamma'+\gamma_{\bar0j}^*-\theta$, it follows from $f=e_{-\theta}$ that $k\neq0$ if and only if $\gamma'+\gamma_{\bar0j}^*-\theta=-\theta$, which is impossible. Then $[u_i^*,[u_j^*,[u_i,w]]]=0$. By the same discussion we can conclude that \begin{equation*}
\begin{split}
&[u_j^*,[u_j,[u_i^*,w]]]=[v_j^*,[v_j,[u_i^*,w]]]=[u_i^*,[v_j^*,[u_i,w]]]\cr
=&[v_j^*,[v_j,[v_i^*,w]]]=[v_i^*,[v_j^*,[v_i,w]]]=0
\end{split}
\end{equation*} for all admissible $i,j$.
The same consideration also applies for $[u_i^*,[u_i,[v_{\frac{\sfr+1}{2}},w]]]$ and $[v_i^*,[v_i,[v_{\frac{\sfr+1}{2}},w]]]$, which are nonzero
if and only if $\gamma'=\frac{1}{2}\theta$.
Moreover, since $\frac{\theta}{2}$ is a simple root by  Convention \ref{conventions}, for all admissible $i,j$ we have $[v_{\frac{\sfr+1}{2}},[u_i^*,w]]$,
$[v_{\frac{\sfr+1}{2}},[v_i^*,w]]\in\mathfrak{n}_\theta$ by weight consideration.

Note that $w$ is a linear combination of $f^*_1,\cdots,f^*_{\frac{s}{2}}, g^*_1,\cdots,g^*_{\frac{\sfr-1}{2}},[v_{\frac{\sfr+1}{2}},e]$, where $[v_{\frac{\sfr+1}{2}},e]$ is a root vector for $\mathfrak{h}$ corresponding to simple root $\frac{\theta}{2}$.
For $w'\in\{f^*_1,\cdots,f^*_{\frac{s}{2}}, g^*_1,\cdots,g^*_{\frac{\sfr-1}{2}}\}\subset\mathfrak{n}_\theta$, by weight consideration we have $[u_i^*,[u_i,w']],[v_i^*,[v_i,w']]$, $[v_{\frac{\sfr+1}{2}},w']$, %$[v_{\frac{\sfr+1}{2}},[u_i^*,w']]$, $[v_{\frac{\sfr+1}{2}},[v_i^*,w']]$,
$[v_{\frac{\sfr+1}{2}},[v_{\frac{\sfr+1}{2}},w']]$,
$[w',f]\in\mathfrak{n}_\theta$ for all admissible $i$. Taking all above into consideration, we know that \eqref{discuss} equals zero in this situation, and then $\Theta_{w'}(1_{\lambda})=0$.

It remains to consider the case with $w=[v_{\frac{\sfr+1}{2}},e]$.
For any $1\leqslant\alpha\leqslant\frac{s}{2}$ or $s+1\leqslant\alpha\leqslant s+\frac{\sfr-1}{2}$ (i.e., for all admissible $i$ of $u_i$'s and $v_i$'s), since \begin{equation*}
[[z_\alpha^*,[z_\alpha,e]],f]=[z_\alpha^*,[z_\alpha,[e,f]]]=[z_\alpha^*,[z_\alpha,h]]=f,
\end{equation*}
we have $[z_\alpha^*,[z_\alpha,e]]+\frac{1}{2}h\in\mathfrak{h}^e$. Set
\begin{align*}
\gamma^*_{|\alpha|,\alpha}:=\left\{\begin{array}{ll}\gamma^*_{\bar0\alpha}&\text{if}~1\leqslant\alpha\leqslant\frac{s}{2};\\
\gamma^*_{\bar1\alpha-s}&\text{if}~s+1\leqslant\alpha\leqslant s+\frac{\sfr-1}{2},\end{array}\right.
\end{align*}where $\gamma^*_{\bar0i}$ for $1\leqslant i\leqslant\frac{s}{2}$ and $\gamma^*_{\bar1j}$ for $1\leqslant j\leqslant\frac{\sfr-1}{2}$  are defined as in \S\ref{1.2}.
Let $x$ be arbitrary element in $\mathfrak{h}^e$. Then $(x,h)=0$, %\theta(x)=0$,
and
\begin{equation}\label{xzzeh}
\bigg(x,[z_\alpha^*,[z_\alpha,e]]+\frac{1}{2}h\bigg)=([x,z_\alpha^*],[z_\alpha,e])=
\gamma^*_{|\alpha|,\alpha}(x)(z_\alpha^*,[z_\alpha,e])=\gamma^*_{|\alpha|,\alpha}(x),
\end{equation}
that is,
\begin{equation}\label{z*ze}
[z_\alpha^*,[z_\alpha,e]]=-\frac{1}{2}h+t_{\bar\gamma^*_{|\alpha|,\alpha}}.
\end{equation}
Therefore,  we have
\begin{equation}\label{zzev1}
[z_\alpha^*,[z_\alpha,[v_{\frac{\sfr+1}{2}},e]]]
=-[[z_\alpha^*,[z_\alpha,e]],v_{\frac{\sfr+1}{2}}]
=[\frac{1}{2}h-t_{\bar\gamma^*_{|\alpha|,\alpha}},v_{\frac{\sfr+1}{2}}]
=\bigg(-\frac{1}{2}+\frac{1}{2}\theta\big(t_{\bar\gamma^*_{|\alpha|,\alpha}}\big)\bigg)v_{\frac{\sfr+1}{2}}
\end{equation}
%As the system of even and odd roots for all basic Lie superalgebras is described in \cite[\S4]{K2}, for each case we will discuss the terms of \eqref{zzev} separately. If $\ggg$ is of type $B(m|n)$, by \cite[Proposition 1.6]{K2} we see that $\theta=2\delta_n$.
For any $t\in\mathfrak{h}^{e_{\theta}}$, we have $\theta(t)=0$ by definition. Then it follows from $\theta\big(t_{\bar\gamma^*_{|\alpha|,\alpha}}\big)=0$ and \eqref{zzev1} that \begin{equation}\label{zzev}
[z_\alpha^*,[z_\alpha,[v_{\frac{\sfr+1}{2}},e]]]=-\frac{1}{2}v_{\frac{\sfr+1}{2}}.
\end{equation}
%Also note that
%\begin{equation*}
%[v_{\frac{\sfr+1}{2}},[v_{\frac{\sfr+1}{2}},e]]
%=[[v_{\frac{\sfr+1}{2}},v_{\frac{\sfr+1}{2}}],e]-[v_{\frac{\sfr+1}{2}},[v_{\frac{\sfr+1}{2}},e]]
%=-h-[v_{\frac{\sfr+1}{2}},[v_{\frac{\sfr+1}{2}},e]],
%\end{equation*}
%thus \begin{equation}\label{vev}
%[v_{\frac{\sfr+1}{2}},[v_{\frac{\sfr+1}{2}},e]]=-\frac{1}{2}h.
%\end{equation}
Since
\begin{equation}\label{vve}
[v_{\frac{\sfr+1}{2}},[v_{\frac{\sfr+1}{2}},e]]=[[v_{\frac{\sfr+1}{2}},e],v_{\frac{\sfr+1}{2}}]=-\frac{1}{2}h
\end{equation}
by \eqref{vev}, then
\begin{equation}\label{vvve}
[v_{\frac{\sfr+1}{2}},[v_{\frac{\sfr+1}{2}},[v_{\frac{\sfr+1}{2}},e]]]=[v_{\frac{\sfr+1}{2}},-\frac{1}{2}h]=-\frac{1}{2}v_{\frac{\sfr+1}{2}},
\end{equation}
and for any $1\leqslant\alpha\leqslant\frac{s}{2}$ or $s+1\leqslant\alpha\leqslant s+\frac{\sfr-1}{2}$ , we have
\begin{equation}\label{zzvve}
[z_\alpha^*,[z_\alpha,[v_{\frac{\sfr+1}{2}},[v_{\frac{\sfr+1}{2}},e]]]]=-[z_\alpha^*,[z_\alpha,\frac{1}{2}h]]
=-\frac{1}{2}[z_\alpha^*,z_\alpha]=-\frac{1}{2}f.
\end{equation}
Also note that
\begin{equation}\label{vef}
[[v_{\frac{\sfr+1}{2}},e],f]=[v_{\frac{\sfr+1}{2}},[e,f]]=[v_{\frac{\sfr+1}{2}},h]
=v_{\frac{\sfr+1}{2}}.
\end{equation}
Combining \eqref{discuss}, \eqref{zzev}---\eqref{vef} with our earlier discussion, one can conclude that
\begin{equation}\label{w+1}
\begin{split}
\Theta_{[v_{\frac{\sfr+1}{2}},e]}(1_{\lambda})=&\bigg([v_{\frac{\sfr+1}{2}},e]-\frac{s}{4}v_{\frac{\sfr+1}{2}}-\Big(-\frac{\sfr-1}{4}v_{\frac{\sfr+1}{2}}
-\frac{1}{2}v_{\frac{\sfr+1}{2}}h\Big)+\frac{1}{3}\Big(\frac{s}{4}v_{\frac{\sfr+1}{2}}+\frac{s}{2}v_{\frac{\sfr+1}{2}}
\\
&-\frac{\sfr-1}{4}v_{\frac{\sfr+1}{2}}-\frac{\sfr-1}{2}v_{\frac{\sfr+1}{2}}-\frac{1}{4}v_{\frac{\sfr+1}{2}}
-2v_{\frac{\sfr+1}{2}}\Big)\bigg)(1_{\lambda})\\
=&\bigg([v_{\frac{\sfr+1}{2}},e]-\frac{3}{4}v_{\frac{\sfr+1}{2}}+\frac{1}{2}
v_{\frac{\sfr+1}{2}}h\bigg)(1_{\lambda}).
\end{split}
\end{equation}
Since \begin{equation}\label{EFh}
\bigg(E-\frac{3}{4}F+\frac{1}{2}Fh\bigg).1_{\lambda}=0
\end{equation}by our assumption (see \S\ref{4.2.2}),
i.e.,
\begin{equation*}
\bigg([v_{\frac{\sfr+1}{2}},e]-\frac{3}{4}v_{\frac{\sfr+1}{2}}+\frac{1}{2}
v_{\frac{\sfr+1}{2}}h\bigg)(1_{\lambda})=0,
\end{equation*}
then \eqref{w+1} entails that $\Theta_{[v_{\frac{\sfr+1}{2}},e]}(1_{\lambda})=0$.
%=\bigg(\Big(-\frac{s}{12}+\frac{\sfr}{3}+\frac{5}{12}\Big)F-\frac{1}{2}
%hF\bigg).1_{\lambda}
Thus we have $\Theta_w(1_{\lambda})=0$ for all positive root vectors $w\in\ggg^e(1)$. Moreover, it follows from \begin{equation}\label{EC}
\begin{split}
&[E-\frac{3}{4}F+\frac{1}{2}Fh,E-\frac{3}{4}F+\frac{1}{2}Fh]\otimes1_\chi\\
=&\bigg([E,E]-\frac{3}{4}[F,E]+\frac{1}{2}[Fh,E]-\frac{3}{4}[E,F]+\frac{9}{16}[F,F]-\frac{3}{8}[Fh,F]
+\frac{1}{2}[E,Fh]-\frac{3}{8}[F,Fh]\\
&+\frac{1}{4}[Fh,Fh]\bigg)\otimes1_\chi\\
=&\bigg([E,E]-\frac{3}{2}[F,E]+[F,E]h+F[h,E]+\frac{9}{16}[F,F]-\frac{3}{4}[F,F]h-\frac{3}{4}F[h,F]+\frac{1}{4}[F,F]h^2\\
&
+\frac{1}{4}F[h,F]h-\frac{1}{4}F[F,h]h\bigg)\otimes1_\chi\\
=&\bigg(2e-\frac{3}{2}h+\frac{1}{2}h^2+FE+\frac{1}{8}\bigg)\otimes1_\chi\\
=&\bigg(C_\theta+\frac{1}{8}\bigg)\otimes1_\chi
\end{split}
\end{equation}that
\begin{equation}\label{ctheta1lambda}
C_\theta.1_{\lambda}=\bigg([E-\frac{3}{4}F+\frac{1}{2}Fh,E-\frac{3}{4}F+\frac{1}{2}Fh]-\frac{1}{8}\bigg).1_{\lambda}=-\frac{1}{8}\cdot1_{\lambda}.
\end{equation}

(iii) Let $C_0$ be the Casimir element of $U(\ggg^e(0))$ as defined in \eqref{C0def}.
In virtue of Theorem \ref{main3}, \eqref{ctheta1lambda} and \eqref{vev}, we have
\begin{equation}\label{C1}
\begin{split}
C(1_{\lambda})=&\bigg(2e+\frac{h^2}{2}-\Big(1+\frac{s-\sfr}{2}\Big)h+C_0
+2\sum\limits_{\alpha\in S(-1)}(-1)^{|\alpha|}[e,z_\alpha^*]z_\alpha\bigg)(1_{\lambda})\\
=&\bigg(2e+\frac{h^2}{2}-\Big(1+\frac{s-\sfr}{2}\Big)h+C_0
+2\sum\limits_{i=1}^{\frac{s}{2}}
[[e,u_i^*],u_i]-2\sum\limits_{i=1}^{\frac{\sfr-1}{2}}
[[e,v_i^*],v_i]\\
&-2[e,v_{\frac{\sfr+1}{2}}]v_{\frac{\sfr+1}{2}}\bigg)(1_{\lambda})\\
=&\bigg(C_\theta+C_0-\frac{s-\sfr+1}{2}h
+2\sum\limits_{i=1}^{\frac{s}{2}}
[[e,u_i^*],u_i]-2\sum\limits_{i=1}^{\frac{\sfr-1}{2}}
[[e,v_i^*],v_i]\bigg)(1_{\lambda})\\
=&\bigg(-\frac{1}{8}+C_0-\frac{s-\sfr+1}{2}h
+2\sum\limits_{i=1}^{\frac{s}{2}}
[[e,u_i^*],u_i]-2\sum\limits_{i=1}^{\frac{\sfr-1}{2}}
[[e,v_i^*],v_i]\bigg)(1_{\lambda}).
\end{split}
\end{equation}
For any $1\leqslant\alpha\leqslant\frac{s}{2}$ or $s+1\leqslant\alpha\leqslant s+\frac{\sfr-1}{2}$,
since
$[[[e,z_\alpha^*],z_\alpha],f]=[[[e,f],z_\alpha^*],z_\alpha]=[[h,z_\alpha^*],z_\alpha]=-f$, one can conclude that $[[e,z_\alpha^*],z_\alpha]-\frac{1}{2}h\in\mathfrak{h}^e$. Let $x$ be an arbitrary element in $\mathfrak{h}^e$. Then
$(x,h)=0$, $\theta(x)=0$, and
\begin{equation*}
\bigg(x,[[e,z_\alpha^*],z_\alpha]-\frac{1}{2}h\bigg)=([x,[e,z_\alpha^*]],z_\alpha)
=\gamma^*_{|\alpha|,\alpha}(x)([e,z_\alpha^*],z_\alpha)=\gamma^*_{|\alpha|,\alpha}(x),
\end{equation*}
that is,
\begin{equation}\label{ez*zalpha}
[[e,z_\alpha^*],z_\alpha]-\frac{1}{2}h=t_{\bar{\gamma}^*_{|\alpha|,\alpha}}.
\end{equation}
Then
\begin{eqnarray}\label{ez}
\bigg(2\sum\limits_{i=1}^{\frac{s}{2}}
[[e,u_i^*],u_i]-2\sum\limits_{i=1}^{\frac{\sfr-1}{2}}
[[e,v_i^*],v_i]-\frac{s-\sfr+1}{2}h\bigg)(1_{\lambda})
&=&\bigg(2\sum\limits_{i=1}^{\frac{s}{2}}t_{\bar{\gamma}^*_{\bar0i}}-
2\sum\limits_{i=1}^{\frac{\sfr-1}{2}}t_{\bar{\gamma}^*_{\bar1i}}\bigg)(1_{\lambda})\nonumber\\
&=&2\bigg(\sum\limits_{i=1}^{\frac{s}{2}}\lambda(t_{\bar{\gamma}^*_{\bar0i}})-\sum\limits_{i=1}^{\frac{\sfr-1}{2}}\lambda(t_{\bar{\gamma}^*_{\bar1i}})\bigg)\cdot1_{\lambda}\nonumber\\
&=&2\bigg(\lambda,\sum\limits_{i=1}^{\frac{s}{2}}\bar{\gamma}^*_{\bar0i}-\sum\limits_{i=1}^{\frac{\sfr-1}{2}}\bar{\gamma}^*_{\bar1i}\bigg)\cdot1_{\lambda}\nonumber\\
&=&4(\lambda,\bar\delta)\cdot1_{\lambda}.
\end{eqnarray}
As $C_0$ is a Casimir element in $U(\ggg^e(0))$, and all positive vectors in $\ggg^e(0)$ annihilate $1_{\lambda}$, by the same discussion as in \cite[Lemma 8.5.3]{M1} we see that
\begin{equation}\label{ez2}
C_0(1_{\lambda})=\bigg(\sum_{i=1}^{k-1}\lambda(h_i)^2+\sum_{\alpha\in\Phi^{+}_{e,0}}(-1)^{|\alpha|}\lambda(h_\alpha)\bigg)(1_{\lambda})=(\lambda,\lambda+2\bar\rho_{e,0})\cdot1_{\lambda},
\end{equation}where $|\alpha|$ denotes the parity of $\alpha$.

We can conclude from \eqref{C1}, \eqref{ez} and \eqref{ez2} that
\begin{equation}\label{c1lambdac}
C(1_{\lambda})=\bigg(-\frac{1}{8}+(\lambda,\lambda+2\bar\rho_{e,0})+4(\lambda,\bar\delta)\bigg)
\cdot1_{\lambda}
=\bigg(-\frac{1}{8}+(\lambda,\lambda+2\bar\rho)\bigg)\cdot1_{\lambda}.
\end{equation}
%{\color{blue} Then under the twisted action of $U(\ggg,e)$ on $\text{Wh}(M(\lambda))$,
%we have
%\begin{equation}\label{c1lambdac22}
%C.1_{\lambda}=(\rho(C)+\epsilon)(1_{\lambda})
%=\bigg(-\frac{1}{8}+(\lambda,\lambda+2\bar\rho)+\epsilon\bigg)\cdot1_{\lambda}.
%\end{equation}
%}

(iv) In virtue of \eqref{calcuoft} and \eqref{c1lambdac}, set
\begin{equation}\label{deflambda'c'}\lambda':=\lambda+\bar\delta,\qquad c':=-\frac{1}{8}+(\lambda,\lambda+2\bar\rho)+\epsilon.
\end{equation} Under the twisted action of $U(\ggg,e)$ on $Z_{U(\ggg,e)}(\lambda',c')$, we have
\begin{equation*}\label{c1lambdacepsilon}
C.v=\text{tw}\,(C)(v)=(C-\epsilon)(v)=(c'-\epsilon)(v)
=\bigg(-\frac{1}{8}+(\lambda,\lambda+2\bar\rho)\bigg)(v)
\end{equation*}for any $v\in Z_{U(\ggg,e)}(\lambda',c')$.
To show that  $(\lambda',c')\in(\mathfrak{h}^e)^*\times\bbc$ is a matchable pair as in \eqref{c0clambda}, we need another description of $c'$ in \eqref{deflambda'c'}. Since $h_1,\cdots,h_{k-1}$ is a orthogonal basis of $\hhh^e$ with respect to $(\cdot,\cdot)$, we can write $t_{\bar{\gamma}^*_{ai}}=\sum_{j=1}^{k-1}l_jh_j$ for $a\in\{\bar0,\bar1\}$, then we have $${\gamma}^*_{ai}(h_m)=(t_{\bar{\gamma}^*_{ai}},h_m)=\sum_{j=1}^{k-1}l_j(h_j,h_m)=l_m,$$ thus
\begin{equation}\label{t*bargamma}
t_{\bar{\gamma}^*_{ai}}=\sum_{i=1}^{k-1}{\gamma}^*_{ai}(h_j)h_j=-\sum_{i=1}^{k-1}\gamma_{ai}(h_j)h_j.
\end{equation}
So we have
\begin{equation}\label{delalambda}
(\lambda,\bar\delta)=\frac{1}{2}\lambda\bigg(\sum\limits_{j=1}^{\frac{s}{2}}t_{\bar{\gamma}^*_{\bar0j}}
-\sum\limits_{j=1}^{\frac{\sfr-1}{2}}t_{\bar{\gamma}^*_{\bar1j}}\bigg)
=\frac{1}{2}\sum\limits_{i=1}^{k-1}\bigg(-\sum_{j=1}^{\frac{s}{2}}\gamma_{\bar0j}(h_i)+
\sum_{j=1}^{\frac{\sfr-1}{2}}\gamma_{\bar1j}(h_i)\bigg)\lambda(h_i)
=\sum\limits_{i=1}^{k-1}\delta(h_i)\lambda(h_i).
\end{equation}
Thanks to the definition of $C_0$ in \eqref{C0def}, and also \eqref{x*x},  \eqref{y*y}, we obtain
\begin{equation}\label{anotherdesC0}
\begin{split}
C_0(1_{\lambda})=&\bigg(\sum_{i=1}^{k-1}h_i^2+\sum_{i=1}^{w}x_ix^*_i+\sum_{i=1}^{w}x^*_ix_i+\sum_{i=1}^{\ell}y_iy^*_i-
\sum_{i=1}^{\ell}y^*_iy_i\bigg)(1_{\lambda})\\
=&\bigg(\sum_{i=1}^{k-1}h_i^2+\sum_{i=1}^{w}[x^*_i,x_i]-\sum_{i=1}^{\ell}[y^*_i,y_i]\bigg)(1_{\lambda})\\
=&\bigg(\sum_{i=1}^{k-1}h_i^2+
\sum_{i=1}^{k-1}\sum_{j=1}^{w}\beta_{\bar0j}(h_i)h_i-\sum_{i=1}^{k-1}\sum_{j=1}^{\ell}\beta_{\bar1j}(h_i)h_i\bigg)(1_{\lambda})\\
=&\bigg(\sum_{i=1}^{k-1}\lambda(h_i)^2+
2\sum_{i=1}^{k-1}\rho_{e,0}(h_i)\lambda(h_i)\bigg)(1_{\lambda}).
\end{split}
\end{equation}

From \eqref{C1}, \eqref{ez}, \eqref{delalambda} and \eqref{anotherdesC0}, we get
\begin{equation}\label{anotherdesC}
C(1_{\lambda})=\bigg(-\frac{1}{8}+\sum_{i=1}^{k-1}\lambda(h_i)^2+
2\sum_{i=1}^{k-1}\rho_{e,0}(h_i)\lambda(h_i)+4\sum\limits_{i=1}^{k-1}\delta(h_i)\lambda(h_i)\bigg)(1_{\lambda}).
\end{equation}
Taking \eqref{c1lambdac} and \eqref{anotherdesC} into consideration, we have
\begin{equation}
\begin{split}
c'=&-\frac{1}{8}+(\lambda,\lambda+2\bar\rho)+\epsilon\\
=&-\frac{1}{8}+\sum_{i=1}^{k-1}\lambda(h_i)^2+
2\sum_{i=1}^{k-1}\rho_{e,0}(h_i)\lambda(h_i)+4\sum\limits_{i=1}^{k-1}\delta(h_i)\lambda(h_i)+
c_0+\frac{1}{8}\\
&+2\sum_{i=1}^{k-1}\rho_{e,0}(h_i)\delta(h_i)
+3\sum_{i=1}^{k-1}\delta(h_i)^2\\
=&c_0+\sum_{i=1}^{k-1}((\lambda+\delta)(h_i))^2+2\sum_{i=1}^{k-1}((\lambda+\delta)(\rho_{e,0}+\delta))(h_i)\\
=&c_0+\sum_{i=1}^{k-1}\lambda'(h_i)^2+2\sum_{i=1}^{k-1}(\lambda'\cdot(\rho_{e,0}+\delta))(h_i),
\end{split}
\end{equation}
which verifies the equation \eqref{c0clambda}.
%It follows from \eqref{c0clambda} and \eqref{anotherdesC} that
%\begin{equation}\label{c0c'lambda'}
%\begin{split}
%0=&\frac{1}{8}-\sum_{i=1}^{k-1}\lambda(h_i)^2-
%2\sum_{i=1}^{k-1}((\rho_{e,0}+2\delta)\cdot\lambda)(h_i)+c_0
%+\sum_{i=1}^{k-1}(\lambda(h_i)+\delta(h_i))^2+2\sum_{i=1}^{k-1}((\lambda+\delta)\\
%&\cdot(\rho_{e,0}+\delta))(h_i)\\
%=&\sum_{i=1}^{k-1}(2\rho_{e,0}\delta+3\delta^2)(h_i)+\frac{1}{8}+c_0.
%\end{split}
%\end{equation}

Denote by $V_0$ the $U(\ggg,e)$-submodule of $M$ generated by $1_{\lambda}$, and let $I_{\lambda',c'}$ be the left ideal of $U(\ggg,e)$ as defined in \S\ref{3.2.1}.
From all the discussion above we know that the left ideal $I_{\lambda',c'}$ of $U(\ggg,e)$ annihilates $1_{\lambda}$. Then $V_0$ is a homomorphic image of the Verma module $Z_{U(\ggg,e)}({\lambda',c'})$.

(v) Now we claim that the restriction $\text{pr}:\text{Wh}(M)\twoheadrightarrow M_0$ to $V_0$ is surjective. Recall that $M_0$ is spanned by all $m({\bf 0,0,\iota,k,l,m,n},0)$ in $M$ with $\iota\in\Lambda_1$, ${\bf k}\in\mathbb{Z}_+^w$, ${\bf l}\in\Lambda_\ell$, ${\bf m}\in\mathbb{Z}_+^{\frac{s}{2}}$, and ${\bf n}\in\Lambda_{\frac{\sfr-1}{2}}$. It is obvious that $m({\bf 0,0},0,{\bf 0,0,0,0},0)=1_{\lambda}\in\text{pr}(V_0)$. Assume that all the vectors $m({\bf 0,0,\iota,k,l,m,n},0)$ of Kazhdan degree $\iota+2(|\mathbf{k}|+|\mathbf{l}|)+3(|\mathbf{m}|+|\mathbf{n}|)<p$ are in $\text{pr}(V_0)$. Set $m({\bf 0,0},\iota,{\bf a,b,c,d},0)\in M_0$ to be such that
$\iota+2(|\mathbf{a}|+|\mathbf{b}|)+3(|\mathbf{c}|+|\mathbf{d}|)=p$ and $\iota+|\mathbf{a}|+|\mathbf{b}|+|\mathbf{c}|+|\mathbf{d}|=q$, and denote by $M_{p,q}$ the span of all $m({\bf i,j,\iota,k,l,m,n},t)$ of Kazhdan degree $p$ with
$|\mathbf{i}|+|\mathbf{j}|+\iota+|\mathbf{k}|+|\mathbf{l}|+|\mathbf{m}|+|\mathbf{n}|+t>q$.
Assume that all vectors $m({\bf 0,0,\iota,k,l,m,n},0)$ of Kazhdan degree $p$ with
$\iota+|\mathbf{k}|+|\mathbf{l}|+|\mathbf{m}|+|\mathbf{n}|>q$ are in $\text{pr}(V_0)$. Since $M$ is a filtrated $U(\ggg)$-module, and also
\begin{equation*}
v_{\frac{\sfr+1}{2}}^2(1_{\lambda})=\frac{1}{2}\cdot1_{\lambda},~~
y_i^2(1_{\lambda})=\frac{1}{2}[y_i,y_i](1_{\lambda}),~~ g_j^2(1_{\lambda})=\frac{1}{2}[g_j,g_j](1_{\lambda})
\end{equation*}
for $1\leqslant i\leqslant\ell$ and $1\leqslant j\leqslant\frac{\sfr-1}{2}$, then it follows from Theorem \ref{main3} that
\begin{equation}\label{leadinglow}
\Theta_{v_{\frac{\sfr+1}{2}}}^{\iota}\cdot\prod_{i=1}^w\Theta_{x_i}^{a_i}\cdot\prod_{i=1}^\ell\Theta_{y_i}^{b_i}\cdot
\prod_{i=1}^{\frac{s}{2}}\Theta_{f_i}^{c_i}\cdot\prod_{i=1}^{\frac{\sfr-1}{2}}
\Theta_{g_i}^{d_i}(1_{\lambda})
\in m({\bf 0,0},\iota,{\bf a,b,c,d},0)+M_{p,q}+M^{p-1}.
\end{equation}
By our assumptions on $p$ and $q$ we can obtain $m({\bf 0,0},\iota,{\bf a,b,c,d},0)\in\text{pr}(V_0+M_{p,q}+M^{p-1})=\text{pr}(V_0)$. Then our claim follows by double induction on $p$ and $q$.
Recall in (2) we have already established that $\text{pr}:\text{Wh}(M)\rightarrow M_0$ is injective, this yields $\text{Wh}(M)=V_0$.

(4) Applying \eqref{leadinglow} it is easy to observe that $\Theta_{v_{\frac{\sfr+1}{2}}}^{\iota}\cdot\prod_{i=1}^w\Theta_{x_i}^{a_i}\cdot\prod_{i=1}^\ell\Theta_{y_i}^{b_i}\cdot
\prod_{i=1}^{\frac{s}{2}}\Theta_{f_i}^{c_i}\cdot\prod_{i=1}^{\frac{\sfr-1}{2}}
\Theta_{g_i}^{d_i}(1_{\lambda})$ with $\iota\in\Lambda_1$, ${\bf a}\in\mathbb{Z}_+^w$, ${\bf b}\in\Lambda_\ell$, ${\bf c}\in\mathbb{Z}_+^{\frac{s}{2}}$, and ${\bf d}\in\Lambda_{\frac{\sfr-1}{2}}$ are linearly independent over $\mathbb{C}$. %Combining Lemma \ref{left ideal2} with
Hence if follows from Lemma \ref{left ideal2}, \eqref{kl+q+1} and the discussion at the beginning of  \S\ref{3.2.2} that $V_0\cong Z_{U(\ggg,e)}(\lambda+\bar{\delta},-\frac{1}{8}+(\lambda+2\bar\rho,\lambda)+\epsilon)$ as $U(\ggg,e)$-modules.
$\hfill\square$
\begin{rem}
We guess that the shift $\epsilon\neq0$ for all basic Lie superalgebras. For example, let $\ggg=\mathfrak{osp}(1|2)$ be as in Lemma \ref{osp12genecom}. Then $\ggg(-1)_{\bar1}=\bbc F$, $s=\text{dim}\,\ggg(-1)_{\bar0}=0$, $\sfr=\text{dim}\,\ggg(-1)_{\bar1}=1$,$F^*=-\frac{1}{2}F$, and $([E,E],f)=2(e,f)=2$. It follows from \eqref{decidec0} that
\begin{equation*}
\begin{split}
c_0=&\frac{1}{([E,E],f)}\bigg(\frac{1}{12}\Big([[[E, F],F],[-\frac{1}{2}F,[-\frac{1}{2}F,E]]]\Big)\otimes1_\chi-\frac{1}{12}([E,E],f)\bigg)\\
=&\frac{1}{24}\bigg(\frac{1}{4}[[h,F],[F,h]]\otimes1_\chi-2(e,f)\bigg)\\
=&\frac{1}{24}\bigg(-\frac{1}{4}[F,F]\otimes1_\chi-2\bigg)
=\frac{1}{24}\bigg(\frac{1}{2}-2\bigg)=-\frac{1}{16}.
\end{split}
\end{equation*}
As $\hhh^e=0$, we get $\epsilon=c_0+\frac{1}{8}=-\frac{1}{16}+\frac{1}{8}=\frac{1}{16}\neq0$. However, the complication for the calculation of $c_0$ in \eqref{decidec0} makes it difficult to verify $\epsilon\neq0$ in general.
\end{rem}
\subsection{}\label{5.4}
This part is devoted to minimal finite $W$-superalgebra $U(\ggg,e)$ of type even. All $\ggg$ corresponding to this type are listed in Table 2. In this type, both $U(\ggg,e)$ and $W_\chi'$ are identical.  %(for which the corresponding $\ggg$ is given in Table 2), which coincides with the corresponding minimal refined $W$-superalgebra $W_\chi'$ of type even.

To determine the composition factors of the Verma modules $Z_{U(\ggg,e)}(\lambda,c)$ with their multiplicities, we will express these $U(\ggg,e)$-modules in terms of the $\ggg$-modules obtained by parabolic induction from Whittaker modules  for $\mathfrak{sl}(2)$. As mentioned in \S\ref{4.1.1}, the latter modules have been studied in much detail in \cite{C,C3}, and Chen-Cheng \cite[Theorem 1]{C3} gave a complete solution to the problem of determining the composition factors of the standard Whittaker modules in terms
of composition factors of Verma modules  in the  category $\mathcal{O}$. As a special case, if $\ggg$ is a basic Lie superalgebra of type I;
or to say, if $\ggg$ is a simple Lie algebra, or $\mathfrak{sl}(m|n)$ with $m\neq n,m\geqslant2$, or $\mathfrak{psl}(m|m)$ with $m\geqslant2$, or $\mathfrak{spo}(2m|2)$   (see Tables 1 and 2), then the corresponding minimal finite $W$-superalgebra $U(\ggg,e)$ is of type even. It follows from \cite[Theorem C]{C} that the  composition of standard Whittaker modules  can be computed by some already known results (e.g., \cite{Bao,BW,Br, CLW2, CLW,  CMW}) on the irreducible characters of the BGG category $\mathcal{O}$. %We are going to rely on Skryabin's equivalence in Theorem \ref{skry}; the Kazhdan filtration of $U(\ggg,e)$ will play an important role too.
%The classification of all simple $\mathfrak{sl}(2)$ modules is known (see \cite{Bl}), where the Whittaker modules of $\mathfrak{sl}(2)$ play a critical role.
%Beside which, the parabolic induced modules of semi-simple Lie algebras from Whittaker modules for $\mathfrak{sl}(2)$ have been studied in much detail in
%\cite{Ba, Mc, Mi}. It is shown that the composition multiplicities of simple modules in parabolic induced modules follows from multiplicities of corresponding modules in the BGG category, and the latter can be calculated by using Kazhdan-Lusztig algorithm. %It is also notable that Soergel's functor plays a key role in the discussion there.
%
%
%
%As mentioned in \S\ref{4.1}, the construction of Whittaker modules for basic classical Lie superalgebras of type I from Whittaker modules for their even part was studied in \cite{Bag}.
%However, the composition series of Whittaker modules for these Lie superalgebras were not extensively discussed there, especially the translation for the corresponding composition multiplicities we referred to in the previous paragraph. It will be a very meaningful task to fill in this gap in the future.

%the study of this problem
%Although the theory related to the parabolic induced modules of Whittaker modules for $\mathfrak{sl}(2)$ is fruitful for the Lie algebra case, little is known
%about the one concerning with Lie superalgebras.

Denote by $\mathfrak{s}_\theta$ the subalgebra of $\ggg$ spanned by $(e,h,f)=(e_\theta,h_\theta,e_{-\theta})$, and put
\begin{equation*}
\mathfrak{p}_\theta:=\mathfrak{s}_\theta+\mathfrak{h}+\sum_{\alpha\in\Phi^+}\bbc e_\alpha,\qquad
\mathfrak{n}_\theta:=\sum_{\alpha\in\Phi^+\backslash\{\theta\}}\bbc e_\alpha,\qquad \widetilde{\mathfrak{s}}_\theta:=\mathfrak{h}^e\oplus\mathfrak{s}_\theta.
\end{equation*}
It is obvious that $\mathfrak{p}_\theta=\widetilde{\mathfrak{s}}_\theta\oplus\mathfrak{n}_\theta$ is a parabolic subalgebra of $\ggg$ with nilradical $\mathfrak{n}_\theta$ and $\widetilde{\mathfrak{s}}_\theta$ is a Levi subalgebra of $\mathfrak{p}_\theta$. Set $C_\theta:=ef+fe+\frac{1}{2}h^2=2ef+\frac{1}{2}h^2-h$ to be a Casimir element of $U(\mathfrak{s}_\theta)$. For $\lambda\in(\mathfrak{h}^e)^*$ and $c\in\bbc$, write $I_\theta(\lambda,c)$ for the left ideal of $U(\mathfrak{p}_\theta)$ generated by $f-1, C_\theta-c$, all $t-\lambda(t)$ with $t\in\mathfrak{h}^e$, and all $e_\gamma$ with $\gamma\in\Phi^+\backslash\{\theta\}$.

Set $Y(\lambda,c):=U(\mathfrak{p}_\theta)/I_\theta(\lambda,c)$ to be a $\mathfrak{p}_\theta$-module with the trivial action of $\mathfrak{n}_\theta$, and let $1_{\lambda,c}$ denote the image of $1$ in $Y(\lambda,c)$. Since $f.1_{\lambda,c}=1_{\lambda,c}$ by definition, then
\begin{equation*}
e.1_{\lambda,c}=\frac{1}{2}\bigg(C_\theta-\frac{1}{2}h^2+h\bigg).1_{\lambda,c}=\bigg(-\frac{1}{4}h^2+\frac{1}{2}h
+\frac{1}{2}c\bigg).1_{\lambda,c}.
\end{equation*}
Combining this with the PBW theorem, we see that the vectors $\{h^k\cdot1_{\lambda,c}\mid k\in\mathbb{Z}_+\}$ form a $\bbc$-basis of $Y(\lambda,c)$. Moreover, one can easily conclude that $Y(\lambda,c)$
is isomorphic to a Whittaker module for $\mathfrak{s}_\theta\cong\mathfrak{sl}(2)$.

It follows from the discussion above that the vectors
\begin{equation*}
\begin{split}
&m({\bf i,j,k,l,m,n},t):=\\
&u_1^{i_1}\cdots u_{\frac{s}{2}}^{i_{\frac{s}{2}}}\cdot v_1^{j_1}\cdots v_{\frac{\sfr}{2}}^{j_{\frac{\sfr}{2}}}\cdot x_1^{k_1}\cdots x_w^{k_w}\cdot y_1^{l_1}\cdots y_\ell^{l_\ell}\cdot f_1^{m_1}\cdots f_{\frac{s}{2}}^{m_{\frac{s}{2}}}\cdot
g_1^{n_1}\cdots g_{\frac{\sfr}{2}}^{n_{\frac{\sfr}{2}}}\cdot h^t(1_{\lambda,c})
\end{split}
\end{equation*}
with ${\bf i,m}\in\mathbb{Z}_+^{\frac{s}{2}}$, ${\bf j,n}\in\Lambda_{\frac{\sfr}{2}}$, ${\bf k}\in\mathbb{Z}_+^w$, ${\bf l}\in\Lambda_\ell$, and $t\in\mathbb{Z}_+$ form a $\bbc$-basis of the induced $\ggg$-module
\begin{equation*}
M(\lambda,c):=U(\ggg)\otimes_{U(\mathfrak{p}_\theta)}Y(\lambda,c).
\end{equation*}
Put \begin{equation}\label{deltarho2}
\begin{split}
\delta=&\frac{1}{2}\bigg(\sum\limits_{i=1}^{\frac{s}{2}}\gamma^*_{\bar0i}-\sum\limits_{i=1}^{\frac{\sfr}{2}}\gamma^*_{\bar1i}\bigg)
=\frac{1}{2}\bigg(\sum\limits_{i=1}^{\frac{s}{2}}(-\theta-\gamma_{\bar0i})+\sum\limits_{i=1}^{\frac{\sfr}{2}}(\theta+\gamma_{\bar1i})\bigg)\\
=&\frac{1}{2}\bigg(-\sum\limits_{i=1}^{\frac{s}{2}}\gamma_{\bar0i}+\sum\limits_{i=1}^{\frac{\sfr}{2}}\gamma_{\bar1i}\bigg)-\frac{s-\sfr}{4}\theta,\\ \rho=&\frac{1}{2}\sum_{\alpha\in\Phi^{+}}(-1)^{|\alpha|}\alpha,\\ \rho_{e,0}=&\rho-2\delta-\bigg(\frac{s-\sfr}{4}+\frac{1}{2}\bigg)\theta=\frac{1}{2}\sum_{\alpha\in\Phi^{+}_{e,0}}(-1)^{|\alpha|}\alpha
=\frac{1}{2}\Big(\sum_{j=1}^{w}\beta_{\bar0j}-\sum_{j=1}^{\ell}\beta_{\bar1j}\Big),
\end{split}
\end{equation}where $\gamma^*_{\bar0i}\in\Phi^+_{\bar0},\gamma^*_{\bar1j}\in\Phi_{\bar1}^+$, $\gamma_{\bar0i}\in\Phi^-_{\bar0},\gamma_{\bar1j}\in\Phi_{\bar1}^-$ for $1\leqslant i \leqslant \frac{s}{2}$ and $1\leqslant j \leqslant \frac{\sfr}{2}$ are defined in \S\ref{1.2}, $\beta_{\bar0i}\in\Phi_{\bar0}^+,\beta_{\bar1j}\in\Phi_{\bar1}^+$ for $1\leqslant i \leqslant w$ and $1\leqslant j \leqslant \ell$ are defined in \S\ref{3.1.1}, and $|\alpha|$ denotes the parity of $\alpha$.
For any $\eta\in(\mathfrak{h}^e)^*$, there exists a unique $t_\eta$ in $\mathfrak{h}^e$ with $\eta=(t_\eta,\cdot)$, and also a non-degenerate bilinear form on $(\mathfrak{h}^e)^*$ via $(\mu,\nu):=(t_\mu,t_\nu)$ for all $\mu,\nu\in(\mathfrak{h}^e)^*$. Given a linear function $\varphi$ on $\mathfrak{h}$ we denote by $\bar{\varphi}$ the restriction of $\varphi$ to $\mathfrak{h}^e$.

Now we have a result parallel to Theorem \ref{Whcchi}, i.e.,
\begin{theorem}\label{wh}Assume that $\sfr$ is even.
Every $\ggg$-module $M(\lambda,c)$ is an object of the category $\mathcal{C}_\chi$. Furthermore, $\text{Wh}(M(\lambda,c))\cong Z_{U(\ggg,e)}(\lambda+\bar{\delta},c+(\lambda+2\bar\rho,\lambda))$ as $U(\ggg,e)$-modules.
\end{theorem}
\begin{proof}
The proof of the theorem is the same as that of Theorem \ref{Whcchi}, while the lack of the element $v_{\frac{\sfr+1}{2}}\in\ggg(-1)_{\bar1}$ here makes the discussion much easier.
In particular, from Theorem \ref{verma} there is no restriction on $\lambda\in(\mathfrak{h}^e)^*$ and $c\in\bbc$. Then the action of $U(\ggg,e)$ on Verma modules need not to be twisted, so the shift $-\epsilon$ on $C$ as in \eqref{twisted} is redundant. The proof will be omitted.
\end{proof}
\begin{rem}
Here we omit the arguments on the minimal refined $W$-superalgebra $W'_{\chi}$ of type odd because there is lack of Skryabin's equivalence for this case.
%In this section, the reason why we did not considers of
%type odd lies in the fact that we have made a critical use of Skryabin's equivalence in Theorem \ref{skry}, in which only finite $W$-superalgebras are involved. It remains to note that refined $W$-superalgebras coincide with finite $W$-superalgebras if and only if  $\sfr$ is even.
\end{rem}

\section{On the category $\mathcal{O}$ for minimal finite $W$-superalgebras of type odd}\label{mathcalO}
All the discussion in previous sections are concentrated on the minimal finite (refined) $W$-superalgebras, for which the generators and their relationship  are given explicitly, and their Verma modules are introduced, which is much like the highest weight theory for $U(\ggg)$. However, the theory can not be applied directly in the general settings. We will manage  to develop the BGG category $\mathcal{O}$ for minimal finite $W$-superalgebra $U(\ggg,e)$, by exploiting the arguments in \cite{BGK} on highest weight theory of finite $W$-algebras to the super case.   %To  study  the representation theory of finite $W$-algebras $U(\ggg,e)$ associated with arbitrary nilpotent element $e\in\ggg$, Brundan-Goodwin-Kleshchev \cite{BGK} introduced the notation of the BGG category $\mathcal{O}$ for $U(\ggg,e)$, which contains all finite-dimensional $U(\ggg,e)$-modules, as well as analogues of the abstract universal highest weight modules.
%Based on the results obtained in \S\ref{3} and \S\ref{4}, we initial to introduce the abstract universal highest weight modules for minimal finite and refined $W$-superalgebras, and then consider the corresponding BGG category $\mathcal{O}$.
During the expositions, we mainly follow the strategy in \cite[\S4]{BGK}, with a lot of modifications.
It should be expected this is an effective attempt for the  general situation.

In this section we will only consider minimal finite $W$-superalgebras  of type odd. Then $\frac{\theta}{2}$ is an odd root of $\ggg$.
\subsection{}\label{5.1.1}
Keep the notations as in previous sections. Recall that $(e,h,f)$ is an $\mathfrak{sl}(2)$-triple in ${\ggg}_{\bar0}$, and ${\ggg}^e$ is the centralizer of $e$ in ${\ggg}$. Write ${\ggg}^h$ for the centralizer of $h$ in ${\ggg}$, which is equal to $\ggg(0)$ by definition. Then $\mathfrak{g}^h\cap\ggg^e=\ggg^e(0)=\ggg(0)^\sharp$ is a Levi factor of ${\ggg}^e$, and  $\mathfrak{h}^e=\mathfrak{h}\cap\ggg^e$ is a Cartan subalgebra of this Levi factor. As in \S\ref{3.1.1}, $\{h_1,\cdots,h_{k-1}\}$ is a basis of $\mathfrak{h}^e$, and $\ggg^e(0)$ has a basis as  in \eqref{basisofge0}.

For $\alpha\in(\mathfrak{h}^e)^*$, let $\ggg_\alpha=\bigoplus_{i\in\mathbb{Z}}\ggg_\alpha(i)$
denote the $\alpha$-weight space of $\ggg$ with respect to $\mathfrak{h}^e$. So
\begin{equation}\label{gg0ePhi}
\ggg=\ggg_0 \oplus \bigoplus_{\alpha\in \Phi'_e} \ggg_\alpha,
\end{equation}
where $\ggg_0$ is the centralizer of $\mathfrak{h}^e$ in $\ggg$, and $\Phi'_e \subset (\mathfrak{h}^e)^*$ is the set of nonzero weights
of $\mathfrak{h}^e$ on $\ggg$.
Since the eigenspace decomposition of $\text{ad}\,h$ gives rise to a short $\mathbb{Z}$-grading of $\ggg$ as in \eqref{short}, and only $\theta(h)=2$ by definition, it is immediate that for all $\alpha\in\Phi\backslash\{\pm\theta\}$ we have $\alpha(h)\in\{-1,0,1\}$. Keep in mind that $\mathfrak{h}=\mathfrak{h}^e\oplus\bbc h$, then $\ggg_0$ is equal to $\widetilde{\mathfrak{s}}_\theta$ as defined in \eqref{pnsodd}. %Moreover, by the definition of $\widetilde{\mathfrak{s}}_\theta$ we have $\Phi=\{\pm\frac{\theta}{2},\pm\theta\}\cup\Phi'_e$.
Since $\frac{\theta}{2}=\alpha_k$ is a simple root in $\Delta=\{\alpha_1,\cdots,\alpha_k\}$ of $\Phi$ by  Convention \ref{conventions}, then $\Phi_k: = \Phi \cap \bbz\alpha_k=\{\pm \alpha_k,\pm 2\alpha_k \}$ is a closed subsystem of $\Phi$ with base
$\Delta_k = \{\alpha_k\}$, which entails that
\begin{lemma}
$\Phi'_e=(\Phi\backslash\{\pm\frac{\theta}{2},\pm\theta\})|_{(\mathfrak{h}^e)^*}
=(\Phi\backslash\Phi_k)|_{(\mathfrak{h}^e)^*}$. \end{lemma}
%In particular, $\Phi'_e$ contains $\Phi_e$ (defined in  \S\ref{3.1.1}) as a proper subset.
\noindent Therefore, $\Phi'_e$ is a restricted root system in the same sense of the non-super case \cite[\S2]{BG}, namely, the set of nonzero restrictions of roots $\alpha\in\Phi$ to $\mathfrak{h}^e$. It is not a root system in the usual sense; for example, since $\theta(\hhh^e)=0$, for $\alpha\in\Phi'_e$ there may $\theta\pm\alpha$ that belong to $\Phi'_e$.
Then for $\alpha\in\Phi'_e$, we can write $\ggg_\alpha=\bigoplus_{i=1}^{I(\alpha)}\bbc e_{\alpha,i}$, where $e_{\alpha,i}$'s with $i\in I(\alpha)$ are the linear independent restricted root vectors which span $\ggg_\alpha$.
Denote by $(\Phi'_e)_{\bar0}$ and $(\Phi'_e)_{\bar1}$ the set of all restricted even roots and odd roots, respectively.
Similarly each of the spaces $\ggg(-1), \ggg(0)$ and $\ggg(1)$ decomposes into $\mathfrak{h}^e$-weight spaces.
There is an induced restricted root decomposition
\begin{equation}\label{ge0alpha}
\ggg^e = \ggg^e_0 \oplus \bigoplus_{\alpha\in \Phi'_e} \ggg^e_\alpha
\end{equation}
of the centralizer $\ggg^e$, where $\ggg^e_0$ is the centralizer of $e$ in $\ggg_0$, and
$\ggg^e_\alpha$ is the centralizer of $e$ in $\ggg_\alpha$.
Writing $\widetilde{\mathfrak{s}}_\theta^e$ for the centralizer of $e$ in $\widetilde{\mathfrak{s}}_\theta$,
it follows from \eqref{pnsodd} that
\begin{lemma}\label{ge0decom}
$\ggg^e_0=\widetilde{\mathfrak{s}}_\theta^e=\mathfrak{h}^e\oplus\bbc e\oplus\bbc E$.
\end{lemma} Moveover, by $\mathfrak{sl}(2)$-representation theory we have ${\ggg}^e\in\bigoplus_{i\geqslant0}{\ggg}(i)$, thus the second summands in \eqref{ge0alpha} can also be considered as chosen in $\Phi_e\backslash\{\frac{\theta}{2}\}$ % which equals $\{\alpha\in\Phi\,\mid\,\alpha(h)\in\{0,1\}\}$
as defined in  \S\ref{3.1.1}.
By the same discussion as in \cite[Lemma 13]{BG}, $\Phi_e\backslash\{\frac{\theta}{2}\}$ is also the set of nonzero weights of
$\mathfrak{h}^e$ on $\ggg^e$, so all the subspaces
$\ggg^e_\alpha= \bigoplus_{i \geqslant 0} \ggg^e_\alpha(i)$ in this decomposition are nonzero.

Recall  that the restricted root system $\Phi'_e$ is the set of nonzero weights of $\mathfrak{h}^e$ on $\ggg=\ggg_0 \oplus \bigoplus_{\alpha\in \Phi'_e} \ggg_\alpha$, and the zero weight space $\ggg_0$ is the centralizer of the toral subalgebra
$\mathfrak{h}^e$ in $\ggg$, so it is a Levi factor of a parabolic subalgebra of $\ggg$. By Bala-Carter theory \cite[Propositions 5.9.3-5.9.4]{Car}, $e$ is a distinguished nilpotent element of $(\ggg_0)_{\bar0}$, i.e., the only semi-simple elements of $(\ggg_0)_{\bar0}$ that centralize $e$ belong to the center of $(\ggg_0)_{\bar0}$, and $h,f$ also lie in $(\ggg_0)_{\bar0}$.
Moreover by \cite[Proposition 5.7.6]{Car} the grading of $(\ggg_0)_{\bar0}$ under the action of $\text{ad}\,h$ is even, i.e., $(\ggg_0(-1))_{\bar0}=(\ggg_0(1))_{\bar0}=0$. In our case, we have $\ggg_0=\widetilde{\mathfrak{s}}_\theta$ is a Levi factor of the  parabolic subalgebra $\mathfrak{p}_\theta$ (defined in \eqref{pnsodd}) of $\ggg$, and $(\ggg_0)_{\bar0}=\mathfrak{h}^e\oplus\bbc e\oplus\bbc h\oplus\bbc f$.
\subsection{}\label{5.1.2}
For the Lie superalgebra $\ggg_0=\widetilde{\mathfrak{s}}_\theta=\mathfrak{h}^e\oplus\mathfrak{s}_\theta$ as in \S\ref{4.2.2}, one can observe that it is also a direct sum decomposition of ideals of $\ggg_0$ with $\mathfrak{h}^e$ being abelian,  and $\mathfrak{s}_\theta\cong\mathfrak{osp}(1|2)$ by Lemma \ref{osp12genecom}.
Let $\mathfrak{m}_0:=\bbc f$ be the ``$\chi$-admissible subalgebra" of $\ggg_0$, and define the corresponding ``extended $\chi$-admissible subalgebra" by $\mathfrak{m}^{\prime}_0:=\bbc F\oplus\bbc f$. Define the generalized Gelfand-Graev ${\ggg}_0$-module associated with $\chi$ by
$(Q_0)_\chi:=U({\ggg}_0)\otimes_{U(\mathfrak{m}_0)}{\bbc}_\chi$, where ${\bbc}_\chi={\bbc}1_\chi$ is a one-dimensional  $\mathfrak{m}_0$-module such that $x.1_\chi=\chi(x)1_\chi$ for all $x\in\mathfrak{m}_0$. Let $(I_0)_\chi$ denote the ${\bbz}_2$-graded left ideal in $U({\ggg}_0)$ generated by all $x-\chi(x)$ with $x\in\mathfrak{m}_0$, and write $\text{Pr}_0:U({\ggg}_0)\twoheadrightarrow U({\ggg}_0)/(I_0)_\chi$ for the canonical homomorphism.
Now we can define the finite $W$-superalgebra $U(\ggg_0,e)$ associated to $e\in\ggg_0$ by
\begin{equation*}
U(\ggg_0,e):=(\text{End}_{\ggg_0}(Q_0)_{\chi})^{\text{op}}\cong(Q_0)_{\chi}^{\ad\,\mmm_0}.
\end{equation*}Recall that $-\theta$ is a minimal root.
As $e_\theta\in\ggg_0$ is  a root vector for $\theta$,  $U(\ggg_0,e)$ is also a minimal finite $W$-superalgebra,
which plays a role similar to ``Cartan subalgebra" in the classical BGG category.  This will be important to the formulation of our BGG category $\mathcal{O}$ for $U(\ggg,e)$.
\subsubsection{}
Let us first look at the structure of $U(\ggg_0,e)$. We have
\begin{prop}\label{fiosp}
The minimal finite $W$-superalgebra $U(\ggg_0,e)$ is generated by
\begin{itemize}
\item[(1)] $\Theta'_{h_i}=h_i\otimes1_\chi$ for $1\leqslant i\leqslant k-1$;
\item[(2)] $\Theta'_E=\big(E+\frac{1}{2}Fh-\frac{3}{4}F\big)\otimes1_\chi$;
\item[(3)] $C'_\theta=\big(2e+\frac{1}{2}h^2-\frac{3}{2}h+FE\big)\otimes1_\chi$;
\item[(4)] $\Theta'_F=F\otimes1_\chi$,
\end{itemize}
subject to the following relations:
\begin{itemize}
\item[(i)] $[\Theta'_E,\Theta'_E]=C'_\theta+\frac{1}{8}\otimes1_\chi$;
\item[(ii)] $[\Theta'_F,\Theta'_F]=-2\otimes1_\chi$,
%$[\Theta_E,\Theta_F]=[C'_\theta,U(\ggg_0,e)]=[\Theta_{h_i},U(\ggg_0,e)]=0$ for $1\leqslant i\leqslant k-1$.
\end{itemize}and the commutators between the other generators are all zero.
\end{prop}
\begin{proof}
The proposition comes as a special case of Theorem \ref{main3}. For the generators of $U(\ggg_0,e)$, we can obtain (1), (2) and (4) by direct computation. The element $C_\theta$ in \eqref{Ctheta} is the Casimir element of $\mathfrak{s}_\theta$, then $C'_\theta=\text{Pr}_0(C_\theta)$  is  in the center of $U(\ggg_0,e)$. Since $\ggg^e_0=\mathfrak{h}^e\oplus\bbc e\oplus\bbc  E$ by Lemma \ref{ge0decom}, then the first part of the proposition follows.

For the second part of the proposition, the commutators of these generators can be calculated directly.
In particular, $[\Theta'_E,\Theta'_E]$ has been calculated in \eqref{EC}.
%Now we will just give the detailed proof of (ii). In fact, we have
%\begin{equation*}
%\begin{split}
%[\Theta'_E,\Theta'_E]=&[E+\frac{1}{2}Fh-\frac{3}{4}F,E+\frac{1}{2}Fh-\frac{3}{4}F]\otimes1_\chi\\
%=&\bigg([E,E]+\frac{1}{2}[Fh,E]-\frac{3}{4}[F,E]+\frac{1}{2}[E,Fh]+\frac{1}{4}[Fh,Fh]
%-\frac{3}{8}[F,Fh]-\frac{3}{4}[E,F]\\
%&-\frac{3}{8}[Fh,F]+\frac{9}{16}[F,F]\bigg)\otimes1_\chi\\
%=&\bigg([E,E]+\frac{1}{2}[F,E]h+\frac{1}{2}F[h,E]-\frac{3}{4}[F,E]+\frac{1}{2}[E,F]h
%-\frac{1}{2}F[E,h]+\frac{1}{4}[F,F]h^2\\
%&+\frac{1}{4}F[h,F]h-\frac{1}{4}F[F,h]h-\frac{3}{8}[F,F]h+\frac{3}{8}F[F,h]-\frac{3}{4}[E,F]-\frac{3}{8}[F,F]h\\
%&-\frac{3}{8}F[h,F]+\frac{9}{16}[F,F]\bigg)\otimes1_\chi\\
%=&\bigg(2e+\frac{1}{2}h^2+FE-\frac{3}{2}h+\frac{1}{8}\bigg)\otimes1_\chi\\
%=&C'_\theta+\frac{1}{8}\otimes1_\chi.
%\end{split}
%\end{equation*}
\end{proof}
\subsubsection{}
We can describe the structure of the center of $U(\ggg_0,e)$ as follows:
\begin{prop}\label{centerg0e}
The center $Z(U(\ggg_0,e))$ of the minimal finite $W$-superalgebra $U(\ggg_0,e)$ is generated by $\Theta'_{h_i}$ for $1\leqslant i\leqslant k-1$ and $C'_\theta$.
\end{prop}
\begin{proof}
One can easily conclude from Proposition \ref{fiosp} that $C'_\theta$ and $\Theta'_{t}$ with $t\in\hhh^e$ lie in the center of $U(\ggg_0,e)$. On the other hand, assume that
\begin{equation}\label{Cabcd}
C':=\sum_{\bf a}\lambda_{{\bf a}}(\Theta'_F)^{a_1}(\Theta'_\mathbf{t})^{\bf a_2}(C'_\theta)^{a_3}(\Theta'_E)^{a_4}\in Z(U(\ggg_0,e))
\end{equation} for $\lambda_{{\bf a}}\in\bbc$ with
${\bf a}=(a_1,{\bf a_2},a_3,a_4)\in\Lambda_1\times\mathbb{Z}_+^{k-1}\times\mathbb{Z}_+\times\Lambda_1$ such that $(\Theta'_\mathbf{t})^{\bf a_2}:=(\Theta'_{h_1})^{a_{21}}\cdots(\Theta'_{h_{k-1}})^{a_{2k-1}}$, is the linear span of the PBW basis of $U(\ggg_0,e)$. In virtue of Proposition \ref{fiosp}, we get
\begin{equation}\label{abcdF}
\begin{split}
0=&\Big[\sum_{\bf a}\lambda_{{\bf a}}(\Theta'_F)^{a_1}(\Theta'_\mathbf{t})^{{\bf a_2}}(C'_\theta)^{a_3}(\Theta'_E)^{a_4},\Theta'_F\Big]\\
=&\Big[\sum_{{\bf a_2},a_3}\lambda_{(0,{\bf a_2},a_3,0)}(\Theta'_\mathbf{t})^{\bf a_2}(C'_\theta)^{a_3},\Theta'_F\Big]+
\Big[\sum_{{\bf a_2},a_3}\lambda_{(1,{\bf a_2},a_3,0)}\Theta'_F(\Theta'_\mathbf{t})^{\bf a_2}(C'_\theta)^{a_3},\Theta'_F\Big]\\
&+\Big[\sum_{{\bf a_2},a_3}\lambda_{(0,{\bf a_2},a_3,1)}(\Theta'_\mathbf{t})^{\bf a_2}(C'_\theta)^{a_3}\Theta'_E,\Theta'_F\Big]+
\Big[\sum_{{\bf a_2},a_3}\lambda_{(1,{\bf a_2},a_3,1)}\Theta'_F(\Theta'_\mathbf{t})^{\bf a_2}(C'_\theta)^{a_3}\Theta'_E,\Theta'_F\Big]\\
=&\sum_{{\bf a_2},a_3}\lambda_{(1,{\bf a_2},a_3,0)}[\Theta'_F,\Theta'_F](\Theta'_\mathbf{t})^{\bf a_2}(C'_\theta)^{a_3}
-\sum_{{\bf a_2},a_3}\lambda_{(1,{\bf a_2},a_3,1)}[\Theta'_F,\Theta'_F](\Theta'_\mathbf{t})^{a_2}(C'_\theta)^{a_3}\Theta'_E\\
=&-2\sum_{{\bf a_2},a_3}\lambda_{(1,{\bf a_2},a_3,0)}(\Theta'_\mathbf{t})^{\bf a_2}(C'_\theta)^{a_3}
+2\sum_{{\bf a_2},a_3}\lambda_{(1,{\bf a_2},a_3,1)}(\Theta'_\mathbf{t})^{\bf a_2}(C'_\theta)^{a_3}\Theta'_E.
\end{split}
\end{equation}
So all the coefficients $\lambda_{{\bf a}}$ with $a_1=1$ in \eqref{Cabcd} equal zero. Taking this into consideration, by the same discussion as in \eqref{abcdF} we have
\begin{equation}\label{abcdE}
\begin{split}
0=&\Big[\sum_{\bf a}\lambda_{{\bf a}}(\Theta'_F)^{a_1}(\Theta'_\mathbf{t})^{\bf a_2}(C'_\theta)^{a_3}(\Theta'_E)^{a_4},\Theta'_E\Big]\\
=&\sum_{{\bf a_2},a_3}\lambda_{(0,{\bf a_2},a_3,1)}(\Theta'_\mathbf{t})^{\bf a_2}(C'_\theta)^{a_3+1}+\frac{1}{8}
\sum_{{\bf a_2},a_3}\lambda_{(0,{\bf a_2},a_3,1)}(\Theta'_\mathbf{t})^{\bf a_2}(C'_\theta)^{a_3}.
\end{split}
\end{equation}
If there exists some $\lambda_{(0,{\bf a_2},a_3,1)}\neq0$ in \eqref{Cabcd}, set $a'_3\in\mathbb{Z}_+$ to be the largest number with this property. Then we have
$\sum_{{\bf a_2}}\lambda_{(0,{\bf a_2},a'_3,1)}(\Theta'_\mathbf{t})^{\bf a_2}(C'_\theta)^{a'_3+1}=0$ by \eqref{abcdE}, which means that $\lambda_{(0,{\bf a_2},a'_3,1)}=0$, a contraction.
Combining this with our earlier discussion, we see that the coefficients $\lambda_{{\bf a}}$ with $a_1\neq0$ or $a_4\neq0$ in \eqref{Cabcd} are all zeros. Then any element in $Z(U(\ggg_0,e))$ can be written as a linear span of $(\Theta'_ {h_1})^{a_{21}}\cdots(\Theta'_{h_{k-1}})^{a_{2k-1}}(C'_\theta)^{a_3}$, completing the proof.
\end{proof}

Let $Z(U(\ggg_0))$ denote the center of $U(\ggg_0)$.
The canonical homomorphism $\text{Pr}_0:U({\ggg}_0)\twoheadrightarrow U({\ggg}_0)/(I_0)_\chi$ we introduced earlier induces an algebra homomorphism
from $Z(U(\ggg_0))$ to $Z(U(\ggg_0,e))$. In fact, we further have
\begin{prop}\label{centertocenter}
The map $\text{Pr}_0$  sends $Z(U(\ggg_0))$ isomorphically onto the center
of $U(\ggg_0,e)$.
\end{prop}
\begin{proof}
Recall that $\ggg_0=\mathfrak{h}^e\oplus\mathfrak{s}_\theta$  is a direct sum decomposition of ideals, with  $\mathfrak{h}^e$ being abelian. It is well-known that the center of $U(\mathfrak{osp}(1|2))$ is generated by its Casimir element, and $\mathfrak{s}_\theta\cong\mathfrak{osp}(1|2)$ by our earlier remark,  then $Z(U(\mathfrak{s}_\theta))$ is also generated by the Casimir element $C_\theta$ defined in \eqref{Ctheta}.
Now we  conclude that $Z(U(\ggg_0))$ is generated by the algebraically independent elements $h_1,\cdots,h_{k-1}$ and $C_\theta$.

By the definition of the map $\text{Pr}_0$, it is readily to check that
\begin{equation*}\label{pr0}
\begin{split}
\text{Pr}_0(h_i)=&h_i\otimes1_\chi=\Theta'_{h_i}\qquad \text{for} ~1\leqslant i\leqslant k-1,\\
\text{Pr}_0(C_\theta)=&(2ef+\frac{1}{2}h^2-\frac{3}{2}h+FE)\otimes1_\chi=\big(2e+\frac{1}{2}h^2-\frac{3}{2}h+FE\big)\otimes1_\chi
=C'_\theta.
\end{split}
\end{equation*}
Since $\Theta'_{h_i}$ for $1\leqslant i\leqslant k-1$ and $C'_\theta$ are also algebraically independent, then the proposition follows from Proposition \ref{centerg0e}.
\end{proof}
\subsubsection{}\label{523}
We will describe the finite-dimensional irreducible modules for $U(\ggg_0,e)$ for later discussion.
Let $V_\lambda:=\bbc v_\lambda\oplus\bbc\Theta'_F(v_\lambda)$ with $\lambda\in(\hhh^e)^*$ be a vector space spanned by $v_\lambda\in (V_\lambda)_{\bar0}$ and $\Theta'_F(v_\lambda)\in (V_\lambda)_{\bar1}$ satisfying $\Theta'_E(v_\lambda)=0$, and $\Theta'_E.\Theta'_F(v_\lambda)=0$, $\Theta'_F.\Theta'_F(v_\lambda)=-v_\lambda$, $C'_\theta(v_\lambda)=-\frac{1}{8}v_\lambda$, $C'_\theta.\Theta'_F(v_\lambda)=-\frac{1}{8}\Theta'_F(v_\lambda)$ (the above four equations are derived from Proposition \ref{fiosp}(i)---(ii)), $\Theta'_t(v_\lambda)=\lambda(t)v_\lambda$, $\Theta'_t.\Theta'_F(v_\lambda)=\lambda(t)\Theta'_F(v_\lambda)$ for all $t\in\hhh^e$. In fact, we have
\begin{theorem}\label{simpleugegmodule}
The set $\{V_\lambda\:|\:\lambda\in(\hhh^e)^*\}$ forms a complete set of pairwise inequivalent finite-dimensional  irreducible $U(\ggg_0,e)$-modules (up to parity switch), all of which are of type $Q$.
\end{theorem}
\begin{proof}
Due to Proposition \ref{fiosp} and the fact that $(\Theta'_F)^2=\frac{1}{2}[\Theta'_F,\Theta'_F]=-1\otimes1_\chi$, it is readily to check that $V_\lambda$ is an irreducible $U(\ggg_0,e)$-module. Obviously $V_\lambda$ is a simple module of type $Q$, for which the odd endomorphism is induced by the element $\Theta'_F$.

If the simple $U(\ggg_0,e)$-modules $V_\lambda$ and $V_{\lambda'}$ are isomorphic, then by parity consideration $\bbc v_\lambda\cong\bbc v_{\lambda'}$ as modules over the commutative subalgebra $\Theta'_{\mathfrak{h}^e}$ of $U(\ggg_0,e)$. So we have $\lambda=\lambda'$.

Let $M$ be a finite-dimensional simple $U(\ggg_0,e)$-module. By the same discussion as the first two paragraphs in Step (3) for the proof of Theorem \ref{verma2}, $M$ decomposes into weight spaces relative to $\Theta_{\mathfrak{h}^e}$, and it contains at least one maximal weight element, for which we put it as $\mu$. For a nonzero vector $m$ in $M_\mu$, we have $\Theta'_t(m)=\mu(t)m$ for $t\in\hhh^e$. Since $M$ is finite-dimensional, we can further assume that $\Theta'_E(m)=0$.
Then there must exist a $U(\ggg_0,e)$-module homomorphism $\xi$ from either $V_\mu$ or $\prod V_\mu$ (Here $\prod$ denotes the parity switching functor) to $M$ such that $\xi(v_\mu)=m$. Moreover, we have $\Theta'_E.\Theta'_F(m)=0$ by Proposition \ref{fiosp}. %, $\Theta'_t.\Theta'_Fm=\mu(t)\Theta'_Fm$ for $t\in\hhh^e$, and $C'_\theta.m=\frac{7}{8}m$, $C'_\theta.\Theta'_Fm=\frac{7}{8}\Theta'_Fm$ by Proposition \ref{fiosp}.
Then the simplicity of $M$ entails that $\xi$ is surjective, thus also injective by the knowledge of linear algebras.
\end{proof}
\begin{rem}
In \cite[Lemma 3.4]{PS2}, Poletaeva-Serganova gave another description of the PBW Theorem
and the irreducible representations of $U(\mathfrak{osp}(1|2),e)$ with $e$ being regular nilpotent (in this case we can also write $e=e_\theta$ with $-\theta$ being a minimal root). Since $\ggg_0\cong\mathfrak{h}^{e}\oplus\mathfrak{osp}(1|2)$ as a direct sum decomposition of ideals, one can compare their results with Proposition \ref{fiosp} and Theorem \ref{simpleugegmodule}.
\end{rem}

The finite $W$-superalgebra $U(\ggg_0,e)$ is going to play the role of Cartan subalgebra in the highest theory. However, just like the non-super case, it does not embed obviously as a subalgebra of the minimal finite $W$-superalgebra $U(\ggg,e)$; instead we will realize it in another way, which is also different from the one applied for non-super case. We will put it in the next part.
%However, this is not obvious for the finite $W$-superalgebra associated with general nilpotent $e\in\ggg_{\bar0}$, even in the non-super version.

\subsection{}\label{5.1.3}
Let us turn back to the minimal finite $W$-superalgebra $U(\ggg,e)$. Let $(\Phi'_e)^+:=\Phi^+\backslash\{\frac{\theta}{2},\theta\}$ be a system of positive roots in the restricted root system $\Phi'_e$. Setting $(\Phi'_e)^-:=-(\Phi'_e)^+$, we define
$\ggg_{\pm} := \bigoplus_{\alpha \in (\Phi'_e)^{\pm}} \ggg_\alpha$,
so that
$$
\ggg = \ggg_- \oplus \ggg_0 \oplus \ggg_+,\qquad
\mathfrak{q} = \ggg_0 \oplus \ggg_+.
$$
The choice $(\Phi'_e)^+$ of positive roots
induces a dominance ordering  $\leqslant$ on $(\mathfrak{h}^e)^*$:
$\mu\leqslant\lambda$ if
$\lambda-\mu\in \mathbb{Z}_{\geqslant 0}(\Phi'_e)^+$, which is exactly the one we defined in \eqref{parord}. Furthermore, since $\theta(\mathfrak{h}^e)=0$, and $\frac{\theta}{2}$ is a simple root in $\Phi$  by  Convention \ref{conventions}, we can always assume that $\mu=\lambda$ on $(\mathfrak{h}^e)^*$ when $\lambda=\mu+k\theta$ for some $k\in\bbc$. Denote by $(\Phi'_e)^+_{\bar0}:=(\Phi'_e)^+\cap(\Phi'_e)_{\bar0}$ and $(\Phi'_e)^+_{\bar1}:=(\Phi'_e)^+\cap(\Phi'_e)_{\bar1}$, respectively.

In this paragraph, write $\mathfrak{a}$ for $\ggg$ or $\ggg^e$.
Recall from \S\ref{5.1.1} that the adjoint actions of
$\mathfrak{h}^e$ on $\mathfrak{a}$ and its universal enveloping algebra $U(\aaa)$
induce decompositions
$\aaa = \aaa_0 \oplus \bigoplus_{\alpha \in \Phi'_e} \aaa_\alpha$
and $U(\aaa) = \bigoplus_{\alpha \in \bbz\Phi'_e} U(\aaa)_\alpha$.
In particular, $U(\aaa)_0$, the zero weight space of $U(\aaa)$ with
respect to the adjoint action, is a subalgebra of $U(\aaa)$.
Let $U(\aaa)_\sharp$ (resp.\ $U(\aaa)_\flat$) denote the left (resp.\ right)
ideal of $U(\aaa)$ generated by the root spaces $\aaa_\alpha$ for
$\alpha \in (\Phi'_e)^+$ (resp.\ $\alpha \in (\Phi'_e)^-$).
Let
$$
U(\aaa)_{0, \sharp} := U(\aaa)_0 \cap U(\aaa)_\sharp,
\qquad
U(\aaa)_{\flat,0} := U(\aaa)_\flat \cap U(\aaa)_0,
$$
which are  left and right ideals of $U(\aaa)_0$, respectively.
By the PBW theorem for  Lie superalgebras, we actually have that
$
U(\aaa)_{0,\sharp} = U(\aaa)_{\flat,0}$,
hence $U(\aaa)_{0,\sharp}$ is a two-sided ideal of $U(\aaa)_0$.
Moreover, $\aaa_0$
is a subalgebra of $\aaa$,
and by the PBW theorem again we have that
$U(\aaa)_0 = U(\aaa_0) \oplus U(\aaa)_{0,\sharp}$.
The projection along this decomposition
defines a surjective algebra homomorphism
\begin{equation}\label{pi}
\pi:U(\aaa)_0 \twoheadrightarrow U(\aaa_0)
\end{equation}
with $\ker\pi = U(\aaa)_{0,\sharp}$.
Hence $U(\aaa)_0 / U(\aaa)_{0,\sharp} \cong U(\aaa_0)$.

Recall in \S\ref{1.2} and \S\ref{3.1.1} that we have chosen a basis consisting of $\hhh^e$-weight vectors
\begin{equation}\label{toeaualweight}
\begin{split}
&x_1,\cdots,x_w,y_1,\cdots,y_\ell,f_1,\cdots,f_{\frac{s}{2}},g_1,\cdots,g_{\frac{\sfr-1}{2}},
h_1,\cdots,h_{k-1},\\
&e,[v_{\frac{\sfr+1}{2}},e],f^*_1,\cdots,f^*_{\frac{s}{2}},g^*_1,\cdots,g^*_{\frac{\sfr-1}{2}},x^*_1,\cdots,x^*_w,y^*_1,\cdots,y^*_\ell
\end{split}
\end{equation}of $\ggg^e$
so that the weights of $x_i, y_j, f_k, g_l$ are respectively $-\beta_{\bar 0i}, -\beta_{\bar 1j}, \theta+\gamma_{\bar0k}, \theta+\gamma_{\bar1l}\in(\Phi'_e)^-$, and the weights of $f^*_k, g^*_l, x^*_i, y^*_j$ are respectively
$\theta+\gamma^*_{\bar0k}, \theta+\gamma^*_{\bar1l}, \beta_{\bar 0i}, \beta_{\bar 1j}\in(\Phi'_e)^+$, while $h_i, e, [v_{\frac{\sfr+1}{2}},e]\in\ggg_0^e$.
Moreover, we have the following PBW basis for $U(\ggg,e)$:
\begin{equation}\label{PBWfW}
\begin{array}{ll}
&\prod_{i=1}^w\Theta_{x_i}^{a_i}\cdot\prod_{i=1}^\ell\Theta_{y_i}^{c_i}\cdot
\prod_{i=1}^{\frac{s}{2}}\Theta_{f_i}^{m_i}\cdot\prod_{i=1}^{\frac{\sfr-1}{2}}
\Theta_{g_i}^{p_i}\cdot\Theta_{v_{\frac{\sfr+1}{2}}}^\iota\cdot\prod_{i=1}^{k-1}
\Theta_{h_i}^{t_i}\\
&\cdot C^{t_k}\cdot\Theta_{[v_{\frac{\sfr+1}{2}},e]}^\varepsilon\cdot\prod_{i=1}^{\frac{s}{2}}\Theta_{f^*_i}^{n_i}\cdot
\prod_{i=1}^{\frac{\sfr-1}{2}}\Theta_{g^*_i}^{q_i}\cdot
\prod_{i=1}^w\Theta_{x^*_i}^{b_i}\cdot
\prod_{i=1}^\ell\Theta_{y^*_i}^{d_i},
\end{array}
\end{equation}
where ${\bf a},{\bf b}\in\mathbb{Z}_+^w$, ${\bf c},{\bf d}\in\Lambda_\ell$, ${\bf m},{\bf n}\in\mathbb{Z}_+^{\frac{s}{2}}$, ${\bf p},{\bf q}\in\Lambda_{\frac{\sfr-1}{2}}, \iota,\varepsilon\in\Lambda_1, {\bf t}\in\mathbb{Z}_+^k$.
Let $v$ be any element in \eqref{toeaualweight} excluding $e$, or let $v=v_{\frac{\sfr+1}{2}}$. Since $\theta(\hhh^e)=0$ by definition, from the explicit description of $\Theta_v$ in Theorem \ref{ge}, we see that $v$ and   $\Theta_v$ have the same $\mathfrak{h}^e$-weight. Also note that the $\mathfrak{h}^e$-weight of $C$ is zero.
Then the subspace $U(\ggg,e)_\alpha$ in the restricted root space decomposition
has a basis given by all the PBW monomials as in \eqref{PBWfW} such that
$\sum_{i}(-a_i+b_i)\beta_{\bar 0i}+\sum_{i}(-c_i+d_i)\beta_{\bar 1i}+\sum_{i}(m_i-n_i)\gamma_{\bar 0i}+\sum_{i}(p_i-q_i)\gamma_{\bar 1i}= \alpha$.

Set $U(\ggg,e)_\sharp$ (resp.\ $U(\ggg,e)_\flat$) to be the
left (resp.\ right) ideal of $U(\ggg,e)$ generated by
\begin{center}
$\Theta_{f^*_1},\cdots,\Theta_{f^*_{\frac{s}{2}}},\Theta_{g^*_1},\cdots,
\Theta_{g^*_{\frac{\sfr-1}{2}}},\Theta_{x^*_1},\cdots,\Theta_{x^*_w},\Theta_{y^*_1},\cdots,
\Theta_{y^*_\ell}$\\
(resp.\ $\Theta_{x_1},\cdots,\Theta_{x_w},\Theta_{y_1},\cdots,
\Theta_{y_\ell}, \Theta_{f_1},\cdots,\Theta_{f_{\frac{s}{2}}},\Theta_{g_1},\cdots,
\Theta_{g_{\frac{\sfr-1}{2}}})$.
\end{center}
Note that $U(\ggg,e)_\sharp$ (resp.\ $U(\ggg,e)_\flat$)
is equivalently the left (resp.\ right) ideal of
$U(\ggg,e)$ generated by all $U(\ggg,e)_\alpha$ for $\alpha \in (\Phi'_e)^+$
(resp.\ $\alpha\in (\Phi'_e)^-$), and it does not depend on the
explicit choice of the basis. Set
\begin{equation*}
U(\ggg,e)_{0,\sharp} := U(\ggg,e)_0 \cap U(\ggg,e)_\sharp,
\qquad
U(\ggg,e)_{\flat,0} := U(\ggg,e)_\flat \cap U(\ggg,e)_0,
\end{equation*}
which are obviously left and right ideals of the zero weight space
$U(\ggg,e)_0$, respectively. The PBW basis of $U(\ggg,e)_\sharp$ (resp.\ $U(\ggg,e)_\flat$) is the monomials as in \eqref{PBWfW} with $({\bf n,q,b,d})\neq\bf 0$ (resp.\ $({\bf a,c,m,p})\neq\bf 0$), and the PBW basis of $U(\ggg,e)_0$ is the monomials as in \eqref{PBWfW} with $\sum_{i}(-a_i+b_i)\beta_{\bar 0i}+\sum_{i}(-c_i+d_i)\beta_{\bar 1i}+\sum_{i}(m_i-n_i)\gamma_{\bar 0i}+\sum_{i}(p_i-q_i)\gamma_{\bar 1i}= 0$. We also have $U(\ggg,e)_{0,\sharp} = U(\ggg,e)_{\flat,0}$ by the PBW theorem, hence it is a two-sided ideal of  $U(\ggg,e)_0$. However, the cosets of  the PBW monomials of the
form $\Theta_{v_{\frac{\sfr+1}{2}}}^\iota\cdot\prod_{i=1}^{k-1}
\Theta_{h_i}^{t_i}\cdot C^{t_k}\cdot\Theta_{[v_{\frac{\sfr+1}{2}},e]}^\varepsilon$
need not span a subalgebra of $U(\ggg,e)$, unlike the situation for the algebras $U(\aaa)$
discussed earlier.

The goal now is to prove the quotient algebra
$U(\ggg,e)_0/U(\ggg,e)_{0,\sharp}$ is canonically isomorphic to
$U(\ggg_0,e)$.
As in \S\ref{5.3}, for a linear function $\varphi$ on $\mathfrak{h}$ we still denote by $\bar{\varphi}$ the restriction of $\varphi$ to $\mathfrak{h}^e$. %The isomorphism involves a shift defined by
%\begin{equation*}
%\gamma:=\sum_{\alpha|_{\hhh^e}\in (\Phi'_e)^-}(-1)^{|\alpha|}\alpha=-\sum_{j=1}^{w}\beta_{\bar0j}+\sum_{j=1}^{\ell}\beta_{\bar1j}
%+2\sum\limits_{j=1}^{\frac{s}{2}}\gamma_{\bar0j}-2\sum\limits_{j=1}^{\frac{\sfr-1}{2}}\gamma_{\bar1j}+\frac{s-\sfr}{2}\theta,
%\end{equation*}where $\beta_{\bar0j}\in\Phi_{\bar0}^+,\beta_{\bar1j}\in\Phi_{\bar1}^+$ and $\gamma_{\bar0j}\in\Phi_{\bar0}^-,\gamma_{\bar1j}\in\Phi_{\bar1}^-$ are defined in \S\ref{2.1} and \S\ref{1.2} respectively, and note that $\theta|_{\hhh^e}=0$ by definition.
Recall in \eqref{deltarho} and \eqref{defofepsilon} we denote by
\begin{equation*}
\begin{split}
\bar\delta=&\frac{1}{2}\bigg(\sum\limits_{i=1}^{\frac{s}{2}}\gamma^*_{\bar0i}-\sum\limits_{i=1}^{\frac{\sfr-1}{2}}\gamma^*_{\bar1i}\bigg)\mid_{\mathfrak{h}^e}
=\frac{1}{2}\bigg(-\sum\limits_{i=1}^{\frac{s}{2}}\gamma_{\bar0i}+\sum\limits_{i=1}^{\frac{\sfr-1}{2}}\gamma_{\bar1i}\bigg)\mid_{\mathfrak{h}^e},\\
\bar\rho_{e,0}=&
\frac{1}{2}\bigg(\sum_{i=1}^{w}\beta_{\bar0i}-\sum_{i=1}^{\ell}\beta_{\bar1i}\bigg)\mid_{\mathfrak{h}^e}\\
\epsilon=&c_0+\frac{1}{8}+\sum_{i=1}^{k-1}(2\rho_{e,0}\delta+3\delta^2)(h_i),
\end{split}
\end{equation*}where $\gamma^*_{\bar0i}\in\Phi^+_{\bar0},\gamma^*_{\bar1j}\in\Phi_{\bar1}^+$, $\gamma_{\bar0i}\in\Phi^-_{\bar0},\gamma_{\bar1j}\in\Phi_{\bar1}^-$ for $1\leqslant i \leqslant \frac{s}{2}$ and $1\leqslant j \leqslant \frac{\sfr-1}{2}$ are defined in \S\ref{1.2},  $\beta_{\bar0i}\in\Phi_{\bar0}^+,\beta_{\bar1j}\in\Phi_{\bar1}^+$ for $1\leqslant i \leqslant w$ and $1\leqslant j \leqslant \ell$ are defined in \S\ref{3.1.1}, and $|\alpha|$ denotes the parity of $\alpha$.   Now we introduce a shift $S_{\epsilon}$ on $U(\ggg_0,e)$
by keeping the other generators  as in Proposition \ref{fiosp} invariant and sending $C'_\theta$ to $C'_\theta+\epsilon$. %Then extend the shift $S_{\epsilon}$ on $U(\ggg_0,e)$.

Keeping the notations as above, we have
\begin{prop}\label{ressur}
The projection $\pi:U(\ggg)_0\twoheadrightarrow U(\mathfrak{g}_0)$ as in \eqref{pi} induces
a surjective homomorphism
\begin{equation*}
\pi:U(\ggg,e)_0\twoheadrightarrow U(\ggg_0,e)
\end{equation*}
with  $\ker\pi= U(\ggg,e)_{0,\sharp}$.
Moreover, there exists an algebras isomorphism
\begin{equation*}
\pi_{\epsilon}:=S_{\epsilon}\circ\pi:U(\ggg,e)_0 / U(\ggg,e)_{0,\sharp} \cong U(\ggg_0,e).
\end{equation*}
\end{prop}
\begin{proof}
(1) Under the linear projection $\pi$, we first discuss the images of $U(\ggg,e)_0$ with respect to the PBW basis of minimal finite $W$-superalgebra $U(\ggg,e)$.

We consider the generators of $U(\ggg,e)$ as in Theorem \ref{main3}.
It is obvious that $\pi(\Theta_{v_{\frac{\sfr+1}{2}}})=\pi(v_{\frac{\sfr+1}{2}})=v_{\frac{\sfr+1}{2}}\otimes1_\chi$.
For $1\leqslant j\leqslant k-1$, note that
$\pi(u_i[u_i^*,h_j])=\pi(v_t[v_t^*,h_j])=0$ for $1\leqslant i \leqslant \frac{s}{2}$, $1\leqslant t \leqslant \frac{\sfr+1}{2}$, and $\pi([u_i^*,h_j]u_i)=\pi([v_t^*,h_j]v_t)=0$ for $\frac{s}{2}+1\leqslant i \leqslant s$,  $\frac{\sfr+3}{2}\leqslant t \leqslant\sfr$,   and also $\pi(h_j)=h_j, \pi(e)=e, \pi(h)=h, \pi(f)=f$ by definition. Then we have
\begin{equation}\label{pithetahj}
\begin{split}
\pi(\Theta_{h_j})=&
\pi(h_j)-\frac{1}{2}\Big(\sum\limits_{i=\frac{s}{2}+1} ^{s}\pi([u_i,[u_i^*,h_j]])+\sum\limits_{i=\frac{\sfr+3}{2}}^{\sfr}\pi([v_i,[v_i^*,h_j]])\Big)\\
=&\pi(h_j)+\frac{1}{2}\sum\limits_{i=1} ^{\frac{s}{2}}\pi([u^*_i,[u_i,h_j]])-\frac{1}{2}\sum\limits_{i=1}^{\frac{\sfr-1}{2}}\pi([v^*_i,[v_i,h_j]])\\
=&\bigg(h_j+\frac{1}{2}(-\gamma_{\bar01}-\cdots-\gamma_{\bar0\frac{s}{2}}
+\gamma_{\bar11}+\cdots+\gamma_{\bar1\frac{\sfr-1}{2}})(h_j)\bigg)\otimes1_\chi\\
=&(h_j+\delta(h_j))\otimes1_\chi.
\end{split}
\end{equation}

By the similar calculation as in \eqref{discuss}, we obtain
\begin{equation}\label{piw}
\begin{split}
\pi(\Theta_{[v_{\frac{\sfr+1}{2}},e]})
=&\pi([v_{\frac{\sfr+1}{2}},e])+\frac{2}{3}\sum\limits_{i=1}^{\frac{s}{2}}\pi([u_i^*,[u_i,[v_{\frac{\sfr+1}{2}},e]]])-\Big(\frac{2}{3}\sum\limits_{i=1}^{\frac{\sfr-1}{2}}
\pi([v_i^*,[v_i,[v_{\frac{\sfr+1}{2}},e]]])\\
&+v_{\frac{\sfr+1}{2}}[v_{\frac{\sfr+1}{2}},\pi([v_{\frac{\sfr+1}{2}},e])]\Big)
+\frac{1}{3}\Big(-2\sum\limits_{i=1}^{\frac{s}{2}} v_{\frac{\sfr+1}{2}}\pi([u_i^*,[u_i,[v_{\frac{\sfr+1}{2}},[v_{\frac{\sfr+1}{2}},e]]]])\\
&+2\sum\limits_{i=1}^{\frac{\sfr-1}{2}}v_{\frac{\sfr+1}{2}} \pi([v_i^*,[v_i,[v_{\frac{\sfr+1}{2}},[v_{\frac{\sfr+1}{2}},e]]]])+\frac{1}{2}[v_{\frac{\sfr+1}{2}},[v_{\frac{\sfr+1}{2}},\pi([v_{\frac{\sfr+1}{2}},e])]]\\
&-2[\pi([v_{\frac{\sfr+1}{2}},e]),f]\Big).
\end{split}
\end{equation}
Taking \eqref{zzev}---\eqref{vef}  into consideration, by \eqref{piw} we have
\begin{equation}\label{pithetave}
\pi(\Theta_{[v_{\frac{\sfr+1}{2}},e]})=\bigg([v_{\frac{\sfr+1}{2}},e]-\frac{3}{4}v_{\frac{\sfr+1}{2}}+\frac{1}{2}
v_{\frac{\sfr+1}{2}}h\bigg)\otimes1_\chi,
\end{equation}

For the Casimir element $C$ of $U(\ggg)$ corresponding to the invariant form $(\cdot,\cdot)$, by the similar discussion as in \eqref{C1}, it follows from \eqref{vve}, \eqref{C0def}, \eqref{x*x}, \eqref{y*y}, \eqref{ez*zalpha} and \eqref{t*bargamma} that
\begin{equation}\label{pithetaC}
\begin{split}
\pi(C)=&\bigg(2e+\frac{h^2}{2}-\Big(1+\frac{s-\sfr}{2}\Big)h+C_0
+2\sum\limits_{\alpha\in S(-1)}(-1)^{|\alpha|}[e,z_\alpha^*]z_\alpha\bigg)\otimes1_\chi\\
=&\bigg(2e+\frac{h^2}{2}-(2+\frac{s-\sfr}{2})h-2v_{\frac{\sfr+1}{2}}[v_{\frac{\sfr+1}{2}},e]+
\sum_{i=1}^{k-1}h_i^2+\sum_{i=1}^{w}\pi([x^*_i,x_i])-\sum_{i=1}^{\ell}\pi([y^*_i,y_i])\\
&+2\sum\limits_{i=1}^{\frac{s}{2}}
\pi([[e,u_i^*],u_i])-2\sum\limits_{i=1}^{\frac{\sfr-1}{2}}
\pi([[e,v_i^*],v_i])\bigg)\otimes1_\chi\\
=&\bigg(2e+\frac{h^2}{2}-\frac{3}{2}h-2v_{\frac{\sfr+1}{2}}[v_{\frac{\sfr+1}{2}},e]+
\sum_{i=1}^{k-1}h_i^2+
\sum_{i=1}^{k-1}\sum_{j=1}^{w}\beta_{\bar0j}(h_i)h_i-\sum_{i=1}^{k-1}\sum_{j=1}^{\ell}\beta_{\bar1j}(h_i)h_i\\
&+2\sum\limits_{i=1}^{\frac{s}{2}}t_{\bar{\gamma}^*_{\bar0i}}
-2\sum\limits_{i=1}^{\frac{\sfr-1}{2}}t_{\bar{\gamma}^*_{\bar1i}}\bigg)\otimes1_\chi\\
=&\bigg(2e+\frac{h^2}{2}-\frac{3}{2}h+\sqrt{-2}v_{\frac{\sfr+1}{2}}[\sqrt{-2}v_{\frac{\sfr+1}{2}},e]+
\sum_{i=1}^{k-1}h_i^2+
\sum_{i=1}^{k-1}\Big(\sum_{j=1}^{w}\beta_{\bar0j}-\sum_{j=1}^{\ell}\beta_{\bar1j}
-2\sum\limits_{j=1}^{\frac{s}{2}}\gamma_{\bar0j}\\
&+2\sum\limits_{j=1}^{\frac{\sfr-1}{2}}\gamma_{\bar1j}\Big)(h_i)h_i\bigg)\otimes1_\chi.
\end{split}
\end{equation}

Recall in Proposition \ref{fiosp} we introduced the PBW theorem of $U(\ggg_0,e)$. Keeping the notations as in \eqref{(e,h,f,E,F)}, it follows from \eqref{pithetahj}---\eqref{pithetaC} that
\begin{equation}\label{ThetatoTheta'}
\begin{split}
\pi(\Theta_{h_i})=&\Theta'_{h_i}+\delta(h_i)\qquad \text{for}~1\leqslant i\leqslant k-1;\\
\pi(\sqrt{-2}\Theta_{[v_{\frac{\sfr+1}{2}},e]})=&\bigg([\sqrt{-2} v_{\frac{\sfr+1}{2}},e]-\frac{3}{4}\sqrt{-2}v_{\frac{\sfr+1}{2}}+\frac{1}{2}
\sqrt{-2}v_{\frac{\sfr+1}{2}}h\bigg)\otimes1_\chi=\Theta'_E;\\
\pi(C)
=&C'_\theta+\sum_{i=1}^{k-1}(\Theta'_{h_i})^2+\sum_{i=1}^{k-1}\Big(\sum_{j=1}^{w}\beta_{\bar0j}-\sum_{j=1}^{\ell}\beta_{\bar1j}
-2\sum\limits_{j=1}^{\frac{s}{2}}\gamma_{\bar0j}+2\sum\limits_{j=1}^{\frac{\sfr-1}{2}}\gamma_{\bar1j}\Big)(h_i)\Theta'_{h_i};\\
\pi(\sqrt{-2}\Theta_{v_{\frac{\sfr+1}{2}}})=&\sqrt{-2}v_{\frac{\sfr+1}{2}}\otimes1_\chi=\Theta'_F.
\end{split}
\end{equation}

(2) Recall that any $\Theta\in U(\ggg,e)_0$ can be written as a linear combination of the monomials as in \eqref{PBWfW} with $\sum_{i}(-a_i+b_i)\beta_{\bar 0i}+\sum_{i}(-c_i+d_i)\beta_{\bar 1i}+\sum_{i}(m_i-n_i)\gamma_{\bar 0i}+\sum_{i}(p_i-q_i)\gamma_{\bar 1i}= 0$. %Moreover, for any $v\in\{x_1,\cdots, x_w, x^*_1,\cdots, x^*_w, y_1,\cdots, y_\ell, y^*_1, \cdots, y^*_\ell, f_1,\cdots, f_{\frac{s}{2}}, f^*_1,\cdots, f^*_{\frac{s}{2}},$\\$ g_1,\cdots,g_{\frac{\sfr-1}{2}}, g^*_1, \cdots, g^*_{\frac{\sfr-1}{2}},h_1,\cdots,h_{k-1},v_{\frac{\sfr+1}{2}},[v_{\frac{\sfr+1}{2}},e]\}$,  by the explicit description of $\Theta_v$ in Theorem \ref{main3} we see that $v$ and $\Theta_v$ have the same $\mathfrak{h}^e$-weight. Also note that the $\mathfrak{h}^e$-weight of $C$ is zero.
Now it follows from the definition of $\pi$ in \eqref{pi}, \eqref{ThetatoTheta'} and Proposition \ref{fiosp} that the restriction $\pi$ defines a surjective homomorphism $\pi:U(\ggg,e)_0\rightarrow U(\ggg_0,e)$, and $U(\ggg,e)_{0,\sharp}\subset\ker\pi$ by definition. Then %it follows from the PBW theorem of $U(\ggg,e)$ in Theorem \ref{main3} that
the quotient $U(\ggg,e)_0 / U(\ggg,e)_{0,\sharp}$ has a basis given by the coset of the PBW monomials
$\Theta_{v_{\frac{\sfr+1}{2}}}^{\iota}(\prod_{i=1}^{k-1}\Theta_{h_i}^{t_i})C^{t_k}\Theta_{[v_{\frac{\sfr+1}{2}},e]}^\varepsilon$ with $\iota,\varepsilon\in\Lambda_1, {\bf t}\in\mathbb{Z}_+^k$. Moreover, \eqref{ThetatoTheta'} and Proposition \ref{fiosp} entail  that
%\begin{equation}\label{Theta'toTheta}
%\begin{split}
%\Theta'_{h_i}=&\pi(\Theta_{h_i})-\delta(h_i),\qquad \qquad \text{for}~1\leqslant i\leqslant k;\\
%\Theta'_E=&\sqrt{-2}\pi(\Theta_{[v_{\frac{\sfr+1}{2}},e]});\qquad \qquad \Theta'_F=\sqrt{-2}\pi(\Theta_{v_{\frac{\sfr+1}{2}}});\\
%C'_\theta=&\pi(C)-\sum_{i=1}^{k-1}(\pi(\Theta_{h_i})-\delta(h_i))^2+\sum_{i=1}^{k-1}\Big(\sum_{j=1}^{w}\beta_{\bar0j}-\sum_{j=1}^{\ell}\beta_{\bar1j}
%-2\sum\limits_{j=1}^{\frac{s}{2}}\gamma_{\bar0j}+2\sum\limits_{j=1}^{\frac{\sfr-1}{2}}\gamma_{\bar1j}\Big)(h_i)\\
%&(\pi(\Theta_{h_i})-\delta(h_i)).
%\end{split}
%\end{equation}
%Therefore,
the monomials
$$\pi(\Theta_{v_{\frac{\sfr+1}{2}}})^{\iota}\bigg(\prod_{i=1}^{k-1}\pi(\Theta_{h_i})^{t_i}\bigg)\pi(C)^{t_k}\pi(\Theta_{[v_{\frac{\sfr+1}{2}},e]})^\varepsilon$$ with $\iota,\varepsilon\in\Lambda_1, {\bf t}\in\mathbb{Z}_+^k$ actually form a basis for $U(\ggg_0,e)$, then the kernel of $\pi$  is no bigger than $U(\ggg,e)_{0,\sharp}$.

(3) Since $S_{\epsilon}$ is a shift on $U(\ggg_0,e)$, by the discussion in Step (2)  it remains to show that $\pi_{\epsilon}=S_{\epsilon}\circ\pi:U(\ggg,e)_0 / U(\ggg,e)_{0,\sharp}\rightarrow U(\ggg_0,e)$ is an algebra homomorphism. Keeping the notations as in \eqref{deltarho}, it follows from  \eqref{keycommuta}, \eqref{ThetatoTheta'} and Theorems \ref{main3}, \ref{fiosp} that
\begin{equation}\label{keycom}
\begin{split}
&\pi_{\epsilon}([\Theta_{[v_{\frac{\sfr+1}{2}},e]},\Theta_{[v_{\frac{\sfr+1}{2}},e]}])\\
=&\pi_{\epsilon}\bigg(-\frac{1}{2}C+\frac{1}{2}c_0+\frac{1}{2}\sum_{i=1}^{k-1}\Theta_{h_i}^2+\sum_{i=1}^{w}\Theta_{x_i}\Theta_{x^*_i}+
\frac{1}{2}\sum_{i=1}^{k-1}\sum_{j=1}^{w}\beta_{\bar0j}(h_i)\Theta_{h_i}
-\frac{1}{2}\sum_{i=1}^{k-1}\sum_{j=1}^{\ell}\beta_{\bar1j}(h_i)\Theta_{h_i}\\
&+\sum_{i=1}^{\ell}\Theta_{y_i}\Theta_{y^*_i}-2\sum_{i=1}^{\frac{s}{2}}\Theta_{[[v_{\frac{\sfr+1}{2}},e],
u_{i}]^{\sharp}}\Theta_{[u_{i}^*,[v_{\frac{\sfr+1}{2}},e]]^{\sharp}}-\frac{1}{2}\sum_{i=1}^{k-1}\sum_{j=1}^{\frac{s}{2}}\gamma_{\bar0j}(h_i)\Theta_{h_i}\\
&-2\sum_{i=1}^{\frac{\sfr-1}{2}}\Theta_{[[v_{\frac{\sfr+1}{2}},e],v_{i}]^{\sharp}}
\Theta_{[v_{i}^*,[v_{\frac{\sfr+1}{2}},e]]^{\sharp}}+\frac{1}{2}\sum_{i=1}^{k-1}\sum_{j=1}^{\frac{\sfr-1}{2}}\gamma_{\bar1j}(h_i)\Theta_{h_i}\bigg)\\
=&-\frac{1}{2}C'_\theta-\frac{1}{2}\sum_{i=1}^{k-1}(\Theta'_{h_i})^2-\frac{1}{2}\sum_{i=1}^{k-1}\Big(\sum_{j=1}^{w}\beta_{\bar0j}-\sum_{j=1}^{\ell}\beta_{\bar1j}
-2\sum\limits_{j=1}^{\frac{s}{2}}\gamma_{\bar0j}+2\sum\limits_{j=1}^{\frac{\sfr-1}{2}}\gamma_{\bar1j}\Big)(h_i)\Theta'_{h_i}-\frac{1}{2}\epsilon\\
&+\frac{1}{2}c_0+\frac{1}{2}\sum_{i=1}^{k-1}(\Theta'_{h_i}+\delta(h_i))^2+
\frac{1}{2}\sum_{i=1}^{k-1}\Big(\sum_{j=1}^{w}\beta_{\bar0j}-\sum_{j=1}^{\ell}\beta_{\bar1j}
-\sum\limits_{j=1}^{\frac{s}{2}}\gamma_{\bar0j}+\sum\limits_{j=1}^{\frac{\sfr-1}{2}}\gamma_{\bar1j}\Big)(\Theta'_{h_i}+\delta(h_i))\\
=&-\frac{1}{2}C'_\theta-\frac{1}{2}\epsilon+\frac{1}{2}c_0+\sum_{i=1}^{k-1}\Big(\rho_{e,0}\delta+\frac{3}{2}\delta^2\Big)(h_i)\\
=&-\frac{1}{2}C'_\theta+\frac{1}{2}\Big(-\epsilon+c_0+\sum_{i=1}^{k-1}(2\rho_{e,0}\delta+3\delta^2)(h_i)\Big)\\
=&-\frac{1}{2}C'_\theta-\frac{1}{16}
=-\frac{1}{2}[\Theta'_E,\Theta'_E]=[\pi_{\epsilon}(\Theta_{[v_{\frac{\sfr+1}{2}},e]}),\pi_{\epsilon}(\Theta_{[v_{\frac{\sfr+1}{2}},e]})],
\end{split}
\end{equation} and $\pi_{\epsilon}([\Theta_{v_{\frac{\sfr+1}{2}}},\Theta_{v_{\frac{\sfr+1}{2}}}])=1\otimes1_\chi=-\frac{1}{2}[\Theta'_F,\Theta'_F]
=[\pi_{\epsilon}(\Theta_{v_{\frac{\sfr+1}{2}}}),\pi_{\epsilon}(\Theta_{v_{\frac{\sfr+1}{2}}})]$,
with the other commutators being trivial.
\end{proof}

\begin{rem}
In the procedure of constructing the isomorphism in Proposition \ref{ressur}, we involve a shift $S_{\epsilon}$ which has great distinction from the one in \cite[Theorem 4.3]{BGK} for the finite $W$-algebra case. %There are two reasons for this. First, the decomposition $\ggg_0=\mathfrak{h}^e\oplus\mathfrak{s}_\theta$ happens to be a direct sum of ideals, i.e., $[\mathfrak{h}^e,\mathfrak{s}_\theta]=0$. So the shift $S_{-\gamma}$ on $\mathfrak{h}^e$ (denoted by $\mathfrak{t}$ there) with $\gamma\in(\hhh^e)^*$ as in \cite[\S4.1]{BGK} is redundant. Second,
%The main reason lies in the fact that the construction of  $\ggg_0$ between the super case and non-super one is quite different.
%To be explicit, %let $\ggg$ to be a semi-simple Lie algebra, with $e\in\ggg$ being arbitrary nilpotent element, and the corresponding finite $W$-algebra $U(\ggg,e)$ and $U(\ggg_0,e)$ are defined. Then up to a shifted for the parabolic subalgebra $\mathfrak{p}_0:=\bigoplus_{i\geqslant 0}{\ggg_0}(i)$ of $\ggg_0$ (the shift extends uniquely to characters of $\mathfrak{p}_0$), $U(\ggg_0,e)$ is isomorphic to the quotient of $U(\ggg,e)$ (see \cite[\S4.1]{BGK} for more details). However, these can not be generalized directly to the super case. In fact,
Recall in \cite[Proposition 5.7.6]{Car} the grading of $\ggg_0$ under the action of $\text{ad}\,h$ is even for any semi-simple Lie algebra $\ggg$, which means that $\ggg_0(i)=\{0\}$ for all odd $i\in\mathbb{Z}$. Then $\gamma\in(\hhh^e)^*$ defined in \cite[\S4.1]{BGK} extends uniquely to a character of the parabolic subalgebra $\mathfrak{p}_0:=\bigoplus_{i\geqslant 0}{\ggg_0}(i)$ of $\ggg_0$ by  \cite[Lemma 4.1]{BGK}. In the meanwhile, the finite $W$-algebra $U(\ggg_0,e)$ can be realized as a subalgebra of $U(\mathfrak{p}_0)$. Under the above settings, the proof for \cite[Theorem 4.3]{BGK} goes through. However, for the basic  Lie superalgebra $\ggg_0$ and a minimal root $-\theta$,  %the root vector $e_{\theta}\in(\ggg_0)_{\bar0}$ with $-\theta$ being a minimal root,
under the action of ad\,$h_{\theta}$  we have $\ggg_0(-1)_{\bar1}=\bbc v_{\frac{\sfr+1}{2}}\neq\{0\}$ and $\ggg_0(1)_{\bar1}=\bbc [v_{\frac{\sfr+1}{2}},e]\neq\{0\}$ under our settings. So the grading of $\ggg_0$ under the action of $\text{ad}\,h$ is not even. Moreover, the proper subspace  $\widetilde{\mathfrak{p}}_0:=\bigoplus_{i\geqslant -1}{\ggg_0}(i)\supset\mathfrak{p}_0$ of $\ggg_0$ is not a good choice. Therefore, the technique in \cite{BGK} is not available here. Fortunately, in the case when minimal roots are concerned,  we have already obtained the precise generators and their relations for $U(\ggg,e)$. This enables us to construct the desired isomorphism. However, there should be a requirement of developing general technique for dealing with finite $W$-superalgebra $U(\ggg,e)$ associated with arbitrary nilpotent element $e\in\ggg_{\bar0}$.

%as $\widetilde{\mathfrak{p}}_0$ is not a Lie subalgebra of $\ggg_0$. Not to mention introducing some element in $(\hhh^e)^*$  that extends to a character of $\widetilde{\mathfrak{p}}_0$, which is substantial in the formulating of Proposition \ref{ressur}. Therefore, when dealing with finite $W$-superalgebras associated with general nilpotents, other methods are needed to draw the similar conclusion as in Proposition \ref{ressur}.
%We hope this result is still true for finite $W$-superalgebras associated with general nilpotents, after a modification for the definition of the shift.
\end{rem}

\subsection{}\label{5.1.4}
For a $U(\ggg,e)$-module $M$ and
$\alpha\in(\hhh^e)^*$, we define the $\alpha$-weight space
\begin{equation}\label{shiftby3delta}
M_\alpha:=\{m\in M\:|\: (\Theta_t-\delta(t))(m) = \alpha(t) m\text{ for all } t \in \hhh^e\}.
\end{equation}
By Theorem \ref{main3}  we have that $U(\ggg,e)_\beta M_\alpha\subset M_{\beta+\alpha}$.
So each $M_\alpha$ is invariant under the action of the subalgebra
$U(\ggg,e)_0$. We call $M_\alpha$  a  maximal weight space of $V$
if $U(\ggg,e)_\sharp V_\alpha = \{0\}$.

Set $M_\alpha$ to be a maximal weight space in
a $U(\ggg,e)$-module $M$. By Proposition \ref{ressur}, the action of $U(\ggg,e)_0$ on
$M_\alpha$ factors through the map $\pi_{\epsilon}$ to make $M_\alpha$ into a $U(\ggg_0,e)$-module
such that $u.m = \pi_{\epsilon}(u) m$ for $u \in U(\ggg,e)_0$ and $m \in M_\alpha$. It follows from Proposition \ref{fiosp} that $\Theta'_{\hhh^e}$ lies in the center of $U(\ggg_0,e)$, then  $\hhh^e$ can be considered as a Lie subalgebra of $U(\ggg_0,e)$.
So by Proposition \ref{fiosp}(1), \eqref{ThetatoTheta'} and \eqref{shiftby3delta} we get the action of $\hhh^e$ on $M_\alpha$ via
\begin{equation*}
  t(m)=\Theta'_{t}(m)=\alpha(t) m
\end{equation*}for all $t \in \hhh^e$, which explains why the additional shift by $-\delta$ in the  definition of the $\alpha$-weight space  of a $U(\ggg,e)$-module is necessary.

A $U(\ggg,e)$-module $M$ is called a highest weight module
if it is generated by a maximal weight space $M_\lambda$
such that $M_\lambda$ is finite-dimensional and irreducible as a $U(\ggg_0,e)$-module.
In that case, as we will see shortly, $\lambda$ is the unique maximal
weight of $M$ in the dominance ordering. Recall that pairwise inequivalent finite-dimensional  irreducible $U(\ggg_0,e)$-modules $V_\lambda$ with $\lambda\in(\hhh^e)^*$ are given in Theorem \ref{simpleugegmodule}.

By definition we see that $U(\ggg,e)_\sharp$ is invariant under left multiplication
by $U(\ggg,e)$ and right multiplication by $U(\ggg,e)_0$, then $U(\ggg,e) / U(\ggg,e)_\sharp$ is
a $(U(\ggg,e), U(\ggg,e)_0)$-bimodule. And then the right action of $U(\ggg,e)_0$ factors through the
map $\pi_{\epsilon}$ from Proposition \ref{ressur} to make
$U(\ggg,e) / U(\ggg,e)_\sharp$ into a $(U(\ggg,e), U(\ggg_0,e))$-bimodule. We introduce the highest weight $U(\ggg,e)$-module with highest weight $\lambda$ as
\begin{equation}\label{highestweigmodu}
 M_e(\lambda):=(U(\ggg,e)/U(\ggg,e)_\sharp)\otimes_{U(\ggg_0,e)} V_\lambda.
\end{equation}
We will show that $M_e(\lambda)$ is universal, meaning that if $M$ is another highest weight module
generated by a maximal weight space $M_\mu$
and $f: V_\lambda \stackrel{\sim}{\rightarrow} M_\mu$
is an even $U(\ggg_0,e)$-module isomorphism,
then there is a unique $U(\ggg,e)$-module
homomorphism $\tilde{f}: M_e(\lambda)\twoheadrightarrow M$ extending $f$. First observe that
\begin{lemma}\label{LRBasis}
As a right $U(\ggg_0,e)$-module,
$U(\ggg,e)/U(\ggg,e)_\sharp$ is free
with basis $$\bigg\{\prod_{i=1}^w\Theta_{x_i}^{a_i}\cdot\prod_{i=1}^\ell\Theta_{y_i}^{b_i}\cdot
\prod_{i=1}^{\frac{s}{2}}\Theta_{f_i}^{c_i}\cdot\prod_{i=1}^{\frac{\sfr-1}{2}}
\Theta_{g_i}^{d_i}\:|\:{\bf a}\in\mathbb{Z}_+^w, {\bf b}\in\Lambda_\ell, {\bf c}\in\mathbb{Z}_+^{\frac{s}{2}}, {\bf d}\in\Lambda_{\frac{\sfr-1}{2}}\bigg\}.$$
\end{lemma}
\begin{proof}
This follows because
the cosets of the PBW monomials of the form $\prod_{i=1}^w\Theta_{x_i}^{a_i}\cdot\prod_{i=1}^\ell\Theta_{y_i}^{c_i}\cdot
\prod_{i=1}^{\frac{s}{2}}\Theta_{f_i}^{m_i}\cdot\prod_{i=1}^{\frac{\sfr-1}{2}}
\Theta_{g_i}^{p_i}\cdot\Theta_{v_{\frac{\sfr+1}{2}}}^\iota\cdot\prod_{i=1}^{k-1}
\Theta_{h_i}^{t_i}\cdot C^{t_k}\cdot\Theta_{[v_{\frac{\sfr+1}{2}},e]}^\varepsilon$ with ${\bf a}\in\mathbb{Z}_+^w$, ${\bf c}\in\Lambda_\ell$, ${\bf m}\in\mathbb{Z}_+^{\frac{s}{2}}$, ${\bf p}\in\Lambda_{\frac{\sfr-1}{2}},\iota,\varepsilon\in\Lambda_1, {\bf t}\in\mathbb{Z}_+^k$,
form a basis for the quotient
$U(\ggg,e)/U(\ggg,e)_\sharp$,
and the cosets of the monomials of the form $\Theta_{v_{\frac{\sfr+1}{2}}}^\iota\cdot\prod_{i=1}^{k-1}
\Theta_{h_i}^{t_i}\cdot C^{t_k}\cdot\Theta_{[v_{\frac{\sfr+1}{2}},e]}^\varepsilon$ with $\iota,\varepsilon\in\Lambda_1, {\bf t}\in\mathbb{Z}_+^k$,
form a basis for  $U(\ggg,e)_0/U(\ggg,e)_{0,\sharp} \cong U(\ggg_0,e)$ by Step (2) in the proof of Proposition \ref{ressur}.
\end{proof}

Now we will introduce the main result of this subsection. We mainly follow the strategy for the non-super case as in \cite[Theorem 4.5]{BGK}, with a lot of modifications.
\begin{theorem}\label{Tverma}
For $\lambda\in(\hhh^e)^*$, let $v_\lambda$ and $\Theta'_F(v_\lambda)=\sqrt{-2}\Theta_{v_{\frac{\sfr+1}{2}}}(v_\lambda)$ be a basis for $U(\ggg_0,e)$-module $V_\lambda$ with $\hhh^e$-weight $\lambda$ as in \S\ref{523}.
\begin{itemize}
\item[(1)]
The vectors $\bigg\{\prod_{i=1}^w\Theta_{x_i}^{a_i}\cdot\prod_{i=1}^\ell\Theta_{y_i}^{b_i}\cdot
\prod_{i=1}^{\frac{s}{2}}\Theta_{f_i}^{c_i}\cdot\prod_{i=1}^{\frac{\sfr-1}{2}}
\Theta_{g_i}^{d_i}\cdot\Theta_{v_{\frac{\sfr+1}{2}}}^\iota(v_\lambda)\:|\:{\bf a}\in\mathbb{Z}_+^w, {\bf b}\in\Lambda_\ell, {\bf c}\in\mathbb{Z}_+^{\frac{s}{2}}, {\bf d}\in\Lambda_{\frac{\sfr-1}{2}},
\iota\in\Lambda_1\bigg\}$ form a basis of $M_e(\lambda)$.
\item[(2)] The weight $\lambda$ is the unique maximal weight of $M_e(\lambda)$
in the dominance ordering,
$M_e(\lambda)$ is generated by the maximal weight space
$M_e(\lambda)_\lambda$, and $M_e(\lambda)_\lambda \cong V_\lambda$ as $U(\ggg_0,e)$-modules.
\item[(3)] The module $M_e(\lambda)$ is a universal highest weight module with highest weight $\lambda$.
\item[(4)] There is a unique maximal proper submodule
$M_e^{\rm max}(\lambda)$ in $M_e(\lambda)$,
\begin{equation}\label{lelambda}
L_e(\lambda):=M_e(\lambda)/M_e^{\rm max}(\lambda)
\end{equation}
is an irreducible module type $Q$, and $\{L_e(\lambda)\:|\: \lambda\in (\hhh^e)^*\}$
is a complete set of pairwise inequivalent irreducible highest weight modules over $U(\ggg,e)$ (up to parity switch).
Moreover, any finite-dimensional simple $U(\ggg,e)$-module  (up to parity switch) is isomorphic to one of the modules $L_e(\lambda)$ for $\lambda\in\Lambda_0^+ := \{\lambda \in (\hhh^e)^*\:|\:
\lambda(h_\alpha)\in\bbz_+~\text{for}~\alpha\in(\Phi^+_{e,0})_{\bar0}\}$.
\end{itemize}
\end{theorem}
\begin{proof}
(1) This is clear from Lemma \ref{LRBasis}.

(2) Since $\theta(\hhh^e)=0$ by definition, then $\prod_{i=1}^w\Theta_{x_i}^{a_i}\cdot\prod_{i=1}^\ell\Theta_{y_i}^{b_i}\cdot
\prod_{i=1}^{\frac{s}{2}}\Theta_{f_i}^{c_i}\cdot\prod_{i=1}^{\frac{\sfr-1}{2}}
\Theta_{g_i}^{d_i}\cdot\Theta_{v_{\frac{\sfr+1}{2}}}^\iota(v_\lambda)$ is of $\hhh^e$-weight $\lambda-\sum_ia_i\beta_{\bar0i}-
\sum_ib_i\beta_{\bar1i}+\sum_ic_i\gamma_{\bar0i}+\sum_id_i\gamma_{\bar1i}$ (Keep in mind that $\gamma_{\bar0i}, \gamma_{\bar1j}\in(\Phi'_e)^-$). Hence the $\lambda$-weight space of $M_e(\lambda)$ is $1 \otimes V_\lambda$
and all other weights of $M_e(\lambda)$ are strictly smaller in the dominance
ordering.

(3) It follows from (1) and (2) that $M_e(\lambda)$ is a highest weight module with highest weight $\lambda$. Let $M$ be another highest weight module generated
by a maximal weight space $M_\mu$ and $f: V_\lambda \rightarrow M_\mu$
be an even $U(\ggg_0,e)$-module isomorphism. By Theorem \ref{simpleugegmodule}
we obtain $\mu = \lambda$.
By adjointness of tensor and hom,
$f$ extends uniquely to a $U(\ggg,e)$-module homomorphism
$\tilde f: M_e(\lambda) \rightarrow M$
such that $\tilde f(1 \otimes v_\lambda) = f(v_\lambda)$.
As $\tilde f(1 \otimes V_\lambda) = f(V_\lambda)$
generates $M$, then $\tilde f$ is surjective.

(4) Let $M$ be a submodule of $M_e(\lambda)$.
Then $M$ is the direct sum of its $\hhh^e$-weight spaces.
If $M_\lambda \neq 0$ then $M_\lambda$ generates all of
$1 \otimes V_\lambda$ as a $U(\ggg_0,e)$-module,
hence it generates all of $M_e(\lambda)$ as a $U(\ggg,e)$-module.
This shows that if $M$ is a proper submodule then
it is contained in $M_e(\lambda)_-:=\bigoplus_{\mu < \lambda} M_e(\lambda)_\mu$.
Hence the sum of all proper submodules of $M_e(\lambda)$ is still a proper
submodule, so $M_e(\lambda)$ has a unique maximal submodule $M_e^{\rm max}(\lambda)$ as claimed.
%Since any $\bbz_2$-graded module of $M_e(\lambda)$ is the direct sum of its $\hhh^e$-weight spaces, then by the same discussion as in the last paragraph of Step (2) for the proof of Theorem \ref{verma2}, $M_e(\lambda)$ has a unique maximal submodule $M_e^{\rm max}(\lambda)$.
By (3) any irreducible highest weight module $N$ of type $\lambda$
is a quotient of $M_e(\lambda)$, hence
$N \cong
L_e(\lambda)$ or $\prod L_e(\lambda)$ (Here $\prod$ denotes the parity switching functor). Moreover $\lambda$ is the unique maximal weight of
$L_e(\lambda)$ by (2) and $L_e(\lambda)_\lambda$ is isomorphic to $N_\lambda$ or $\prod N_\lambda$ as $U(\ggg_0,e)$-modules.
Thus $\lambda$ is uniquely determined by $N$. Since $N_\lambda$ is an irreducible $U(\ggg_0,e)$-module of  type $Q$ by Theorem \ref{simpleugegmodule}, then $L_e(\lambda)$ is also of type $Q$.

By Theorem \ref{main3} (more precisely, Theorem \ref{maiin1}(1)), when we restrict $L_e(\lambda)$ to the $\mathfrak{sl}(2)$-triple $(\Theta_{e_\alpha},\Theta_{h_\alpha},\Theta_{e_{-\alpha}})\subset U(\ggg,e)$ with $\alpha\in(\Phi^+_{e,0})_{\bar0}$, we get that $\lambda(h_\alpha)\in\mathbb{Z}_+$ for any $\alpha\in(\Phi^+_{e,0})_{\bar0}$.
\end{proof}

\subsection{}\label{5.1.5}
Let $Z(U(\ggg))$ denote the center of $U(\ggg)$, and write $Z(U(\ggg,e))$ for the center of $U(\ggg,e)$.
It is immediate from the definition of $U(\ggg,e)$ that the restriction of the linear map $\text{Pr}:U({\ggg})\twoheadrightarrow U({\ggg})/I_\chi$ defines a natural algebra homomorphism $\text{Pr}:Z(U(\ggg))\rightarrow Z(U(\ggg,e))$.
Denote by the representation of $U(\ggg)$ in $\text{End}(Q_\chi)$ by $\rho_{\chi}$.
Thanks to \cite[Corollary 13.2.2]{M1}, the Harish-Chandra homomorphism $Z(U(\ggg))\rightarrow U(\mathfrak{h})$ is injective, so is the restriction of $\rho_{\chi}: U(\ggg)\rightarrow\text{End}(Q_\chi)$ to $Z(U(\ggg))$, i.e., the map  $\text{Pr}:Z(U(\ggg))\rightarrow Z(U(\ggg,e))$ is injective (The discussion here comes by the same strategy as in its non-super case \cite[\S6.2]{P2}, and one can also refer to the proof of \cite[Lemma 3.2]{SZc} for more details). Moreover,
when $\ggg$ is a complex semi-simple Lie algebra, as explained in the footnote to \cite[Question 5.1]{P3} (see also \cite[Remark 2.1]{PPY}), the map $\text{Pr}$ is actually an algebra isomorphism. But it seems more complicated in the super case. We try to raise a conjecture:
\begin{conjecture}\label{centercon}
The center of finite $W$-superalgebra $U(\ggg,e)$ coincides with the image of $Z(U(\ggg))$ in $U(\ggg,e)$ under the projection $\text{Pr}:Z(U(\ggg))\rightarrow Z(U(\ggg,e))$.
\end{conjecture}
We can give positive answer to Conjecture \ref{centercon} in some special cases. First, as an immediate consequence of Proposition \ref{centertocenter}, the conjecture is true for $\ggg=\mathfrak{osp}(1|2)$ with $e$ being regular nilpotent (which  can also be considered as a root vector $e_\theta$ with $-\theta$ being a minimal root). Second, if $\ggg=\mathfrak{gl}(m|n)$ with $e$ being regular nilpotent, \cite[Theorem 3.21]{BBG3} ensures the conjecture. Third, the conjecture for the Lie superalgebra $\mathfrak{q}(n)$  with $e$ being  regular nilpotent  is confirmed in \cite[Corollary 5.10]{PS2}. However, whether this conclusion holds for the general situation is still an open problem.

Let us turn back to the minimal case. %Although it is difficult to give an explicit description of the center of minimal finite $W$-superalgebra, we can still study it in another way.
%In the following, we will show that $Z(U(\ggg,e))$ embeds into $Z(U(\ggg_0,e))$ under the map $\pi$ as defined in \eqref{pi}. To achieve this, we first need some preparations.
For further discussion, we will compute the action of the Casimir element $C\in Z(U(\ggg,e))$ on the highest weight module $M_e(\lambda)$.
\begin{lemma}\label{Conhigestwtmodule}
Retain the notations as in \S\ref{3.2.1} and \S\ref{5.3}. For all $\lambda\in(\hhh^e)^*$, $C$ acts on the highest weight module $M_e(\lambda)$ as the scalar $c_0+(\lambda,\lambda+2\bar\rho)+\sum_{i=1}^{k-1}(2\rho_{e,0}\delta+3\delta^2)(h_i)$.
\end{lemma}
\begin{proof}
Since all positive root vectors annihilate $v_\lambda$, and $C'_\theta(v_\lambda)=-\frac{1}{8}v_\lambda$ by Proposition \ref{fiosp}(i), the similar discussion as in \eqref{C1}---\eqref{c1lambdac} entails that
\begin{equation*}
\begin{split}
C.v_\lambda=&\pi_{\epsilon}(C)(v_\lambda)=(C+\epsilon)(v_\lambda)\\
=&\bigg(C'_\theta+(\lambda,\lambda+2\bar\rho)+c_0+\frac{1}{8}+\sum_{i=1}^{k-1}(2\rho_{e,0}\delta+3\delta^2)(h_i)\bigg)(v_\lambda)\\
=&\bigg(c_0+(\lambda,\lambda+2\bar\rho)+\sum_{i=1}^{k-1}(2\rho_{e,0}\delta+3\delta^2)(h_i)\bigg)(v_\lambda).
\end{split}
\end{equation*}
Since $M_e(\lambda)$ is a cyclic module generated by $v_\lambda$, then the conclusion follows.
%Recall in \S\ref{3.1.1} that the restriction of $(\cdot,\cdot)$ to $\mathfrak{h}^e$ is non-degenerate, and $\{h_1,\cdots,h_{k-1}\}$ is a basis of $\mathfrak{h}^e$ such that $(h_i,h_j)=\delta_{i,j}$ for $1\leqslant i,j\leqslant k-1$. For any $\mu\in(\hhh^e)^*$, write $t_\mu=\sum_{j=1}^{k-1}a_jh_j$, then $\mu(h_i)=(t_\mu,h_i)=\sum_{j=1}^{k-1}a_j(h_j,h_i)=a_i$. Thus \begin{equation}
%(\mu,\mu)=(t_\mu,t_\mu)=\sum_{i,j=1}^{k-1}a_ia_j(h_i,h_j)=\sum_{i=1}^{k-1}\mu(h_i)^2.
%\end{equation}
%For all $\alpha\in\Phi^+$, the elements $e_\alpha\in\ggg_\alpha, e_{-\alpha}\in\ggg_{-\alpha}$ are chosen so that $[e_\alpha,e_{-\alpha}]=t_\alpha$. Then we obtain bases $\{X_i\},\{Y_j\}$ for $\ggg$ with $(X_i,Y_j)=\delta_{i,j}$ by taking the elements
%\begin{equation*}
%\{h_i,h_i\},\{e_\alpha,e_{-\alpha}\}_{\alpha\in\Phi_{\bar0}},\{e_\alpha,e_{-\alpha}\}_{\alpha\in\Phi_{\bar1}^+},
%\{e_{-\alpha},-e_\alpha\}_{\alpha\in\Phi_{\bar1}^+},
%\end{equation*}and $C=\sum_{i=1}^{k-1}h_i^2+\sum_{\alpha\in\Phi_{\bar0}^+}(e_\alpha e_{-\alpha}+e_{-\alpha}e_\alpha)+\sum_{\alpha\in\Phi^+_{\bar1}}(e_{-\alpha}e_\alpha-e_\alpha e_{-\alpha})$. We have $C.v_\lambda=cv_\lambda$ for some $c\in\bbc$. Furthermore,
\end{proof}

Similar to \eqref{anotherdesC}, we also have another description of  Lemma \ref{Conhigestwtmodule}. Taking Proposition \ref{fiosp} into consideration,  by the similar discussion as in \eqref{pithetaC} we obtain
\begin{equation}\label{anotherdesofC}
\begin{split}
C.v_\lambda=&\pi_{\epsilon}(C)(v_\lambda)=(C+\epsilon)(v_\lambda)\\
=&\bigg(C'_\theta+\sum_{i=1}^{k-1}h_i^2+\sum_{i=1}^{k-1}\Big(\sum_{j=1}^{w}\beta_{\bar0j}-\sum_{j=1}^{\ell}\beta_{\bar1j}
-2\sum\limits_{j=1}^{\frac{s}{2}}\gamma_{\bar0j}+2\sum\limits_{j=1}^{\frac{\sfr-1}{2}}\gamma_{\bar1j}\Big)(h_i)h_i+c_0+\frac{1}{8}\\
&+\sum_{i=1}^{k-1}(2\rho_{e,0}\delta+3\delta^2)(h_i)\bigg)(v_\lambda)\\
=&\bigg(c_0+\sum_{i=1}^{k-1}\Big(\lambda(h_i)^2+(2\rho_{e,0}+4\delta)(h_i)\lambda(h_i)+(2\rho_{e,0}\delta+3\delta^2)(h_i)\Big)\bigg)(v_\lambda)
\end{split}
\end{equation}
This conclusion will be used in the formulation of Theorem \ref{numberis finite}.

A $U(\ggg,e)$-module $V$ is of central character $\psi: Z(U(\ggg,e))\rightarrow\bbc$ if $z(v)=\psi(z)v$ for all $z\in Z(U(\ggg,e))$ and $v\in V$. %Analogously, we say that a %$U(\ggg_0,e)$-module $V$ is of
%central character $\psi_0:Z(U(\ggg_0,e)) \rightarrow \bbc$
%if $z(v) = \psi_0(v)$ for all $z \in Z(U(\ggg_0,e))$ and $v \in V$.
%Thanks to Proposition \ref{an embedding center}, we can relate the central character of %an irreducible
%highest weight module over $U(\ggg,e)$  to the central character of its maximal weight %space over
%$U(\ggg_0,e)$. To be explicit,
%For $\lambda\in(\hhh^e)^*$ let $\psi^\lambda: Z(U(\ggg,e))\rightarrow\bbc$ %and %$\psi_0^\lambda: Z(U(\ggg_0,e)) \rightarrow \bbc$
%be the corresponding central character. %, respectively and let denote %$\psi_0^\lambda\circ\pi_{\epsilon}$.
For the highest weight $U(\ggg,e)$-module $M_e(\lambda)$ with highest weight $\lambda\in(\hhh^e)^*$, let $\psi^\lambda: Z(U(\ggg,e))\rightarrow\bbc$ be the corresponding central character. %it is readily to see that it is of
%central character $\psi^\lambda: Z(U(\ggg,e))\rightarrow\bbc$.
Under the above settings, we have
\begin{theorem}\label{numberis finite}
The number of isomorphism classes of irreducible
highest weight modules for $U(\ggg,e)$ with prescribed central character $\psi: Z(U(\ggg,e))\rightarrow\bbc$ is finite, i.e., the set
$\{\lambda \in (\hhh^e)^*\:|\:\psi^\lambda = \psi\}$ is finite.
\end{theorem}
\begin{proof}
Given any irreducible
highest weight $U(\ggg,e)$-module $L_e(\lambda)$ with prescribed central character $\psi: Z(U(\ggg,e))\rightarrow\bbc$, we can write $\psi=\psi^\lambda$.
%By the discussion above the theorem, $\psi^\lambda=\psi$ implies $\psi_0^\lambda=\psi_0$ for some central character $\psi_0:Z(U(\ggg_0,e)) \rightarrow \bbc$ such that $\psi_0\circ\pi_{\epsilon}=\psi$.
Let $\mu\in(\hhh^e)^*$ be such that $\psi^\mu=\psi^\lambda$, then the action of $C\in Z(U(\ggg,e))$ on $v_\mu$ of $L_e(\mu)=U(\ggg,e)v_\mu$ coincides with that on
$v_\lambda$ of $L_e(\lambda)=U(\ggg,e)v_\lambda$, i.e., $\psi^\mu(C)=\psi^\lambda(C)$. Thanks to \eqref{anotherdesofC}, we have
\begin{equation}\label{mulambda}
\begin{split}
0=&\psi^\mu(C)-\psi^\lambda(C)\\
=&\sum_{i=1}^{k-1}\Big(\mu(h_i)^2+(2\rho_{e,0}+4\delta)(h_i)\mu(h_i)\Big)
-\sum_{i=1}^{k-1}\Big(\lambda(h_i)^2+(2\rho_{e,0}+4\delta)(h_i)\lambda(h_i)\Big)\\
=&\sum_{i=1}^{k-1}\Big(h_i^2+(2\rho_{e,0}+4\delta)(h_i)h_i\Big)(\mu)-
\sum_{i=1}^{k-1}\Big(\lambda(h_i)^2+(2\rho_{e,0}+4\delta)(h_i)\lambda(h_i)\Big),
\end{split}
\end{equation}
where the last equation of \eqref{mulambda} is considered as a polynomial in indeterminates $h_1,\cdots,h_{k-1}$ at $\mu$. By the knowledge of linear algebra we know that the solution to \eqref{mulambda} is finite, i.e., the set $\{\lambda\in(\hhh^e)^*\:|\:\psi^\lambda = \psi\}$ is finite.
%Thanks to Proposition \ref{centerg0e}, Theorem \ref{simpleugegmodule} and  the discussion above it, $\psi_0^\lambda$ is uniquely determined by each $\lambda \in (\hhh^e)^*$. Moreover, Proposition \ref{an embedding center} entails that the restriction of $\pi_{\epsilon}$ on $Z(U(\ggg,e))$ is injective, %and $(\lambda-\frac{\theta}{2})\mid_{\mathfrak{h}^e}=\lambda\mid_{\mathfrak{h}^e}$ by definition,
%then the conclusion follows.
\end{proof}

For the simple root system $\Delta=\{\alpha_1,\cdots,\alpha_k\}$ for $\Phi$,
write $\beta_i:=\alpha_i|_{(\mathfrak{h}^e)^*}$ for $1\leqslant i \leqslant k-1$, then $\{\beta_1,\cdots,\beta_{k-1}\}$ is a system of restricted simple roots for $(\Phi'_e)^+$ by our earlier discussion. Denote by
$(Q^e)^+:=\sum_{i=1}^{k-1}\mathbb{Z}_+\beta_i$.
%Now we can prove that the highest weight modules have finite length, i.e.,
As a corollary of Theorem \ref{numberis finite}, we have
\begin{corollary}\label{highestcompser}
For $\lambda \in (\hhh^e)^*$, the highest weight module $M_e(\lambda)$ has composition series.
\end{corollary}
\begin{proof}
The corollary can be proved by imitating the standard argument in the classical case from
\cite[\S7.6.1]{Dix}. Let us put it explicitly.

Let $N$ and $N'$ be sub-$U(\ggg,e)$-modules of $M_e(\lambda)$ such that $N'\subseteq N$ and $N/N'$ is simple. Since $M_e(\lambda)=\bigoplus_{\mu'\in(\hhh^e)^*}M_e(\lambda)_{\mu'}$, we have
$N/N'=\bigoplus_{\mu'\in(\hhh^e)^*}(N/N')_{\mu'}$, and  every $\mu'$ belongs to $\lambda-(Q^e)^+$.
So there exists a weight $\mu$ of $N/N'$ such that, for all $\alpha\in(\Phi'_e)^+$, $\mu+\alpha$ is not a weight of $N/N'$. If $v$ is a nonzero element of $(N/N')_\mu$,
we have $U(\ggg,e)_\sharp(v) = 0$, and hence $N/N'$ is isomorphic to $L_e(\mu)$ (up to parity switch) by Theorem \ref{Tverma}.  The central characters of $N/N'$ and $M_e(\lambda)$ are equal, whence $\mu$ is in the finite set $\{\lambda \in (\hhh^e)^*\:|\:\psi^\lambda = \psi\}$ by Theorem \ref{numberis finite}.

Recall that $U(\ggg,e)$ is Noetherian by Proposition \ref{zero-divisor}.
Since $M_e(\lambda)$ is a cycle $U(\ggg,e)$-module, then $M_e(\lambda)$ is a Noetherian $U(\ggg,e)$-module.
Every nonzero sub-$U(\ggg,e)$-module $N$ of $M_e(\lambda)$ hence contains a sub-$U(\ggg,e)$-module $N'$ such that $N/N'$ is simple. If $M_e(\lambda)$ had no composition series, there would exist an infinite decreasing sequence
$N_0\supset N_1\supset\cdots$ of sub-$U(\ggg,e)$-modules of $M_e(\lambda)$ such that every $N_i/N_{i+1}$ was simple.
From the previous paragraph, infinitely many of these quotients would be isomorphic to each
other. Then one of the $\hhh^e$-weights of $M_e(\lambda)$ would have infinite multiplicity,
contrary to Theorem \ref{Tverma}(1).
\end{proof}
\subsection{}\label{5.1.6}
Now we introduce an analogue of the BGG (short for Bernstein-Gelfand-Gelfand) category $\mathcal{O}$. Let $\mathcal{O}(e)=\mathcal{O}(e;\mathfrak{h},\mathfrak{q})$ denote the category of all finitely generated $U(\ggg,e)$-modules $V$, that are semi-simple over $\mathfrak{h}^e$
with finite-dimensional $\mathfrak{h}^e$-weight spaces, such that the set
$\{\lambda \in (\mathfrak{h}^e)^*\:|\:V_\lambda \neq \{0\}\}$ is contained in a finite union of sets of the form $\{\nu\in(\mathfrak{h}^e)^*\:|\: \nu\leqslant\mu\}$ for $\mu \in (\mathfrak{h}^e)^*$.

In virtue of the results we obtained in this section, we introduce the proof of Theorem \ref{proofcateO}, which can be checked routinely as a counterpart  for the ordinary BGG category $\mathcal{O}$ (see, e.g., \cite{Hu2,M1}). Let us put it explicitly.
\vskip0.2cm
\noindent\textbf{The proof of Theorem \ref{proofcateO}}
First note that $U(\ggg,e)$ is Noetherian by Proposition \ref{zero-divisor}, then $\mathcal{O}(e)$ is closed under the operations of taking submodules, quotients and finite direct sums.

(1) Let $L$ be a simple object in $\mathcal{O}(e)$, we will show that
$L\cong L_e(\lambda)$ for some $\lambda\in (\hhh^e)^*$.

Since $L$ is $\mathbb{Z}_2$-graded, there exists $\mu\in(\hhh^e)^*$ such that $L_\mu$ contains a nonzero
homogeneous element $v$. Since $N = U(\ggg,e)_\sharp(v)$ is finite-dimensional by the definition of $\mathcal{O}(e)$, there
exists $\lambda\in(\hhh^e)^*$ such that $N_\lambda\neq0$, but $N_{\lambda+\alpha} = 0$ for all restricted positive roots $\alpha\in(\hhh^e)^*$. Since $N_\lambda$ is finite-dimensional, we can further assume that
$\Theta_{[v_{\frac{\sfr+1}{2}},e]}.N_\lambda=0$.   Now $N_\lambda$ is a finite-dimensional $\mathbb{Z}_2$-graded $U(\ggg,e)_0 / U(\ggg,e)_{0,\sharp}\cong U(\ggg_0,e)$-module, so $N_\lambda$ contains
a $U(\ggg_0,e)$-submodule isomorphic to $V_\lambda$  by Theorem \ref{simpleugegmodule}. Since
$U(\ggg,e)V_\lambda$ is a $U(\ggg,e)$-submodule of $L$, it equals $L$ by simplicity. Then $L\cong L_e(\lambda)$ follows from Theorem \ref{Tverma}(3)-(4).

(2) Let $M$ be any nonzero module in $\mathcal{O}(e)$. We claim that $M$ has a finite filtration
$0 \subset M_1 \subset M_2 \subset \cdots \subset M_n = M$ with nonzero quotients each of which is a highest weight module, then Statement (2) follows from Corollary \ref{highestcompser}.

In fact, since $M$ can be generated by finitely many weight vectors,  and for each $m\in M$, the subspace
$U(\ggg,e)_\sharp(m)$ is finite-dimensional by definition, then the $U(\ggg,e)_\sharp$-submodule $V$ generated by such a generating
family of weight vectors is finite-dimensional. If $\dim V=1$, it is clear that
$M$ itself is a highest weight module. Otherwise proceed by induction on
$\dim V$.

Start with a nonzero weight vector $v\in V$ of weight $\lambda\in(\mathfrak{h}^e)^*$ which is maximal
among all weights of $V$ and is therefore a maximal vector in $M$. It generates
a submodule $M_1$, while the quotient $\overline M := M/M_1$ again lies in $\mathcal{O}(e)$ and is generated
by the image $\overline V$ of $V$. Since $\dim\overline V< \dim V$, the induction hypothesis
can be applied to $M$, yielding a chain of highest weight submodules whose
pre-images in $M$ are the desired $M_2, \cdots,  M_n$.

(3) Define a subspace of $M$ for each fixed $\psi^{\lambda}$ by
$$M^{\psi^{\lambda}}:=\{m\in M\:|\:(z-\psi^{\lambda}(z))^n(m)=0~\text{for some}~n>0~\text{depending on}~z\in Z(U(\ggg,e))\}.$$
It is clear that $M^{\psi^{\lambda}}$ is a $U(\ggg,e)$-submodule of $M$, while the subspaces
$M^{\psi^{\lambda}}$ are independent.

Now $Z(U(\ggg,e))$ stabilizes each weight space $M_\mu$, as $Z(U(\ggg,e))$ and $U(\mathfrak{h}^e)$ commute.
Then a standard result from linear algebra on families of commuting operators implies that
$M_\mu=\bigoplus_{\mu}(M_\mu\cap M^{\psi^{\lambda}})$. Because $M$ is generated by finitely many weight
vectors, it must therefore be the direct sum of finitely many nonzero submodules
$M^{\psi^{\lambda}}$. %Thanks to  Proposition \ref{centerg0e}  and Proposition \ref{an embedding center},
Thanks to Statement (1), each central character
$\psi$ occurring in this way must be of the form $\psi^\lambda$ for some $\lambda\in(\mathfrak{h}^e)^*$.

Denote by $\mathcal O_{\psi^\lambda}(e)$  the (full) subcategory of $\mathcal O(e)$ whose objects are the modules
$M$ for which $M=M^{\psi^{\lambda}}$. Then Statement (3) follows.
$\hfill\square$

%In particular, this shows that the irreducible objects in $\mathcal O(e)$
%are all of the form $L_e(\lambda)$ for $\lambda \in (\hhh^e)^*$.
At the extreme, if $\ggg=\ggg_0\cong\mathfrak{h}^e\oplus\mathfrak{osp}(1|2)$,
then $e$ is a distinguished nilpotent element
of $\ggg$, i.e., the only semi-simple elements of $\ggg$ that centralize $e$ belong to the center of $\ggg$. In this case, $\mathcal O(e)$ is
the category of all finite-dimensional $U(\ggg,e)$-modules that are semi-simple
over the Lie algebra center of $\ggg$.

\begin{rem}
Let $U(\ggg,e)$ be the finite $W$-algebra associated with a complex semi-simple Lie algebra $\ggg$ and its nilpotent element $e$. In \cite{Los}, by the method of quantized symplectic actions, Losev established an equivalence of the BGG category $\mathcal O$ for $U(\ggg,e)$ in \cite{BGK} with certain category of $\ggg$-modules. In the case when $e$ is of principal Levi type, the category of $\ggg$-modules in interest is the category of generalized Whittaker modules. Recall in \S\ref{4}  we have achieved this goal for Verma modules of minimal finite $W$-superalgebras.  It will be an interesting topic to generalize all these to finite $W$-superalgebras associated with arbitrary even nilpotent elements.
\end{rem}
\begin{rem}\label{compare}
For the minimal finite $W$-superalgebra $U(\ggg,e)$ of type odd, it can be checked that both  the Verma module $Z_{U(\ggg,e)}(\lambda,c)$ introduced in \S\ref{proof of thverma} and the highest weight module $M_e(\lambda)$ constructed in this section, are essentially identical. Actually, the action of $\Theta_{v_{\frac{\sfr+1}{2}}}$, $\Theta_{h_i}$ for $1\leqslant i \leqslant k-1$, $C$ and $\Theta_{[v_{\frac{\sfr+1}{2}},e]}$ on $Z_{U(\ggg,e)}(\lambda,c)$ is translated into the action of $U(\ggg_0,e)$ on $M_e(\lambda)$. And the restriction \eqref{c0clambda} for $Z_{U(\ggg,e)}(\lambda,c)$ is converted into the one in Lemma \ref{Conhigestwtmodule} (or \eqref{anotherdesofC}) for $M_e(\lambda)$.

However, each of the two definitions has its own advantages. On one hand, the $U(\ggg,e)$-module $Z_{U(\ggg,e)}(\lambda,c)$ is easy to construct, and it is much like the highest weight theory for $U(\ggg)$. Moreover, it is more convenient to establish a link between $Z_{U(\ggg,e)}(\lambda,c)$ and the $\ggg$-modules obtained by parabolic induction from Whittaker modules for $\mathfrak{osp}(1|2)$ (i.e., the standard Whittaker modules)  as in \S\ref{4.1}. On the other hand, one can observe that the related theory for $U(\ggg,e)$-module $M_e(\lambda)$ is more fruitful, not only allows us to defined
the corresponding BGG category $\mathcal{O}$, but also provides a method that may be generalized to finite $W$-superalgebras associated with other nilpotent elements.
%All above discussions apply for their irreducible quotients, and also .

The above discussion also apply for the minimal refined $W$-superalgebra $W'_\chi$ of both types. See the forthcoming Appendix \ref{5.2} for more details.
\end{rem}

\begin{appendix}
\renewcommand{\thesection}{\Alph{section}}
\section{On the category $\mathcal{O}$ for minimal refined $W$-superalgebras}\label{5.2}
As a counterpart of \S\ref{mathcalO}, we will consider the abstract universal highest weight modules for minimal refined $W$-superalgebra $W'_\chi$ of both types, and then consider the corresponding BGG category $\mathcal{O}$ in this appendix.
Since the discussion for them are similar as in \S\ref{mathcalO}, we will just sketch them, omitting the proofs.

\subsection{}\label{5.2.1}
We first consider the minimal refined $W$-superalgebra $W'_\chi$ of type odd.

Keep the notations as in \S\ref{4.1} and \S\ref{mathcalO}. Recall that
for the  restricted root decomposition $\ggg=\ggg_0 \oplus \bigoplus_{\alpha\in \Phi'_e} \ggg_\alpha$ with respect to $\mathfrak{h}^e$, we have $\ggg_0\cong\mathfrak{h}^e\oplus\mathfrak{osp}(1|2)$ as in \S\ref{4.2.2}. Moreover, $(\Phi'_e)^+=(\Phi^+\backslash\{\frac{\theta}{2},\theta\})|_{(\mathfrak{h}^e)^*}$ is a system of positive roots in the restricted root system $\Phi'_e$, and  $(\Phi'_e)^-=-(\Phi'_e)^+$. Define the minimal refined $W$-superalgebra $(W_0)'_\chi$ associated to $e\in\ggg_0$ by
$(W_0)'_\chi:=(Q_0)_{\chi}^{\ad\,\mmm'_0}$, where $\mmm'_0$ is the ``extended $\chi$-admissible subalgebra" introduced in \S\ref{5.1.2}.   By the same discussion as in \S\ref{5.1.2}, we have
\begin{prop}\label{refoddstr}The following statements  hold:
\begin{itemize}\item[(1)]the minimal refined $W$-superalgebra $(W_0)'_\chi$ is generated by
\begin{itemize}
\item[(i)] $\Theta'_{h_i}=h_i\otimes1_\chi$ for $1\leqslant i\leqslant k-1$;
\item[(ii)] $\Theta'_E=\big(E+\frac{1}{2}Fh-\frac{3}{4}F\big)\otimes1_\chi$;
\item[(iii)] $C'_\theta=\big(2e+\frac{1}{2}h^2-\frac{3}{2}h+FE\big)\otimes1_\chi$,
\end{itemize}
subject to the relation: $[\Theta'_E,\Theta'_E]=C'_\theta+\frac{1}{8}\otimes1_\chi$,
and the commutators between  other generators are all zero;
\item[(2)] the center $Z((W_0)'_\chi)$ of $(W_0)'_\chi$ is generated by $\Theta'_{h_i}$ for $1\leqslant i\leqslant k-1$ and $C'_\theta$;
\item[(3)] the map $\text{Pr}_0:U({\ggg}_0)\twoheadrightarrow U({\ggg}_0)/(I_0)_\chi$  sends $Z(U(\ggg_0))$ isomorphically onto the center
of $(W_0)'_\chi$;
\item[(4)] for $\lambda\in(\hhh^e)^*$, let $V_\lambda:=\bbc v_\lambda$ be a vector space spanned by $v_\lambda\in (V_\lambda)_{\bar0}$ satisfying $\Theta'_E(v_\lambda)=0$, $C'_\theta(v_\lambda)=-\frac{1}{8}v_\lambda$, and $\Theta'_t(v_\lambda)=\lambda(t)v_\lambda$ for all $t\in\hhh^e$. Then
the set $\{V_\lambda\:|\:\lambda\in(\hhh^e)^*\}$ forms a complete set of pairwise inequivalent finite-dimensional  irreducible $(W_0)'_\chi$-modules, all of which  are of type $M$.
\end{itemize}
\end{prop}

We have the following PBW basis for $W'_\chi$:
\begin{equation}\label{PBWRW}
\begin{array}{ll}
&\prod_{i=1}^w\Theta_{x_i}^{a_i}\cdot\prod_{i=1}^\ell\Theta_{y_i}^{c_i}\cdot
\prod_{i=1}^{\frac{s}{2}}\Theta_{f_i}^{m_i}\cdot\prod_{i=1}^{\frac{\sfr-1}{2}}
\Theta_{g_i}^{p_i}\cdot\prod_{i=1}^{k-1}
\Theta_{h_i}^{t_i}\cdot C^{t_k}\\
&\cdot\Theta_{[v_{\frac{\sfr+1}{2}},e]}^\varepsilon\cdot\prod_{i=1}^{\frac{s}{2}}\Theta_{f^*_i}^{n_i}\cdot
\prod_{i=1}^{\frac{\sfr-1}{2}}\Theta_{g^*_i}^{q_i}\cdot
\prod_{i=1}^w\Theta_{x^*_i}^{b_i}\cdot
\prod_{i=1}^\ell\Theta_{y^*_i}^{d_i},
\end{array}
\end{equation}
where ${\bf a},{\bf b}\in\mathbb{Z}_+^w$, ${\bf c},{\bf d}\in\Lambda_\ell$, ${\bf m},{\bf n}\in\mathbb{Z}_+^{\frac{s}{2}}$, ${\bf p},{\bf q}\in\Lambda_{\frac{\sfr-1}{2}}, \varepsilon\in\Lambda_1, {\bf t}\in\mathbb{Z}_+^k$.
Let $v$ be any element in \eqref{toeaualweight} excluding $e$. Since $\theta(\hhh^e)=0$ by definition, from the explicit description of $\Theta_v$ in Theorem \ref{ge}, we see that $v$ and $\Theta_v$ have the same $\mathfrak{h}^e$-weight. Also note that the $\mathfrak{h}^e$-weight of $C$ is zero.
As $\hhh^e$-root spaces,  $W'_\chi$ can be decomposed into $W'_\chi=\bigoplus_{\alpha \in \bbz\Phi'_e}(W'_\chi)_\alpha$.
The restricted root space $(W'_\chi)_\alpha$
has a basis given by all the PBW monomials as in \eqref{PBWRW} such that
$\sum_{i}(-a_i+b_i)\beta_{\bar 0i}+\sum_{i}(-c_i+d_i)\beta_{\bar 1i}+\sum_{i}(m_i-n_i)\gamma_{\bar 0i}+\sum_{i}(p_i-q_i)\gamma_{\bar 1i}= \alpha$. In particular, $(W'_\chi)_0$ is the zero weight space of $W'_\chi$ spanned all above with $\alpha=0$.
Set
$(W'_\chi)_\sharp$ (resp.\ $(W'_\chi)_\flat$)
to be the left (resp.\ right) ideal of
$W'_\chi$ generated by all $(W'_\chi)_\alpha$ for $\alpha \in (\Phi'_e)^+$
(resp.\ $\alpha\in (\Phi'_e)^-$), and denote by $(W'_\chi)_{0,\sharp} := (W'_\chi)_0 \cap (W'_\chi)_\sharp,
(W'_\chi)_{\flat,0} := (W'_\chi)_\flat \cap (W'_\chi)_0$. By the PBW theorem we have $(W'_\chi)_{0,\sharp}=(W'_\chi)_{\flat,0}$, hence it is a two-sided ideal of  $(W'_\chi)_0$.

Let  $S_{\epsilon}$ be a shift on $W'_\chi$
by keeping all other generators  as in Proposition \ref{refoddstr}(1) invariant and sending   $C'_\theta$ to $C'_\theta+\epsilon$ with $\epsilon$ having the same meaning as in \eqref{defofepsilon}. %Then extend the shift $S_{\epsilon}$ on $W'_\chi$.
By the similar discussion as in the proof of Proposition \ref{ressur}, we obtain
\begin{prop}\label{ressur2}
The projection $\pi:U(\ggg)_0\twoheadrightarrow U(\mathfrak{g}_0)$ as in \eqref{pi} induces
a surjective homomorphism $\pi: (W'_\chi)_0\twoheadrightarrow (W_0)'_\chi$
with  $\ker\pi= (W'_\chi)_{0,\sharp}$, and  there exists an algebras isomorphism
\begin{equation*}
\pi_{\epsilon}:=S_{\epsilon}\circ\pi:(W'_\chi)_0 / (W'_\chi)_{0,\sharp} \cong (W_0)'_\chi.
\end{equation*}
\end{prop}

For a $W'_\chi$-module $M$ and
$\alpha\in(\hhh^e)^*$, we define the $\alpha$-weight space
\begin{equation*}
M_\alpha:=\{m\in M\:|\: (\Theta_t-\delta(t))(m) = \alpha(t) m\text{ for all } t \in \hhh^e\}.
\end{equation*}
Since the right action of $(W'_\chi)_0$ factors through the
map $\pi_{\epsilon}$ from Proposition \ref{ressur2} to make
$W'_\chi / (W'_\chi)_\sharp$ into a $(W'_\chi, (W_0)'_\chi)$-bimodule, then the highest weight $W'_\chi$-module with highest weight $\lambda$ can be defined as
\begin{equation*}
M_e(\lambda):=(W'_\chi / (W'_\chi)_\sharp)\otimes_{(W_0)'_\chi} V_\lambda,
\end{equation*}where the  $(W_0)'_\chi$-module $V_\lambda$ is one-dimensional, defined as in Proposition \ref{refoddstr}(4).  Moreover, we have
\begin{theorem}\label{Tverma2}
For $\lambda\in(\hhh^e)^*$,
let $v_\lambda$ be a basis for $(W_0)'_\chi$-module $V_\lambda$ with $\hhh^e$-weight $\lambda$ as in Proposition \ref{refoddstr}(4).
\begin{itemize}
\item[(1)]
The vectors $\bigg\{\prod_{i=1}^w\Theta_{x_i}^{a_i}\cdot\prod_{i=1}^\ell\Theta_{y_i}^{b_i}\cdot
\prod_{i=1}^{\frac{s}{2}}\Theta_{f_i}^{c_i}\cdot\prod_{i=1}^{\frac{\sfr-1}{2}}
\Theta_{g_i}^{d_i}(v_\lambda)\:|\:{\bf a}\in\mathbb{Z}_+^w, {\bf b}\in\Lambda_\ell, {\bf c}\in\mathbb{Z}_+^{\frac{s}{2}}, {\bf d}\in\Lambda_{\frac{\sfr-1}{2}}\bigg\}$ form a basis of $M_e(\lambda)$.
\item[(2)] The weight $\lambda$ is the unique maximal weight of $M_e(\lambda)$
in the dominance ordering,
$M_e(\lambda)$ is generated by the maximal weight space
$M_e(\lambda)_\lambda$, and $M_e(\lambda)_\lambda \cong V_\lambda$ as $(W_0)'_\chi$-modules.
\item[(3)] The module $M_e(\lambda)$ is a universal highest weight module with highest weight $\lambda$, i.e., if $M$ is another highest weight module
generated by a maximal weight space $M_\mu$
and $f: V_\lambda \stackrel{\sim}{\rightarrow} M_\mu$
is a $(W_0)'_\chi$-module isomorphism,
then there is a unique $W'_\chi$-module
homomorphism $\tilde{f}: M_e(\lambda)\twoheadrightarrow M$ extending $f$.
\item[(4)] There is a unique maximal proper submodule
$M_e^{\rm max}(\lambda)$ in $M_e(\lambda)$,
\begin{equation}\label{irrtypeoddM}
L_e(\lambda):=M_e(\lambda)/M_e^{\rm max}(\lambda)
\end{equation}
is an irreducible module type $M$, and $\{L_e(\lambda)\:|\: \lambda\in (\hhh^e)^*\}$
is a complete set of pairwise inequivalent irreducible highest weight modules over $W'_\chi$.
Moreover, any finite-dimensional simple $W'_\chi$-module is isomorphic to one of the modules $L_e(\lambda)$ for $\lambda\in\Lambda_0^+= \{\lambda \in (\hhh^e)^*\:|\:
\lambda(h_\alpha)\in\bbz_+~\text{for}~\alpha\in(\Phi^+_{e,0})_{\bar0}\}$.
\end{itemize}
\end{theorem}

%By the same discussion as in Proposition \ref{an embedding center}, we get
%\begin{prop}\label{an embedding center 2}
%There exists an embedding $\pi_{\epsilon}: Z(W'_\chi)\hookrightarrow Z((W_0)'_\chi)$.
%\end{prop}
We say that a $W'_\chi$-module  $V$ is of central character $\psi: Z(W'_\chi)\rightarrow\bbc$ %(resp. $\psi_0:Z((W_0)'_\chi) \rightarrow \bbc$)
if $z(v)=\psi(z)v$ %(resp. $z(v) = \psi_0(v)$)
for all $z\in Z(W'_\chi)$ %(resp. $z \in Z((W_0)'_\chi)$)
and $v\in V$.
%In virtue of Proposition \ref{an embedding center 2},
%For $\lambda\in(\hhh^e)^*$ let $\psi^\lambda: Z(W'_\chi)\rightarrow\bbc$
%(resp. $\psi_0^\lambda: Z((W_0)'_\chi) \rightarrow \bbc$)
%be the corresponding central character.
For the highest weight $W'_\chi$-module $M_e(\lambda)$ with highest weight $\lambda\in(\hhh^e)^*$, let $\psi^\lambda: Z(W'_\chi)\rightarrow\bbc$ be the corresponding central character. Repeat verbatim the proofs of Theorem \ref{numberis finite} and Corollary \ref{highestcompser}, we obtain
\begin{theorem}\label{numberis finite 2}
The number of isomorphism classes of irreducible
highest weight
modules for $W'_\chi$ with prescribed central character $\psi: Z(W'_\chi)\rightarrow\bbc$  is finite, i.e., the set
$\{\lambda \in (\hhh^e)^*\:|\:\psi^\lambda = \psi\}$  is finite.
\end{theorem}
\begin{corollary}\label{highestcompser2}
For each $\lambda \in (\hhh^e)^*$, the highest weight module $M_e(\lambda)$ has composition series.
\end{corollary}

Now an analogue of the BGG category $\mathcal{O}$ can also be introduced. Denote by $\mathfrak{q} = \ggg_0\oplus\bigoplus_{\alpha \in (\Phi'_e)^{+}}\ggg_\alpha$, and let $\mathcal{O}(e)=\mathcal{O}(e;\mathfrak{h},\mathfrak{q})$ denote the category of all finitely generated $W'_\chi$-modules $V$, that are semi-simple over $\mathfrak{h}^e$
with finite-dimensional $\mathfrak{h}^e$-weight spaces, such that the set
$\{\lambda \in (\mathfrak{h}^e)^*\:|\:V_\lambda \neq \{0\}\}$ is contained in a finite union of sets of the form $\{\nu\in(\mathfrak{h}^e)^*\:|\: \nu\leqslant\mu\}$ for $\mu \in (\mathfrak{h}^e)^*$.
Then we have
\begin{theorem}\label{BGGcateO}
For the category $\mathcal O(e)$, the following statements hold:
\begin{itemize}
\item[(1)] There is a complete set of isomorphism classes of simple objects which is $\{L_e(\lambda)\:|\: \lambda\in (\hhh^e)^*\}$ as in \eqref{irrtypeoddM}.
\item[(2)] The category $\mathcal O(e)$ is Artinian. In particular, every object has finite length of composition series.
\item[(3)] The category $\mathcal O(e)$ has a block decomposition as $\mathcal O(e) = \bigoplus_{\psi^\lambda}  \mathcal O_{\psi^\lambda}(e)$, where the direct sum is over all central characters $\psi^\lambda:Z(U(\ggg,e)) \rightarrow \bbc$,
and $\mathcal O_{\psi^\lambda}(e)$ denotes the Serre subcategory of $\mathcal O(e)$
generated by the irreducible modules
$\{L_e(\mu)\:|\:\mu \in (\hhh^e)^*~\text{such that}~\psi^\mu=\psi^\lambda\}$.
\end{itemize}
\end{theorem}
\subsection{}
It remains to deal with the minimal refined (equivalently, finite) $W$-superalgebra $W'_\chi$ of type even. In this case, it is much like the situation of finite $W$-algebras as in \cite[\S4]{BGK}, and we still just give a  brief description.

Keep the notations as in \S\ref{5.4} and \S\ref{5.1.1}. The adjoint action of
$\mathfrak{h}^e$ on $\mathfrak{g}$
induces the  restricted root decompositions
$\ggg = \ggg_0 \oplus \bigoplus_{\alpha \in \Phi'_e} \ggg_\alpha$, and we  have $\ggg_0=\mathfrak{h}^e\oplus\mathfrak{s}_\theta\cong\mathfrak{h}^e\oplus\mathfrak{sl}(2)$ as  defined in \S\ref{5.4}. Let $(\Phi'_e)^+$ be a system of positive roots in the restricted root system $\Phi'_e$, and $(\Phi'_e)^-=-(\Phi'_e)^+$.  Let $\mathfrak{m}_0:=\bbc f$ be the ``$\chi$-admissible subalgebra" of $\ggg_0$, $\mathfrak{p}_0:=\mathfrak{h}^e\oplus\bbc h\oplus\bbc e$ a parabolic subalgebra of $\ggg_0$ with nilradical $\bbc e$ and Levi subalgebra $\mathfrak{h}$. Define
$(Q_0)_\chi:=U({\ggg}_0)\otimes_{U(\mathfrak{m}_0)}{\bbc}_\chi$, where ${\bbc}_\chi={\bbc}1_\chi$ is a one-dimensional  $\mathfrak{m}_0$-module such that $x.1_\chi=\chi(x)1_\chi$ for all $x\in\mathfrak{m}_0$. Let $(I_0)_\chi$ denote the ${\bbz}_2$-graded left ideal in $U({\ggg}_0)$ generated by all $x-\chi(x)$ with $x\in\mathfrak{m}_0$, and write $\text{Pr}_0:U({\ggg}_0)\twoheadrightarrow U({\ggg}_0)/(I_0)_\chi$ for the canonical homomorphism.
Recall the minimal refined $W$-superalgebra $(W_0)'_\chi$ associated to $e$ is defined by
\begin{equation*}
(W_0)'_\chi:=(\text{End}_{\ggg_0}(Q_0)_{\chi})^{\text{op}}\cong(Q_0)_{\chi}^{\ad\,\mmm_0}.
\end{equation*}
So $(W_0)'_\chi\cong\{u\in U(\mathfrak{p}_0)\mid\text{Pr}_0([x,u])=0~\text{for all}~x\in\mathfrak{m}_0\}$, which is a subalgebra of $U(\mathfrak{p}_0)$.  By the similar discussion as in \S\ref{5.1.2}, we have
\begin{prop}\label{refpbweven2}The following statements  hold:
\begin{itemize}\item[(1)]
the minimal refined $W$-superalgebra $(W_0)'_\chi$ is generated by
\begin{itemize}
\item[(i)] $\Theta'_{h_i}=h_i\otimes1_\chi$ for $1\leqslant i\leqslant k-1$;
\item[(ii)] $C'_\theta=\big(2e+\frac{1}{2}h^2-h\big)\otimes1_\chi$,
\end{itemize}
and the commutators between the generators are all zero.
\item[(2)] the algebra $(W_0)'_\chi$ is commutative;
\item[(3)] the map $\text{Pr}_0$  sends $Z(U(\ggg_0))$ isomorphically onto  $(W_0)'_\chi$;
\item[(4)] for $\lambda\in(\hhh^e)^*$ and $c\in\bbc$, let $V_{\lambda,c}:=\bbc v_{\lambda,c}$ be a vector space spanned by $v_{\lambda,c}\in (V_{\lambda,c})_{\bar0}$ satisfying $C'_\theta(v_{\lambda,c})=cv_{\lambda,c}$, and $\Theta'_t(v_{\lambda,c})=\lambda(t)v_{\lambda,c}$ for all $t\in\hhh^e$. Then
the set $\{V_{\lambda,c}\:|\:\lambda\in(\hhh^e)^*,c\in\bbc\}$ forms a complete set of pairwise inequivalent finite-dimensional  irreducible $(W_0)'_\chi$-modules, all of which are of type $M$.
\end{itemize}
\end{prop}

The adjoint action of
$\mathfrak{h}^e$ on the universal enveloping algebra $U(\ggg)$
induces decomposition
 $U(\ggg) = \bigoplus_{\alpha \in \bbz\Phi'_e} U(\ggg)_\alpha$.
Then $U(\ggg)_0$, the zero weight space of $U(\ggg)$ with
respect to the adjoint action, is a subalgebra of $U(\ggg)$.
Let $U(\ggg)_\sharp$ (resp.\ $U(\ggg)_\flat$) denote the left (resp.\ right)
ideal of $U(\ggg)$ generated by the root spaces $\ggg_\alpha$ for
$\alpha \in (\Phi'_e)^+$ (resp.\ $\alpha \in (\Phi'_e)^-$).
Let
$$
U(\ggg)_{0, \sharp} := U(\ggg)_0 \cap U(\ggg)_\sharp,
\qquad
U(\ggg)_{\flat,0} := U(\ggg)_\flat \cap U(\ggg)_0,
$$
which are  left and right ideals of $U(\ggg)_0$, respectively.
By the PBW theorem,  $U(\ggg)_{0,\sharp}$ is a two-sided ideal of $U(\ggg)_0$, and
$U(\ggg)_0 = U(\ggg_0) \oplus U(\ggg)_{0,\sharp}$.
The projection along this decomposition
defines a surjective algebra homomorphism
\begin{equation}\label{pi2}
\pi:U(\ggg)_0 \twoheadrightarrow U(\ggg_0)
\end{equation}
with $\ker\pi = U(\ggg)_{0,\sharp}$.
Hence $U(\ggg)_0 / U(\ggg)_{0,\sharp} \cong U(\ggg_0)$ as $\bbc$-algebras.

Recall in \S\ref{1.2} and \S\ref{3.1.1} that we have a basis consisting of $\hhh^e$-weight vectors
\begin{equation}\label{toeaualweight2}
\begin{split}
&x_1,\cdots,x_w,y_1,\cdots,y_\ell,f_1,\cdots,f_{\frac{s}{2}},g_1,\cdots,g_{\frac{\sfr}{2}},
h_1,\cdots,h_{k-1},\\
&e,f^*_1,\cdots,f^*_{\frac{s}{2}},g^*_1,\cdots,g^*_{\frac{\sfr}{2}},x^*_1,\cdots,x^*_w,y^*_1,\cdots,y^*_\ell
\end{split}
\end{equation}of $\ggg^e$
so that the weights of $x_i, y_j, f_k, g_l$ are respectively $-\beta_{\bar 0i}, -\beta_{\bar 1j}, \theta+\gamma_{\bar0k}, \theta+\gamma_{\bar1l}\in(\Phi'_e)^-$, and the weights of $f^*_k, g^*_l, x^*_i, y^*_j$ are respectively
$\theta+\gamma^*_{\bar0k}, \theta+\gamma^*_{\bar1l}, \beta_{\bar 0i}, \beta_{\bar 1j}\in(\Phi'_e)^+$,  while $h_i, e\in\ggg_0^e$.
Moreover, we have the following PBW basis for $W'_\chi$:
\begin{equation}\label{PBWRWeven}
\begin{array}{ll}
&\prod_{i=1}^w\Theta_{x_i}^{a_i}\cdot\prod_{i=1}^\ell\Theta_{y_i}^{c_i}\cdot
\prod_{i=1}^{\frac{s}{2}}\Theta_{f_i}^{m_i}\cdot\prod_{i=1}^{\frac{\sfr}{2}}
\Theta_{g_i}^{p_i}\cdot\prod_{i=1}^{k-1}
\Theta_{h_i}^{t_i}\\
&\cdot C^{t_k}\cdot\prod_{i=1}^{\frac{s}{2}}\Theta_{f^*_i}^{n_i}\cdot
\prod_{i=1}^{\frac{\sfr}{2}}\Theta_{g^*_i}^{q_i}\cdot
\prod_{i=1}^w\Theta_{x^*_i}^{b_i}\cdot
\prod_{i=1}^\ell\Theta_{y^*_i}^{d_i},
\end{array}
\end{equation}
where ${\bf a},{\bf b}\in\mathbb{Z}_+^w$, ${\bf c},{\bf d}\in\Lambda_\ell$, ${\bf m},{\bf n}\in\mathbb{Z}_+^{\frac{s}{2}}$, ${\bf p},{\bf q}\in\Lambda_{\frac{\sfr}{2}}, {\bf t}\in\mathbb{Z}_+^k$.
Let $v$ be one of the element in \eqref{toeaualweight2} excluding $e$. Since $\theta(\hhh^e)=0$ by definition, from the explicit description of $\Theta_v$ in Theorem \ref{ge}, we see that $v$ and $\Theta_v$ have the same $\mathfrak{h}^e$-weight. Also note that the $\mathfrak{h}^e$-weight of $C$ is zero.
As $\hhh^e$-root spaces,  $W'_\chi$ can be decomposed as $W'_\chi=\bigoplus_{\alpha \in \bbz\Phi'_e}(W'_\chi)_\alpha$.
The restricted root space $(W'_\chi)_\alpha$
has a basis given by all the PBW monomials as in \eqref{PBWRWeven} such that
$\sum_{i}(-a_i+b_i)\beta_{\bar 0i}+\sum_{i}(-c_i+d_i)\beta_{\bar 1i}+\sum_{i}(m_i-n_i)\gamma_{\bar 0i}+\sum_{i}(p_i-q_i)\gamma_{\bar 1i}= \alpha$, and $(W'_\chi)_0$ is the zero weight space of $W'_\chi$ spanned all above with $\alpha=0$. Set
$(W'_\chi)_\sharp$ (resp.\ $(W'_\chi)_\flat$)
to be the left (resp.\ right) ideal of
$W'_\chi$ generated by all $(W'_\chi)_\alpha$ for $\alpha \in (\Phi'_e)^+$
(resp.\ $\alpha\in (\Phi'_e)^-$), and denote by $(W'_\chi)_{0,\sharp} := (W'_\chi)_0 \cap (W'_\chi)_\sharp,
(W'_\chi)_{\flat,0} := (W'_\chi)_\flat \cap (W'_\chi)_0$.
We have $(W'_\chi)_{0,\sharp}=(W'_\chi)_{\flat,0}$, then it is a two-sided ideal of  $(W'_\chi)_0$.

Retain the notation $\delta$ as in \eqref{deltarho2}.
%Let  $S_{-\delta}$ be a shift on $W'_\chi$
%by keeping $C'_\theta$ as in Proposition \ref{refpbweven2}(1) invariant and sending $\Theta'_{h_i}$ to $\Theta'_{h_i}-\delta(h_i)$ for $1\leqslant i \leqslant k-1$. Extend it to an automorphism of $(W_0)'_\chi$.
By the similar discussion as in the proof of Proposition \ref{ressur}, we have
\begin{prop}\label{ressur3}
The projection $\pi:U(\ggg)_0\twoheadrightarrow U(\mathfrak{g}_0)$ as in \eqref{pi2} induces
a surjective homomorphism $\pi: (W'_\chi)_0\twoheadrightarrow (W_0)'_\chi$
with  $\ker\pi= (W'_\chi)_{0,\sharp}$. Hence  there exists an algebras isomorphism
\begin{equation*}
\pi:(W'_\chi)_0 / (W'_\chi)_{0,\sharp} \cong (W_0)'_\chi.
\end{equation*}
\end{prop}
\begin{rem}
Comparing Proposition \ref{ressur3} with Proposition \ref{ressur}, one can observe significant difference in the absence of the shift $S_{\epsilon}$. This is because the element $\Theta_{[v_{\frac{\sfr+1}{2}},e]}$ does not exist for the present case. So we need not to consider \eqref{keycom}, which  directly leads to the emergence of $S_{\epsilon}$ in Proposition \ref{ressur}.
\end{rem}

For a $W'_\chi$-module $M$ and
$\alpha\in(\hhh^e)^*$, we define the $\alpha$-weight space
\begin{equation*}\label{shiftby1delta}
M_\alpha:=\{m\in M\:|\: (\Theta_t-\delta(t))(m) = \alpha(t) m\text{ for all } t \in \hhh^e\}.
\end{equation*}
By the same consideration as in \S\ref{5.1.4}, we get the action of $\hhh^e$ on $M_\alpha$ via
\begin{equation*}
  t(m)=\Theta'_{t}(m)=\alpha(t) m
\end{equation*}for all $t \in \hhh^e$, which explains why the additional shift by $-\delta$ in the  definition of the $\alpha$-weight space  of a $W'_\chi$-module is necessary.

Associated with the  $(W_0)'_\chi$-module $V_{\lambda,c}$ as in Proposition \ref{refpbweven2}(4), we can define the highest weight $W'_\chi$-module of type $(\lambda,c)$ as
\begin{equation*}
M_e(\lambda,c):=(W'_\chi / (W'_\chi)_\sharp)\otimes_{(W_0)'_\chi} V_{\lambda,c}.
\end{equation*}
\begin{theorem}\label{Tverma3}
For $\lambda\in(\hhh^e)^*$ and $c\in\bbc$,
let $v_{\lambda,c}$ be a basis for $(W_0)'_\chi$-module $V_{\lambda,c}$ of level $c$ with $\hhh^e$-weight $\lambda$ as in Proposition \ref{refpbweven2}(4).
\begin{itemize}
\item[(1)]
The vectors $\bigg\{\prod_{i=1}^w\Theta_{x_i}^{a_i}\cdot\prod_{i=1}^\ell\Theta_{y_i}^{b_i}\cdot
\prod_{i=1}^{\frac{s}{2}}\Theta_{f_i}^{c_i}\cdot\prod_{i=1}^{\frac{\sfr}{2}}
\Theta_{g_i}^{d_i}(v_\lambda)\:|\:{\bf a}\in\mathbb{Z}_+^w, {\bf b}\in\Lambda_\ell, {\bf c}\in\mathbb{Z}_+^{\frac{s}{2}}, {\bf d}\in\Lambda_{\frac{\sfr}{2}}\bigg\}$ form a basis of $M_e(\lambda,c)$.
\item[(2)] The weight $\lambda$ is the unique maximal weight of $M_e(\lambda,c)$
in the dominance ordering,
$M_e(\lambda,c)$ is generated by the maximal weight space
$M_e(\lambda,c)_\lambda$, and $M_e(\lambda,c)_\lambda \cong V_{\lambda,c}$ as $(W_0)'_\chi$-modules.
\item[(3)] The module $M_e(\lambda,c)$ is a universal highest weight module of
type $(\lambda,c)$, i.e., if $M$ is another highest weight module
generated by a maximal weight space $M_{\mu,c}$
and $f: V_{\lambda,c} \stackrel{\sim}{\rightarrow} M_{\mu,c}$
is a $(W_0)'_\chi$-module isomorphism,
then there is a unique $W'_\chi$-module
homomorphism $\tilde{f}: M_e(\lambda,c)\twoheadrightarrow M$ extending $f$.
\item[(4)] There is a unique maximal proper submodule
$M_e^{\rm max}(\lambda,c)$ in $M_e(\lambda,c)$,
\begin{equation}\label{irrtypeevenM}
L_e(\lambda,c):=M_e(\lambda,c)/M_e^{\rm max}(\lambda,c)
\end{equation}
is an irreducible module type $M$, and $\{L_e(\lambda,c)\:|\: \lambda\in (\hhh^e)^*,c\in\bbc\}$
is a complete set of pairwise inequivalent irreducible highest weight modules over $W'_\chi$.
Moreover, any finite-dimensional simple $W'_\chi$-module is isomorphic to one of the modules $L_e(\lambda,c)$ for $\lambda\in\Lambda_0^+ = \{\lambda \in (\hhh^e)^*\:|\:
\lambda(h_\alpha)\in\bbz_+~\text{for}~\alpha\in(\Phi^+_{e,0})_{\bar0}\}$ and $c\in\bbc$. We further have that $c$ is a rational number in the case when $\ggg$ is a simple Lie algebra except type $A(m)$, or when $\ggg=\mathfrak{psl}(2|2)$, or when $\mathfrak{spo}(2|m)$ with $m$ being even such that $\ggg^e(0)=\mathfrak{so}(m)$, or when $\mathfrak{osp}(4|2m)$ with $\ggg^e(0)=\mathfrak{sl}(2)\oplus\mathfrak{sp}(2m)$, or when $\ggg=\mathfrak{osp}(5|2m)$ with $\ggg^e(0)=\mathfrak{osp}(1|2m)\oplus\mathfrak{sl}(2)$, or when $\ggg=D(2,1;\alpha)$ with $\alpha\in\overline\bbq$, or when $\ggg=F(4)$ with $\ggg^e(0)=\mathfrak{so}(7)$.
\end{itemize}
\end{theorem}

%By the similar discussion as in Proposition \ref{an embedding center}, we get
%\begin{prop}\label{an embedding center 3}
%There exists an embedding $\pi_{-\delta}: Z(W'_\chi)\hookrightarrow %Z((W_0)'_\chi)=(W_0)'_\chi$.
%\end{prop}
We say that a $W'_\chi$-module %(resp. $(W_0)'_\chi$-module)
$V$ is of central character $\psi: Z(W'_\chi)\rightarrow\bbc$ %(resp. $\psi_0:Z((W_0)'_\chi) \rightarrow \bbc$)
if $z(v)=\psi(z)v$ %(resp. $z(v) = \psi_0(v)$)
for all $z\in Z(W'_\chi)$ %(resp. $z \in Z((W_0)'_\chi)$)
and $v\in V$.
%In virtue of Proposition \ref{an embedding center 3},
%For $\lambda\in(\hhh^e)^*$ and $c\in\bbc$ let $\psi^{\lambda,c}: %Z(W'_\chi)\rightarrow\bbc$
%(resp. $\psi_0^{\lambda,c}: Z((W_0)'_\chi) \rightarrow \bbc$)
%be the corresponding central character.
For the highest weight $W'_\chi$-module $M_e(\lambda,c)$ of type $(\lambda,c)\in (\hhh^e)^*\times\bbc$, let $\psi^{\lambda,c}: Z(W'_\chi)\rightarrow\bbc$ be the corresponding central character. By the similar discussion as in  Theorem \ref{numberis finite} and Corollary \ref{highestcompser}, we obtain
\begin{theorem}\label{numberis finite 3}
The number of isomorphism classes of irreducible
highest weight
modules for $W'_\chi$ with prescribed central character $\psi: Z(W'_\chi)\rightarrow\bbc$  is finite, i.e., the set
$\{(\lambda,c) \in (\hhh^e)^*\times\bbc\:|\:\psi^{\lambda,c} = \psi\}$  is finite.
\end{theorem}
\begin{corollary}\label{highestcompser2}
For each $(\lambda,c) \in (\hhh^e)^*\times\bbc$, the highest weight module $M_e(\lambda,c)$ has composition series.
\end{corollary}

Now an analogue of the BGG category $\mathcal{O}$ can also be introduced. Denote by $\mathfrak{q} = \ggg_0\oplus\bigoplus_{\alpha \in (\Phi'_e)^{+}}\ggg_\alpha$, and let $\mathcal{O}(e)=\mathcal{O}(e;\mathfrak{h},\mathfrak{q})$ denote the category of all finitely generated $W'_\chi$-modules $V$, that are semi-simple over $\mathfrak{h}^e$
with finite-dimensional $\mathfrak{h}^e$-weight spaces, such that the set
$\{\lambda \in (\hhh^e)^*\:|\: V_{\lambda} \neq \{0\}\}$ is contained in a finite union of sets of the form $\{\nu\in(\mathfrak{h}^e)^*\:|\: \nu\leqslant\mu\}$ for $\mu \in (\mathfrak{h}^e)^*$.
Then we have
\begin{theorem}\label{BGGcateOeven}
For the category $\mathcal O(e)$, the following statements hold:
\begin{itemize}
\item[(1)] There is a complete set of isomorphism classes of simple objects which is \\$\{L_e(\lambda,c)\:|\: \lambda\in (\hhh^e)^*,c\in\bbc\}$ as in \eqref{irrtypeevenM}.
\item[(2)] The category $\mathcal O(e)$ is Artin. In particular, every object has finite length of composition series.
\item[(3)] The category $\mathcal O(e)$ has a block decomposition as $\mathcal O(e) = \bigoplus_{\psi^{\lambda,c}}  \mathcal O_{\psi^{\lambda,c}}(e)$, where the direct sum is over all central characters $\psi^{\lambda,c}:Z(W'_\chi) \rightarrow \bbc$,
and $\mathcal O_{\psi^{\lambda,c}}(e)$ denotes the Serre subcategory of $\mathcal O(e)$
generated by the irreducible modules
\\$\{L_e(\mu,c)\:|\:\mu \in (\hhh^e)^*~\text{such that}~\psi^{\mu,c}=\psi^{\lambda,c}\}$.
\end{itemize}
\end{theorem}
%Since the element $\Theta_{[v_{\frac{\sfr+1}{2}},e]}$ does not appear in $W'_\chi$ of type even,
%Moreover, it is readily to check that all the discussion from Proposition \ref{an embedding center 2} to Corollary \ref{BGGcateO} also go through in the present case. %, after removing the assumption $\epsilon\neq0$ in \eqref{defofepsilon}.
%In particular, the corresponding BGG category $\mathcal{O}$ can be defined similarly as in \S\ref{5.1.6}, for which we will not repeat here.
\begin{rem}
It is worth mentioning that there are also some other methods to interpret the representation theory of finite $W$-superalgebras, which are only suitable to some basic Lie superalgebras but associated with arbitrary even nilpotents. To be explicitly, given a basic Lie superalgebra $\ggg=\ggg_{\bar0}\oplus\ggg_{\bar1}$ of type I and arbitrary nilpotent element $e\in\ggg_{\bar0}$, denote by $\mathcal{W}$ the corresponding finite $W$-superalgebras.
In \cite{X}, Xiao introduced an extended $W$-superalgebra $\tilde{\mathcal{W}}$ which contains $\mathcal{W}$ as a subalgebra, and then established a bijection between the isomorphism classes of finite-dimensional irreducible $\mathcal{W}$-modules and that of $\tilde{\mathcal{W}}$-modules. With aid of this bijection,  the ``Verma modules" (which are much like the Kac-modules formulated for basic Lie superalgebras)  for $\tilde{\mathcal{W}}$ were also defined there, the finite-dimensional simple $\mathcal{W}$-modules with integral central character were classified. Furthermore, %in virtue of Kac functor and generalized Soergel functor,
an algorithm for computing their characters was given. If $e=e_\theta$ with $-\theta$ being a minimal root, it corresponds to the cases when $\ggg$ is a simple Lie algebra, or when $\mathfrak{sl}(m|n)$ with $m\neq n,m\geqslant2$, $\mathfrak{psl}(m|m)$ with $m\geqslant2$, or when $\mathfrak{spo}(2m|2)$. All these correspond to a subclass of the minimal refined $W$-superalgebra $W_\chi'$ of type even in Table 2, for which we have already studied directly; see the beginning of \S\ref{5.4} for more details.
\end{rem}

\section{Proof of Lemma \ref{left ideal2}}\label{lengthy proof}
This appendix is contributed to the proof of Lemma \ref{left ideal2}.  We mainly follow Premet's strategy on finite $W$-algebras as in \cite[Lemma 7.1]{P3}, with a few modifications. Compared with the non-super situation, one can observe significantly difference for the emergence of the restriction \eqref{c0clambda}.
\subsection{}\label{6.1}
For ${\bf a},{\bf b}\in\mathbb{Z}_+^w$, ${\bf c},{\bf d}\in\Lambda_\ell$, ${\bf
m},{\bf n}\in\mathbb{Z}_+^{\frac{s}{2}}$, ${\bf p},{\bf q}\in\Lambda_{\frac{\sfr-1}{2}},
{\bf t}\in\mathbb{Z}_+^k$, $\iota,\varepsilon\in\Lambda_1$, set
\begin{equation*}
\begin{split}
\Theta({\bf a,b,c,d,m,n,p,q,t},\iota,\varepsilon):=&
\prod_{i=1}^w\Theta_{x_i}^{a_i}\cdot\prod_{i=1}^\ell\Theta_{y_i}^{c_i}\cdot
\prod_{i=1}^{\frac{s}{2}}\Theta_{f_i}^{m_i}\cdot\prod_{i=1}^{\frac{\sfr-1}{2}}
\Theta_{g_i}^{p_i}\cdot\Theta_{v_{\frac{\sfr+1}{2}}}^\iota\cdot \prod_{i=1}^{k-1}
\Theta_{h_i}^{t_i}\cr&\cdot C^{t_k}\cdot\Theta_{[v_{\frac{\sfr+1}{2}},e]}^\varepsilon
\prod_{i=1}^{\frac{s}{2}}\Theta_{f^*_i}^{n_i}\cdot
\prod_{i=1}^{\frac{\sfr-1}{2}}\Theta_{g^*_i}^{q_i}\cdot
\prod_{i=1}^w\Theta_{x^*_i}^{b_i}\cdot
\prod_{i=1}^\ell\Theta_{y^*_i}^{d_i}.
\end{split}
\end{equation*}
By Theorem \ref{PBWQC}, the PBW monomials $\Theta({\bf a,b,c,d,m,n,p,q,t},\iota,\varepsilon)$ form a $\bbc$-basis of $U(\ggg,e)$. Note that
$\text{deg}_e(\Theta({\bf a,b,c,d,m,n,p,q,t},\iota,\varepsilon))=4t_k+3(|{\bf m}|+|{\bf n}|+|{\bf p}|+|{\bf q}|+\varepsilon)
+2(|{\bf a}|+|{\bf b}|+|{\bf c}|+|{\bf d}|)+2\sum_{i=1}^{k-1}t_i+\iota$. Recall that $[v_{\frac{\sfr+1}{2}},e]$ is a root vector corresponding to odd simple root $\frac{\theta}{2}\in\Phi_{e,1}^+$, and
$[\Theta_{v_{\frac{\sfr+1}{2}}},\Theta_v]=[\Theta_{v_{\frac{\sfr+1}{2}}},\Theta_w]
=[\Theta_{v_{\frac{\sfr+1}{2}}},C]=0$
for all $v\in\ggg^e(0)$ and $w\in\ggg^e(1)$ by Theorem \ref{main3}.

\subsection{}\label{6.2}
As $C-c$ is in the center of finite $W$-superalgebra $U(\ggg,e)$, we have
$$\Theta({\bf a,b,c,d,m,n,p,q,t},\iota,\varepsilon)(C-c)\in I_{\lambda,c}.$$ On the other hand, it follows from Theorem \ref{main3} (more precisely, Theorem \ref{maiin1}(1)---(2)) that $\Theta({\bf a,b,c,d,m,n,p,q,t},\iota,\varepsilon)\cdot(\Theta_{h_i}-\lambda(h_i))\in I_{\lambda,c}$ for $1\leqslant i\leqslant k-1$.
Moreover, since $\Theta_{\mathfrak{n}^+(0)}$ is a Lie subalgebra of $\Theta_{\ggg^e(0)}$, Theorem \ref{main3} entails that $\Theta({\bf a,b,c,d,m,n,p,q,t},\iota,\varepsilon)\cdot\Theta_{e_\alpha}\in I_{\lambda,c}$ for all $\alpha\in\Phi^+_{e,0}$.

\subsection{}
It remains to show that $\Theta({\bf a,b,c,d,m,n,p,q,t},\iota,\varepsilon)\cdot\Theta_{H}\in I_{\lambda,c}$ with $H\in\{f^*_1,\cdots,f^*_{\frac{s}{2}}$, $g^*_1,\cdots,g^*_{\frac{\sfr-1}{2}},[v_{\frac{\sfr+1}{2}},e]\}$. To prove this, we will use induction on $\text{deg}_e(\Theta({\bf a,b,c,d,m,n,p,q,t},\iota,\varepsilon))$. Obviously it is true for the case with $\text{deg}_e(\Theta({\bf a,b,c,d,m,n,p,q,t},\iota,\varepsilon))=0$. From now on we assume that $\text{deg}_e(\Theta({\bf a,b,c,d,m,n,p,q,t},\iota,\varepsilon))=N$ and $\text{F}_kU(\ggg,e)\cdot\Theta_{H}\in I_{\lambda,c}$ for all $k<N$.

\subsubsection{}  Note that the span of $f_1^*,\cdots,f_{\frac{s}{2}}^*, g_1^*,\cdots,g_{\frac{\sfr-1}{2}}^*,[v_{\frac{\sfr+1}{2}},e]$ equals $\mathfrak{n}^+(1)$, hence is stable under the adjoint action of $\mathfrak{n}^+(0)$.
As we have already established that $U(\ggg,e)\cdot\Theta_{e_\alpha}\in I_{\lambda,c}$ for all $\alpha\in\Phi^+_{e,0}$, it follows from Theorem \ref{main3} (more precisely, Theorem \ref{maiin1}(2)) that
\begin{equation}\label{00q0}
\Theta({\bf a,b,c,d,m,n,p,q,t},\iota,\varepsilon)\cdot\Theta_{H}\in\Theta({\bf a,0,c,0,m,n,p,q,t},\iota,\varepsilon)\cdot\Theta_{\mathfrak{n}^+(1)}+I_{\lambda,c}.
\end{equation}
Thus we can assume that ${\bf b=d=0}$.
\subsubsection{}\label{Step b}(i)
If $q_j=0$ for all $j\geqslant i$, it is immediate that
\begin{equation}\label{q+ei1}
\Theta({\bf a,0,c,0,m,n,p,q,t},\iota,\varepsilon)\cdot\Theta_{g^*_i}=\Theta({\bf a,0,c,0,m,n,p,q+e}_i{\bf,t},\iota,\varepsilon)\in I_{\lambda,c}.
\end{equation} So we just need to consider the case ${\bf q}=(q_1,\cdots,q_k,0,\cdots,0)$ for some $q_k=1$ and $k\geqslant i$. In virtue of \cite[Theorem 3.7(3)]{ZS5} and our induction assumption we have
\begin{equation}\label{igigk}
\begin{split}
\Theta({\bf a,0,c,0,m,n,p,q,t},\iota,\varepsilon)\cdot\Theta_{g^*_i}\in&\Theta({\bf a,0,c,0,m,n},{\bf p,q-e}_k,{\bf t},\iota,\varepsilon)[\Theta_{g_k^*},\Theta_{g_i^*}]\\
&+(-1)^{q_{i+1}+\cdots+q_{k}}\Theta({\bf a,0,c,0,m,n},{\bf p,q+e}_i,{\bf t},\iota,\varepsilon)\\
&+\text{F}_{N-2}U(\ggg,e)\cdot\Theta_{g^*_k}\\
\subset&\Theta({\bf a,0,c,0,m,n},{\bf p,q-e}_k,{\bf t},\iota,\varepsilon)[\Theta_{g_k^*},\Theta_{g_i^*}]\\
&+(-1)^{q_{i+1}+\cdots+q_{k}}\Theta({\bf a,0,c,0,m,n},{\bf p,q+e}_i,{\bf t},\iota,\varepsilon)\\
&+I_{\lambda,c}.
\end{split}
\end{equation}

Now we consider the first term in the last equation of \eqref{igigk}. Since $\frac{\theta}{2}$ is an odd simple root by  Convention \ref{conventions}, then for $1\leqslant i,k\leqslant\frac{\sfr-1}{2}$, we have $[z_\alpha^*,g_i^*]^{\sharp}$, $[[g_i^*,z_\alpha^*]^{\sharp},[z_\alpha,g_k^*]^{\sharp}]\in\bigcup_{\beta\in\Phi^+_{e,0}}\bbc e_\beta$ for $1\leqslant\alpha\leqslant\frac{s}{2}$ and $s+1\leqslant\alpha\leqslant s+\frac{\sfr+1}{2}$. %, and also $[[g_i^*,z_\alpha]^{\sharp},[z_\alpha^*,g_k^*]^{\sharp}]=(-1)^{|\alpha|}[[g_i^*,z_\alpha^*]^{\sharp},[z_\alpha,g_k^*]^{\sharp}]\in\bigcup_{\beta\in\Phi^+_{e,0}}\bbc e_\beta$ for $\frac{s}{2}+1\leqslant\alpha\leqslant s$ and $s+1+\frac{\sfr+1}{2}\leqslant\alpha\leqslant s+\sfr$.
As
\begin{equation*}
\begin{split}
([g_k^*,g_i^*],f)=&([[e,v_k^*],[e,v_i^*]],f)=([e,v_k^*],[[e,v_i^*],f])\\=&
([e,v_k^*],[e,[v_i^*,f]])+([e,v_k^*],[[e,f],v_i^*])\\=&-([e,v_k^*],v_i^*)=-(e,[v_k^*,v_i^*])=0,
\end{split}
\end{equation*}
 one can conclude from Theorem \ref{main3} and the discussion in Appendix \ref{6.2} that
\begin{equation}\label{gkgi}
[\Theta_{g_k^*},\Theta_{g_i^*}]=-\frac{1}{2}\sum\limits_{\alpha\in S(-1)}(\Theta_{[g_k^*,z_\alpha]^{\sharp}}\Theta_{[z_\alpha^*,g_i^*]^{\sharp}} +\Theta_{[g_i^*,z_\alpha]^{\sharp}}\Theta_{[z_\alpha^*,g_k^*]^{\sharp}})\in\sum_{\alpha\in\Phi^+_{e,0}}
U(\ggg,e)\cdot\Theta_{e_\alpha}\in I_{\lambda,c}.
\end{equation}
In particular, if $i=k$, the right-hand side of \eqref{igigk} equals $\frac{1}{2}\Theta({\bf a,0,c,0,m,n},{\bf p,q-e}_i,{\bf t},\iota,\varepsilon)\cdot[\Theta_{g_i^*},\Theta_{g_i^*}]+I_{\lambda,c}$ by definition, which is contained in $I_{\lambda,c}$. %by our earlier discussion.

Let us consider the second term in the last equation of \eqref{igigk}. For $1\leqslant i\leqslant\frac{\sfr-1}{2}$,
if $q_i+e_i=1$, then
\begin{equation}\label{q+ei}
\Theta({\bf a,0,c,0,m,n},{\bf p,q+e}_i,{\bf t},\iota,\varepsilon)\in I_{\lambda,c}
\end{equation} by definition. If $q_i+e_i=2$, we have $\Theta_{g_i^*}^2=\frac{1}{2}[\Theta_{g_i^*},\Theta_{g_i^*}]=-\frac{1}{2}\sum\limits_{\alpha\in S(-1)}\Theta_{[g_i^*,z_\alpha]^{\sharp}}\Theta_{[z_\alpha^*,g_i^*]^{\sharp}}$. For $1\leqslant i<k\leqslant\frac{\sfr-1}{2}$, since $[z_\alpha^*,g_i^*]^{\sharp}$,
%$[z_\alpha^*,g_i^*]^{\sharp}\in\bigcup_{\beta\in\Phi^+_{e,0}}\bbc e_\beta$,
$[[z_\alpha^*,g_i^*]^{\sharp},g^*_k]$, $[[g_i^*,z_\alpha^*]^{\sharp},[z_\alpha,g_i^*]^{\sharp}]$, $[[[g_i^*,z_\alpha^*]^{\sharp},[z_\alpha,g_i^*]^{\sharp}],g^*_k]\in\bigcup_{\beta\in\Phi^+_{e,1}}\bbc e_\beta$ for $1\leqslant\alpha\leqslant\frac{s}{2}$ and $s+1\leqslant\alpha\leqslant s+\frac{\sfr+1}{2}$, %and also %$[[g_i^*,z_\alpha]^{\sharp},[z_\alpha^*,g_i^*]^{\sharp}]\in\bigcup_{\beta\in\Phi^+_{e,0}}\bbc e_\beta$,
%$[[[g_i^*,z_\alpha]^{\sharp},[z_\alpha^*,g_i^*]^{\sharp}],g^*_k]$ for $\frac{s}{2}+1\leqslant\alpha\leqslant s$ and $s+1+\frac{\sfr+1}{2}\leqslant\alpha\leqslant s+\sfr$,
then it follows from \eqref{gkgi}, Theorem \ref{main3}, the discussion in Appendix \ref{6.2} and our induction assumption that
\begin{equation}\label{n-4}
\begin{split}
\Theta({\bf a,0,c,0,m,n},{\bf p,q+e}_i,{\bf t},\iota,\varepsilon)\in&\sum_{\alpha\in\Phi^+_{e,0}}
U(\ggg,e)\cdot\Theta_{e_\alpha}
+\text{F}_{N-4}U(\ggg,e)\cdot\Theta_{g^*_k}\\
&+\sum_{\substack{1\leqslant\alpha\leqslant\frac{s}{2},\\ s+1\leqslant\alpha\leqslant s+\frac{\sfr+1}{2}}}\text{F}_{N-4}U(\ggg,e)\cdot\Theta_{[[z_\alpha^*,g_i^*]^{\sharp},g^*_k]}\\
&+\sum_{\substack{1\leqslant\alpha\leqslant\frac{s}{2},\\ s+1\leqslant\alpha\leqslant s+\frac{\sfr-1}{2}}}\text{F}_{N-6}U(\ggg,e)\cdot\Theta_{[[[g_i^*,z_\alpha^*]^{\sharp},[z_\alpha,g_i^*]^{\sharp}],g^*_k]}\in I_{\lambda,c}.
\end{split}
\end{equation}

Therefore, \eqref{q+ei1}---\eqref{n-4} %and the discussion in (1)
show that
\begin{equation}\label{onlyg}
\Theta({\bf a,0,c,0,m,n,p,q,t},\iota,\varepsilon)\cdot\Theta_{g^*_i}\in I_{\lambda,c}
\end{equation} for all $1\leqslant i\leqslant\frac{\sfr-1}{2}$.

(ii) By parity consideration we know that $([g_k^*,f_i^*],f)=0$  for all $1\leqslant k\leqslant\frac{\sfr-1}{2}$ and $1\leqslant i\leqslant\frac{s}{2}$, and it is obvious that $[z_\alpha^*,f_i^*]^{\sharp}$, $[z_\alpha^*,g_k^*]^{\sharp}$, $[[g_k^*,z_\alpha^*]^{\sharp},[z_\alpha,f_i^*]^{\sharp}]$, $[[f_i^*,z_\alpha^*]^{\sharp},[z_\alpha,g_k^*]^{\sharp}]\in\bigcup_{\beta\in\Phi^+_{e,0}}\bbc e_\beta$ for $1\leqslant\alpha\leqslant\frac{s}{2}$ and $s+1\leqslant\alpha\leqslant s+\frac{\sfr+1}{2}$. Then Theorem \ref{main3} yields
\begin{equation}\label{gf}
[\Theta_{g_k^*},\Theta_{f_i^*}]=-\frac{1}{2}\sum\limits_{\alpha\in S(-1)}\bigg(\Theta_{[g_k^*,z_\alpha]^{\sharp}}\Theta_{[z_\alpha^*,f_i^*]^{\sharp}} -\Theta_{[f_i^*,z_\alpha]^{\sharp}}\Theta_{[z_\alpha^*,g_k^*]^{\sharp}}\bigg)\in\sum_{\alpha\in\Phi^+_{e,0}}
U(\ggg,e)\cdot\Theta_{e_\alpha}.
\end{equation}As $U(\ggg,e)\cdot\Theta_{e_\alpha}\in I_{\lambda,c}$ for all $\alpha\in\Phi^+_{e,0}$,
if there exists $q_l\neq0$ for $1\leqslant l\leqslant\frac{\sfr-1}{2}$, then
\eqref{onlyg} and \eqref{gf} entail that
\begin{equation}\label{00q1}
\Theta({\bf a,0,c,0,m,n,p,q,t},\iota,\varepsilon)\cdot\Theta_{f^*_i}\in I_{\lambda,c}.%\Theta({\bf a,0,c,0,m,n,p,0,t},\iota,\varepsilon)\cdot\Theta_{\mathfrak{n}^+(1)}+
\end{equation}

(iii) By weight consideration we know that $([g_i^*,[v_{\frac{\sfr+1}{2}},e]],f)=0$ for $1\leqslant i\leqslant\frac{\sfr-1}{2}$. Since $[z_\alpha^*,g_i^*]^{\sharp}$, $[z_\alpha^*,[v_{\frac{\sfr+1}{2}},e]]^{\sharp}$, $[[g_i^*,z_\alpha^*]^{\sharp},[z_\alpha,[v_{\frac{\sfr+1}{2}},e]]^{\sharp}]$, $[[[v_{\frac{\sfr+1}{2}},e],z_\alpha^*]^{\sharp},[z_\alpha,g_i^*]^{\sharp}]
\in\bigcup_{\beta\in\Phi^+_{e,0}}\bbc e_\beta$ for $1\leqslant\alpha\leqslant\frac{s}{2}$ and $s+1\leqslant\alpha\leqslant s+\frac{\sfr+1}{2}$,  applying Theorem \ref{main3} again  we have
\begin{equation}\label{g*ve}
\begin{split}
[\Theta_{g_i^*},\Theta_{[v_{\frac{\sfr+1}{2}},e]}]=&-\frac{1}{2}\sum\limits_{\alpha\in S(-1)}\bigg(\Theta_{[g_i^*,z_\alpha]^{\sharp}}\Theta_{[z_\alpha^*,[v_{\frac{\sfr+1}{2}},e]]^{\sharp}} +\Theta_{[[v_{\frac{\sfr+1}{2}},e],z_\alpha]^{\sharp}}\Theta_{[z_\alpha^*,g_i^*]^{\sharp}}\bigg)\cr\in&\sum_{\alpha\in\Phi^+_{e,0}}
U(\ggg,e)\cdot\Theta_{e_\alpha}\in I_{\lambda,c}.
\end{split}
\end{equation}If there exists $q_l\neq0$ for $1\leqslant l\leqslant\frac{\sfr-1}{2}$, then \eqref{onlyg} and \eqref{g*ve} yield
\begin{equation}\label{00q2}
\Theta({\bf a,0,c,0,m,n,p,q,t},\iota,\varepsilon)\cdot\Theta_{[v_{\frac{\sfr+1}{2}},e]}\in I_{\lambda,c}.%\Theta({\bf a,0,c,0,m,n,p,0,t},\iota,\varepsilon)\cdot\Theta_{\mathfrak{n}^+(1)}+
\end{equation}

In virtue of \eqref{00q0}, \eqref{onlyg}, \eqref{00q1} and \eqref{00q2}, we may further assume that ${\bf b=d=q=0}$.

\subsubsection{} Repeat verbatim the discussions in Appendix \ref{Step b} but substitute $\Theta_{g_i^*},\Theta_{g_j^*}$ with $\Theta_{f_i^*},\Theta_{f_j^*}$, Theorem \ref{main3} shows that
\begin{equation}\label{0001}
\Theta({\bf a,0,c,0,m,n,p,0,t},\iota,\varepsilon)\cdot\Theta_{H}\in I_{\lambda,c}%\Theta({\bf a,0,c,0,m,0,p,0,t},\iota,\varepsilon)\cdot\Theta_{\mathfrak{n}^+(1)}+
\end{equation}
for $H\in\{f^*_1,\cdots,f^*_{\frac{s}{2}},[v_{\frac{\sfr+1}{2}},e]\}$.
Then by \eqref{0001} we may assume that ${\bf b=d=n=q=0}$.

\subsubsection{}\label{Step d}Thanks to  \eqref{0001}, it remains to show that \begin{equation}\label{0000}
\Theta({\bf a,0,c,0,m,0,p,0,t},\iota,\varepsilon)\cdot\Theta_{[v_{\frac{\sfr+1}{2}},e]}\in I_{\lambda,c}.
\end{equation}

(i) First note that
\begin{equation*}
\begin{split}
[[v_{\frac{\sfr+1}{2}},e],[v_{\frac{\sfr+1}{2}},e]]=&[[[v_{\frac{\sfr+1}{2}},e],v_{\frac{\sfr+1}{2}}],e]-
[v_{\frac{\sfr+1}{2}},[[v_{\frac{\sfr+1}{2}},e],e]]\cr
=&[[[v_{\frac{\sfr+1}{2}},e],v_{\frac{\sfr+1}{2}}],e]\cr
=&[[[v_{\frac{\sfr+1}{2}},v_{\frac{\sfr+1}{2}}],e],e]
+[[v_{\frac{\sfr+1}{2}},[e,v_{\frac{\sfr+1}{2}}]],e]\cr
=&[[f,e],e]+[[[e,v_{\frac{\sfr+1}{2}}],v_{\frac{\sfr+1}{2}}],e]\cr
=&-2e-[[[v_{\frac{\sfr+1}{2}},e],v_{\frac{\sfr+1}{2}}],e].
\end{split}
\end{equation*}
As a result, \begin{equation}\label{veve}
[[v_{\frac{\sfr+1}{2}},e],[v_{\frac{\sfr+1}{2}},e]]=[[[v_{\frac{\sfr+1}{2}},e],v_{\frac{\sfr+1}{2}}],e]=-e.
\end{equation}
In virtue of Theorem \ref{main3} (more precisely, Theorem \ref{maiin1}(3)) and \eqref{veve}, we have
\begin{equation}\label{Thetavr+12vr+12}
\begin{split}
\Theta_{[v_{\frac{\sfr+1}{2}},e]}^2=&\frac{1}{2}[\Theta_{[v_{\frac{\sfr+1}{2}},e]},\Theta_{[v_{\frac{\sfr+1}{2}},e]}]=\frac{1}{4}
([[v_{\frac{\sfr+1}{2}},e],[v_{\frac{\sfr+1}{2}},e]],f)(C-\Theta_{\text{Cas}}-c_0)\cr
&-\frac{1}{2}\sum\limits_{\alpha\in S(-1)}\bigg(\Theta_{[[v_{\frac{\sfr+1}{2}},e],z_\alpha]^{\sharp}}\Theta_{[z_\alpha^*,[v_{\frac{\sfr+1}{2}},e]]^{\sharp}}
\bigg)\cr
=&-\frac{1}{4}(C-c_0)+\frac{1}{4}\Theta_{\text{Cas}}-\frac{1}{2}\sum\limits_{\alpha\in S(-1)}\bigg(\Theta_{[[v_{\frac{\sfr+1}{2}},e],z_\alpha]^{\sharp}}\Theta_{[z_\alpha^*,[v_{\frac{\sfr+1}{2}},e]]^{\sharp}}
\bigg).
\end{split}
\end{equation}

Second, by the definition of $\Theta_{\text{Cas}}$ in \eqref{castheta}, Theorem \ref{main3} (more precisely, Theorem \ref{maiin1}(1)) implies that
\begin{equation}\label{explicitcas}
\begin{split}
\Theta_{\text{Cas}}=&\sum_{i=1}^{k-1}\Theta_{h_i}^2+\sum_{i=1}^{w}\Theta_{x_i}\Theta_{x^*_i}+
\sum_{i=1}^{w}\Theta_{x^*_i}\Theta_{x_i}+\sum_{i=1}^{\ell}\Theta_{y_i}\Theta_{y^*_i}-
\sum_{i=1}^{\ell}\Theta_{y^*_i}\Theta_{y_i}\cr
=&\sum_{i=1}^{k-1}\Theta_{h_i}^2+2\sum_{i=1}^{w}\Theta_{x_i}\Theta_{x^*_i}+
\sum_{i=1}^{w}[\Theta_{x^*_i},\Theta_{x_i}]+2\sum_{i=1}^{\ell}\Theta_{y_i}\Theta_{y^*_i}-
\sum_{i=1}^{\ell}[\Theta_{y^*_i},\Theta_{y_i}]\cr
=&\sum_{i=1}^{k-1}\Theta_{h_i}^2+2\sum_{i=1}^{w}\Theta_{x_i}\Theta_{x^*_i}+
\sum_{i=1}^{w}\Theta_{[x^*_i,x_i]}+2\sum_{i=1}^{\ell}\Theta_{y_i}\Theta_{y^*_i}-
\sum_{i=1}^{\ell}\Theta_{[y^*_i,y_i]}.
\end{split}
\end{equation}
Note that $x^*_i,x_i,y^*_j,y_j$ are in the Lie algebra $\ggg^e(0)$ for all $i,j$, then both $[x^*_i,x_i]$ and $[y^*_j,y_j]$ are linear combinations of $h_1,\cdots,h_{k-1}$ by weight consideration. Moreover, if we write $[x^*_i,x_i]=\sum_{j=1}^{k-1}l_jh_j$, then for any $1\leqslant r \leqslant k-1$ we have
\begin{equation*}
l_r=\sum_{j=1}^{k-1}l_j(h_r,h_j)=(h_r,[x^*_i,x_i])=([h_r,x^*_i],x_i)=\beta_{\bar0i}(h_r)(x^*_i,x_i)=\beta_{\bar0i}(h_r),
\end{equation*}
which shows that
\begin{equation}\label{x*x}
[x^*_i,x_i]=\sum_{j=1}^{k-1}\beta_{\bar0i}(h_j)h_j.
\end{equation} As a result,
\begin{equation}\label{thetax*x}
\Theta_{[x^*_i,x_i]}=\sum_{j=1}^{k-1}\beta_{\bar0i}(h_j)\Theta_{h_j}.
\end{equation}
By the same discussion as above, we can also obtain
\begin{equation}\label{y*y}
[y^*_i,y_i]=\sum_{j=1}^{k-1}\beta_{\bar1i}(h_j)h_j,
\end{equation} and
\begin{equation}\label{thetay*y}
\Theta_{[y^*_i,y_i]}=\sum_{j=1}^{k-1}\beta_{\bar1i}(h_j)\Theta_{h_j}.
\end{equation}
%Correspondingly, both $\Theta_{[x^*_i,x_i]}$ and $\Theta_{[y^*_j,y_j]}$ are linear combinations of $\Theta_{h_1},\cdots,\Theta_{h_{k-1}}$ as in Theorem \ref{ge}(1).
Taking \eqref{thetax*x} and \eqref{thetay*y} into consideration, \eqref{explicitcas} shows that
\begin{equation}\label{explicitcas2}
\Theta_{\text{Cas}}
=\sum_{i=1}^{k-1}\Theta_{h_i}^2+2\sum_{i=1}^{w}\Theta_{x_i}\Theta_{x^*_i}+
\sum_{i=1}^{k-1}\sum_{j=1}^{w}\beta_{\bar0j}(h_i)\Theta_{h_i}
-\sum_{i=1}^{k-1}\sum_{j=1}^{\ell}\beta_{\bar1j}(h_i)\Theta_{h_i}
+2\sum_{i=1}^{\ell}\Theta_{y_i}\Theta_{y^*_i}.
\end{equation}

For the last term in the final equation of \eqref{Thetavr+12vr+12}, as
%\begin{equation}\label{vuexchange}
%[[[v_{\frac{\sfr+1}{2}},e],u_i]^{\sharp},[u_i^*,
%[v_{\frac{\sfr+1}{2}},e]]^{\sharp}]
%=-[[[v_{\frac{\sfr+1}{2}},e],u_{s+1-i}^*]^{\sharp},[u_{s+1-i},
%[v_{\frac{\sfr+1}{2}},e]]^{\sharp}]
%\end{equation}for $\frac{s}{2}+1\leqslant i\leqslant s$.
%Since
\begin{equation*}
[[v_{\frac{\sfr+1}{2}},e],v_{\frac{\sfr+1}{2}}]
=[v_{\frac{\sfr+1}{2}},[e,v_{\frac{\sfr+1}{2}}]]+
[[v_{\frac{\sfr+1}{2}},v_{\frac{\sfr+1}{2}}],e]
=-[[v_{\frac{\sfr+1}{2}},e],v_{\frac{\sfr+1}{2}}]-h,
\end{equation*}
we have
\begin{equation}\label{vev}
[[v_{\frac{\sfr+1}{2}},e],v_{\frac{\sfr+1}{2}}]=-\frac{h}{2},
\end{equation}then
\begin{equation}\label{hes}
\Theta_{[[v_{\frac{\sfr+1}{2}},e],v_{\frac{\sfr+1}{2}}]^{\sharp}}\Theta_{[v_{\frac{\sfr+1}{2}}^*,[v_{\frac{\sfr+1}{2}},e]]^{\sharp}}
=\Theta_{(-\frac{h}{2})^{\sharp}}\Theta_{(-\frac{h}{2})^{\sharp}}=0.
\end{equation}
Since $[[v_{\frac{\sfr+1}{2}},e],u_i^*]^{\sharp}, [u_i,
[v_{\frac{\sfr+1}{2}},e]]^{\sharp}\in\ggg^e(0)$ for all $1\leqslant i\leqslant\frac{s}{2}$, then $[[[v_{\frac{\sfr+1}{2}},e],u_i^*]^{\sharp},[u_i,
[v_{\frac{\sfr+1}{2}},e]]^{\sharp}]$ is a linear combination of $h_1,\cdots,h_{k-1}$ by weight consideration. For any $t\in\mathfrak{h}^e$, by definition we have $\theta(t)=0$.
Taking \eqref{vev} into consideration, if we write
\begin{equation*}
[[[v_{\frac{\sfr+1}{2}},e],u_i^*]^{\sharp},[u_i,
[v_{\frac{\sfr+1}{2}},e]]^{\sharp}]=\sum_{j=1}^{k-1}l_jh_j,
\end{equation*} then for any $1\leqslant r \leqslant k-1$ and $1\leqslant i\leqslant\frac{s}{2}$, we have
\begin{equation}\label{hveuuve}
\begin{split}
l_r=&\sum_{j=1}^{k-1}l_j(h_r,h_j)=(h_r,[[[v_{\frac{\sfr+1}{2}},e],u_i^*]^{\sharp},[u_i,
[v_{\frac{\sfr+1}{2}},e]]^{\sharp}])
=(h_r,[[[v_{\frac{\sfr+1}{2}},e],u_i^*],[u_i,
[v_{\frac{\sfr+1}{2}},e]]])\\
=&([h_r,[[v_{\frac{\sfr+1}{2}},e],u_i^*]],[u_i,
[v_{\frac{\sfr+1}{2}},e]])
=\bigg(\frac{\theta}{2}+\gamma^*_{\bar0i}\bigg)(h_r)([[v_{\frac{\sfr+1}{2}},e],u_i^*],[u_i,
[v_{\frac{\sfr+1}{2}},e]])\\
=&-\gamma_{\bar0i}(h_r)([[[[v_{\frac{\sfr+1}{2}},e],u_i^*],u_i],
v_{\frac{\sfr+1}{2}}],e)
=-\gamma_{\bar0i}(h_r)([[[[v_{\frac{\sfr+1}{2}},e],
v_{\frac{\sfr+1}{2}}],u_i^*],u_i],e)\\
=&\gamma_{\bar0i}(h_r)([[\frac{h}{2},u_i^*],u_i],e)
=-\frac{\gamma_{\bar0i}}{2}(h_r)([u_i^*,u_i],e)=
-\frac{\gamma_{\bar0i}}{2}(h_r).
\end{split}
\end{equation}
%
%\begin{equation}\label{veuuve}
%[[[v_{\frac{\sfr+1}{2}},e],u_i]^{\sharp},[u_i^*,
%[v_{\frac{\sfr+1}{2}},e]]^{\sharp}]=\frac{1}{2}\sum_{j=1}^{k-1}\gamma_{\bar0s+1-i}(h_j)h_j.
%\end{equation}
%for $\frac{s}{2}+1\leqslant i\leqslant s$. Then for $\frac{s}{2}+1\leqslant i\leqslant s$, it follows from \eqref{vuexchange}, \eqref{veuuve} and that
For $\frac{s}{2}+1\leqslant i\leqslant s$, it follows from \eqref{hveuuve} and Theorem \ref{main3} (more precisely, Theorem \ref{maiin1}(1)) that
\begin{equation}\label{thetaveuuve}
\begin{split}
\Theta_{[[v_{\frac{\sfr+1}{2}},e],u_i]^{\sharp}}\Theta_{[u_i^*,
[v_{\frac{\sfr+1}{2}},e]]^{\sharp}}
=&-\Theta_{[[v_{\frac{\sfr+1}{2}},e],u_{s+1-i}^*]^{\sharp}}\Theta_{[u_{s+1-i},
[v_{\frac{\sfr+1}{2}},e]]^{\sharp}}\cr
=&\Theta_{[u_{s+1-i},
[v_{\frac{\sfr+1}{2}},e]]^{\sharp}}\Theta_{[[v_{\frac{\sfr+1}{2}},e],u_{s+1-i}^*]^{\sharp}}-\Theta_{[[[v_{\frac{\sfr+1}{2}},e],u_{s+1-i}^*]^{\sharp},[u_{s+1-i},
[v_{\frac{\sfr+1}{2}},e]]^{\sharp}]}\cr
=&\Theta_{[u_{s+1-i},
[v_{\frac{\sfr+1}{2}},e]]^{\sharp}}\Theta_{[[v_{\frac{\sfr+1}{2}},e],u_{s+1-i}^*]^{\sharp}}+
\frac{1}{2}\sum_{j=1}^{k-1}\gamma_{\bar0s+1-i}(h_j)\Theta_{h_j}\cr
=&\Theta_{[[v_{\frac{\sfr+1}{2}},e],
u_{s+1-i}]^{\sharp}}\Theta_{[u_{s+1-i}^*,[v_{\frac{\sfr+1}{2}},e]]^{\sharp}}+
\frac{1}{2}\sum_{j=1}^{k-1}\gamma_{\bar0s+1-i}(h_j)\Theta_{h_j}.
\end{split}
\end{equation}

By the same discussion as in \eqref{thetaveuuve}, we have
\begin{equation}\label{thetavevvve}
\Theta_{[[v_{\frac{\sfr+1}{2}},e],v_i]^{\sharp}}\Theta_{[v_i^*,
[v_{\frac{\sfr+1}{2}},e]]^{\sharp}}%&
%=\Theta_{[[v_{\frac{\sfr+1}{2}},e],v_{\sfr+1-i}^*]^{\sharp}}\Theta_{[v_{\sfr+1-i},
%[v_{\frac{\sfr+1}{2}},e]]^{\sharp}}\cr
%&=\Theta_{[v_{\sfr+1-i},
%[v_{\frac{\sfr+1}{2}},e]]^{\sharp}}\Theta_{[[v_{\frac{\sfr+1}{2}},e],v_{\sfr+1-i}^*]^{\sharp}}
%+\Theta_{[[[v_{\frac{\sfr+1}{2}},e],v_{\sfr+1-i}^*]^{\sharp},[v_{\sfr+1-i},
%[v_{\frac{\sfr+1}{2}},e]]^{\sharp}]}\cr
=\Theta_{[[v_{\frac{\sfr+1}{2}},e],v_{\sfr+1-i}]^{\sharp}}\Theta_{[v_{\sfr+1-i}^*,[v_{\frac{\sfr+1}{2}},e]]^{\sharp}}
-\frac{1}{2}\sum_{j=1}^{k-1}\gamma_{\bar1\sfr+1-i}(h_j)\Theta_{h_j}
%
%
%[[[v_{\frac{\sfr+1}{2}},e],v_i]^{\sharp},[v_i^*,
%[v_{\frac{\sfr+1}{2}},e]]^{\sharp}]=-\frac{1}{2}\sum_{j=1}^{k-1}\gamma_{\bar0\sfr+1-i}(h_j)h_j.
\end{equation}
for $\frac{\sfr+3}{2}+1\leqslant i\leqslant\sfr$.

Now combining \eqref{Thetavr+12vr+12},
\eqref{explicitcas2}, \eqref{hes}, \eqref{thetaveuuve} with \eqref{thetavevvve}, we obtain
\begin{equation}\label{keycommuta}
\begin{split}
\Theta_{[v_{\frac{\sfr+1}{2}},e]}^2
=&-\frac{1}{4}C+\frac{1}{4}c_0+\frac{1}{4}\sum_{i=1}^{k-1}\Theta_{h_i}^2+\frac{1}{2}\sum_{i=1}^{w}\Theta_{x_i}\Theta_{x^*_i}+
\frac{1}{4}\sum_{i=1}^{k-1}\sum_{j=1}^{w}\beta_{\bar0j}(h_i)\Theta_{h_i}\\
&-\frac{1}{4}\sum_{i=1}^{k-1}\sum_{j=1}^{\ell}\beta_{\bar1j}(h_i)\Theta_{h_i}+\frac{1}{2}\sum_{i=1}^{\ell}\Theta_{y_i}\Theta_{y^*_i}-\sum_{i=1}^{\frac{s}{2}}\Theta_{[[v_{\frac{\sfr+1}{2}},e],
u_{i}]^{\sharp}}\Theta_{[u_{i}^*,[v_{\frac{\sfr+1}{2}},e]]^{\sharp}}\\
&-\frac{1}{4}\sum_{i=1}^{k-1}\sum_{j=1}^{\frac{s}{2}}\gamma_{\bar0j}(h_i)\Theta_{h_i}-\sum_{i=1}^{\frac{\sfr-1}{2}}\Theta_{[[v_{\frac{\sfr+1}{2}},e],v_{i}]^{\sharp}}
\Theta_{[v_{i}^*,[v_{\frac{\sfr+1}{2}},e]]^{\sharp}}+\frac{1}{4}\sum_{i=1}^{k-1}\sum_{j=1}^{\frac{\sfr-1}{2}}\gamma_{\bar1j}(h_i)\Theta_{h_i}
\end{split}
\end{equation}

(ii) Now we introduce the proof of \eqref{0000}. If $\varepsilon=0$, then \eqref{0000} follows by definition. Now we will consider the case with $\varepsilon=1$.
Since $C-c$ is in the center of $U(\ggg,e)$,
and $[v_{\frac{\sfr+1}{2}}^*,
[v_{\frac{\sfr+1}{2}},e]]^{\sharp}=[v_{\frac{\sfr+1}{2}},
[v_{\frac{\sfr+1}{2}},e]]^{\sharp}=(-\frac{h}{2})^{\sharp}=0$ by \eqref{vev}, then
by \eqref{keycommuta} we have
\begin{equation}\label{1vr+12fir}
\begin{split}
&\Theta({\bf a,0,c,0,m,0,p,0,t},\iota,1)\cdot\Theta_{[v_{\frac{\sfr+1}{2}},e]}
=\Theta({\bf a,0,c,0,m,0,p,0,t},\iota,0)\cdot\Theta_{[v_{\frac{\sfr+1}{2}},e]}^2\cr
=&-\frac{1}{4}\Theta({\bf a,0,c,0,m,0,p,0,t},\iota,0)(C-c)-\frac{1}{4}(c-c_0)\Theta({\bf a,0,c,0,m,0,p,0,t},\iota,0)\cr
&+\frac{1}{4}\bigg(\sum_{i=1}^{k-1}\Theta({\bf a,0,c,0,m,0,p,0,t},\iota,0)\cdot(\Theta_{h_i}-\lambda(h_i))^2\cr
&+2\sum_{i=1}^{k-1}\lambda(h_i)\Theta({\bf a,0,c,0,m,0,p,0,t},\iota,0)\cdot(\Theta_{h_i}-\lambda(h_i))\cr
&+\sum_{i=1}^{k-1}(\lambda(h_i))^2\Theta({\bf a,0,c,0,m,0,p,0,t},\iota,0)
+2\sum_{i=1}^{w}\Theta({\bf a,0,c,0,m,0,p,0,t},\iota,0)\cdot\Theta_{x_i}\Theta_{x^*_i}\cr
&+\sum_{i=1}^{k-1}\sum_{j=1}^{w}\beta_{\bar0j}(h_i)\Theta({\bf a,0,c,0,m,0,p,0,t},\iota,0)\cdot(\Theta_{h_i}-\lambda(h_i))\cr
&+\sum_{i=1}^{k-1}\sum_{j=1}^{w}\lambda(h_i)\beta_{\bar0j}(h_i)\Theta({\bf a,0,c,0,m,0,p,0,t},\iota,0)
+2\sum_{i=1}^{\ell}\Theta({\bf a,0,c,0,m,0,p,0,t},\iota,0)\cr
&\cdot\Theta_{y_i}\Theta_{y^*_i}-\sum_{i=1}^{k-1}\sum_{j=1}^{\ell}\beta_{\bar1j}(h_i)\Theta({\bf a,0,c,0,m,0,p,0,t},\iota,0)\cdot(\Theta_{h_i}-\lambda(h_i))\cr
&-\sum_{i=1}^{k-1}\sum_{j=1}^{\ell}\lambda(h_i)\beta_{\bar1j}(h_i)\Theta({\bf a,0,c,0,m,0,p,0,t},\iota,0)\bigg)+\Theta({\bf a,0,c,0,m,0,p,0,t},\iota,0)\cr
&\cdot\bigg(-\sum_{i=1}^{\frac{s}{2}}\Theta_{[[v_{\frac{\sfr+1}{2}},e],
u_{i}]^{\sharp}}\Theta_{[u_{i}^*,[v_{\frac{\sfr+1}{2}},e]]^{\sharp}}-\frac{1}{4}\sum_{i=1}^{k-1}\sum_{j=1}^{\frac{s}{2}}\gamma_{\bar0j}(h_i)(\Theta_{h_i}-\lambda(h_i))
-\frac{1}{4}\sum_{i=1}^{k-1}\sum_{j=1}^{\frac{s}{2}}\gamma_{\bar0j}(h_i)\lambda(h_i)\\
&-\sum_{i=1}^{\frac{\sfr-1}{2}}\Theta_{[[v_{\frac{\sfr+1}{2}},e],v_{i}]^{\sharp}}
\Theta_{[v_{i}^*,[v_{\frac{\sfr+1}{2}},e]]^{\sharp}}+\frac{1}{4}\sum_{i=1}^{k-1}
\sum_{j=1}^{\frac{\sfr-1}{2}}\gamma_{\bar1j}(h_i)(\Theta_{h_i}-\lambda(h_i))+
\frac{1}{4}\sum_{i=1}^{k-1}\sum_{j=1}^{\frac{\sfr-1}{2}}\gamma_{\bar1j}(h_i)\lambda(h_i)\bigg)
\end{split}
\end{equation}
\begin{equation}\label{1vr+12}
\begin{split}
=&-\frac{1}{4}\Theta({\bf a,0,c,0,m,0,p,0,t+e}_k,\iota,0)+\frac{1}{4}\Bigg(\sum_{i=1}^{k-1}\Theta({\bf a,0,c,0,m,0,p,0,t+}2{\bf e}_i,\iota,0)\cr
&+2\sum_{i=1}^{w}\Theta({\bf a,0,c,0,m,0,p,0,t},\iota,0)\cdot\Theta_{x_i}\Theta_{x^*_i}+2\sum_{i=1}^{\ell}\Theta({\bf a,0,c,0,m,0,p,0,t},\iota,0)\cdot\Theta_{y_i}\Theta_{y^*_i}\cr
&+\bigg(\sum_{i=1}^{k-1}\Big(2\lambda
+\sum_{j=1}^{w}\beta_{\bar0j}
-\sum_{j=1}^{\ell}\beta_{\bar1j}-\sum_{j=1}^{\frac{s}{2}}\gamma_{\bar0j}+\sum_{j=1}^{\frac{\sfr-1}{2}}\gamma_{\bar1j}\Big)(h_i)\bigg)\Theta({\bf a,0,c,0,m,0,p,0,t+e}_i,\iota,0)\cr
&+\bigg(\sum_{i=1}^{k-1}\Big(\lambda^2+\lambda
\big(\sum_{j=1}^{w}\beta_{\bar0j}
-\sum_{j=1}^{\ell}\beta_{\bar1j}-\sum_{j=1}^{\frac{s}{2}}\gamma_{\bar0j}+\sum_{j=1}^{\frac{\sfr-1}{2}}\gamma_{\bar1j}\big)\Big)(h_i)-c+c_0\bigg)\Bigg)\\
&\cdot\Theta({\bf a,0,c,0,m,0,p,0,t},\iota,0)
-\Theta({\bf a,0,c,0,m,0,p,0,t},\iota,0)\cdot\bigg(\sum_{i=1}^{\frac{s}{2}}\Theta_{[[v_{\frac{\sfr+1}{2}},e],u_i]^{\sharp}}\Theta_{[u_i^*,
[v_{\frac{\sfr+1}{2}},e]]^{\sharp}}\\
&+\sum_{i=1}^{\frac{\sfr-1}{2}}\Theta_{[[v_{\frac{\sfr+1}{2}},e],v_i]^{\sharp}}\Theta_{[v_i^*,
[v_{\frac{\sfr+1}{2}},e]]^{\sharp}}\bigg)
\end{split}
\end{equation}

Now we discuss the terms in \eqref{1vr+12}. By definition we have
$\Theta({\bf a,0,c,0,m,0,p,0,t+e}_k,\iota,0) \in I_{\lambda,c}$, and also
$\Theta({\bf a,0,c,0,m,0,p,0,t+e}_i,\iota,0),\Theta({\bf a,0,c,0,m,0,p,0,t+}2{\bf e}_i,\iota,0)\in I_{\lambda,c}$
for $1\leqslant i\leqslant k-1$. For all $i,j$,  by definition we have $x^*_i\in\mathfrak{n}^+_{\bar0}(0)$ and $y^*_j\in\mathfrak{n}^+_{\bar1}(0)$ respectively, thus $x^*_i, y^*_j$ are all in the span of $e_\alpha$ with $\alpha\in\Phi_{e,0}^+$. Therefore, $\Theta({\bf a,0,c,0,m,0,p,0,t},\iota,0)\cdot\Theta_{x_i}\Theta_{x^*_i}$ and $\Theta({\bf a,0,c,0,m,0,p,0,t},\iota,0)\cdot\Theta_{y_i}\Theta_{y^*_i}$ are in $\sum_{\alpha\in\Phi^+_{e,0}}U(\ggg,e)\cdot\Theta_{e_\alpha}$ of $I_{\lambda,c}$.
%Therefore,
%\begin{equation*}
%\Theta({\bf a,0,c,0,m,0,p,0,t},0)\cdot\Theta_{[x^*_i,x_i]},\Theta({\bf a,0,c,0,m,0,p,0,t},0)\cdot\Theta_{[y^*_i,y_i]}\in I_{\lambda,c}.
%\end{equation*}
%Therefore,
%\begin{equation*}
%\Theta({\bf a,0,c,0,m,0,p,0,t},0)\cdot\Theta_{[x^*_i,x_i]},\Theta({\bf a,0,c,0,m,0,p,0,t},0)\cdot\Theta_{[y^*_i,y_i]}\in I_{\lambda,c}.
%\end{equation*}
Moreover, for $1\leqslant i\leqslant\frac{s}{2}$ and
$1\leqslant j\leqslant\frac{\sfr-1}{2}$, since $[u_i^*,[v_{\frac{\sfr+1}{2}},e]]^{\sharp},[v_j^*,[v_{\frac{\sfr+1}{2}},e]]^{\sharp}\in\bigcup_{\beta\in\Phi^+_{e,0}}\bbc e_\beta$, then both
$\Theta({\bf a,0,c,0,m,0,p,0,t},\iota,0)\cdot\Theta_{[[v_{\frac{\sfr+1}{2}},e],u_i]^{\sharp}}\Theta_{[u_i^*,[v_{\frac{\sfr+1}{2}},e]]^{\sharp}}$ and
$\Theta({\bf a,0,c,0,m,0,p,0,t},\iota,0)\cdot\Theta_{[[v_{\frac{\sfr+1}{2}},e],v_j]^{\sharp}}\Theta_{[v_j^*,[v_{\frac{\sfr+1}{2}},e]]^{\sharp}}$ are also in $\sum_{\alpha\in\Phi^+_{e,0}}U(\ggg,e)\cdot\Theta_{e_\alpha}$ of $ I_{\lambda,c}$.
Apart from all the terms mentioned above, what remains in \eqref{1vr+12} are
\begin{equation}\label{remainder}
\frac{1}{4}\bigg(\sum_{i=1}^{k-1}\Big(\lambda^2+\lambda\cdot
\big(\sum_{j=1}^{w}\beta_{\bar0j}
-\sum_{j=1}^{\ell}\beta_{\bar1j}-\sum_{j=1}^{\frac{s}{2}}\gamma_{\bar0j}+\sum_{j=1}^{\frac{\sfr-1}{2}}\gamma_{\bar1j}\big)\Big)(h_i)-c+c_0\bigg)\Theta({\bf a,0,c,0,m,0,p,0,t},0).
\end{equation}
Thanks to our assumption in \eqref{c0clambda}, \eqref{remainder} must be zero.

Taking all above into consideration, we finally obtain \eqref{0000}.

\end{appendix}

\end{document}